\documentclass[reqno,11pt,a4paper]{amsart}

\usepackage{amsmath,amsthm,amssymb,mathtools}
\usepackage{multicol,multirow,hhline,comment,graphicx,color,url}
\usepackage{mathrsfs}	
\usepackage[margin=1in]{geometry}
\usepackage{accents} 
\usepackage[foot]{amsaddr}
\usepackage{hyperref}
\usepackage{cellspace}		
\usepackage{enumitem} 
\setlist[enumerate,1]{leftmargin=1cm}

\theoremstyle{plain}

\newtheorem{theorem}{Theorem}[section]

\newtheorem{proposition}[theorem]{Proposition}
\newtheorem{corollary}[theorem]{Corollary}
\newtheorem{lemma}[theorem]{Lemma}
\newtheorem{conjecture}{Conjecture}

\theoremstyle{definition}
\newtheorem{definition}[theorem]{Definition}
\newtheorem{notation}[theorem]{Definition}

\theoremstyle{remark}
\newtheorem{remark}[theorem]{Remark}

\newcommand{\gen}{\mathcal{G}}
\newcommand{\abs}[1]{\left\lvert #1 \right\rvert}
\newcommand{\infsimp}{{\nabla}_{\infty}}
\newcommand{\infsimpcl}{\overline{\nabla}_{\infty}}
\newcommand{\ualgebra}{\mathfrak{F}}
\newcommand{\petrovgen}{\mathcal{B}}

\newcommand{\ltwo}{\mathbf{L}^2}
\newcommand{\difftwo}{\widetilde{\Delta}}
\newcommand{\iprod}[1]{\left\langle #1 \right\rangle}
\newcommand{\norm}[1]{\left\lVert #1 \right\rVert}
\newcommand{\ranked}{\textsc{ranked}}
\newcommand{\PV}{\mathbf{P}}
\newcommand{\Prm}{\mathrm{P}}
\newcommand{\procvgen}{\mathcal{A}}

\makeatletter
 \DeclareRobustCommand{\checkarg}{\@ifnextchar[{\@witharg}{}}
 \DeclareRobustCommand{\@witharg}[1][]{\ensuremath{\left(#1\right)}}
 
 \DeclareRobustCommand{\scaleGen}[1]{\@ifnextchar[{\@scalewithargs{#1}}{\odot^{}_{#1}}}
 \def\@scalewithargs#1[#2][#3]{#2 \odot^{}_{#1} #3}
\makeatother

\def\IPspace{\mathcal{I}}
\def\dI{d_{\IPspace}}
\def\HIPspace{\IPspace_H}
\def\dH{d_H}
\def\dJ{d_{\cJ}}

\def\cJ{\mathcal{J}}
\def\dJ{d_{\cJ}}

\def\fs{\mathbf{s}}

\def\Exc{\mathcal{E}}

\def\SExc{\Sigma(\Exc)}

\def\cExc{\SExc}
\def\mBxc{\nu_{\textnormal{BES}}}	

\def\cC{\mathcal{C}}	
\def\ScS{\Sigma(\cC)}
\def\cCRI{\cC([0,\infty),\IPspace)}
\def\cCRIi{\cC([0,\infty),\IPspace_1)}

\def\cD{\mathcal{D}}
\def\ScD{\Sigma(\cD)}
\def\DS{\cD_{\textnormal{stb}}}

\def\SDS{\Sigma(\DS)}

\def\DSxc{\cD_{\textnormal{exc}}}
\def\SDSxc{\Sigma(\DSxc)}

\def\ScDSxc{\SDSxc}

\def\mSxc{\nu_{\textnormal{stb}}}	

\def\H{\mathcal{N}^{\textnormal{\,sp}}}
\def\SH{\Sigma(\H)}
\def\cH{\Sigma(\H)}
\def\Hfin{\H_{\textnormal{fin}}}
\def\SHfin{\Sigma(\Hfin)}
\def\cHfin{\Sigma(\Hfin)}
\def\Hxc#1{\H_{#1\textnormal{cld}}}
\def\SHxc#1{\Sigma\left(\Hxc{#1}\right)}
\def\cHxc#1{\SHxc{#1}}
\def\mClade{\nu_{\textnormal{cld}}}	

\def\Hs{\mathcal{N}^{\textnormal{\,sp},*}}

\def\Hfins{\Hs_{\textnormal{fin}}}
\def\SHfins{\Sigma(\Hfins)}

\def\cN{\mathcal{N}}

\def\cNRHf{\mathcal{N}\big([0,\infty)\times\Hfin\big)}

\def\cNRE{\mathcal{N}\big([0,\infty)\times\Exc\big)}
\def\ScNRE{\Sigma\big(\cNRE\big)}
\def\cNS{\mathcal{N}(\cS)}

\def\cS{\mathcal{S}}
\def\ScS{\Sigma(\mathcal{S})}
\def\cT{\mathcal{T}}
\def\ScT{\Sigma(\mathcal{T})}

\def\len{\textnormal{len}}			
\def\life{\zeta}					
\def\dis{\textnormal{dis}}			
\def\IPmag#1{\left\|\vphantom{I}#1\right\|}		

\def\skewer{\textsc{skewer}}		
\def\skewerP{\widebar{\skewer}}		

\def\cutoffL#1#2{\textsc{cutoff}^{\leq #1}_{#2}}
\def\cutoffG#1#2{\textsc{cutoff}^{\geq #1}_{#2}}
\def\cutoffLB#1#2{\textsc{cutoff}^{\leq #1}_{#2}}
\def\cutoffGB#1#2{\textsc{cutoff}^{\geq #1}_{#2}}

\def\Dirac#1{\delta\left( #1 \right)}
\def\DiracBig#1{\delta\big( #1 \big)}	

\def\reverse{\mathcal{R}}			
\def\reverseexc{\reverse_{\textnormal{BES}}}		
\def\reverseincr{\reverse_{\textnormal{stb}}}		
\def\reverseH{\reverse_{\textnormal{cld}}}			

\def\scale{\scaleGen{}}	
\def\scaleB{\scaleGen{\textnormal{BES}}}		
\def\scaleS{\scaleGen{\textnormal{stb}}}		
\def\scaleH{\scaleGen{\textnormal{cld}}}		

\def\scaleI{\scaleGen{\textnormal{IP}}}	

\def\ShiftRestrict#1#2{#1\big|^{\from}_{#2}} 
\def\shiftrestrict#1#2{#1|^{\from}_{#2}}

\def\Restrict#1#2{#1\big|_{#2}}
\def\restrict#1#2{#1|_{#2}}

\def\Concat{ \mathop{ \raisebox{-2pt}{\Huge$\star$} } }
\def\ConcatIL{ \mbox{\huge $\star$} }
\def\concat{\star}

\def\bN{\mathbf{N}}			
\def\bF{\mathbf{F}}			
\def\bG{\mathbf{G}}			
\def\bX{\mathbf{X}}			
\newcommand{\IPLT}{\mathscr{D}}
\def\bM{\mathbf{M}}			
\def\bff{\mathbf{f}}		

\newcommand{\td}[1]{\widetilde{#1}}
\def\tdN{\widetilde{\mathbf{N}}}
\def\tdF{\widetilde{\mathbf{F}}}
\def\tdX{\widetilde{\bX}}

\def\tdl{\widetilde{\ell}}

\newcommand{\wh}[1]{\widehat{#1}}
\def\whN{\widehat{\mathbf{N}}}
\def\whF{\widehat{\mathbf{F}}}
\def\whX{\widehat{\bX}}

\newcommand{\ol}[1]{\widebar{#1}}
\def\olN{\widebar{\mathbf{N}}}
\def\olF{\widebar{\mathbf{F}}}
\def\olX{\widebar{\bX}}

\def\oll{\widebar{\ell}}

\def\BR{\mathbb{R}}				
\def\BN{\mathbb{N}}				
\def\BQ{\mathbb{Q}}				
\def\Leb{\textnormal{Leb}}		
\def\e{\epsilon}
\def\d{\delta}
\def\to{\rightarrow}
\def\downto{\downarrow}
\def\upto{\uparrow}
\def\from{\leftarrow}

\def\cf{\mathbf{1}}				
\def\Pr{\mathbf{P}}				
\def\BPr{\mathbb{P}}			
\def\bQ{\mathbf{Q}}				
\def\BQ{\mathbb{Q}}				
\def\EV{\mathbf{E}}				
\def\cF{\mathcal{F}}			
\def\cG{\mathcal{G}}			
\def\cH{\mathcal{H}}			
\def\indep{ \mathop{\perp\hspace{-5.5pt}\perp} }

\def\cA{\mathcal{A}}	

\def\distribfont#1{\texttt{\upshape #1}}

\def\ExpDist{\distribfont{Exponential}\checkarg}
\def\GammaDist{\distribfont{Gamma}\checkarg}
\def\InvGammaDist{\distribfont{InverseGamma}\checkarg}
\def\GeomDist{\distribfont{Geometric}\checkarg}
\def\BetaDist{\distribfont{Beta}\checkarg}

\def\PoiDir{\distribfont{PD}\checkarg}

\def\PRM{\distribfont{PRM}\checkarg}

\def\Stable{\distribfont{Stable}\checkarg}
\def\BESQ{\distribfont{BESQ}\checkarg}
\def\PDIP{\distribfont{PDIP}\checkarg}
\def\EKP{\distribfont{EKP}\checkarg}
\def\WF{\distribfont{WF}\checkarg}
\DeclareRobustCommand{\CRP}{\texttt{CRP}\checkarg}
\def\CRPAT{\CRP[\alpha,\theta]}

\def\cadlag{c\`adl\`ag}

\newcommand{\cev}[1]{\accentset{\leftharpoonup}{#1}}
\newcommand{\vecc}[1]{\accentset{\rightharpoonup}{#1}}

\makeatletter
\let\save@mathaccent\mathaccent
\newcommand*\if@single[3]{%
  \setbox0\hbox{${\mathaccent"0362{#1}}^H$}%
  \setbox2\hbox{${\mathaccent"0362{\kern0pt#1}}^H$}%
  \ifdim\ht0=\ht2 #3\else #2\fi
  }
\newcommand*\rel@kern[1]{\kern#1\dimexpr\macc@kerna}
\newcommand{\widebar}{}
\DeclareRobustCommand*\widebar[1]{\@ifnextchar^{\wide@bar{#1}{0}}{\wide@bar{#1}{1}}}
\newcommand*\wide@bar[2]{\if@single{#1}{\wide@bar@{#1}{#2}{1}}{\wide@bar@{#1}{#2}{2}}}
\newcommand*\wide@bar@[3]{%
  \begingroup
  \def\mathaccent##1##2{%
    \let\mathaccent\save@mathaccent
    \if#32 \let\macc@nucleus\first@char \fi
    \setbox\z@\hbox{$\macc@style{\macc@nucleus}_{}$}%
    \setbox\tw@\hbox{$\macc@style{\macc@nucleus}{}_{}$}%
    \dimen@\wd\tw@
    \advance\dimen@-\wd\z@
    \divide\dimen@ 3
    \@tempdima\wd\tw@
    \advance\@tempdima-\scriptspace
    \divide\@tempdima 10
    \advance\dimen@-\@tempdima
    \ifdim\dimen@>\z@ \dimen@0pt\fi
    \rel@kern{0.6}\kern-\dimen@
    \if#31
      \overline{\rel@kern{-0.6}\kern\dimen@\macc@nucleus\rel@kern{0.4}\kern\dimen@}%
      \advance\dimen@0.4\dimexpr\macc@kerna
      \let\final@kern#2%
      \ifdim\dimen@<\z@ \let\final@kern1\fi
      \if\final@kern1 \kern-\dimen@\fi
    \else
      \overline{\rel@kern{-0.6}\kern\dimen@#1}%
    \fi
  }%
  \macc@depth\@ne
  \let\math@bgroup\@empty \let\math@egroup\macc@set@skewchar
  \mathsurround\z@ \frozen@everymath{\mathgroup\macc@group\relax}%
  \macc@set@skewchar\relax
  \let\mathaccentV\macc@nested@a
  \if#31
    \macc@nested@a\relax111{#1}%
  \else
    \def\gobble@till@marker##1\endmarker{}%
    \futurelet\first@char\gobble@till@marker#1\endmarker
    \ifcat\noexpand\first@char A\else
      \def\first@char{}%
    \fi
    \macc@nested@a\relax111{\first@char}%
  \fi
  \endgroup
}
\makeatother

\numberwithin{equation}{section}
\numberwithin{figure}{section}
\numberwithin{table}{section}

\allowdisplaybreaks[3]

\begin{document}
\ \vspace{-22pt}

\title[Diffusions on a space of interval partitions]
{Diffusions on a space of interval partitions\\ with Poisson-Dirichlet stationary distributions}
%
%
%
\author[N.~Forman]{Noah Forman$^1$}
\address{$^1$ Department of Statistics\\ University of Oxford\\ 24-29 St Giles'\\ Oxford, OX1 3LB\\ UK}
\email{noah.forman@gmail.com}
\author[S.~Pal]{Soumik Pal$^2$}
\address{$^2$ Department of Mathematics\\ University of Washington\\ Seattle, WA 98195\\ USA}
\email{soumikpal@gmail.com}
\author[D.~Rizzolo]{Douglas Rizzolo$^3$}
\address{$^3$ Department of Mathematics\\ University of Delaware\\ Newark, DE 19716\\ USA}
\email{drizzolo@udel.edu}
\author[M.~Winkel]{Matthias Winkel$^1$}
\email{winkel@stats.ox.ac.uk}


\thanks{This research is partially supported by NSF grants {DMS-1204840, DMS-1308340, DMS-1612483}, UW-RRF grant A112251, and EPSRC grant EP/K029797/1}
\date{\today}
\begin{abstract}
 We construct a pair of related diffusions on a space of interval partitions of the unit interval $[0,1]$ that are stationary with the Poisson-Dirichlet laws with parameters $\left(\frac12,0\right)$ and $\left(\frac12,\frac12\right)$ respectively. These are two particular cases of a general construction of such processes obtained by decorating the jumps of a spectrally positive L\'evy process with independent squared Bessel excursions. 
 The processes of ranked interval lengths of our partitions are members of a two parameter family of diffusions introduced by Ethier and Kurtz (1981) and Petrov (2009). The latter diffusions are continuum limits of up-down Markov chains on Chinese restaurant processes. 
 Our construction 
  is also a step towards describing a diffusion on the space of real trees whose existence has been conjectured by Aldous. 
\end{abstract}

\maketitle

\ \vspace{-22pt}

\section{Introduction}
\label{sec:intro}

We define interval partitions, following Aldous \cite[Section 17]{AldousExch} and Pitman \cite[Chapter 4]{CSP}.

\begin{definition}\label{def:IP_1}
 An \emph{interval partition} is a set $\beta$ of disjoint, open subintervals of some finite real interval $[0,M]$, that cover $[0,M]$ up to a Lebesgue-null set. We write $\IPmag{\beta}$ to denote $M$. We refer to the elements of an interval partition as its \emph{blocks}. The Lebesgue measure of a block is called its \emph{width} or \emph{mass}.
\end{definition}

An interval partition represents a totally ordered and summable collection of real numbers, for example, the interval partition generated naturally by the range of a subordinator (see Pitman and Yor \cite{PitmYor92}), or the partition of $[0,1]$ given by the complement of the zero-set of a Brownian bridge (Gnedin and Pitman \cite[Example 3]{GnedPitm05}). They also arise from the so-called stick-breaking schemes; see \cite[Example 2]{GnedPitm05}. More generally, interval partitions occur as limits of \textit{compositions} of natural numbers $n$, i.e.\ sequences of positive integers with sum $n$. Interval partitions serve as extremal points in paintbox representations of composition structures on $\mathbb{N}$; see Gnedin \cite{Gnedin97}. 

In this paper we construct two related diffusion processes on a space of interval partitions. These can be thought of as continuum analogues of natural up-down Markov chains on discrete partitions and admit as stationary laws certain members of a canonical family of probability distributions called the Poisson-Dirichlet interval partitions. On partitions with blocks ordered by decreasing mass, related diffusions 
have been introduced by Ethier and Kurtz in \cite{EthiKurt81} and, more recently, by Petrov in \cite{Petrov09}. Other known processes of interval partitions such as Bertoin's \cite{Bertoin02} are neither path-continuous nor stationary.

The construction requires two ingredients: (i) a spectrally positive L\'evy process that we call a \textit{scaffolding}, and (ii) a family of independent squared Bessel (\BESQ) excursions, called \textit{spindles}, one for each of the countably many jumps of the L\'evy process. For each jump of the scaffolding, the corresponding excursion has a length given by the height of that jump. This allows us to imagine the spindles decorating the jumps. See Figure \ref{fig:skewer_1}, where we consider a scaffolding of finite variation and the spindles are represented by the laterally symmetric spindle-like shapes attached to the jumps. 

The diffusion process at time $y$ is obtained as the output of a \emph{skewer} map at level $y$, as in Figure \ref{fig:skewer_1}. Let us give an informal description of this map. As we move from left to right along the horizontal dotted line in Figure \ref{fig:skewer_1}, we encounter a sequence of spindles. Consider the widths of these spindles when intersected by this line, arrange them sequentially on the positive half-line, and slide them (as if on a skewer) towards the origin to remove gaps between them. The collection of the intervals of these widths now produce an interval partition. As $y$ varies, we get a continuous process of interval partitions, which is our primary interest.    

\begin{figure}[t]
 \centering
 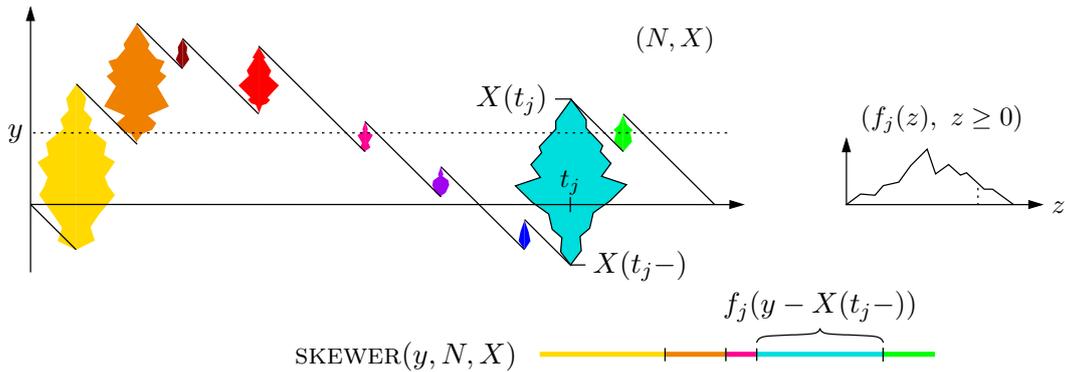
 \caption{Left: The slanted black lines comprise the graph of the scaffolding $X$. Shaded blobs decorating jumps describe the corresponding spindles: points $(t_j,f_j)$ of $N$. Right: Graph of one spindle. Bottom: A skewer, with blocks shaded to correspond to spindles; not drawn to scale.\label{fig:skewer_1}}
\end{figure}

Let us formulate the above ideas in the language of point processes that will be used throughout the rest of the paper. 
Recall that a continuous (positive) \textit{excursion} is a continuous function 
$f\colon \BR \to [0,\infty)$ with the property that, for some $z>0$, we have $f(x) > 0$ if and only if $x\in (0,z)$. That is, the function escapes up from zero at time zero and is killed upon its first return. We write $\zeta(f) = z$; this is the \emph{lifetime} of the excursion. Let $\Exc$ denote a suitable space containing continuous excursions.

For $n \in \BN$, take $0 \leq t_1 < \cdots < t_n \leq T$ and let $f_1,\ldots,f_n$ denote continuous excursions. We represent this collection of pairs $(t_j,f_j)$ in a counting measure $N = \sum_{j=1}^n \Dirac{t_j,f_j}$. Here, $\Dirac{t,f}$ denotes a Dirac point mass at $(t,f)\in [0,\infty)\times\Exc$. For some constant $c_0>0$ we define
\begin{equation}
 X(t) := -c_0t+\int_{[0,t]\times \Exc}\zeta(f)dN(u,f)= -c_0 t + \sum_{j=1}^n \zeta(f_j)\cf\left\{ 0\le t_j\le t \right\} \qquad \text{for }t\in [0,T].\label{eq:discrete_JCCP_eg}
\end{equation}
When $t_j$s arrive at rate 1 and $f_j$s are i.i.d.\ $\nu$, then $X$ is a spectrally positive L\'evy process and $N$ is a Poisson random measure with intensity $\Leb\otimes\nu$, both stopped at $T$. Here, $\Leb$ denotes Lebesgue measure.

\begin{definition}\label{def:skewer}
 For $y\in\BR$, $t\in [0,T]$, the \emph{aggregate mass} in $(N,X)$ at level $y$, up to time $t$ is
 \begin{equation}
  M_{N,X}^y(t) := \int_{[0,t]\times \Exc}f(y - X(s-))dN(s,f).\label{eq:agg_mass_from_spindles}
 \end{equation}
 The \emph{skewer} of $(N,X)$ at level $y$ is
 \begin{equation}
  \skewer(y,N,X) := \left\{\left(M^y_{N,X}(t-),M^y_{N,X}(t)\right)\!\colon t\in [0,T],\ M^y_{N,X}(t-) < M^y_{N,X}(t)\right\}.\label{eq:skewer_def}
 \end{equation}
 The \emph{skewer process} associated with $(N,X)$ is $\skewerP(N,X) := \big( \skewer(y,N,X),\ y\geq 0\big)$.
\end{definition}

Figure \ref{fig:skewer_1} does not capture the level of complexity we require for our results. Specifically, our spindles will be sampled from the $\sigma$-finite excursion measure of the squared Bessel process of dimension $-1$, denoted by $\mBxc$. 
Therefore, the jump distribution of the scaffolding is not a probability distribution. Since lengths of excursions arising from $\mBxc$ are not summable, we will replace the definition of $X$ in \eqref{eq:discrete_JCCP_eg} with a compensated limit. The resulting scaffolding will be 
a spectrally positive \Stable[\frac32] L\'evy process. 

\begin{theorem}\label{thm:diffusion_0}
 Let $\bX$ denote a spectrally positive \Stable[\frac32] process, stopped at a level 0 inverse local time $T$. Then there exists a correspondingly stopped Poisson random measure $\bN$ with intensity $\Leb\otimes\mBxc$ such that 
  $\skewerP(\bN,\bX)$ is a path-continuous strong Markov process in a suitable metric space $(\IPspace,\dI)$ of interval partitions.
\end{theorem}

The skewer process in the previous theorem is an example of what we call a \emph{type-1 evolution} with a particular initial distribution. To start a type-1 evolution from a given $\beta\in\IPspace$ we define a pair $(\bN_{\beta},\bX_{\beta})$ as follows. For each interval $V\in\beta$ let $\bff_V$ denote a \BESQ[-1] process starting from $\Leb(V)$ and killed upon hitting $0$. We denote its lifetime by $\life(\bff_V)$. Let $\bX_V$ be an independent \Stable[\frac32] process starting from $\life(\bff_V)$, killed upon hitting $0$. We form $\bN_V$ by decorating jumps of $\bX_V$ with independent \BESQ[-1] excursions, and we decorate the time 0 jump to $\life(\bff_V)$ with $\bff_V$. 
We do this independently for each $V\in\beta$. The scaffolding $\bX_\beta$ is then formed by concatenating the excursions $\bX_V$ in the order in which the $V$s arise in $\beta$. We similarly form $\bN_{\beta}$ by concatenating the $\bN_V$s. A continuous version of $\skewerP(\bN_{\beta},\bX_{\beta})$ is a type-1 evolution starting from $\beta$. This construction is made precise in Definition \ref{constr:type-1}.

We will also define \emph{type-0 evolutions}. Informally, the difference between the type-1 and type-0 evolutions is the following. While type-1 evolutions arise from scaffolding processes 
$(X(t),\,t\in [0,T])$, for type-0 evolutions we consider scaffoldings that ``come down from $\infty$'' from time $-\infty$. This is described in detail at the start of Section \ref{sec:type-0}.

\begin{theorem}\label{thm:diffusion}
 Type-1 and type-0 evolutions exist as path-continuous Hunt processes in $(\IPspace,\dI)$.
\end{theorem}

Our processes are self-similar. Their explicit transition kernels are stated in Propositions \ref{prop:type-1:gen_transn} and \ref{prop:type-0:transn}. Among their many remarkable properties, the following are worth stating here. 

\begin{theorem}\label{thm:BESQ_total_mass}
 Consider a type-1 (respectively type-0) evolution $\left( \beta^y,\; y\ge 0  \right)$ with $\beta^0 \in \IPspace$. The total mass process $\left(\IPmag{\beta^y},\,y\ge0\right)$ is a \BESQ[0] (respectively \BESQ[1]) diffusion starting at $\IPmag{\beta^0}$.
\end{theorem}

Theorems \ref{thm:diffusion_0} and \ref{thm:BESQ_total_mass} together can be viewed as a Ray-Knight theorem for a discontinuous L\'evy process. 
The local time of the stopped 
L\'evy process is not Markov in level \cite{EiseKasp93}, but our marking of jumps and skewer map fill in the missing information about jumps to construct a larger Markov process.  Moreover, the local time of the L\'evy process can be measurably recovered from the skewered process; see \cite[Theorem 28]{Paper0} or Theorem \ref{thm:LT_property_all_levels} below. The appearance of \BESQ[0] total mass is an additional connection to the second Brownian Ray-Knight theorem \cite[Theorem XI.(2.3)]{RevuzYor}, in which local time evolves as \BESQ[0].

It is well-known (see \cite{PitmYor82,ShigaWata73}) that the family of laws of $\BESQ$ processes of nonnegative dimensions running on a common time axis satisfies an additivity property. This additivity property does not extend to negative dimensions. Theorem \ref{thm:BESQ_total_mass} states that the sum of countably many squared $\BESQ$ excursions of dimension $-1$ anchored at suitably random positions on the time axis gives a $\BESQ[0]$ process. This can be interpreted as an extension of the additivity of $\BESQ$ processes to negative dimensions. 

To obtain stationary diffusions on partitions of the unit interval
, we employ a procedure that we call \textit{de-Poissonization}. The resulting invariant distributions are members of a two-parameter family of random partitions of the unit interval that we call 
the Poisson-Dirichlet interval partitions, $\PDIP[\alpha, \theta]$. See Pitman and Winkel \cite{PitmWink09} or Proposition \ref{prop:PDIP} below for more details. 
Consider the total mass process $\left(\IPmag{\beta^y},\; y\ge 0 \right)$ from Theorem \ref{thm:BESQ_total_mass} and the time-change
\begin{equation}\label{eq:de_Pois_tc_intro}
 \rho(u) := \inf\left\{  y\ge 0:\;  \int_0^y \IPmag{\beta^z}^{-1} dz > u   \right\}, \quad u \ge 0.
\end{equation}

\begin{theorem}\label{thm:stationary}
 Let $(\beta^y,\,y\ge0)$ denote a type-1 (respectively, type-0) evolution with initial state $\beta^0 \in \IPspace\setminus\{\emptyset\}$. Let $\IPmag{\beta^y}^{-1} \scale \beta^y$ denote the interval partition obtained by dividing each block in $\beta^y$ by the scalar $\IPmag{\beta^y}$. Then the process
 \[
  (\ol\beta^u,\,u\ge0) := \left(  \IPmag{\beta^{\rho(u)}}^{-1} \scale \beta^{\rho(u)},\; u \ge 0  \right)
 \]
 is a path-continuous Hunt process on $(\IPspace_1,\dI)$, where $\IPspace_1 := \{\beta\in\IPspace\colon \IPmag{\beta}=1\}$, with a stationary distribution given by $\PDIP[\frac12,0]$ (respectively, $\PDIP[\frac12, \frac{1}{2}]$).
\end{theorem}

For this reason the type-1 and type-0 evolutions might alternatively be called $\left(\frac12,0\right)$- and $\left(\frac12,\frac12\right)$-interval partition diffusions. 
In order to visualize $\PDIP[\frac 12,0]$ consider a Brownian motion during time $[0,1]$, time-reverse the process, and consider the complement of the zero-set. This is an interval partition of the unit interval that is distributed according to $\PDIP[\frac 12,0]$ \cite[Example 4]{GnedPitm05}. Notice that this interval partition contains a special leftmost block that corresponds to the meander. A similar construction for the Brownian bridge (which is reversible) gives us $\PDIP[\frac 12, \frac 12]$; see \cite[Example 3]{GnedPitm05}.  


Petrov \cite{Petrov09} introduced a two-parameter family of diffusions on the Kingman simplex,
\begin{equation}\label{eq:infsimp_def}
 \infsimp := \left\{ x=\left(x_1, x_2, \ldots   \right)\colon x_1 \ge x_2 \ge \cdots \ge 0,\; \sum_{i=1}^\infty x_i = 1     \right\},
\end{equation}
extending a one-parameter family due to Ethier and Kurtz \cite{EthiKurt81}. For $0\le \alpha < 1$ and $\theta>-\alpha$, these processes are characterized by the following generator acting on symmetric polynomials: 
\begin{equation}
\petrovgen=\sum_{i\ge 1}x_i\frac{\partial^2}{\partial x_i^2}-\sum_{i,j\ge 1}x_ix_j\frac{\partial^2}{\partial x_i\partial x_j}-\sum_{i\ge 1}(\theta x_i+\alpha)\frac{\partial}{\partial x_i}.
\end{equation}
We denote the laws of this two-parameter family by \EKP[\alpha, \theta]. Diffusions with these laws are stationary and reversible with respect to the Poisson-Dirichlet distributions, which we denote by \PoiDir[\alpha,\theta]. 

\begin{theorem}\label{thm:petroviden}
 Consider the de-Poissonized type-1 evolution $\left(\ol{\beta}^u,\, u \ge 0\right)$ of Theorem \ref{thm:stationary}. Let $W(u)$ be the vector of decreasing order statistics of $\left(  \mathrm{Leb}(U),\; U \in \ol{\beta}^u \right)$. Starting from any point in $\infsimp$, the law of the stochastic process $\left(W(u/2), \, u \ge 0\right)$ is $\EKP[\frac{1}{2},0]$. The corresponding statement for the type-0 evolutions gives us $\EKP[\frac{1}{2},\frac{1}{2}]$.
\end{theorem}

In light of Theorem \ref{thm:petroviden}, the de-Poissonized evolutions of Theorem \ref{thm:stationary} may be viewed as labeled variants of \EKP[\frac12,0] and \EKP[\frac12,\frac12] 
diffusions. See Feng and Sun \cite{FengSun10} for a related conjecture. Also see \cite{CdBERS16, FengSunWangXu11}. 
In contrast to the analytic methods in these papers, our Poissonian construction gives a pathwise realization of this diffusion.

Following \cite[equations (82) and (83)]{Pitman03}, for $\alpha\in (0,1)$, the \emph{$\alpha$-diversity} is
\begin{equation}
 \IPLT^{\alpha}_x := \lim_{h\downto 0} \Gamma(1-\alpha)h^{\alpha}\#\{i\geq 1\colon x_i>h\} \qquad \text{for }x\in\infsimp.\label{eq:intro:diversity}
\end{equation}
This may be understood as a continuum analogue to the number of blocks in a partition of $n$. A constant multiple of this is sometimes called the \emph{local time} of $x$ \cite[equation (24)]{PitmYorPDAT}. These quantities arise in a variety of contexts \cite{CSP,PitmWink09}. Ruggiero et\ al.\ \cite{RuggWalkFava13} have studied processes related to \EKP\ diffusions for which $\alpha$-diversity evolves as a diffusion. The following result will be shown to be a corollary of Theorems \ref{thm:stationary} and \ref{thm:petroviden}.

\begin{corollary}\label{cor:Petrov_diversity}
 For $x\in\infsimp$, if $\IPLT^{1/2}_x$ exists, then under \EKP[\frac12,0] or \EKP[\frac12,\frac12] starting from $x$, the diversity $\big(\IPLT^{1/2}_{W(u)},\,u\ge0\big)$ evolves continuously a.s..
\end{corollary}

Since \EKP\ diffusions are reversible, the evolving sequences of ranked block masses in our de-Poissonized evolutions are reversible as well. We make the following conjecture.

\begin{conjecture}\label{conj:reverse_0}
 The de-Poissonized type-0 evolution of Theorem \ref{thm:stationary} is reversible with respect to \PDIP[\frac12,\frac12].
\end{conjecture}

In \cite{Pal11}, Wright-Fisher diffusions are obtained as a de-Poissonization of a vector of independent \BESQ\ processes via the same time-change as in \eqref{eq:de_Pois_tc_intro}. In that paper, the sum of the \BESQ\ processes turns out to be independent of the de-Poissonized process. See also \cite{Pal13}. A similar result was previously found in a construction of Jacobi diffusions in \cite{WarrYor98}.

\begin{conjecture}\label{conj:mass_struct}
 Consider a type-1 or type-0 evolution $(\beta^y,\,y\ge0)$ starting from $\beta\in\IPspace$. Its total mass process, as in Theorem \ref{thm:BESQ_total_mass}, is independent of its de-Poissonization, as in Theorem \ref{thm:stationary}.
\end{conjecture}

See Theorem \ref{thm:pseudostat_strong} for a weaker result.

\subsection{An ordered Chinese restaurant process with reseating}
\label{sec:intro:CRP}

We now describe a Markov chain on ordered partitions of $[n]$, for $n \in \BN$, which can be thought of as a discrete approximation of our diffusion on interval partitions. A rigorous convergence theorem, however, is beyond the scope of the current paper.  

The two-parameter Chinese Restaurant Process (denoted by \CRPAT) is a well-known sequential construction for a family of exchangeable random partitions of the natural numbers, due to Dubins and Pitman. See the survey in \cite[Section 3.2]{CSP}. As usual, we represent the blocks of the partition as tables and the natural numbers within the blocks as customers. 

Start with customer 1 sitting at some table. Subsequent customers follow the \CRPAT\ \textit{seating rule}, defined when the parameters $(\alpha,\theta)$ satisfy $0\le \alpha\leq 1$, $\theta \ge 0$. If a table has $m$ customers then the next customer to enter joins that table with probability proportional to $m-\alpha$, or sits alone with probability proportional to $k\alpha+\theta$, where $k$ is the number of tables currently occupied. The \CRPAT\ \emph{with reseating} is a Markov chain on partitions of $[n]$. We start with $n$ numbered customers already seated. At each step, a uniform random customer stands up and selects a new seat according to the \CRPAT\ seating rule.

Petrov \cite{Petrov09} showed that the process of ranked table sizes in this Markov chain converges in a scaling limit, as $n$ tends to infinity, to the \EKP[\alpha,\theta] diffusion of Theorem \ref{thm:petroviden}. There has been considerable recent interest in these and related models \cite{CdBERS16, FengBook, FengSun10, FengSunWangXu11, FengWang07, Fulman05, Ruggiero14, RuggWalk09, RuggWalkFava13}. However, births and deaths of ``tables'' are invisible in Petrov's diffusion limit.

We rectify this by ordering tables from left to right in a special way, following \cite{PitmWink09}. We also turn the \CRP\ into a continuous-time Markov chain on vectors of occupied table sizes. If a table has $m$ customers then a new customer joins that table after an exponential time with rate $m-\alpha$. For each occupied table, a new customer enters and begins a new table immediately to its right after an exponential time with rate $\alpha$. Also, a new customer enters to begin a new table to the left of the leftmost table with rate $\theta$. This corresponds to the usual \CRPAT\ because, starting from a time when there are $k$ tables seating $n$ customers, the total rate at which a customer enters and starts a new table is $k\alpha + \theta$. 
This process gives rise to a regenerative composition structure \cite{GnedPitm05} and is related to a two-parameter family of continuum random trees \cite{PitmWink09}. 

The \emph{ordered Poissonized down-up} \CRPAT\ is defined by modifying this continuous-time Markov chain by having each customer exit after an independent exponential time with rate 1. Consequently the population at each table is a birth-and-death chain with birth rate $m-\alpha$ and death rate $m$, where $m$ is the number of customers currently seated at that table.


The two diffusions described by the main theorems above are continuum analogues of this process, with $(\alpha, \theta)=\left(\frac12,0\right)$ for type-1 evolutions or $\left(\frac12,\frac12\right)$ for type-0 evolutions. These two cases are of special interest, as \PDIP[\frac12,0] and \PDIP[\frac12,\frac12] arise as the interval partitions formed by excursions of Brownian motion and Brownian bridge, respectively. This construction can be generalized to all $\alpha\in (0,1)$, $\theta\ge0$. The authors will take this up in subsequent work.

We relate the ordered Poissonized down-up \CRP\ to scaffoldings with spindles, as in Figure \ref{fig:skewer_1}, by introducing the following tree structure. Whenever a new table is begun immediately to the right of an existing table, we say the older table is the parent of the new one. Then the number of tables evolves like a Crump-Mode-Jagers (CMJ) branching process. See \cite[Chapter 6]{JagersBranching} for a discussion of this family of processes and related references; also see \cite{Nerman81}.

\begin{figure}[t]
 \centering
 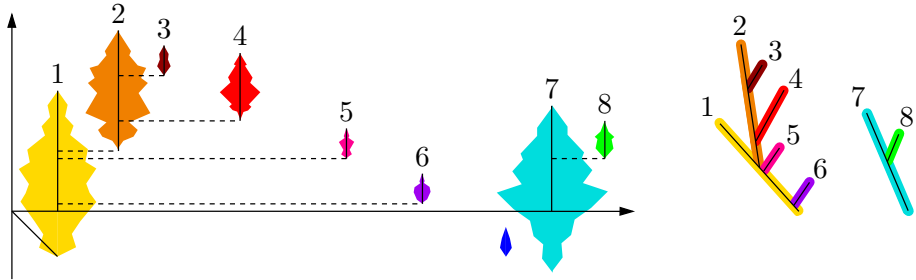
 \caption{The forest of splitting trees corresponding to the JCCP of Figure \ref{fig:skewer_1}.\label{fig:marked_chron_tree}}
\end{figure}

CMJ branching processes are often associated with \emph{chronological trees} or \emph{splitting trees} \cite{Geiger95,GeigKers97,Lambert10}. In a chronological tree, edges have lengths representing time. Branch points represent birth events, leaves represent deaths, and each individual is represented by a certain path directed away from the root, with the sum of edge lengths along the path describing the lifetime of the individual. Consider Figure \ref{fig:marked_chron_tree}, which represents a family tree. The vertical line labeled $1$ represents an individual with three children, labeled $\{2,5,6\}$. Individual $2$ gives rise to two more children $\{3,4\}$. The chronological tree is drawn on the right. 

Note that at each branch point, the path describing the parent continues up in a straight line, while the path describing the child branches off to the right. We may represent a splitting tree by a jumping chronological contour process (JCCP). Imagine a jumping flea traveling around the tree as follows. It begins at the root and jumps up the left-hand side to the leftmost leaf. It walks continuously down at constant speed until it reaches a branch point. Then, it jumps up to the left-most leaf in the resulting branch, and so on. The JCCP is the distance from the root to the flea as it varies over time. Note that the lifespans of individuals correspond to the jumps of the JCCP. If each individual has an i.i.d. lifetime and has children at a constant rate then the JCCP is a L\'evy process \cite{Lambert10}.

There is an equivalent representation of the ordered Poissonized down-up \CRP\ as a JCCP (scaffolding) whose jumps are marked with the evolving number of customers seated at an existing table (discrete spindles). Each table population (spindle) is an excursion of a birth-and-death chain, starting with one individual and dying when it hits zero, independent of every other table. Now suppose the parameters are chosen such that this excursion law has a scaling limit under which it converges vaguely to a $\sigma$-finite measure on $\Exc$. Also assume that the JCCP itself converges, under a similar scaling (and perhaps, compensation), to another spectrally positive L\'evy process. Then we recover the continuous spindles-on-scaffolding picture described around Definition \ref{def:skewer}. It is obvious that the natural skewer map on the discrete JCCP recovers the vector of sizes of tables in the ordered Poissonized down-up \CRP. This inspires a similar definition in continuum.

Notice that, when the scaffolding is of infinite variation (as in our case), the phylogenetic situation is more complex, as ancestors do not have immediate children. Rather, between each ancestor and descendant there are infinitely many intervening generations, most of them extremely short-lived. Recently, the convergence and description of similar phylogenetic structures have been studied in \cite{DelaporteII,Delaporte15} and \cite{LambertBravo16}, including specific discussion of the \Stable[\frac32] case.

\subsection{Bigger picture: Aldous diffusion on continuum trees}
\label{sec:intro:AD}

\begin{figure}\centering
 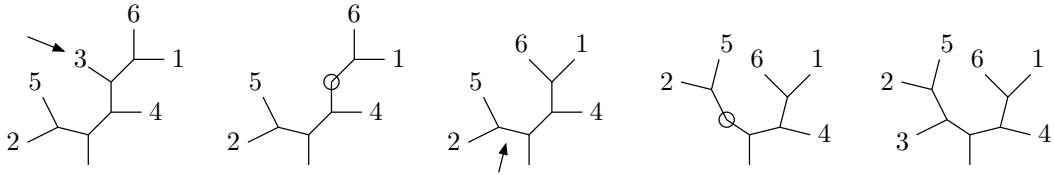
 \caption{From left to right, one Aldous down-up move.\label{fig:AC_move}}
\end{figure}

The Aldous chain \cite{Aldous00,Schweinsberg02} is a Markov chain on the space of (rooted, for our purposes) binary trees with $n$ labeled leaves. Each transition of the chain, called a down-up move, has two stages. First, a uniform random leaf is deleted, and the resulting degree-two branch point is contracted away. Next, a uniform random edge is selected, we insert a new branch point into the middle of that edge, and we extend a new leaf-edge out from that branch point. This is illustrated in Figure \ref{fig:AC_move} where $n=6$ and the leaf labeled $3$ is deleted and re-inserted. The stationary distribution of this chain is the uniform distribution on rooted binary trees with $n$ labeled leaves.

We are interested in the scaling limit of this stochastic process on the space of trees without leaf labels as $n$ tends to infinity, where each edge is endowed with length $1/\sqrt{n}$. It is a long-standing conjecture due to Aldous \cite{AldousDiffusionProblem} that the limit is a diffusion on some space of real trees, with invariant distribution given by the Brownian continuum random tree \cite{AldousCRT1}. 

The current article is the first of a sequence of projects whose goal is to construct and study this conjectured limit, superseding the unpublished preprint \cite{PalPreprint}. We mention this connection only briefly since it will be taken up elsewhere in more detail. First we \textit{Poissonize} the Aldous chain. Attach independent exponential clocks of rate one to each edge and independent exponential clocks of rate two to each leaf. If a leaf clock rings, we perform deletion of the leaf, and if an edge clock rings we perform an insertion at that edge. The idea is that although Poissonization changes the Markov chain, the limiting diffusion of the Poissonized chains should merely differ by a rescaling of time and space from the conjectured limit of the Aldous chain. Hence, a de-Poissonization (as in Theorem \ref{thm:stationary}) of the limit will produce the desired process. For a justification of this choice of rates and discussion of de-Poissonization see \cite{Pal13}.

Now, suppose we mark a uniform random pair of distinct leaves and decompose the tree into: (i) the path, called the \emph{spine}, from the root up to the branch point that separates the two marked leaves, (ii) the sequence of subtrees growing out from the spine, and (iii) the two subtrees above the branch point that respectively contain the two marked leaves. The sequence of leaf-counts in these subtrees, excluding the spine itself, which contains no leaves, will evolve as an ordered \CRP[\frac12,-\frac12]. Although these parameters lie outside the usual domain of the \CRP, this can nevertheless be made rigorous. 
Much of the technical groundwork in order to do this is covered in this article. Our approach towards constructing the Aldous diffusion is then to generalize the two-leaf argument to $k$ leaves, prove consistency, and pass to the limit. 

This inspires the names of our diffusions. The ordered \CRP[\frac12,-\frac12] will have as its continuum analogue a \emph{type-2 evolution}, with two leftmost blocks at any fixed time corresponding to the two top subtrees, whereas the type-1 has a single leftmost block and the type-0 has none.

\tableofcontents
\section{The state space $\IPspace$: interval partitions with diversity}
\label{sec:prelim_IP}

\subsection{Diversity and operations on interval partitions}

\begin{definition}\label{def:diversity_property} Let $\HIPspace$ denote the set of all interval partitions in the sense of Definition \ref{def:IP_1}. 
  We say that an interval partition $\beta\in\HIPspace$ of a finite interval $[0,M]$ has the \emph{$\frac12$-diversity property}, or that $\beta$ is an 
  \emph{interval partition with diversity}, if the following limit exists for every $t\in [0,M]$:
  \begin{equation}
   \IPLT_{\beta}(t) := \sqrt{\pi}\lim_{h\downto 0}\sqrt{h}\#\{(a,b)\in \beta\colon\ |b-a|>h,\ b\leq t\}.\label{eq:IPLT}
  \end{equation}
  We denote by $\IPspace\subset\HIPspace$ the set of interval partitions $\beta$ that possess the $\frac12$-diversity property. We call $\IPLT_{\beta}(t)$ the 
  \emph{diversity} of the interval partition up to $t\in[0,M]$. For $U\in\beta$, $t\in U$, we write $\IPLT_{\beta}(U)=\IPLT_{\beta}(t)$, and we write 
  $\IPLT_{\beta}(\infty) := \IPLT_{\beta}(M)$ to denote the \emph{total ($\frac12$-)diversity} of $\beta$.
\end{definition}

Note that $\IPLT_\beta(U)$ is well-defined, since $\IPLT_{\beta}$ is constant on each interval $U \in \beta$, as the intervals are disjoint. 
For $\beta$ a partition of the unit interval and $x = (x_1,x_2,\ldots)$ the order statistics of its block masses, the total $\frac12$-diversity $\IPLT_{\beta}(\infty)$ defined above agrees with $\IPLT^{1/2}_x$, as in \eqref{eq:intro:diversity}.

\begin{proposition}\label{prop:IP:Stable} Let $Y=(Y(s),s\geq 0)$ denote a \Stable[\frac12] subordinator with Laplace
  exponent $\Phi(\lambda) = \lambda^{1/2}$. Then, for any $T>0$,  
 \begin{equation}\label{eq:IPs_from_subord}
  \beta := \{(Y(s-),Y(s))\colon\ s\in [0,T),\ Y(s-)<Y(s)\}.
 \end{equation}
 is an interval partition with diversity $\IPLT_\beta(\infty)=T$ a.s.. We call $\beta$ a \Stable[\frac12] interval partition. 
\end{proposition}
\begin{proof}
 This follows from the Strong Law of Large Numbers for the Poisson process of jumps and the monotonicity of $\IPLT_{\beta}(t)$ in $t$.
\end{proof}

There are various natural operations for interval partitions:


\begin{definition}\label{def:IP:concat}
 We define a \emph{reversal involution} $\reverse_{\textnormal{IP}}\colon\HIPspace\rightarrow\HIPspace$ and a \emph{scaling map} $\scaleI\colon (0,\infty)\times\HIPspace\rightarrow\HIPspace$ by saying, for $c>0$ and $\beta\in\HIPspace$,
 \begin{equation}
  \reverse_{\textnormal{IP}}(\beta)=\left\{\left(\IPmag{\beta}-b,\,\IPmag{\beta}-a\right),\,(a,b)\in\beta\right\}, \qquad \scaleI[c][\beta] = \{(ca,cb)\colon (a,b)\in \beta\}.\label{eq:IP:scale_def}
 \end{equation}
 Let $(\beta_a)_{a\in\mathcal{A}}$ denote a family of interval partitions indexed by a totally ordered set $(\cA,\preceq)$. For the purpose of this definition, let 
 $S_{\beta}(a-) := \sum_{b\prec a}\IPmag{\beta_b}$ for $a\in\mathcal{A}$.
 If $S_{\beta}(a-) < \infty$ for every $a\in \mathcal{A}$, then we define the \emph{concatenation}
 \begin{equation}\label{eq:IP:concat_def}
  \Concat_{a\in\mathcal{A}}\beta_a := \{(x+S_{\beta}(a-),y+S_{\beta}(a-)\colon\ a\in\mathcal{A},\ (x,y)\in \beta_a\}.
 \end{equation}
 When $\mathcal{A}=\{a_1,a_2\}$, we denote this by $\beta_{a_1}\concat\beta_{a_2}$. 
 We call $(\beta_a)_{a\in\mathcal{A}}$ \emph{summable} if $\sum_{a\in\mathcal{A}}\IPmag{\beta_a} < \infty$. 
 It is then \emph{strongly summable} if the concatenated partition satisfies the diversity property \eqref{eq:IPLT}. 
%
\end{definition}

Before specifying our chosen metric on $\IPspace$, we present a few more examples of random interval partitions with the diversity property, selected from the literature.

\begin{proposition}\label{prop:PDIP}\begin{enumerate}[label=(\roman*), ref=(\roman*)]
 \item Consider the zero-set $\mathcal{Z}=\{t\in[0,1]\colon B^{\rm br}_t=0\}$ of standard Brownian bridge $B^{\rm br}$. Then $[0,1)\setminus\mathcal{Z}$ is a union of
   disjoint open intervals that form an interval partition $\gamma$ with diversity a.s.. The ranked interval lengths have \PoiDir[\frac12,\frac12] distribution. We call 
   $\gamma$ a Poisson-Dirichlet interval partition \PDIP[\frac12,\frac12].
 \item For Brownian motion $B$, the interval partition $\gamma^\prime$ of $[0,1]$ associated with its zero-set has diversity a.s.. The ranked
   interval lengths have $\PoiDir[\frac12,0]$ distribution. We call the reversed interval partition $\reverse_{\textnormal{IP}}(\gamma^\prime)$ a \PDIP[\frac12,0].
 \item Let $Y$ be a \Stable[\frac12] subordinator. 
   Let $S$ be independent of $Y$ with \ExpDist[\lambda] distribution and
   $T := \inf\{s > 0\colon\ Y(s) > S\}$. Then $Y(T-)$ and $S-Y(T-)$ are independent \GammaDist[\frac12,\lambda] variables. For $\beta$ as in 
   \eqref{eq:IPs_from_subord} and $\beta^\prime=\{(0,S-Y(T-))\}\concat\beta$, \label{item:PDIP:Stable}
   $$\reverse_{\textnormal{IP}}(\overline{\beta})\stackrel{d}{=}\overline{\beta}:=\IPmag{\beta}^{-1} \scaleI \beta\sim\PDIP[\textstyle\frac12,\frac12]\quad\text{and}\quad\overline{\beta}^\prime:=\IPmag{\beta^\prime}^{-1}\scaleI\beta^\prime\sim\PDIP[\textstyle\frac12,0].$$
  \end{enumerate}
\end{proposition}
\begin{proof} By Proposition \ref{prop:IP:Stable} and the definition of the scaling map $\scaleI$, both $\beta$ and $\overline{\beta}$ are interval partitions with
  diversity. This easily entails the same for $\beta^\prime$ and $\overline{\beta}^\prime$. Recall that the inverse local time of $B$ is a \Stable[\frac12] subordinator. 
  Hence, the remainder is well-known; see for example \cite[Lemma 3.7]{PermPitmYor92}, which states that the last zero $G$ of $B$ is a \BetaDist[\frac12,\frac12] 
  variable independent of a Brownian bridge $(B(uG)/\sqrt{G},0\le u\le 1)$. Finally, the \PoiDir[\frac12,\frac12]\ and \PoiDir[\frac12,0]\ distributions can be read from 
  \cite[Corollary 4.9]{CSP}.
\end{proof}

\subsection{The metric spaces $(\IPspace,\dI)$ and $(\HIPspace,d_H)$}

\begin{definition} \label{def:IP:metric}
 We adopt the standard discrete mathematics notation $[n] := \{1,2,\ldots,n\}$.
 For $\beta,\gamma\in \IPspace$, a \emph{correspondence} from $\beta$ to $\gamma$ is a finite sequence of ordered pairs of intervals $(U_1,V_1),\ldots,(U_n,V_n) \in \beta\times\gamma$, $n\geq 0$, where the sequences $(U_j)_{j\in [n]}$ and $(V_j)_{j\in [n]}$ are each strictly increasing in the left-to-right ordering of the interval partitions.

 The \emph{distortion} of a correspondence $(U_j,V_j)_{j\in [n]}$ from $\beta$ to $\gamma$, denoted by $\dis(\beta,\gamma,(U_j,V_j)_{j\in [n]})$, is defined to be the maximum of the following four quantities:
 \begin{enumerate}[label=(\roman*), ref=(\roman*)]
  \item $\sup_{j\in [n]}|\IPLT_{\beta}(U_j) - \IPLT_{\gamma}(V_j)|$,
  \item $|\IPLT_{\beta}(\infty) - \IPLT_{\gamma}(\infty)|$,
  \item $\sum_{j\in [n]}|\Leb(U_j)-\Leb(V_j)| + \IPmag{\beta} - \sum_{j\in [n]}\Leb(U_j)$, \label{item:IP_m:mass_1}
  \item $\sum_{j\in [n]}|\Leb(U_j)-\Leb(V_j)| + \IPmag{\gamma} - \sum_{j\in [n]}\Leb(V_j)$. \label{item:IP_m:mass_2}
 \end{enumerate}
 Note that the second of these quantities depends only on the partitions $\beta$ and $\gamma$ and not on the correspondence. 
 
 For $\beta,\gamma\in\IPspace$ we define
 \begin{equation}\label{eq:IP:metric_def}
  \dI(\beta,\gamma) := \inf_{n\ge 0,\,(U_j,V_j)_{j\in [n]}}\dis\big(\beta,\gamma,(U_j,V_j)_{j\in [n]}\big),
 \end{equation}
 where the infimum is over all correspondences from $\beta$ to $\gamma$.
\end{definition}

\begin{proposition}\label{prop:IP_metric}
 The map $\dI: \IPspace^2\to [0,\infty)$ is a metric on $\IPspace$.
\end{proposition}

\begin{proof}
 Symmetry is built into the definition, and we leave positive-definiteness as an exercise for the reader. We will prove that $\dI$ satisfies the triangle inequality.
 
 Suppose that $\dI(\alpha,\beta) = a$ and $\dI(\beta,\gamma) = b$. Then
 \begin{equation}\label{eq:divbound}
  |\IPLT_{\alpha}(\infty)-\IPLT_{\gamma}(\infty)| \leq |\IPLT_{\alpha}(\infty)-\IPLT_{\beta}(\infty)| + |\IPLT_{\beta}(\infty)-\IPLT_{\gamma}(\infty)| \leq a+b.
 \end{equation}
 Now take $\e > 0$. It suffices to show that $\dI(\alpha,\gamma) \leq a+b+2\e$.
 
 There exist correspondences $(U_j,V_j)_{j\in[m]}$ and $(W_j,X_j)_{j\in[n]}$, from $\alpha$ to $\beta$ and from $\beta$ to $\gamma$ respectively, with distortions less than $a+\e$ and $b+\e$ respectively. We will split these two sequences into two parts each. Let $(\hat V_j)_{j\in[k]} = (\hat W_j)_{j\in[k]}$ denote the subsequence of intervals that appear in both $(V_j)_{j\in[m]}$ and $(W_j)_{j\in[n]}$; note that $k$ may equal zero, i.e.\ the overlap may be empty. For each $j\in[k]$, let $\hat U_j$ and $\hat X_j$ denote the intervals in $\alpha$ and $\gamma$ respectively that are paired with $\hat V_j = \hat W_j$ in the two correspondences. Then, let $(\hat U_j,\hat V_j)_{j\in [m]\setminus[k]}$ denote the remaining terms in the first correspondence not accounted for in the intersection, and let $(\hat W_j,\hat X_j)_{j\in [n]\setminus[k]}$ denote the remaining terms in the second correspondence. So overall, the sequences $(\hat U_j,\hat V_j)_{j\in[m]}$ and $(\hat V_j,\hat W_j)_{j\in[n]}$ are reorderings of the two correspondences.
 
 We will show that the correspondence $(\hat U_j,\hat X_j)_{j\in[k]}$ has distortion less than $a+b+2\e$. There are four quantities, listed in Definition \ref{def:IP:metric}, that we must bound. Quantity (ii) has already been bounded in (\ref{eq:divbound}). To bound (i), observe that
 $$
  \sup_{j\in[k]}|(\IPLT_{\alpha}(\hat U_j) - \IPLT_{\gamma}(\hat X_j)| \leq \sup_{j\in[k]}\left(|\IPLT_{\alpha}(\hat U_j) - \IPLT_{\beta}(\hat V_j)| + |\IPLT_{\beta}(\hat W_j) - \IPLT_{\gamma}(\hat X_j)|\right)
  	< a+b+2\e.
 $$
 We now go about bounding (iii), which is more involved. By the triangle inequality,
 \begin{align*}
  &\sum_{j\in[k]}|\Leb(\hat U_j) - \Leb(\hat X_j)| + \IPmag{\alpha} - \sum_{j\in[k]}\Leb(\hat U_j)\\
  	&\leq \sum_{j\in[m]}|\Leb(U_j) - \Leb(V_j)| + \sum_{j\in[n]}|\Leb(W_j) - \Leb(X_j)|\\
  	&\quad - \sum_{j\in[m]\setminus[k]}|\Leb(\hat U_j) - \Leb(\hat V_j)|
  	 + \left(\IPmag{\alpha} - \sum_{j\in[m]}\Leb(\hat U_j)\right) + \sum_{j\in[m]\setminus[k]}\Leb(\hat U_j).
 \end{align*}
 Since the $(\hat V_j)_{j\in[m]\setminus[k]}$ are members of $\beta$ not listed in $(W_j)_{j\in[n]}$,
 $$
  \sum_{j=k+1}^m\Leb(\hat U_j) - \sum_{j=k+1}^m|\Leb(\hat U_j) - \Leb(\hat V_j)| \leq \sum_{j=k+1}^m\Leb(\hat V_j)
  	\leq \IPmag{\beta} - \sum_{j=1}^n\Leb(W_j),
 $$
 again by the triangle inequality. Thus,
 \begin{align*}
  &\sum_{j\in[k]}|\Leb(\hat U_j) - \Leb(\hat X_j)| + \IPmag{\alpha} - \sum_{j\in[k]}\Leb(\hat U_j)\\
  	&\leq \sum_{j\in[m]}\!|\Leb(U_j)\! -\! \Leb(V_j)|\! +\! \IPmag{\alpha}\! -\! \sum_{j\in[m]}\Leb(U_j)
  	\!+\! \sum_{j\in[n]}\!|\Leb(W_j)\! -\! \Leb(X_j)|\! +\! \IPmag{\beta}\! -\! \sum_{j\in[n]}\Leb(W_j)\\
  	&\leq a+\e + b+\e.
 \end{align*}
 This is the desired bound on quantity (iii) in Definition \ref{def:IP:metric}. The same argument bounds (iv):
 $$
  \sum_{j\in[k]}|\Leb(\hat U_j) - \Leb(\hat X_j)| + \IPmag{\gamma} - \sum_{j\in[k]}\Leb(\hat X_j) \leq a+b+2\e.
 $$
 Therefore $\dI(\alpha,\gamma) \leq a+b+2\e$, as desired.
\end{proof}

\begin{theorem}\label{thm:Lusin}
 $(\IPspace,\dI)$ is Lusin, i.e.\ homeomorphic to a Borel subset of a compact metric space. 
\end{theorem}

We give a detailed account of the topological properties of $(\IPspace,\dI)$ in Appendix \ref{sec:IP_space}, including the proof of this proposition.
It follows from Theorem \ref{thm:Lusin} that $(\IPspace,\dI)$ is a Borel space, i.e. bi-measurably in bijective correspondence with a Borel subset of $[0,1]$. In this setting, regular conditional distributions exist; see Kallenberg \cite[Theorem A1.2, Theorem 6.3]{Kallenberg}.

The reader may wonder: Why take the metric space $(\IPspace,\dI)$? Why restrict to interval partitions with diversity? As noted in the introduction, we will define two  
diffusions on $(\IPspace,\dI)$. The transition kernels are such that, if we start one of these processes from any state in $\HIPspace$, then the 
evolving interval partition will immediately enter the subspace $\IPspace$. Furthermore, the evolution of diversities will also be continuous after entering $\IPspace$. 
If we start in $\IPspace$, continuity also holds at time 0. Ruggiero et al.\ \cite{RuggWalkFava13} have studied a related family of diffusions on the Kingman simplex, driven by continuously evolving diversity. 
For us, diversities will play an important role in setting up the Aldous diffusion as an evolution of the Brownian CRT, where
diversities correspond to lengths in the tree. In this context, the particular metric $\dI$ relates to Gromov-Hausdorff and Gromov-Hausdorff-Prohorov
metrics on spaces of (weighted) real trees.

We have elected to represent interval partitions as sets of disjoint open intervals, but they may be alternatively represented as the complements of the unions of these sets. We will write
$$
 G(\alpha) := \big[0,\IPmag{\alpha}\big] \setminus \bigcup\nolimits_{U\in \alpha}U \qquad \text{for }\alpha\in\HIPspace.
$$
Such sets $G(\alpha)$ are closed and bounded, with Lebesgue measure zero. Note that $G$ is injective.

\begin{definition}\label{def:Hausdorff}
 For $\alpha,\beta\in\HIPspace$, we define
 \begin{equation}
  d_H(\alpha,\beta) := \inf\Bigg\{ r>0\colon\ G(\alpha) \subseteq \bigcup_{x\in G(\beta)}(x-r,x+r),\ G(\beta) \subseteq \bigcup_{x\in G(\alpha)}(x-r,x+r) \Bigg\}.
 \end{equation}
 We call $d_H$ the \emph{Hausdorff metric on $\HIPspace$}. 
 Recalling Definition \ref{def:IP:metric} of correspondences and distortion, we define the \emph{Hausdorff distortion} of a correspondence between $\beta,\gamma\in\HIPspace$ to be the maximum of quantities \ref{item:IP_m:mass_1} and \ref{item:IP_m:mass_2} in that definition. We denote this by $\dis_H$. Then we define
 \begin{equation}\label{eq:IP:Hausdorff_2_def}
  d'_H(\beta,\gamma) := \inf_{n\ge 0,\,(U_j,V_j)_{j\in [n]}}\dis_H\big(\beta,\gamma,(U_j,V_j)_{j\in [n]}\big).
 \end{equation}
\end{definition}

Note that $d_H$ is a metric, as it is the pullback via $G$ of the Hausdorff metric on compact subsets of $\BR$. It follows from the proof of Proposition \ref{prop:IP_metric} that $d_H'$ is a metric as well.

\begin{proposition}
 $(\HIPspace,d_H)$ is Lusin. Furthermore, $\HIPspace$ is not closed, but a Borel subset of the locally compact space $(\mathcal{C},d_H)$ of (collections of disjoint open intervals that form the complements of) compact subsets of $[0,\infty)$.
\end{proposition}

\begin{proof}
 By \cite[Theorem 7.3.8]{BuraBuraIvan01}, the subspace of compact subsets of $[0,M]$ is compact, hence $(\mathcal{C},d_H)$ is locally compact. $\HIPspace$ is not closed since $\{(j-1)2^{-n},j2^{-n}),j\in[2^n]\}\cup\{(1,2)\}$, $n\ge 1$, is a sequence in $(\HIPspace,d_H)$, even in $(\IPspace,d_H)$, but with limit $\{(1,2)\}$ in $\mathcal{C}\setminus\HIPspace$. The proof that $\HIPspace\subset\mathcal{C}$ is a Borel subset is left to the reader.
\end{proof}   

\begin{proposition}[$d_H$ is equivalent to $d_H'$ and weaker than $\dI$]\label{prop:Hausdorff}
 \begin{enumerate}[label=(\roman*), ref=(\roman*)]
  \item For every $\epsilon > 0$ there exist some $\beta,\gamma\in\IPspace$ for which $d_H(\beta,\gamma) < \epsilon$ and $\dI(\beta,\gamma) > 1/\epsilon$.\label{item:Haus:neq}
  \item For $\beta,\gamma\in\IPspace$, we have $d_H'(\beta,\gamma)\leq \dI(\beta,\gamma)$.\label{item:Haus:leq}
  \item The metrics $d_H$ and $d_H'$ generate the same topology on $\HIPspace$.\label{item:Haus:equiv}
  \item Each of $\dI$, $d_H$, and $d_H'$ generates the same Borel $\sigma$-algebra on $\IPspace$.\label{item:Haus:sig}
 \end{enumerate}
\end{proposition}

We prove (i)-(iii) at the beginning of Appendix \ref{sec:IP_space} and (iv) at the end. The following two lemmas may be deduced easily from the definitions of $\dI$, $\dH$, and $\dH'$.

\begin{lemma}\label{lem:IP:concat}
 \begin{enumerate}[label=(\roman*), ref=(\roman*)]
  \item Any totally ordered collection of interval partitions in $\IPspace$ (respectively $\HIPspace$) in which only finitely many are non-empty is strongly summable 
    (respectively summable).
  \item If $(\beta_a)_{a\in \mathcal{A}}\in \IPspace^{\mathcal{A}}$ and $(\gamma_a)_{a\in \mathcal{A}}\in \IPspace^{\mathcal{A}}$ are strongly summable then
  \begin{equation}\label{eq:IP:concat_dist}
   \dI\left(\Concat_{a\in\mathcal{A}}\beta_a,\ \Concat_{a\in\mathcal{A}}\gamma_a\right) \leq \sum_{a\in\mathcal{A}}\dI(\beta_a,\gamma_a).
  \end{equation}
  The same holds for $d_H$ and $d_H'$ when $(\beta_a)_{a\in \mathcal{A}},(\gamma_a)_{a\in \mathcal{A}}\in \HIPspace^{\mathcal{A}}$ are summable.
 \end{enumerate}
\end{lemma}

\begin{lemma}\label{lem:IP:scale}
 For $c>0$, the scaling functions $\beta\mapsto\scaleI[c][\beta]$ are bijections on $\IPspace$ and on $\HIPspace$; specifically, partitions in the image of $\IPspace$ 
 possess the diversity property with
 \begin{equation}
  \IPLT_{\scaleI[c][\beta]}(ct) = \sqrt{c}\IPLT_{\beta}(t) \qquad \text{for }\beta\in\IPspace,\ t>0,\ c>0.\label{eq:IP:LT_scaling}
 \end{equation}
 Moreover, for $\beta,\gamma\in\IPspace$,
 \begin{gather}
  \dH'(\beta,\scaleI[c][\beta]) = \left|c-1\right|\IPmag{\beta}, \qquad \dH'(c\scaleI\beta,c\scaleI\gamma) = c\dH'(\beta,\gamma),\label{eq:IP:Haus_scale}\\
  \dI(\beta,\scaleI[c][\beta]) \leq \max\left\{\left|\sqrt{c} - 1\right|\IPLT_{\beta}(\infty), \left|c-1\right|\IPmag{\beta}\right\},\label{eq:IP:scaling_dist_1}\\
  \text{and}\quad\min\{c,\sqrt{c}\}\dI(\beta,\gamma) \leq \dI(\scaleI[c][\beta],\scaleI[c][\gamma]) \leq \max\{c,\sqrt{c}\}\dI(\beta,\gamma).\label{eq:IP:scaling_dist_2}
 \end{gather}
\end{lemma}

\section{Ingredients for the Poissonian construction}
\label{sec:prelim}

The Poissonian construction will be based on several $\sigma$-finite measures and their disintegrations along certain one-dimensional statistics or more general
measurable functions. We set up a common framework for disintegrations before turning to the ingredients for the construction. 

\begin{definition}[Pushforward and disintegration]\label{def:disintegration} Let $(\cS,\ScS,\mu)$ be a $\sigma$-finite measure space and $\phi\colon\cS\to\cT$
  a measurable function into a measurable space $(\cT,\ScT)$. We use notation $\mu(\phi\in B):=\phi_*\mu(B):=\mu(\phi^{-1}(B))$, $B\in\ScT$, for the \em pushforward \em 
  of $\mu$ under $\phi$. A \emph{$\phi$-disintegration of $\mu$} is a map $y\mapsto\mu_y$ from $\cT$ to probability measures on $(\cS,\ScS)$ with the following properties.
 \begin{enumerate}[label=(\roman*), ref=(\roman*)]
  \item For each $A\in\ScS$ fixed, the map $y\mapsto \mu_y(A)$ is $\ScT$-measurable.
  \item For each $y\in \cT$, we have $\mu_y\{x\in \cS\colon\phi(x)\neq y\} = 0$.
  \item Writing $\nu=\phi_*\mu$, we have 
 \begin{align}
  \mu(A) = \int \mu_y(A)d\nu(y) \qquad \text{for every }A\in\ScS.\label{eq:scl_ker:integration}
 \end{align}
 \end{enumerate}
\end{definition}

This is narrower in several respects than the most general definition of a disintegration. In general, disintegrations are not required to give probability kernels. For an excellent, broader treatment, see \cite{ChanPoll97}.

In the setting of excursion theory, it is a standard technique to view a $\sigma$-finite excursion measure as a product between a $\sigma$-finite measure on $(0,\infty)$ and a probability distribution on the space of normalized excursions. This is possible when the excursion measure has an invariance property relative to some scaling map. This product decomposition is equivalent to a unique disintegration: to each $a>0$ we associate a probability distribution on excursions of length $a$, by scaling the random normalized excursion by a factor of $a$. See \cite{Chaumont97} for work related to this method in the setting of stable L\'evy processes. Williams' decomposition of Brownian excursion fits into this framework, with excursions normalized by amplitude rather than length; see \cite[Theorem 22.15]{Kallenberg} or \cite[Theorems XII.4.2 and XII.4.5]{RevuzYor}.

Here is a more general statement of this principle that scaling maps give rise to unique disintegrations. We leave its proof to the reader.

\begin{lemma}[Scaling and disintegration]\label{lem:scl_ker} Let $(\cS,\ScS,\mu)$ be a $\sigma$-finite measure space and $\theta\colon \cS\to (0,\infty)$ a measurable 
 function. Suppose $\scale\colon (0,\infty)\times \cS \to \cS$ is a measurable \emph{scaling map} in the sense that there exists some $p>0$ such that, for every $b,c>0$, $x\in\cS$, and $A\in \ScS$,
 \begin{gather}
  b \scale \left(c \scale x\right) = \scale[bc][x], \qquad \scale[1][x] = x, \quad \text{and} \quad \mu(\scale[c][A]) = c^{-p}\mu(A).\label{eq:scl_ker:inv}
 \end{gather}
 Suppose further that there exists some $q>0$ such that for every $c>0$ and $x\in \cS$,
 \begin{equation}
  \theta(\scale[c][x])  = c^q\theta(x) \quad \text{and} \quad \mu\{x\in\cS\colon \theta(x) > 1\} < \infty.\label{eq:scl_ker:finite}
 \end{equation}
 
 Then there exists a $\theta$-disintegration of $\mu$, which we denote by $a\mapsto \mu_a$, unique with the property that for every $a>0$, if $X$ has law $\mu_a$ then 
 $\scale[a^{-1/q}][X]$ has law $\mu_1$. Moreover, for each $b>0$, the measure $\mu_1$ equals the pushforward of $\mu(\,\cdot\;|\,\theta(x) > b)$ under the map
 $x\mapsto\scale[\theta(x)^{-1/q}][x]$.
\end{lemma}

 In the context of Lemma \ref{lem:scl_ker} we write $\mu(\,\cdot\;|\,\theta=a)$ to denote the measure called $\mu_a$ in the statement of the lemma. We denote the resulting $\theta$-disintegration by $\mu(\,\cdot\;|\,\theta)$. The role of the scaling map is suppressed in this notation, but this will not pose a problem as we will not associate multiple scaling maps with any single measure space.

\subsection{Spindles: \BESQ ($-1$) excursions for use as block size evolutions}
\label{sec:BESQ}

Let $(\cD,d_\cD)$ denote the Skorohod space of real-valued c\`adl\`ag functions. Recall that its Borel $\sigma$-algebra $\Sigma(\cD)$ is generated by the evaluation
maps $g\mapsto g(t)$, $t\in\mathbb{R}$; see \cite[Theorem 14.5]{Billingsley}. Let $\Exc$ be the subset of non-negative real-valued excursions that are continuous, possibly excepting \cadlag\ jumps at their times of birth (time 0 as elements of $\Exc$) and death:
 \begin{equation}
  \Exc := \left\{f\colon\BR\to [0,\infty)\ \middle| \begin{array}{c}
    \displaystyle \exists\ z\in(0,\infty)\textrm{ s.t.\ }\restrict{f}{(-\infty,0)\cup [z,\infty)} = 0,\\[0.2cm]
    \displaystyle \restrict{f}{(0,z)} > 0\text{, and }\restrict{f}{[0,z)}\textrm{ is continuous}
   \end{array}\right\}.\label{eq:cts_exc_space_def}
 \end{equation}
Let $\SExc$ denote the Borel $\sigma$-algebra on $\Exc$ generated by $d_\cD$. We define the \emph{lifetime} and \emph{amplitude} 
$\life,A\colon \Exc \to (0,\infty)$ via
 \begin{equation}
  \life(f) = \sup\{s\geq 0\colon f(s) > 0\},\quad\mbox{and}\quad A(f)=\sup\{f(s),0\le s\le\zeta(f)\}.
 \end{equation}

Squared Bessel processes (\BESQ) are a one-parameter family of diffusions. They contain the Feller diffusion, which is a continuous-state 
branching process, when the dimension parameter $\delta\in\BR$ is $\delta=0$, with immigration when $\delta>0$. The squared norm of a $\delta$-dimensional Brownian 
motion is a \BESQ[\delta] starting from 0, when $\delta\in\BN$. The case $\delta=-1$ can be interpreted as emigration at unit rate. In this case (as when $\delta=0$), 
the boundary point 0 is not an entrance boundary, while exit at 0 (we then force absorption) happens almost surely. See \cite{PitmYor82,GoinYor03,Pal13}. 

\begin{lemma}[Equation (13) in \cite{GoinYor03}]\label{lem:BESQ:length}
  Let $Z=(Z_s,s\ge 0)$ be a \BESQ[-1] process starting from $z >0$. Then the absorption time $\zeta(Z)=\inf\{s\ge 0\colon Z_s=0\}$ has distribution 
  \InvGammaDist[3/2,z/2], i.e.\ $z/2\zeta(Z)$ has density $(\Gamma(3/2))^{-1}x^{1/2}e^{-x}$, $x\in(0,\infty)$. 
\end{lemma}

Pitman and Yor \cite{PitmYor82} constructed excursion measures $\Lambda$ for diffusions even when there is no reflecting extension (to replace absorption at 0) that has $\Lambda$ as its 
It\^o excursion measure. They gave several descriptions, the first of which yields the following for the special case of \BESQ[-1]. For the purpose of the following, we define \emph{first passage times} $H^a\colon\Exc\to[0,\infty]$ via $H^a(f)=\inf\{s\ge 0\colon f(s)=a\}$, $a>0$.

\begin{lemma}[Section 3 of \cite{PitmYor82}]\label{lem:BESQ:existence}
  There is a measure $\Lambda$ on $\Exc$ such that $\Lambda\{f\in\Exc\colon f(0)\neq 0\}=0$, 
  $\Lambda\{H^a<\infty\}=a^{-3/2}$, $a>0$, and under $\Lambda(\,\cdot\;|\,H^a<\infty)$, the restricted canonical process 
  $f|_{[0,H^a]}$ is a \BESQ[5] process starting from 0 and stopped at the first passage time of $a$, independent of $f(H^a+\cdot\,)$, which is a \BESQ[-1] process 
  starting from $a$. 
\end{lemma}  

\begin{proposition}\label{prop:E_bdedly_finite}
  We define a metric $d_{\cD}^A$ on $\Exc$ via $d_{\cD}^A(f,g) := d_{\cD}(f,g) + \big|A(f)^{-1} - A(g)^{-1}\big|$, where $d_{\cD}$ is the Skorokhod metric and $A$ the 
  amplitude. This metric is topologically equivalent to $d_{\cD}$ on $\Exc$. Moreover, $(\Exc,d_{\cD}^A)$ is separable and complete and $\Lambda$ is boundedly finite 
  under this metric, i.e.\ $\Lambda(B)<\infty$ for all measurable $B\subset\Exc$ that are bounded in $(\Exc,d_\cD^A)$.
\end{proposition}
\begin{proof}
 Note that $\Exc\cup\{0\}$ is complete under $d_{\cD}$ and $1/A$ is continuous under $d_{\cD}$ on $\Exc$. It is easily confirmed that a $d_{\cD}$-Cauchy sequence in $\Exc$ with non-zero limit converges to the same limit in $d_{\cD}^A$, and a sequence converging to $0$ under $d_{\cD}$ cannot be Cauchy under $d_{\cD}^A$. This proves the completeness of $d_{\cD}^A$ and its topological equivalence to $d_{\cD}$. Separability follows from topological equivalence. The reader may confirm that $\Lambda$ is boundedly finite under $d_{\cD}^A$.
\end{proof}

We will consider a constant multiple of $\Lambda$ as an intensity of a Poisson random measure on $[0,\infty)\times\Exc$, using the framework of 
\cite{DaleyVereJones1,DaleyVereJones2} as our reference. 
In the setting of scaffoldings and spindles discussed in the introduction, changing the intensity by a constant corresponds to time-changing the scaffolding, which will not impact the skewer map of Definition \ref{def:skewer} or our interval partition diffusions. 
It will help with calculations 
to make the following choice.

\begin{definition}
 We define our \BESQ[-1] \emph{excursion measure} as $\mBxc=(3/2\sqrt{\pi})\Lambda$, where $\Lambda$ is the Pitman-Yor excursion measure of Lemma \ref{lem:BESQ:existence}. We call continuous elements of $\Exc$ such as $\mBxc$-a.e.\ $f\in\Exc$ \emph{spindles} and elements with a discontinuity at birth or death \emph{broken spindles}.
\end{definition}

While every spindle $f\in\Exc$ has an intrinsic lifetime $\life(f)\in[0,\infty)$, the scaffolding of Section \ref{sec:prelim:JCCP} will shift spindles to non-zero birth
times that are not intrinsic to each spindle.   

\begin{lemma}\label{lem:BESQ:exc_length} 
  $\displaystyle\mBxc\{f\in\Exc\colon A \geq m\}=\frac{3}{2\sqrt{\pi}}m^{-3/2}$ 
  and $\displaystyle\mBxc\{f\in\Exc\colon \life(f) \geq y\}= \frac{1}{\pi\sqrt{2}}y^{-3/2}$.
\end{lemma}
\begin{proof} The first formula follows straight from Lemma \ref{lem:BESQ:existence}. We can calculate the second one using \cite[Description (3.2)]{PitmYor82} to express
  $\Lambda\{\life>s\}$ in terms of a
  \BESQ[5] process $Z$ starting from 0, whose probability density function at time $s$ is given in \cite[Equation (50)]{GoinYor03}: 
  $$\Lambda(\life>s)=\EV[Z_s^{-3/2}]=\int_0^\infty y^{-3/2}(2s)^{-5/2}\Gamma(5/2)^{-1}y^{3/2}e^{-y/2s}dy=\frac{\sqrt{2}}{3\sqrt{\pi}}s^{-3/2}.$$\vspace{-0.8cm}
  
\end{proof}

\begin{definition}\label{def:BESQ:scaling_def}
 We define a \emph{reversal involution} $\reverseexc\colon\Exc\rightarrow\Exc$ and a \emph{scaling map} $\scaleB\colon (0,\infty)\times\Exc\rightarrow\Exc$ by saying, for $c>0$ and $f\in\Exc$,
 \begin{equation}
  \reverseexc (f) := \big(f\big((\life(f)-y)-\big),\ y\in\BR\big) \qquad \text{and} \qquad  c \scaleB f := \left(cf\left(y/c\right),\ y\in\BR\right).
 \end{equation}
\end{definition}

\begin{lemma}\label{lem:BESQ:invariance}
 For $A\in \cExc$ and $c>0$,
 \begin{equation}
  \mBxc(\reverseexc(A)) = \mBxc(A)\quad\text{and}\quad
  \mBxc(\scaleB[c][A]) = c^{-3/2}\mBxc(A).\label{eq:BESQ:scaling_inv}
 \end{equation}
\end{lemma}
\begin{proof} The time reversibility can be read from \cite[(3.3)]{PitmYor82}. The scaling relation follows from Lemma \ref{lem:BESQ:existence} and the scaling
  properties of \BESQ[-1] and \BESQ[5] as noted e.g.\ in \cite[A.3]{GoinYor03}. 
\end{proof}

Note that, by Definition \ref{def:BESQ:scaling_def} and \eqref{eq:BESQ:scaling_inv}, the scaling and lifetime maps $\scaleB$ and $\life$ satisfy the hypotheses of Lemma \ref{lem:scl_ker}, with $p = 3/2$ and $q=1$. Thus, the conclusions of that lemma apply to these two maps.

\begin{corollary}\label{cor:BESQ:scl_ker}
  There exists a $\life$-disintegration of $\mBxc$, denoted by $\mBxc(\,\cdot\;|\,\life)$, that is unique with the following property. For every $a,b>0$, if $f$ is random with law $\mBxc(\,\cdot\;|\,\life = a)$ then $\scaleB[(b/a)][f]$ has law $\mBxc(\,\cdot\;|\,\life = b)$.
\end{corollary}

Finally, we will require better smoothness than just continuity for \BESQ[-1] excursions.

\begin{definition}
 \begin{enumerate}[label=(\roman*), ref=(\roman*)]
  \item A function $f\colon\cS\to\cT$ between metric spaces $(\cS,d_\cS)$ and $(\cT,d_\cT)$ is \emph{H\"older continuous with parameter $\theta$} if, for some $C\in (0,\infty)$,
  \begin{align}
   \sup_{r,s\in \cS}\frac{d_\cT(f(r),f(s))}{d_\cS(r,s)^{\theta}} \leq C.
  \end{align}
  For brevity, we will say that $f$ is \emph{H\"older-$\theta$} with \emph{H\"older constant $C$}.
  
  \item A family of functions $(f_a)_{a\in \cA}$ is \emph{uniformly H\"older-$\theta$} if, for some $C\in(0,\infty)$, 
  \begin{align}
   \sup_{a\in\cA}\sup_{r,s\in \cS}\frac{d_\cT(f_a(r),f_a(s))}{d_\cS(r,s)^{\theta}} \leq C.
  \end{align}
  
  \item A bivariate function $f\colon \cS\times \cT \to \mathcal{U}$ into a third metric space $(\mathcal{U},d_\mathcal{U})$ is said to be \emph{uniformly H\"older-$\theta$ in its first coordinate} if the family of functions $(f(\cdot,t))_{t\in \cT}$ is uniformly H\"older-$\theta$, and correspondingly for the second coordinate.
 \end{enumerate}
\end{definition}

\begin{lemma}[e.g.\ Corollary 3 of \cite{Paper0}]\label{lem:BESQ:Holder}
 For every $\theta\in (0,\frac12)$, $\mBxc$-a.e.\ excursion is H\"older-$\theta$.
\end{lemma}

In Lemma \ref{lem:spindle_piles}, we exhibit subfamilies of the \BESQ[-1] excursions of a \PRM[\Leb\otimes\mBxc] that are uniformly H\"older-$\theta$. 
Here, \PRM[\mu] is short for Poisson random measure with intensity $\mu$.

\subsection{Scaffolding: \Stable[\frac32] processes to describe births and deaths of blocks}
\label{sec:prelim:JCCP}

Let $\bN$ denote a \PRM[\Leb\otimes\mBxc] on $[0,\infty)\times \Exc$. The key idea of the pair $(N,X)$ and the skewer map of Definition \ref{def:skewer} is to 
associate with each atom $(t,f)$ of $N$ a spindle birth time $X(t-)$ and to extract $f(y-X(t-))$ as its mass at level $y\in\BR$. Furthermore, the scaffolding $X$ is such
that $\Delta X(t):= X(t)-X(t-)=\zeta(f)$. By Lemma \ref{lem:BESQ:exc_length} and standard mapping of Poisson random measures, $\int\delta(s,\zeta(f))d\bN(s,f)$ is a
\PRM[\Leb\otimes\mBxc(\zeta\in\cdot\,)], where 
$$\int_{(z,\infty]}x\,\mBxc(\zeta\in dx)=\int_\Exc\cf\{\zeta(f)>z\}\zeta(f)d\mBxc(f) = \frac{3}{\pi\sqrt{2}}z^{-1/2}\longrightarrow\infty\qquad\text{as }z\downarrow 0.$$
Thus, if we take the lifetimes of the spindles $f$ that occur as points in $\bN$ to be the heights of jumps for a \cadlag\ path, in the manner of 
\eqref{eq:discrete_JCCP_eg}, then these jumps are almost surely not summable. To define a path $\bX$ associated with $\bN$ in this manner, we require a limit with
compensation. We give a general definition that will also apply to random measures constructed in various ways from independent copies of $\bN$. Our reference for  
measures on Polish metric spaces is \cite{DaleyVereJones1,DaleyVereJones2}.

\begin{definition}\label{def:JCCP} For a complete, separable metric space $(\cS,d_\cS)$, denote by 
  $\cN(\cS)$ the set of counting measures $N$ on $\cS$ that are boundedly finite: $N(B)<\infty$ for all bounded Borel sets $B\subset\cS$. We equip 
  $\cN(\cS)$ with the $\sigma$-algebra $\Sigma(\cN(\cS))$ generated by evaluation maps $N\mapsto N(B)$. 

  Recall $(\Exc,d_{\cD}^A)$ from Proposition \ref{prop:E_bdedly_finite}. We metrize $[0,\infty)\times\Exc$ via the sum of the Euclidean metric in the first coordinate and $d_{\cD}^A$ in the second. For $N\in \cNRE$, let 
 \begin{equation}
  \len(N) := \inf\Big\{t>0\colon N\big([t,\infty)\times\Exc\big) = 0\Big\} \in [0,\infty].\label{eq:assemblage:len_def}
 \end{equation}
 When the following limit exists for $t\in [0,\len(N)]\cap [0,\infty)$, we further define 
 \begin{equation}
  \xi_N(t) := \lim_{z\downto 0}\left(\int_{[0,t]\times\{g\in\Exc\colon\zeta(g) > z\}}\life(f)dN(s,f) - \frac{3tz^{-1/2}}{\pi\sqrt{2}}\right).\label{eq:JCCP_def}
 \end{equation}   
 When typographically convenient, we will also write
 $\xi(N) := \big( \xi_N(t),\ t\in [0,\len(N)]\cap [0,\infty) \big)$.
\end{definition}

Note that the limit in \eqref{eq:JCCP_def} does not exist for every $N$ and $t$.

\begin{proposition}\label{prop:stable_JCCP}
  For $\bN$ a \PRM[\Leb\otimes\mBxc] on $[0,\infty)\times \Exc$, the convergence in \eqref{eq:JCCP_def} holds a.s.\ uniformly in $t$ on any bounded interval.   
  Moreover, the scaffolding $\xi(\bN)$ is a spectrally positive stable L\'evy process of index $3/2$, with L\'evy measure and Laplace exponent given by 
 \begin{equation}
  \mBxc(\zeta\in dx)=\frac{3}{2\pi\sqrt{2}}x^{-5/2}dx\qquad\mbox{and}\qquad\psi(\lambda) = \sqrt{\frac{2}{\pi}}\lambda^{3/2}.\label{eq:JCCP:Laplace}
 \end{equation}
\end{proposition}

\begin{proof} By Lemma \ref{lem:BESQ:exc_length} and elementary Poisson random measure arguments, the pre-limiting quantity is a compensated compound Poisson process.
  By the L\'evy-It\^o decomposition of L\'evy processes, e.g.\ in \cite[Theorem 19.2]{Sato}, the remaining conclusions follow. Specifically, we obtain 
  $$\psi(\lambda)=\int_0^\infty(e^{-\lambda x}-1+\lambda x)\mBxc(\life\in dx)=\int_0^\infty(e^{-\lambda x}-1+\lambda x)\frac{3}{2\pi\sqrt{2}}x^{-5/2}dx=\sqrt{\frac{2}{\pi}}\lambda^{3/2}.\quad$$
  \vspace{-0.8cm}
  
\end{proof}

\begin{notation}
 Henceforth we write ``\Stable[\frac32]'' to refer exclusively to L\'evy processes with the Laplace exponent specified in \eqref{eq:JCCP:Laplace}. In particular, such processes are spectrally positive. 
  We write $\bX := \xi(\bN)$.
\end{notation}

\begin{definition}[$\H$, $\Hfin$, point processes of spindles]\label{def:assemblage_m}
 Let $\Hfin\subset \cNRE$ denote the set of all counting measures $N$ on $[0,\infty)\times \Exc$ with the following additional properties:
 \begin{enumerate}[label=(\roman*), ref=(\roman*)]
  \item $N\big( \{t\}\times\Exc \big)\leq 1$ for every $t\in [0,\infty)$,
  \item $N\big( [0,t]\times\{f\in\Exc\colon \life(f) > z\} \big) < \infty$ for every $t,z > 0$, 
  \item the \emph{length} of $N$, defined to be $\len(N)$ in the sense of Definition \ref{def:JCCP}, is finite and the convergence in (\ref{eq:JCCP_def}) holds 
    uniformly in $t\in[0,\len(N)]$.
 \end{enumerate}
 
 We define $\H\subset\cNRE$ by saying $N\in\H$ if and only if $\restrict{N}{[0,t]} \in \Hfin$ for every $t>0$. We call the members of $\H$ \emph{point processes of spindles}. We denote by $\cH$ and $\cHfin$ the restrictions of $\ScNRE$ to the subsets of $\H$ and $\Hfin$ respectively.
\end{definition}

\begin{proposition}\label{prop:H_Lusin} 
 The $\sigma$-algebra $\ScNRE$ on $\cNRE$ generated by the evaluation maps is the Borel
 $\sigma$-algebra of a Polish topology. The sets $\H$ and $\Hfin$ are Borel-measurable subsets of $\cN\big([0,\infty)\times\Exc\big)$.
\end{proposition}

\begin{proof}
 By Proposition \ref{prop:E_bdedly_finite}, $(\Exc,d_\cD^A)$ is complete and separable. Thus, the first assertion follows from \cite[Theorem A2.6.III]{DaleyVereJones1}. 
 The measurability of condition (i) in Definition \ref{def:assemblage_m} follows from the existence of measurable enumerations of points of counting measures; see 
 \cite[Proposition 9.1.XII]{DaleyVereJones2}. Finally, the measurability of conditions (ii) and (iii) in Definition \ref{def:assemblage_m} follows from the measurability
 of the evaluation maps $N\mapsto N(B)$, $B\in\ScNRE$.
\end{proof}

\begin{proposition}\label{prop:JCCP_meas}
 The map $\xi\colon \H \to \cD$ is well-defined and measurable, where $\cD$ is the space of real-valued \cadlag\ functions equipped with the Skorohod topology.
\end{proposition}

\begin{proof}
 Definition \ref{def:assemblage_m} (iii) ensures that for $N\in\H$, the convergence to $\xi_N(t)$ in \eqref{eq:JCCP_def} holds uniformly in $t$ for any bounded interval.
 Uniform limits preserve the \cadlag\ property, so $\xi(N)$ exists and is \cadlag. By definition of $\SH$, the function  
 $N\mapsto\xi_N(t)$ is measurable for each $t\geq 0$. By \cite[Theorem 14.5]{Billingsley}, the Borel $\sigma$-algebra $\ScD$ associated with the Skorohod topology is 
 generated by evaluation maps $g\mapsto g(t)$, $t\ge 0$, so we conclude that $\xi$ is measurable.
\end{proof}

Most of the constructions in this paper begin with a point process $N\in\H$ and from there obtain a scaffolding $X=\xi(N)$. However, it is useful to be able to go in the other direction, to begin with a scaffolding $X$ and to define a point process $N\in\H$ by marking the jumps of $X$ with continuous excursions. 

\begin{proposition}[The \PRM\ of spindles via marking jumps]\label{prop:marking_jumps}
 Let $\bX$ denote a \Stable[\frac32] process with Laplace exponent as in \eqref{eq:JCCP:Laplace}. Let $\bM = \sum_{t\geq 0\colon\Delta \bX(t)>0}\Dirac{t,\Delta\bX(t)}$.
 Use the marking kernel $x\mapsto\mBxc(\,\cdot\;|\,\zeta=x)$ to mark each point $(t,\Delta\bX(t))$ of $\bM$ by an excursion $f_t$ with length $\life(f_t)=\Delta\bX(t)$.
 Then $\bN := \sum_{t\geq 0\colon\Delta\bX(t)>0}\Dirac{t,f_t}$ is a \PRM[\Leb\otimes\mBxc] and $\bX = \xi(\bN)$.
\end{proposition}

\begin{proof} 
 Since $\bX$ is a L\'evy process, $\bM$ is a \PRM. By \eqref{eq:JCCP:Laplace}, its intensity is $\mBxc(\life\in\cdot\,)$. 
 It is well-known that marking constructions like that above result in \PRM s; see \cite[Proposition 6.4.VI]{DaleyVereJones1}. Thus, $\bN$ is a \PRM. 
 Since $\mBxc(\,\cdot\;|\,\life)$ is a disintegration of $\mBxc$,
 it follows from the defining property \eqref{eq:scl_ker:integration} of disintegrations that $\bN$ has intensity $\Leb\otimes\mBxc$.
\end{proof}

\begin{definition}\label{def:concat}
 Let $(N_a)_{a\in\mathcal{A}}$ denote a family of elements of $\H$ indexed by a totally ordered set $(\cA,\preceq)$. For the purpose of this definition, set
 \begin{equation}\label{eq:length_partial_sums}
  S(a) := \sum_{b\preceq a}\len(N_b) \quad \text{and} \quad S(a-) := \sum_{b\prec a}\len(N_b) \quad \text{for each }a\in \mathcal{A}.
 \end{equation}
 If $S(a-) < \infty$ for every $a\in \mathcal{A}$ and if for every consecutive $a\prec b$ in $\mathcal{A}$ we have $N_a(\{\len(N_a)\}\times\Exc) + N_b(\{0\}\times\Exc)\leq 1$, then we define the \emph{concatenation} of $(N_a)_{a\in\mathcal{A}}$ to be the counting measure
 $$\Concat_{a\in\mathcal{A}}N_a := \sum_{a\in\mathcal{A}}\int\Dirac{S(a-)+t,f}dN_a(t,f).$$
\end{definition}

We can now give a formal version of the construction stated before Theorem \ref{thm:diffusion}.

\begin{definition}[$\Pr^1_{\beta}$, pre-type-1 evolution]\label{constr:type-1}
 Take $\beta\in\IPspace$. If $\beta = \emptyset$ then $\bN_{\beta} := 0$. Otherwise, for each $U\in\beta$ we carry out the following construction independently. Let $\bN$ denote a \PRM[\Leb\otimes\mBxc], let $\bff$ be an independent \BESQ[-1] started from $\Leb(U)$ and absorbed at 0, and consider the hitting time $T := \inf\{t>0\colon \xi_{\bN}(t) = -\life(\bff)\}$. Let $\bN_U := \Dirac{0,\bff}+\restrict{\bN}{[0,T]}$. Let $\bN_{\beta} := \ConcatIL_{U\in\beta}\bN_U$. 
 We write $\Pr^1_{\beta}$ to denote the law of $\bN_{\beta}$ on $\Hfin$. For probability distributions $\mu$ on $\IPspace$, we write $\Pr^1_{\mu} := \int \Pr^1_{\beta}\mu(d\beta)$ to denote the $\mu$-mixture of the laws $\Pr^1_{\beta}$. 
 Recalling Definition \ref{def:skewer}, 
 we call $(\alpha^y,\,y\ge0) := \skewerP(\bN_{\beta},\xi(\bN_{\beta}))$ a \emph{pre-type-1 evolution} starting from $\beta$.
\end{definition}

We confirm in Proposition \ref{prop:clade_lengths_summable} that $\sum_{U\in\beta}\len(\bN_U)<\infty$ a.s.\ for each $\beta\in\IPspace$ and that $\beta\mapsto\Pr^1_{\beta}$ is a stochastic kernel. In Lemma \ref{lem:type-1:wd} we show that $\alpha^y\in\IPspace$ for every $y\geq 0$, a.s.. Finally, in Proposition \ref{prop:type-1:cts}, we show that pre-type-1 evolutions admit continuous versions; following Definition \ref{def:type-1}, the latter will be called type-1 evolutions.

\subsection{Scaffolding levels: excursion theory for \Stable[\frac32] processes}
\label{sec:exc}

Excursion theory for Markov processes was first developed by It\^o \cite{Ito72}. Bertoin \cite[Ch.\ IV]{BertoinLevy} offers a nice treatment of this theory in the setting of L\'evy processes. 

\begin{definition}\label{def:LT}
 For a \cadlag\ function $g\colon [0,\infty) \to\BR$, a bivariate measurable function $(y,t)\mapsto \ell^y(t)$ from $\BR\times [0,\infty]$ to $[0,\infty]$, is an \emph{(occupation density) local time} for $g$ if $t\mapsto\ell^y(t)$ is increasing for all $y\in\BR$ and if for every bounded and measurable $f\colon \BR\to[0,\infty)$,
 \begin{equation}
  \int_{-\infty}^{\infty} f(y)\ell^y(t)dy = \int_0^t f\big(g(s)\big)ds.\label{eq:LT_int_ident}
 \end{equation}
 We call $t$ the \emph{time parameter} and $y$ the \emph{space parameter} and say $\ell^y(t)$ is the local time of $g$ at level $y$, up to time $t$.
\end{definition}

\begin{theorem}[Boylan \cite{Boylan64}, equations (4.4) and (4.5)]\label{thm:Boylan}
 For $\bX$ a \Stable[\frac32] process, there exists an a.s.\ unique jointly continuous stochastic process $\ell=(\ell^y(t);\ y\in\BR,t\geq 0)$ such that $\ell$ is almost surely a local time for $\bX$. It is a.s.\ the case that for every $t\geq 0$ and $y\in\BR$
 \begin{equation}
  \ell^y(t) = \lim_{h\downto 0} \frac{1}{h}\int_0^t \cf\{y-h<\bX(s)<y\}ds=\lim_{h\downto 0} \frac{1}{h}\int_0^t \cf\{y<\bX(s)<y+h\}ds.\label{eq:LT_def}
 \end{equation}
 Moreover, for every $\theta\in (0,\frac15)$, $\theta'\in (0,\frac14)$, and each bounded space-time rectangle $R$, the restriction of $(y,t)\mapsto\ell^y(t)$ to $R$ is uniformly H\"older-$\theta$ in the time coordinate and uniformly H\"older-$\theta'$ in the space coordinate.
\end{theorem}

\begin{definition}\label{def:stable_exc}
 We define the following subsets of the Skorokhod space $\cD$:
 \begin{align}\label{eq:stable_exc_space_def}
  \DS &:= \xi(\H)=\left\{\xi(N)\colon N\in\H\right\},\\[0.2cm]
  \DSxc &:= \left\{\xi(N)\colon N\in\Hfin,\ \xi_N|_{(0,\len(N))}\neq 0,\ \xi_N(\len(N))=0\right\}.
 \end{align}
 We take $\SDS$ and $\SDSxc$ to denote the $\sigma$-algebras on these spaces generated by the evaluation maps. We say that members of $\DS$ are \emph{\Stable[\frac32]-like processes} and members of $\DSxc$ are \emph{\Stable[\frac32]-like excursions}.
 
 For $g=\xi(N)\in\DS$, we define the \emph{length} of $g$ to be $\len(g)=\len(N)$. We will use the convention $g(0-)=0$.  
\end{definition}

Note that if $g\in \DSxc$ then $\len(g) = \sup\{t>0\colon g(t)\neq 0\}.$ 

\begin{proposition}\label{prop:stable_Skor_meas}
 The sets $\DS$ and $\DSxc$ are measurable subsets of the Skorohod space $\cD$.
\end{proposition}

This is a straightforward exercise in topology and measure theory, starting by expressing the conditions on $N$ of Definition \ref{def:assemblage_m} in
terms of the \cadlag\ function $g=\xi(N)$.

\begin{definition}\label{def:inverse_LT}
 For $y\in\BR$ and $g\in\DS$ we make the following definitions.
 \begin{enumerate}[label=(\roman*), ref=(\roman*)]
  \item We define the set $V^y(g)$ of intervals of complete excursions, i.e.\ those beginning and ending at $y$ in the following weak sense,\label{item:i_LT:exc_int}
    $$
    V^y(g) := \left\{[a,b]\subset [0,\len(g)]\ \middle|\begin{array}{r}
    		 a < b < \infty;\hfill g(t-)\neq y\neq g(t)\text{ for }t\in (a,b);\\[.2cm]
    		 g(a-) = y \text{ or }g(a) = y; 
    		 \text{ and } g(b-) = y \text{ or }g(b) = y
         	\end{array}\!\!\right\}.$$
  \item We define the set $V^y_0(g)\supseteq V^y(g)$ to include incomplete first and/or last excursions. Let
  \begin{equation*}
  \begin{split}
   T^y(g)   &:= \inf\left(\left\{t\in [0,\len(g)]\colon g(t)=y \text{ or }g(t-)=y\right\}\cup\{\len(g)\}\right),\\
   T^y_*(g) &:= \sup\left(\left\{t\in [0,\len(g)]\colon g(t)=y \text{ or }g(t-)=y\right\}\cup\{0\}\right)
  \end{split}
  \end{equation*}
  If $y\neq 0$ then we include $[0,T^y(g)]\cap [0,\infty)$ in $V^y_0(g)$. If $T^y_*(g) < \len(g)$ or $g(\len(g)) \neq y$, then we include $[T^y_*(g),\len(g)]\cap [0,\infty)$ in $V^y_0(g)$.
  \label{item:i_LT:exc_int_0}
  \item If the level $y$ local time associated with $g$ exists, in the sense that the limits \eqref{eq:LT_def} exist for all $t\ge 0$, with $g$ in the place of $\bX$, 
    then we denote this by $(\ell^y_g(t),t\geq 0)$. Then the \emph{(right-continuous) inverse local time} is $\tau^y(s) := \inf\{t \geq 0\colon\ell^y_g(t) > s\}$ for $s\geq 0$.\label{item:i_LT:LT}
 \end{enumerate}
 In each of the preceding notations we may replace $g$ with $N\in\H$ to denote the corresponding object with $g = \xi(N)$.
\end{definition}

In the sequel, we will suppress the `$g$' in the above notations when we refer to these objects applied to $g = \bX$, where $\bX$ denotes the 
\Stable[\frac32] scaffolding process $\xi(\bN)$, as in Proposition \ref{prop:stable_JCCP}. Let $\ell$ denote the jointly H\"older continuous version of local time 
specified in Theorem \ref{thm:Boylan}.
 Note that for $y\in\BR$ fixed, $T^y = \tau^y(0)$ a.s., but this is not simultaneously the case for all $y\in\BR$.

\begin{definition}\label{def:shifted_restriction}
  Let $[a,b]$ be an interval, $f\colon\BR\to\cS$ a function and $N\in\cN([0,\infty)\times\cS)$ a counting measure. We define \emph{shifted restrictions} by setting for $x\in\BR$, $B\in\ScS$ and $I\subset\BR$ Borel
 $$
 \ShiftRestrict{f}{[a,b]}(x):=\Restrict{f}{[a,b]}(x+a),\quad\mbox{and}\quad\ShiftRestrict{N}{[a,b]}(I\times B):=\Restrict{N}{[a,b]}((I+a)\times B).
 $$
 We will use similar notation with open and half-open intervals $(a,b)$, $[a,b)$, and $(a,b]$.

  Adapting the definition of $d_\cD^A$ of Proposition \ref{prop:E_bdedly_finite}, we equip $\DSxc\setminus\{0\}$ with the
  metric $d_\cD^A(f,g)=d_\cD(f,g)+|A(f)^{-1}-A(g)^{-1}|$, where 
  $A(g):=\sup\{|g(t)|,0\le t\le\len(g)\}$ with $\len(g)$ as in Definition \ref{def:stable_exc}.
  We define an excursion counting measure and an associated intensity
 \begin{align}
  \mathbf{G}^y &:= \sum_{(a,b)\in V^y} \Dirac{\ell^y(a),\ShiftRestrict{\bX}{[a,b]}};\label{eq:stable:exc_PRM}\\
  \mSxc(B) &:= \EV[\mathbf{G}^0([0,1]\times B)] \quad \text{for }B\in\Sigma(\DSxc\setminus\{0\}).\label{eq:stable_exc_measure}
 \end{align}
\end{definition}

\begin{theorem}[
see e.g.\ \cite{BertoinLevy} Theorems IV.8 and IV.10]\label{thm:excursion_PRM} 
\begin{enumerate}[label=(\roman*), ref=(\roman*)]\item[(i)] The measure $\mSxc$ is boundedly finite for $(\DSxc\setminus\{0\},d_\cD^A)$, and it satisfies $\int (\len(g)\wedge 1)d\mSxc(g) < \infty$.
 \item[(ii)] For each $y\in\BR$, the inverse local-time $\tau^y(s)$ is a.s.\ finite for every $s\ge 0$, and $\mathbf{G}^y$ is a \PRM[\Leb\otimes\mSxc] on 
   $[0,\infty)\times\DSxc$.
\end{enumerate}
\end{theorem}

This theorem has the consequence that, for fixed $y\in\BR$,
\begin{align}
 V^y = \big\{[\tau^y(t-),\tau^y(t)]\colon\ t>0,\ \tau^y(t-)\neq \tau^y(t)\big\}\quad\text{a.s.}\label{eq:inv_LT_intervals}
\end{align}

\begin{definition}\label{def:stable:xforms}
 Let $\reverseincr$ denote the increment-reversal involution on excursions $g\in \DSxc$:
 \begin{align}
  \reverseincr(g) = \big(-g\big((\len(g)-t)-\big),\ t\in [0,\len(g)]\big).\label{eq:stable:reversal_def}
 \end{align}
 Let $\scaleS$ denote the \Stable[\frac32] scaling map from $(0,\infty)\times\DS$ to $\DS$:
 \begin{align}
  c \scaleS g := \left(cg\left(c^{-3/2}t\right),\ t\in [0,c^{3/2}\len(g)]\right).\label{eq:stable:scaling_def}
 \end{align}
\end{definition}

\begin{lemma}[Invariance of \Stable $(\frac32)$ excursions]
\label{lem:stable:invariance}
 For $B\in \SDSxc$ and $c>0$,
 \begin{align}
  \mSxc(\reverseincr(B)) = \mSxc(B)\quad \text{and} \quad \mSxc(\scaleS[c][B]) = c^{-1/2}\mSxc(B).\label{eq:stable:scaling_inv}
 \end{align}
\end{lemma}
\begin{proof} The increment-reversal invariance was obtained by Getoor and Sharpe \cite[Theorem (4.8)]{GetoShar81}. The scaling invariance follows from the scaling invariance of $\bX$. 
\end{proof}


It is important for us to distinguish between degenerate phenomena that may occur on a null set of levels and those that occur nowhere. To that end we offer the following two results.

\begin{proposition}\label{prop:exc_intervals}
 It is a.s.\ the case that for every $y\in\BR$, the following properties hold.
 \begin{enumerate}[label=(\roman*), ref=(\roman*)]
  \item $\displaystyle V^y = \left\{[a,b]\subset (0,\infty)\ \middle|\ a < b;\ \bX(a-) = y = \bX(b)\text{; and }\bX(t)\neq y\text{ for }t\in (a,b)\right\}.$
  \item For $I,J\in V^y_0$, $I\neq J$, the set $I\cap J$ is either empty or a single shared endpoint.
  \item If two intervals $[a,b],[b,c]\in V^y_0$ share an endpoint $b$ then $\bX$ does not jump at time $b$.
  \item For every $t\notin \bigcup_{I\in V^y_0}I$, we find $\bX(t-) = \bX(t) = y$.
  \item $\displaystyle\Leb\left( [0,\infty)\setminus \bigcup\nolimits_{I\in V^y_0}I\right) = 0$.
 \end{enumerate}
\end{proposition}

\begin{proof}
 (i), (ii), and (iii). These properties follow from a common observation. In the terminology of Bertoin \cite{BertoinLevy}, $0$ is regular for $(-\infty,0)$ and for $(0,\infty)$ after and, by time reversal, before each of the countably many jump times.
 
 (iv). Take $t>T^y$ and set $a = \sup\{s \leq t\colon \bX(s-) = y\}$ and $b = \inf\{s \geq t\colon \bX(s) = y\}$. If $a=b=t$ then, by the \cadlag\ property of $\bX$, we have $\bX(t-)= \bX(t) = y$. Otherwise, by assertion (i), $[a,b]\in V^y$ and $t\in [a,b]$.
 
 (v). This follows from (iv) and the a.s.\ existence of occupation density local time at all levels, per Theorem \ref{thm:Boylan}. Since occupation measure therefore has a derivative in level, it cannot jump at any level.
\end{proof}

\begin{definition}\label{def:typical_exc}
 We call an excursion $g\in\DSxc$ \emph{typical} if there exists some time $T_0^+(g) \in (0,\len(g))$ such that: (i) $g(t) < 0$ for $t\in (0,T_0^+(g))$, and (ii) $g(t) > 0$ for $t\in [T_0^+(g),\len(g))$. We call $g$ \emph{degenerate} if it is not typical.
\end{definition}

A typical excursion may be decomposed around $T_0^+$ into: (1) an initial escape downwards from zero, (2) a single jump up across zero, and (3) a final passage back down to zero. 

\begin{proposition}\label{prop:nice_level}
 For each $y\in\BR$ it is a.s.\ the case that level $y$ is \emph{nice} for $\bX$ in the following sense.
 \begin{enumerate}[label=(\roman*), ref=(\roman*)]
  \item There are no degenerate excursions of $\bX$ about level $y$.
  \item Local times $(\ell^y(t),\,t\ge0)$ exist. For $[a,b],[c,d] \in V^y_0$, $\ell^y(a)\neq \ell^y(c)$ unless $[a,b]=[c,d]$.
  \item If $y>0$, we also have $T^y>T^{\geq y}:=\inf\{t\ge 0\colon\bX\geq y\}$.
 \end{enumerate}
\end{proposition}

\begin{proof} 
 (i) There are four cases of potential degeneracy: start with a jump or creep up from the starting level; end with a jump or creep up to the end level. Millar \cite{Millar73} showed that spectrally positive \Stable[\frac32] processes a.s.\ do not creep up to a fixed level. The distributions of pre-jump levels and jump levels are absolutely continuous, so a.s.\ no jump ends at a fixed level. Hence, there is a.s.\ no degeneracy at ends of excursions. By time reversal, the same holds at the start of excursions. 
 
 (ii) Existence of local times has been addressed in Theorem \ref{thm:Boylan}. The Poisson random measure $\bG^y$ of Theorem \ref{thm:excursion_PRM} places all excursions at different local times a.s..
 
 (iii) As noted in (i), there is a.s.\ no creeping up to a level. Hence, $T^y>T^{\geq y}$ a.s..
\end{proof}

Nonetheless, each jump of $\bX$, say at time $T$, occurs at the start of a (possibly incomplete) excursion about level $\bX(T-)$, which includes no initial escape down 
from the level; and it occurs at the end of a (possibly incomplete) excursion about level $\bX(T)$, which has no part above that level. So while there are a.s.\ no 
degenerate excursions about any fixed level, there are also a.s.\ infinitely many degenerate excursions that occur at random levels. There are also degenerate excursions
that begin by creeping continuously up from a level or end by creeping continuously up to a level, or levels at which two excursions occur at the same local time, such as around
local extrema of $\bX$. In fact, there are uncountably many levels that are not nice, e.g.\ in the closure of the ladder height set 
$\{\bX(T^{\geq y}),y\ge 0\}$.

\subsection{Bi-clades: level filtrations via excursions of scaffolding with spindles}
\label{sec:clade_filtration}

In the preceding section, we have looked at excursions of the \Stable[\frac32] scaffolding process. In this section, we consider such excursions with jumps marked by \BESQ[-1] spindles.

\begin{definition}\label{def:bi-clade}
 We define
 \begin{gather*}
  \Hxc{\pm} := \big\{N\in \Hfin\colon \xi(N)\in\DSxc\big\} = \big\{N\in \Hfin\colon \inf\{t > 0\colon \xi_N(t)=0\} = \len(N)\big\},\\
  \Hxc{+}	:= \big\{N\in \Hxc{\pm}\colon \inf\nolimits_t\xi_N(t)=0\big\}, \quad \text{and} \quad \Hxc{-}	:= \big\{N\in \Hxc{\pm}\colon \sup\nolimits_t\xi_N(t)=0\big\},
 \end{gather*}
 where $\Hfin$ is as in Definition \ref{def:assemblage_m}. Let $\Sigma(\Hxc{\pm})$, $\Sigma(\Hxc{+})$, and $\Sigma(\Hxc{-})$ denote the restrictions of $\ScNRE$ to subsets of $\Hxc{\pm}$, $\Hxc{+}$, and $\Hxc{-}$ respectively. We call the members of $\Hxc{+}$ the \emph{clades} and those of $\Hxc{-}$ the \emph{anti-clades}. Members of $\Hxc{\pm}$ are called \emph{bi-clades}.
\end{definition}

\begin{figure}
 \centering
 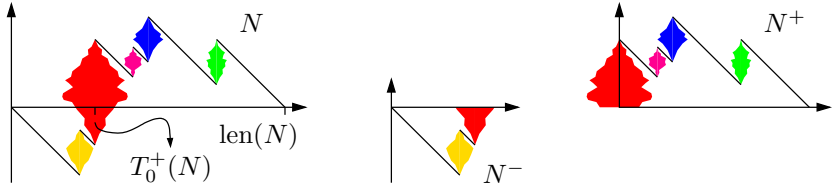
 \caption{Decomposition of a bi-clade into an anti-clade and a clade.\label{fig:bi-clade_decomp}}
\end{figure}

Recall Definition \ref{def:typical_exc} of a typical \Stable[\frac32]-like excursion $g\in\DSxc$. Such an excursion may be decomposed into an initial escape downwards, a jump up across zero, and a final first-passage descent. Correspondingly, as illustrated in Figure \ref{fig:bi-clade_decomp}, a bi-clade $N\in\Hxc{\pm}$ for which $\xi(N)$ is typical may be split into two components around the jump of $\xi(N)$ across zero. The initial component, corresponding to the negative path-segment of $\xi(N)$, is an anti-clade, and the subsequent component, on which $\xi(N)$ is positive, is a clade. In order to make this decomposition, we must break the spindle marking the middle jump of the excursion into two parts, above and below level zero.

\begin{figure}
 \centering
 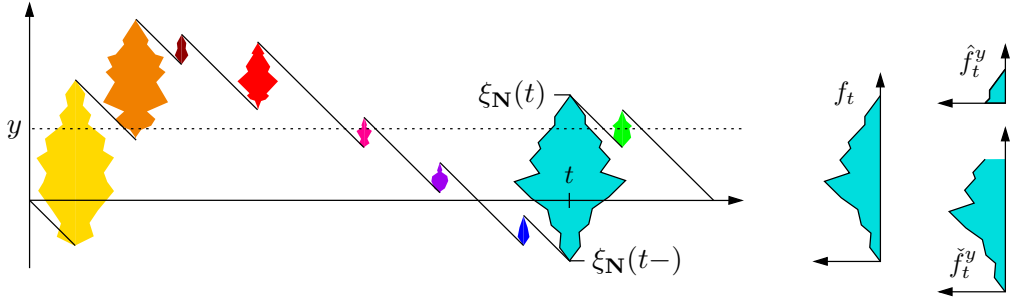
 \caption{Splitting a spindle of some $N\in \H$ about some level $y$, as in \eqref{eq:spindle_split}.\label{fig:spindle_split}}
\end{figure}

If a measure $N\in \H$ has a point $(t,f_t)$ with $y\in (\xi_N(t-),\xi_N(t))$, then we define
\begin{equation}
 \left.\begin{array}{r@{\ :=\ }l}
     \check f_{t}^y(z) & f_{t}(z)\cf\{z \in [0,y-\xi_N(t-))\}\\[0.2cm]
     \text{and}\quad\hat f_{t}^y(z) & f_{t}(y - \xi_N(t-) + z)\cf\{z \in [0,\xi_N(t)-y]\}.
 \end{array}\right.\label{eq:spindle_split}
\end{equation}
This splits the spindle $f_{t}$ into two parts, corresponding to the part of the jump of $\xi(N)$ that goes from $\xi_N(t-)$ up to $y$, and the part extending from $y$ up to $\xi_N(t)$. This is illustrated in Figure \ref{fig:spindle_split}. Following Definition \ref{def:typical_exc}, for $N\in\Hxc{\pm}$ the \emph{crossing time} is
\begin{equation}
 T_0^+(N) := \inf\{t \in (0,\len(N)] \colon \xi_N(t) \geq 0\}.\label{eq:clade:crossing_def}
\end{equation}

Fix $N\in\Hxc{\pm}\setminus(\Hxc{-}\cup\Hxc{+})$. For the purposes of the following definitions, we abbreviate the crossing time $T_0^+ := T_0^+(N)$. We split the bi-clade into anti-clade and clade components, denoted by $(N^-,N^+)$, as follows:
\begin{equation}\label{eq:bi-clade_split}
 N^- := \Restrict{N}{[0,T_0^+)} + \Dirac{T_0^+,\check f_{T_0^+}^0} \qquad \text{and} \qquad N^+ := \Dirac{0,\hat f_{T_0^+}^0} + \ShiftRestrict{N}{(T_0^+,\infty)}.
\end{equation}
For $N\in\Hxc{-}$, let $(N^-,N^+) := (N,0)$. For $N\in\Hxc{+}$, let $(N^-,N^+) := (0,N)$.

More generally, we may define scaffolding and spindles \emph{cut off} above and below a level $y\in\BR$. These processes are illustrated in Figure \ref{fig:cutoff}. For the purpose of the following, for $N\in\H$ and $y\in\BR$, let
\begin{equation*}
 \sigma^y_N(t) := \Leb\{u\leq t\colon \xi_N(u)\leq y\} \qquad \text{for }t\ge0.
\end{equation*}
In other words, $\sigma^y_N(t)$ is the amount of time that $\xi(N)$ spends below level $y$, up to time $t$. Then for $t\ge 0$,
\begin{equation}
\begin{split}
  &\cutoffL{y}{\xi(N)}(t) := \xi_N\big( \sup\{t\ge 0\colon \sigma^y_N(t) \le s\} \big) - \min\{y,0\},\\
  &\cutoffG{y}{\xi(N)}(t) := \xi_N\big( \sup\{t\ge 0\colon t - \sigma^y_N(t) \le s\} \big) - \max\{y,0\},\\
 \label{eq:cutoff_def}
  &\cutoffL{y}{N} := \sum_{\text{points }(t,f_t)\text{ of }N} \left(\begin{array}{r@{\,}l}
 		\cf\big\{y\in (\xi_N(t-),\xi_N(t))\big\}&\DiracBig{\sigma^y_N(t),\check f^y_t}\\[3pt]
 		+\; \cf\big\{\xi_N(t) \leq y\big\}&\DiracBig{\sigma^y_N(t),f_t}
 		\end{array}\right),\\
 \text{and} \quad &\cutoffG{y}{N} := \sum_{\text{points }(t,f_t)\text{ of }N} \left(\begin{array}{r@{\,}l}
 		\cf\big\{y\in (\xi_N(t-),\xi_N(t))\big\}&\DiracBig{t-\sigma^y_N(t),\hat f^y_t}\\[3pt]
 		+\; \cf\big\{\xi_N(t-) \geq y\big\}&\DiracBig{t-\sigma^y_N(t),f_t}
 		\end{array}\right).
\end{split} 
\end{equation}

\begin{figure}
 \centering
 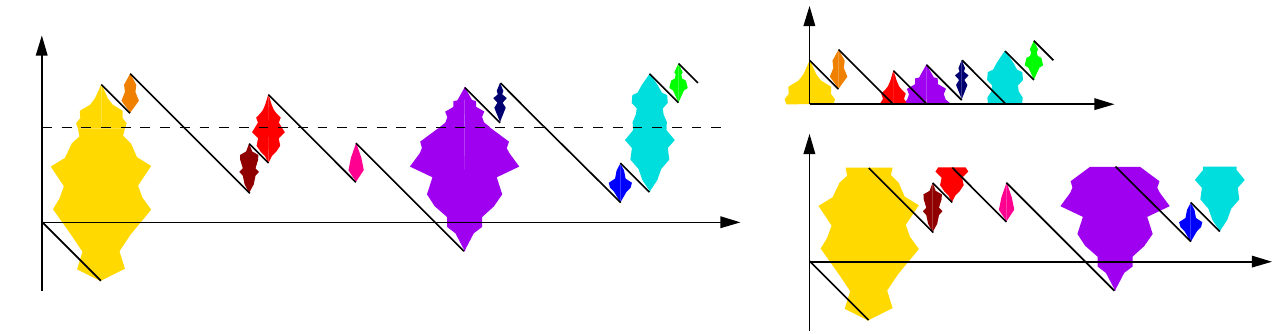\vspace{-.3cm}
 \caption{Left: $(N,\xi(N))$. Right: $\big(\cutoffG{y}{N},\cutoffG{y}{\xi(N)}\big)$ (above) and $\big(\cutoffL{y}{N}, \cutoffL{y}{\xi(N)}\big)$, as in \eqref{eq:cutoff_def}.\label{fig:cutoff}}
\end{figure}

We note the following elementary result.

\begin{lemma}\label{lem:cutoff_meas}
 $N\mapsto \left(\cutoffL{y}{\xi(N)}, \cutoffG{y}{\xi(N)}, \cutoffL{y}{N}, \cutoffG{y}{N}\right)$ is a measurable map.
\end{lemma}

Recall Figure \ref{fig:skewer_1} and Definition \ref{def:skewer} of the skewer map. We are ultimately interested in processes $\big(\skewer(y,\bN_{\beta},\xi(\bN_{\beta})),\ y\geq 0\big)$, where $\bN_{\beta}$ is as in Definition \ref{constr:type-1}. We view such processes as evolving in \emph{level} rather than in \emph{time}, as the parameter $y$ of this process corresponds to values, or levels, in the scaffolding function $\xi(\bN_{\beta})$. From this standpoint, $\cutoffL{y}{N}$ describes the past up to level $y$, and $\cutoffG{y}{N}$ describes the future beyond level $y$. This motivates the following. Throughout, superscripts refer to level whereas subscripts refer to time.

\begin{definition}\label{def:filtrations}
 \begin{enumerate}[label=(\roman*), ref=(\roman*)]
  \item We define the \emph{filtration in level} on $\H$, denoted by $(\cF^y)_{y\geq 0}$, to be the least right-continuous filtration under which $N\mapsto \left(\cutoffL{y}{\xi(N)},\cutoffL{y}{N}\right)$ is $\cF^y$-measurable for $y\geq 0$; see \cite[Section 1.3]{DuquLeGall02} for a similar definition on Skorokhod space.
  
  \item The \emph{filtration in time} on $\H$, denoted by $(\cF_t,\ t\geq 0)$, is defined to be the least right-continuous filtration under which $N\mapsto \restrict{N}{[0,t]}$ is $\cF_t$-measurable for every $t\geq 0$.
  
  \item We write $(\cF^{y-})$ and $(\cF_{t-})$ to denote left-continuous versions of the filtrations.
 \end{enumerate}
\end{definition}

In Lemma \ref{lem:cutoff_vs_PP} we will alternatively form the cutoff point processes by concatenating (anti-) clades. For $N\in\H$ and $y\in\BR$, recall Definition \ref{def:inverse_LT} of the sets $V^y(N)$ and $V^y_0(N)$ of excursion intervals of $g:=\xi(N)\in\DS$. For $[a,b]\in V^y_0(N)$ we define $I^y_N(a,b)$ to equal one of $[a,b]$, $(a,b]$, $[a,b)$, or $(a,b)$, as follows. We exclude the endpoint $a$ from $I^y_N(a,b)$ if and only if both $a<b$ and $g(a-) < y = g(a)$. We exclude $b$ if and only if both $a < b$ and $g(b-) = y < g(b)$. 

We make the following elementary observation. 

\begin{proposition}\label{prop:partn_spindles}
 Take $y\in\BR$ and $N\in\H$. Then for $[a,b]\in V^y(N)$, the process $\shiftrestrict{N}{I^y_N(a,b)}$ is a bi-clade. Moreover, the set $\big\{\shiftrestrict{N}{I^y_N(a,b)}\colon [a,b]\in V_0^y(N)\big\}$ partitions the spindles of $N$, in the sense that for each point $(t,f_t)$ of $N$ there is a unique $[a,b]\in V^y_0(N)$ for which $t\in I_N^y(a,b)$. 
\end{proposition}

Recall Definition \ref{def:concat} of concatenation on $\Hfin$. To form $\cutoffL{y}{N}$, we concatenate all of the andi-clades of $N$ below level $y$, along with potentially incomplete anti-clades at the start $[0,T^y]$ and/or end $[T^y_*,\len(N)]$, as in Definition \ref{def:inverse_LT} \ref{item:i_LT:exc_int_0}. In order to formally describe these incomplete anti-clades and the corresponding incomplete clades, we specify their crossing times:
\begin{equation*}
\begin{split}
 					T^{\geq y} &:= \inf\big(\{t\in [0,\len(N)]\colon \xi_N(t)  \geq y\}\cup\{\len(N)\}\big)\\
 \text{and} \quad T^{\geq y}_* &:= \sup\big(\{t\in [0,\len(N)]\colon \xi_N(t-) \leq y\}\cup\{0\}      \big).
\end{split}
\end{equation*}
Note that $T^{\geq y} = T^{\geq y}_*$ if and only if $\xi(N)$ is a single incomplete excursion about level $y$ that neither begins nor ends at $y$. 
To avoid duplication in our formulas, we adopt the convention that in this case, this sole incomplete bi-clade is called the last, and there is no first.
\begin{equation*}
\!\!\begin{array}{l@{\ :=\ }l@{\;+\;}l}
 N^{\leq y}_{\textnormal{first}}	& \Big(\Restrict{N}{[0,T^{\geq y})}				&\cf\big\{\xi_N\left(T^{\geq y}-\right) < y\big\}
 	\Dirac{T^{\geq y},\check f^y_{T^{\geq y}}}\!\Big)\cf\!\left\{T^{\geq y}\neq T^{\geq y}_*\right\}\!,\\[.2cm]
 N^{\geq y}_{\textnormal{first}}	& \Big(\ShiftRestrict{N}{(T^{\geq y},T^y]}		&\cf\big\{y\neq 0;\;\xi_N\!\left(T^{\geq y}\right) > (y\vee 0)\big\}
 	\Dirac{0,\hat f^y_{T^{\geq y}}}\!\Big)\cf\!\left\{T^{\geq y}\neq T^{\geq y}_*\right\}\!,\\[.2cm]
 N^{\leq y}_{\textnormal{last}} 	& \ShiftRestrict{N}{[T^y_*,T^{\geq y}_*)}	&\cf\big\{y\!\neq\!\xi_N(\len(N));\;\xi_N\big(T^{\geq y}_*\!-\!\big)\!<\big(y\wedge \xi_N\big(T^{\geq y}_*\big)\big)\big\}
 	\delta\!\left(T^{\geq y}_*\!-T^y_*,\check f^y_{T^{\geq y}_*}\right)\!,\\[.2cm]
 N^{\geq y}_{\textnormal{last}}		& \ShiftRestrict{N}{(T^{\geq y}_*,\len(N)]}	&\cf\big\{\xi_N \big(T^{\geq y}_*\big) > y\big\}
 	\Dirac{0,\hat f^y_{T^{\geq y}_*}}\!.
\end{array}
\end{equation*}
The first bi-clade is complete if and only if $y=0$, in which case $N^{\leq y}_{\textnormal{first}}=N^{\geq y}_{\textnormal{first}} = 0$. Similarly, the last bi-clade is complete if and only if $y = \xi_N(\len(N))$, in which case $N^{\leq y}_{\textnormal{last}} = N^{\geq y}_{\textnormal{last}} = 0$.

\begin{definition}\label{def:bi-clade_PP}
 Take $N\in\H$ and $y\in\BR$. If the level $y$ local time $(\ell^y_N(t),\ t\geq 0)$ of Definition \ref{def:inverse_LT} \ref{item:i_LT:LT} exists, we define the following counting measures of (bi-/anti-)clades.
  \begin{gather*}
   F^{y}(N) := \sum_{[a,b]\in V^y(N)}\Dirac{\ell^y(a),\shiftrestrict{N}{I^y_N(a,b)}}, \qquad
   F^{\geq y}(N) := \sum_{[a,b]\in V^y(N)}\Dirac{\ell^y(a),\left(\shiftrestrict{N}{I^y_N(a,b)}\right)^+},\\
   F^y_0(N) := \sum_{[a,b]\in V^y_0(N)}\Dirac{\ell^y(a),\shiftrestrict{N}{I^y_N(a,b)}},\qquad
   F^{\leq y}(N) := \sum_{[a,b]\in V^y(N)}\Dirac{\ell^y(a),\left(\shiftrestrict{N}{I^y_N(a,b)}\right)^-},\\
   F^{\leq y}_0(N) := F^{\leq y}(N) + \cf\left\{N^{\leq y}_{\textnormal{first}} \neq 0\right\}\Dirac{0,N^{\leq y}_{\textnormal{first}}} + \cf\left\{N^{\leq y}_{\textnormal{last}} \neq 0\right\}\Dirac{\ell^y(T^y_*),N^{\leq y}_{\textnormal{last}}},\\
   F^{\geq y}_0(N) := F^{\geq y}(N) + \cf\left\{N^{\geq y}_{\textnormal{first}} \neq 0\right\}\Dirac{0,N^{\geq y}_{\textnormal{first}}} + \cf\left\{N^{\geq y}_{\textnormal{last}} \neq 0\right\}\Dirac{\ell^y(T^y_*),N^{\geq y}_{\textnormal{last}}}.
  \end{gather*}
  If $\ell^y_N(t)$ is undefined for some $t\in [0,\len(N)]$ then we set all six of these measures equal to zero.
\end{definition}

Recall the \PRM\ $\bN$ studied in Sections \ref{sec:prelim:JCCP} and \ref{sec:exc}. We plan to use standard techniques from the study of counting measures, as in \cite{DaleyVereJones1,DaleyVereJones2}, to manipulate the measures of Definition \ref{def:bi-clade_PP}. To justify the use of such techniques, we require the following.

\begin{lemma}\label{lem:bi-clade_PP_meas}
 There exists a metric $d^A_{\cN}$ on $\Hfin$ and a set $\td\cN^{\textnormal{sp}}\subset\H$ with the following properties.
 \begin{enumerate}[label=(\roman*), ref=(\roman*)]
  \item $(\Hfin,d^A_{\cN})$ is isometric to a Borel subset of a complete, separable metric space.  \label{item:bcPPm:Lusin}
  \item The Borel $\sigma$-algebra generated by $d^A_{\cN}$ equals that generated by the evaluation maps on measurable subsets of $[0,\infty)\times\Exc$.\label{item:bcPPm:metric}
  \item For $y\in\BR$, the maps $F^y,\ F^y_0,\ F^{\leq y},\ F^{\leq y}_0,\ F^{\geq y}$, and $F^{\geq y}_0$ are measurable maps from $\td\cN^{\textnormal{sp}}$ to $\cNRHf$, where the latter is the space of counting measures that are boundedly finite on $[0,\infty)\times\Hfin$, in the sense of Definition \ref{def:JCCP}. \label{item:bcPPm:meas_maps}
  \item $\Hfin\subset \td\cN^{\textnormal{sp}}$, and the law of the \PRM[\Leb\otimes\mBxc] on $[0,\infty)\times\Exc$ is supported on $\td\cN^{\textnormal{sp}}$.\label{item:bcPPm:PRM}
 \end{enumerate}
\end{lemma}

\begin{proof}
 By Proposition \ref{prop:H_Lusin} there is a complete, separable metric $d_{\cN}$ on $\cNRE$ that generates the same $\sigma$-algebra as the evaluation maps. Following Proposition \ref{prop:E_bdedly_finite} and Definition \ref{def:shifted_restriction}, we define a modified metric $d^A_{\cN}(N,M) := d_{\cN}(N,M) + |\wh A(\xi(N))^{-1}-\wh A(\xi(M))^{-1}|$ on $\Hfin$, where for $g\neq 0$, $\wh A(g) = A(g)$ denotes amplitude, and we set $\wh A(0) := -1$. Let
 $$\td\cN^{\textnormal{sp}} := \bigcap_{y\in\BR}\cN^y \qquad \text{where} \qquad \cN^y := \Hfin\cup\Big\{N\in\H\colon \ell^y_N(t)\text{ exists }\forall t>0,\ \lim_{t\upto\infty}\ell^y_N(t) = \infty\Big\}.$$
 
 (i) Consider $\cNRE\times (\{-1\}\cup [0,\infty))$ under the sum of $d_{\cN}$ in the first coordinate plus the Euclidean metric in the second coordinate. This space is complete and separable. Moreover, $(\Hfin,d^A_{\cN})$ is isometric to a measurable subset of this space via the map $N\mapsto (N,\wh A(\xi(N))^{-1})$.
 
 (ii) Since $d^A_{\cN}$ is stronger than $d_{\cN}$, it follows that $\SHfin$ is contained in the Borel $\sigma$-algebra generated by $d^A_{\cN}$. By separability of $(\Hfin,d^A_{\cN})$, open sets under $d^A_{\cN}$ can be described as countable unions of sets of the form $B_{\cN}(N,r)\cap B^A(N,r)$, where this denotes the intersection of a ball under $d_{\cN}$ with a ball under the pseudometric $|\wh A(\xi(\cdot))^{-1} - \wh A(\xi(\cdot))^{-1}|$. Thus, it suffices to confirm that $B^A(N,r)\in\SHfin$ for every $N\in\Hfin$ and $r>0$. Indeed, balls in the pseudometric $|A(\cdot)^{-1} - A(\cdot)^{-1}|$ are open in the Skorokhod topology on $\{g\in \DS\colon \len(g)\in (0,\infty)\}$. Finally, since Proposition \ref{prop:JCCP_meas} indicates that $\xi\colon \Hfin\setminus\{0\}\to \{g\in \DS\colon \len(g)\in (0,\infty)\}$ is a measurable map, the $B^A(N,r)$ are measurable.
 
 (iii) Fix $y\in\BR$, $N\in\cN^y$, and $\epsilon > 0$. By the right-continuity of $\xi(N)$, this process can only have finitely many excursions of amplitude at least $\epsilon$ up to any inverse local time $\tau^y_N(s)$, $s\ge0$. Thus, the desired bounded finiteness follows from the property that sequences of bi-clades $N_i$ with $\lim_{i\upto\infty}A(\xi(N_i)) = 0$ cannot be bounded in $d^A_{\cN}$. As for measurability, this follows from the measurability of the local time process $(\ell^y_N(t),\ t\geq 0)$, Definition \ref{def:inverse_LT} of $V^y_0(N)$, and the straightforward measurability of restriction maps on $\H$ and $\Exc$.
 
 (iv) It is immediate from our definition of $\td\cN^{\textnormal{sp}}$ that $\Hfin\subset \td\cN^{\textnormal{sp}}$. The second claim, that the \PRM\ law is supported on $\td\cN^{\textnormal{sp}}$, follows from Theorem \ref{thm:Boylan} and \cite[Lemma 20]{Paper0}.
\end{proof}

\begin{definition}\label{def:scaff_concat}
 Let $(N_a)_{a\in\mathcal{A}}$ denote a family of elements of $\Hfin$ indexed by a totally ordered set $(\cA,\preceq)$, with all but finitely many being bi-clades. Let $S\colon \mathcal{A}\to[0,\infty]$ be as in \eqref{eq:length_partial_sums}. We require that: (i) $S(a-)<\infty$ for all $a\in\mathcal{A}$ and (ii) there is no infinite $\mathcal{B}\subseteq\mathcal{A}$ with $\inf_{b\in \mathcal{B}}\sup_t|\xi_{N_b}(t)| > 0$. 
 Then we define the \emph{concatenation of scaffoldings} by setting,
 \begin{equation*}
  \text{for }t\in \left[0,\sum_{a\in\mathcal{A}}\len(N_a)\right],\quad
  \Big(\Concat_{a\in\mathcal{A}}\xi(N_a)\Big)(t) := \sum_{a\in\mathcal{A}}\left\{\begin{array}{ll}
  		\xi_{N_a}(\len(N_a))	& \text{if }S(a)\le t,\\[2pt]
  		\xi_{N_a}(t - S(a-))	& \text{if }t\in [S(a-),S(a)),\\[2pt]
  		0						& \text{otherwise.}
  	\end{array}\right.
 \end{equation*}
 Note that $\xi_{N_a}(\len(N_a))=0$ if $N_a$ is a bi-clade.
\end{definition}

\begin{lemma}\label{lem:cutoff_vs_PP}
 Take $N\in\H$. If level $y\in\BR$ is nice for $\xi(N)$ as in Proposition \ref{prop:nice_level}, then
 \begin{equation*}
 \begin{split}
  \cutoffL{y}{N} = \Concat_{\text{points }(s,N^-_s)\text{ of }F^{\leq y}_0(N)}N^-_s, &\qquad \cutoffL{y}{\xi(N)} = \Concat_{\text{points }(s,N^-_s)\text{ of }F^{\leq y}_0(N)}\xi(N^-_s),\\
  \cutoffG{y}{N} = \Concat_{\text{points }(s,N^+_s)\text{ of }F^{\geq y}_0(N)}N^+_s, &\qquad \cutoffG{y}{\xi(N)} = \Concat_{\text{points }(s,N^+_s)\text{ of }F^{\geq y}_0(N)}\xi(N^+_s).
 \end{split}
 \end{equation*}
 In this event, $\cutoffL{y}{N}$ is a measurable function of $F^{\leq y}_0(N)$, and likewise for $F^{\geq y}_0(N)$ and $\cutoffG{y}{N}$. Moreover, $F^{\leq y}_0$ generates $\cF^y$ up to \PRM[\Leb\otimes\mBxc]-null sets.
\end{lemma}

\begin{proof}
 If level $y$ is nice then, in particular, there is no $s$ for which $F^y_0(N)$ has multiple points at local time $s$. Thus, the points are totally ordered by local time. In light of this, the claimed formulas are clear, asserting decompositions of the cutoff processes corresponding to the decompositions of $N$ into bi-clades about level $y$ and of $\xi(N)$ into excursions. 
 
 We establish measurability by expressing the concatenation formulas in terms of integrals. For convenience, we write $F^{\leq y}_{0,N} := F^{\leq y}_0(N)$. Setting
 \begin{equation*}
 \begin{split}
  \tau^{\leq y}_N(s-) &:= \int_{[0,s)\times\H}\len(N')dF^{\leq y}_{0,N}(r,N'),\\
  \cutoffL{y}{N} &= \int\int\Dirac{\tau^{\leq y}_N(s-)+t,f}dN'(t,f)dF^{\leq y}_{0,N}(s,N'),
 \end{split}
 \end{equation*}
 and correspondingly for $\cutoffG{y}{N}$. Finally, $F^{\leq y}_0$ generates $\cF^y$ up to events on which level $y$ is not nice for $N$, which by Proposition \ref{prop:nice_level} are null for the \PRM.
\end{proof}

We would like to say that $\cutoffL{y}{\xi(N)} = \xi\big(\cutoffL{y}{N}\big)$. We will find in \eqref{eq:cutoff_scaffold_commute} that this holds a.s.\ in the setting of the processes $\bN_{\beta}$ of Definition \ref{constr:type-1}. A more general result of this type may be true, but we do not need it here. A challenge to proving such a result is that in general, the scaffolding map does not commute with concatenation: $\xi(\ConcatIL_a N_a) \neq \ConcatIL_a \xi(N_a)$.  Consider, for example, the excursions of a \Stable[\frac32] process $\xi(\bN)$ above its past minimum. This decomposition into positive excursions corresponds to a decomposition of a \PRM\ $\bN$ of spindles into degenerate clades: clades lacking a (broken) leftmost spindle, unlike the clade in Figure \ref{fig:bi-clade_decomp}. For each of these degenerate clades $N$, the associated scaffolding $\xi(N)$ is a degenerate \Stable[\frac32]-like excursion above level zero; but when we concatenate these clades, in the manner of Definition \ref{def:concat}, to recover the full point process $\bN$, the resulting scaffolding $\xi(\bN)$ sets these clades above progressively lower levels.

\section{The type-1 evolution in a \Stable$\left(\frac32\right)$ process with spindles}
\label{sec:clades_type-1}

\subsection{Bi-clade It\^o measure and invariance}
\label{sec:biclade_PRM}

Let $\bN$ denote a \PRM[\Leb\otimes\mBxc] on $[0,\infty)\times \Exc$. As in Section \ref{sec:exc}, we adopt the convention of suppressing the parameter $\bN$ when referring various functions of $\bN$, including the local time $(\ell^y(t))$, inverse local time $(\tau^y(s))$, hitting and crossing times $T^y$ and $T^{\geq y}$, and sets of excursion intervals $V^y$ and $V^y_0$. Refer back to Section \ref{sec:exc} for definitions of these objects. We write $\bX := \xi(\bN)$. We restrict to the a.s.\ events that: (a) $V^y$ and $V^y_0$ have the properties enumerated in Proposition \ref{prop:exc_intervals} for every $y\in\BR$, (b) $(\ell^y(t),\ y\in\BR,\ t\geq 0)$ is H\"older continuous in $(y,t)$, as in Theorem \ref{thm:Boylan}, and (c) $\bN\in\td\cN^{\textnormal{sp}}$, as in Lemma \ref{lem:bi-clade_PP_meas} \ref{item:bcPPm:PRM}, so that the counting measures of Definition \ref{def:bi-clade_PP} are all boundedly finite. We use notation such as $\bF^y := F^y(\bN)$ and $\bF^{\geq y} := F^{\geq y}(\bN)$ for those counting measures. 

\begin{table}
 \centering
 \begin{tabular}{|c|c|c|Sc|c|}\hline
  		&	PRM of spindles				&	Spindle intensity				&	PRM of bi-clades				&	Bi-clade intensity\\\hline
  $\bX$				&	$\sum\Dirac{t,\Delta X_t}$	&	$(3/2\pi\sqrt{2})x^{-5/2}dx$	&	$\bG^y = \sum \Dirac{s,g^y_s}$	&	$\mSxc$\\[1pt]\hline
  $(\bN,\xi(\bN))$	&	$\bN = \sum\Dirac{t,f_t}$	&	$\mBxc$							&	$\bF^y = \sum \Dirac{s,N^y_s}$	&	$\mClade$\\[1pt]\hline
 \end{tabular}\vspace{3pt}
 \caption{Objects from L\'evy process (excursion) theory and analogous objects in the setting of bi-clades.\label{tbl:exc_vs_cld}}
\end{table}

\begin{definition}\label{def:bi-clade_Ito}
 We define the \emph{It\^o measures on bi-clades}, \emph{clades}, and \emph{anti-clades} respectively by saying that for $A\in\cHxc{\pm}$, $B\in\cHxc{+}$, and $C\in\cHxc{-}$,
 \begin{equation*}
  \mClade(A) := \EV\big[\bF^0([0,1]\times A)\big], \quad \mClade^+(B) := \EV\big[\bF^{\geq 0}([0,1]\times B)\big], \quad\mClade^-(C) := \EV\big[\bF^{\leq 0}([0,1]\times C)\big].
 \end{equation*}
\end{definition}

In Proposition \ref{prop:marking_jumps} we construct $\bN$ by marking jumps of the scaffolding $\bX$ with independent \BESQ[-1] excursion spindles. After an auxiliary lemma, we give a similar description of $\mClade$.

\begin{lemma}\label{lem:meas_marking}
 Consider two complete and separable metric spaces $(\cS,d_\cS)$ and $(\cT,d_\cT)$, equipped with their Borel $\sigma$-algebras $\ScS$ and $\ScT$, respectively. Let $\kappa\colon\cS\times\ScT\rightarrow[0,1]$ denote a stochastic kernel. Let $M\in\cNS$ be a counting measure on $\cS$. Consider for each point $x$ of $M$ an independent mark $m_x$ with distribution $\kappa(x,\cdot)$. Then the map that associates with $M$ the distribution of the marked point process $\int\delta(x,m_x)dM(x)\in\cN(\cS\times\cT)$ is Borel measurable.
\end{lemma}

\begin{proof}
 This is a direct consequence of the following results. First, there exists a measurable enumeration map $\phi\colon\cNS\rightarrow\cS^* := \bigcup_{0\le n\le\infty}\cS^n$ that sends a counting measure to a list of all its points; see \cite[Proposition 9.1.XII]{DaleyVereJones2}. Second, the marking kernel $\kappa$ induces a natural kernel $\kappa^*\colon\cS^*\times\Sigma((\cS\times\cT)^*)\to [0,1]$ marking each of the points in the sequence independently. Finally, the map $\Lambda\colon(\cS\times\cT)^*\rightarrow\cN(\cS\times\cT)$ that sends $((x_j,m_j)) \mapsto \sum_j\delta(x_j,m_j)$ is measurable.
\end{proof}

\begin{proposition}[Bi-clade It\^o measure via marking jumps]\label{prop:bi-clade_PRM}
 For $g\in\DSxc$, let $\bN_g$ be derived from $g$ in the manner in which $\bN$ is derived from $\bX$ in Proposition \ref{prop:marking_jumps} -- i.e.\ by passing from a \cadlag\ path to a point process of jumps and marking jumps of height $z$ with excursions with law $\mBxc(\,\cdot\;|\;\life = z)$.
 \begin{enumerate}[label=(\roman*), ref=(\roman*)]
  \item For every $g\in\DSxc$, this $\bN_g$ is a random member of $\Hxc{\pm}$. Let $\mu_g$ denote its law.
  \item The map $g\mapsto\mu_g$ on $\DSxc$ is a $\xi$-disintegration of $\mClade$, in the sense of Definition \ref{def:disintegration}. As in Lemma \ref{lem:scl_ker}, we denote this disintegration by $\mClade(\,\cdot\;|\;\xi)$.
  \item For every $y\in\BR$, $\bF^y$ is a \PRM[\Leb\otimes\mClade] on $[0,\infty)\times \Hxc{\pm}$.
 \end{enumerate}
\end{proposition}

The reader may find Table \ref{tbl:exc_vs_cld} helpful regarding the counting measures that we have introduced.

\begin{proof}
 (i). Refer back to the marking construction in Proposition \ref{prop:marking_jumps}. By definition, the laws $\mBxc(\,\cdot\;|\;\life = z)$ are supported on $\{f\in\Exc\colon \life(f) = z\}$, for each $z$. Thus $\bN_g\in\Hxc{\pm}$ with $\xi(\bN_g) = g$. That $\bN_g$ is a random variable follows from Lemma \ref{lem:meas_marking} via measurability of the map from \cadlag\ functions to point processes of jumps \cite[Proposition II.(1.16)]{JacodShiryaev}.
 
 (ii) and (iii). We saw $\mu_g(\xi=g)=1$ in (i). Also, the map $g\mapsto\mu_g$ is measurable by Lemma \ref{lem:meas_marking}. To complete the proof of (ii), we need to check (\ref{eq:scl_ker:integration}). We derive this together with (iii). Take $B\in\SDSxc$ with $\mSxc(B)<\infty$. Fix $y\in\BR$ and let $\bG^y$ denote the point process of excursions of $\bX$ about level $y$, as in \eqref{eq:stable:exc_PRM}. By Theorem \ref{thm:excursion_PRM}, $\bG^y$ is a \PRM[\Leb\otimes\mSxc]. From elementary properties of Poisson point processes, 
 $\restrict{\bG^y}{[0,\infty)\times B} = \sum_{j\geq 1}\Dirac{S_j,X_j}$
 for some sequence of random pairs $\big((S_j,X_j),\ j\geq 1\big)$, where the $S_j$ are strictly increasing. It is a standard result, e.g.\ in \cite[Theorem 2.1]{GreePitm80}, that the $(X_j)$ are i.i.d.\ with law $\mSxc(\,\cdot\;|\;B)$ and the $(S_j)$ have i.i.d.\ \ExpDist[\mSxc(B)] differences.
 
 Let $A := \xi^{-1}(B) \subset \Hxc{\pm}$. By Proposition \ref{prop:nice_level}, level $y$ is a.s.\ nice for $\bX$. Moreover, by the description of $V^y$ in \eqref{eq:inv_LT_intervals}, it follows that
 \begin{equation*}
  \Restrict{\bF^y}{[0,\infty)\times A} = \sum_{j\geq 1}\Dirac{S_j,\bN_j} \qquad \text{where} \qquad \bN_j := \shiftrestrict{\bN}{\left[\tau^y(S_j-),\tau^y(S_j)\right]}, \qquad \text{almost surely.}
 \end{equation*}
 Note in particular that $\xi(\bN_j) = X_j$ for every $j\geq 1$, almost surely. Thus, by Proposition \ref{prop:marking_jumps}, there are $f_t$ for all $t\ge 0$ with $\Delta\bX_t>0$, conditionally independent given $\bX$, such that
 \begin{equation*}
  \bN_j = \sum_{t\in[\tau^y(S_j-),\tau^y(S_j)]\colon\Delta\bX_t>0}\Dirac{t - \tau(S_j-),f_t}\cf\{\ell^y(t) = S_j\} \qquad \text{for }j\geq 1.
 \end{equation*}
 Therefore, the $(\bN_j,\ j\geq 1)$ are conditionally independent given $\bX$, with conditional laws $\bN_j \sim \mu_{X_j}$, where $\mu_g$ is as in assertion (i), above. Set
 \begin{equation}\label{eq:shape_disint_int}
  \mu'_A(C) := \int\mu_g(C\cap A)d\mSxc(g) =  \int\mu_g(C)\cf\{g\in B\}d\mSxc(g) \qquad \text{for }C\in \cHxc{\pm}.
 \end{equation}
 Then the sequence $(\bN_j)$ is i.i.d.\ with law $\mu'_A(\,\cdot\,)/\mu'_A(A)$. Since the inter-arrival times are exponential, referring again to 
 \cite[Theorem 2.1]{GreePitm80}, we conclude that $\restrict{\bF^y}{[0,\infty)\times A}$ is a \PRM[\Leb\otimes\mu'_A].
 
 Now consider an increasing sequence of subsets $B_i\in\ScDSxc$ with $\mSxc(B_i)<\infty$ and $\bigcup_i B_i = \DSxc$, and for each $i$ set $A_i = \xi^{-1}(B_i)$. The previous argument applies to each $A_i$, and the resulting measures $\mu'_{A_i}$ are consistent in the sense that $\mu'_{A_i} = \Restrict{\mu'_{A_j}}{A_i}$ for $j>i$. 
 Thus, by \cite[Theorem 2.1]{GreePitm80}, $\bF^y$ itself is a \PRM[\Leb\otimes\mu'], where $\mu'$ is the increasing limit of the $\mu'_{A_i}$. By Definition \ref{def:bi-clade_Ito}, we conclude that $\mClade = \mu'$, completing the proof of (iii). Finally, by the definition of $\mu'$ via \eqref{eq:shape_disint_int}, and since $\mu_g(\xi=g)=1$,
 $$\mClade(\xi\in B) = \lim_{i\rightarrow\infty}\mu_{A_i}^\prime(\xi\in B) = \int\mu_g(\xi\in B)d\mSxc(g) = \int\cf\{g\in B\}d\mSxc(g) = \mSxc(B).$$
 Therefore, by Definition \ref{def:disintegration} of disintegration, $g\mapsto\mu_g$ is a $\xi$-disintegration of 
 $\mClade$, as desired.
\end{proof}

\begin{corollary}\label{cor:clade_PRM}
 $\bF^{\geq y}$ is a \PRM[\Leb\otimes\mClade^+] on $[0,\infty)\times \Hxc{+}$. Correspondingly, $\bF^{\leq y}$ is a \PRM[\Leb\otimes\mClade^-] on $[0,\infty)\times \Hxc{-}$.
\end{corollary}

We define a time-reversal involution and a scaling operator via
\begin{equation}\label{eq:clade:xform_def}
 \begin{split}
  \reverseH(N) &:= \int \Dirac{\len(N)-t,\reverseexc(f)}dN(t,f) \qquad \text{for }N\in\Hfin\\
  \text{and} \quad \scaleH[c][N] &:= \int \Dirac{c^{3/2}t,\scaleB[c][f]}dN(t,f) \qquad \text{for }N\in\H,\ c>0,
\end{split}
\end{equation}
where $\scaleB$ and $\reverseexc$ are as in Definition \ref{def:BESQ:scaling_def}. The map $\reverseH$, in particular, reverses the order of spindles and reverses time within each spindle.

\begin{lemma}[Bi-clade invariance properties]\label{lem:clade:invariance}
 For $A\in\cHxc{\pm}$ and $c>0$,
 $$\mClade(\reverseH(A)) = \mClade(A) \quad \text{and} \quad \mClade(\scaleH[c][A]) = c^{-1/2}\mClade(A).$$
 Moreover, for $\bN$ a \PRM\ as above, 
 $\scaleH[c][\bN] \stackrel{d}{=} \bN$. 
\end{lemma}

\begin{proof}
 For the first identity, observe that in the notation of Proposition \ref{prop:bi-clade_PRM}, we have 
 \begin{equation*}
 \begin{split}
  \reverseH{\bN_g}&=\sum_{t\in[0,\len(g)]\colon\Delta g(t)>0}\delta(\len(g)-t,\reverseexc(f_{\len(g)-t}))\\
                 &\stackrel{d}{=}\sum_{t\in[0,\len(g)]\colon\Delta g(t)>0}\delta(\len(g)-t,f_{\len(g)-t})
                 =\sum_{s\in[0,\len(g)]\colon\Delta\reverseincr(g)(s)>0}\delta(s,f_s) =: \bN_{\reverseincr(g)},
 \end{split}
 \end{equation*}
 where the equality in distribution follows from the \BESQ[-1] excursion invariance properties in Lemma \ref{lem:BESQ:invariance} and, again recalling Proposition \ref{prop:bi-clade_PRM}, $\bN_{\reverseincr(g)}$ has law $\mu_{\reverseincr(g)}$. By the \Stable[\frac32] invariance properties in Lemma \ref{lem:stable:invariance},
 \begin{equation*}
 \begin{split}
  \mClade(\reverseH(A))&=\int_{\DSxc}\Pr(\bN_g\in\reverseH(A))d\mSxc(g)=\int_{\DSxc}\Pr(\reverseH(\bN_g)\in A)d\mSxc(g)\\
                        &=\int_{\DSxc}\Pr(\bN_{\reverseincr(g)}\in A)d\mSxc(g)=\int_{\DSxc}\Pr(\bN_{g}\in A)d\mSxc(g)=\mClade(A).
 \end{split}
 \end{equation*}
 By a monotone class theorem, the scaling invariance of Lemma \ref{lem:stable:invariance} yields that, for $G$ bounded and $\SDSxc$-measurable, 
 $\int_{\DSxc} G(\scaleS[(1/c)][g])d\mSxc(g)=c^{-1/2}\int_{\DSxc} G(g)d\mSxc(g)$. Hence,
 \begin{equation*}
 \begin{split}
  \mClade(\scaleH[c][A])&=\int_{\DSxc}\Pr(\bN_g\in\scaleH[c][A])d\mSxc(g)=\int_{\DSxc}\Pr(\scaleH[(1/c)][\bN_g]\in A)d\mSxc(g)\\
                         &=c^{-1/2}\int_{\DSxc}\Pr(\bN_g\in A)d\mSxc(g)=c^{-1/2}\mClade(A),
 \end{split}
 \end{equation*}
 which is the second identity. Finally, the third identity follows from Lemma \ref{lem:BESQ:invariance} and the definition of $\bN$ as a \PRM[\Leb\otimes\mBxc]. 
\end{proof}

\subsection{Mid-spindle Markov property and conditioning bi-clade It\^o measure}
\label{sec:clade_anti-clade}

Take $N\in\H$. A spindle $f_t$ that arises at time $t$ in $N$ is said to be born at level $\xi_N(t-)$ and die at level $\xi_N(t)$. Thus, at each level $z\in\BR$ it has mass $f_t(z-\xi_N(t-))$. In particular, the spindle crosses level $z$ only if $f_t(z-\xi_N(t-)) > 0$. In a bi-clade $N$ for which $\xi(N)$ is typical, in the sense of Definition \ref{def:typical_exc}, there is a single spindle that crosses level $0$. Otherwise, if $\xi(N)$ is degenerate, there is no such spindle. The following formula isolates the level-$0$ mass of this unique spindle, when it exists. Moreover, the formula is sufficiently general that it may be applied to clades and anti-clades as well. The \emph{(central spindle) mass} of $N\in\Hxc{\pm}$ is
\begin{equation}
 m^0(N) := \int \max\Big\{\ f\big((-\xi_N(s-))-\big),\ f\big(-\xi_N(s-)\big)\ \Big\}dN(s,f).\label{eq:clade:mass_def}
\end{equation}
See Figure \ref{fig:clade_stats} for an illustration highlighting this and other quantities. Consider $N\in\Hxc{\pm}$ for which $\xi(N)$ is typical. Recalling the definition of broken spindles in \eqref{eq:spindle_split}, $f_{T_0^+}(-\xi_N(T_0^+-)) = \hat f_{T_0^+}^0(0) = \check f_{T_0^+}^0\big((-\xi_N(T_0^+-))-\big)$. Thus, $m^0(N) = m^0(N^+) = m^0(N^-)$.

\begin{lemma}\label{lem:heavy_clades_discrete}
 Under $\mClade$, the variable $m^0$ satisfies $\mClade\{m^0 > 1\} < \infty$.
\end{lemma}

\begin{proof}
 Since $m^0(N)$ evaluates a single spindle in $N$ at a single point,
 \begin{equation*}
  \bF^0\big([0,1]\times \{m^0 > 1\}\big) \leq \bN\left( (0,\tau^0(1))\times\left\{f\in \Exc\colon \sup\nolimits_{y\in\BR}f(y) > 1\right\}\right).
 \end{equation*}
 Theorem \ref{thm:excursion_PRM} implies that $\tau^0(1)$ is a.s.\ finite. Since $\bN$ is a \PRM[\Leb\otimes\mBxc], by Lemma \ref{lem:BESQ:exc_length} the right-hand side is a.s.\ finite. The desired formula follows from the \PRM\ description of $\bF^0$ in Proposition \ref{prop:bi-clade_PRM}.
\end{proof}

Fix $y\in\BR$ and $n,j\in\BN$. For the purpose of the following, let
\begin{equation}\label{eq:Tnj}
 T^{y}_{n,j} := \inf\left\{ t > 0\colon \int_{s\in [0,t]}\cf\left\{f\big( y-\bX(s-) \big) > 1/n\right\} d\bN(s,f) \geq j \right\}.
\end{equation}
This is the $j^{\text{th}}$ time at which a spindle of $\bN$ crosses level $y$ with mass at least $1/n$.

\begin{lemma}[Mid-spindle Markov property]\label{lem:mid_spindle_Markov}
 Let $T$ be either the stopping time $T^{\geq y}$ for some $y > 0$ or $T^y_{n,j}$ for some $y\in\BR,\ n,j\in\BN$. Let $f_T$ denote the spindle of $\bN$ at this time. Let $\hat f^y_T$ and $\check f^y_T$ denote the split of this spindle about level $y$, as in \eqref{eq:spindle_split}. Then, given $f_T(y-\bX(T-)) = a > 0$,
  \begin{equation*}
   \left( \Restrict{\bN}{[0,T)},\check f^y_T\right) \quad \text{is conditionally independent of} \quad \left( \ShiftRestrict{\bN}{(T,\infty)},\hat f^y_T\right).
  \end{equation*}
  Under this conditional law, $\shiftrestrict{\bN}{(T,\infty)}$ is a \PRM[\Leb\otimes\mBxc] independent of $\hat f^y_T$, which is a \BESQ[-1] started at $a$ and killed upon hitting zero.
\end{lemma}

\begin{proof}
 We start by proving the case $T = T^y_{n,j}$. By the strong Markov property of $\bN$, it suffices to prove this with $j=1$. For the purpose of the following, let $\Exc_n := \big\{f\in\Exc\colon \sup_{u}f(u) > \frac1n\big\}$, where $\Exc$ is the space of excursions of \eqref{eq:cts_exc_space_def}. Lemma \ref{lem:BESQ:exc_length} asserts that $\mBxc(\Exc_n) < \infty$. Thus, we may sequentially list the points of $\bN$ in $\Exc_n$:
 \begin{equation*}
  \Restrict{\bN}{[0,\infty)\times \Exc_n} = \sum_{i=1}^\infty \Dirac{T_i,f_i} \quad \text{with} \quad T_1 < T_2 < \ldots.
 \end{equation*}
 
 First, note that each time $T_i$ is a stopping time in the time-filtration $(\cF_t)$; thus, by the Poisson property of $\bN$, each $f_i$ is independent of $\cF_{T_i-}$. Also the $(f_i)$ are i.i.d.\ with the law $\mBxc(\,\cdot\;|\;\Exc_n)$. We define first passage times of $f_i$, $H_i := \inf\{z > 0\colon f_{i}(z) = 1/n\}$. Then by Lemma \ref{lem:BESQ:existence}, for each $i$ the process $\big(f_i(H_i+z),\ z\geq 0\big)$ is a \BESQ[-1] starting from $1/n$. We define a stopping $\rho_i$ for $f_i$ as follows. If $f_i(y-\bX(T_i-)) > 1/n$, set $\rho_i := y-\bX(T_i-)$; otherwise, set $\rho_i := \zeta(f_i)$. Thus, $\rho_i$ is always greater than $H_i$, and hence $\rho_i-H_i$ is a stopping time for $\big(f_i(H_i+z),\ z\geq 0\big)$.
 
 Recall Definition \ref{def:filtrations} of $(\cF_t,\ t\geq 0)$. For the purpose of the following, for $i\geq 1$ let $\cG_i := \sigma(\cF_{T_i-},\restrict{f_{i}}{(-\infty,\rho_i)})$. The sequence of pairs $\big(\bX(T_i-),\ \restrict{f_i}{(-\infty,\rho_i)}\big)$ is a Markov chain in this filtration. Indeed, in the case $\rho_i = \life(f_i)$, the process $\bX$ simply runs forward from its value $\bX(T_i) = \bX(T_i-)+\life(f_i)$ until the $(\cF_t)$-stopping time $T_{i+1}$. In the case $\rho_i < \life(f_i)$, we have $f_i(\rho_i) > 1/n$. Then by the Markov property of $\big(f_i(H_i+z),\ z\geq 0\big)$ at $\rho_i-H_i$, conditionally given $\cG_i$ the process $\hat f^y_i = \shiftrestrict{f_i}{[\rho_i,\infty)}$ is a \BESQ[-1] starting from $f_i(\rho_i)$. In particular, $\hat f^y_i$ is conditionally independent of $\cG_i$ given $f_i(\rho_i)$. Then $\bX(T_i) = y + \life\big(\hat f^y_i\big)$.
 
 Let $J := \inf\big\{i \geq 1\colon \rho_i < \life(f_i)\big\}$, so $T_J = T$. This $J$ is a stopping time for $(\cG_i)$. Therefore, conditionally given $f_J(\rho_J)$, the process $\hat f^y_J$ is independent of $\cG_J$, distributed like a \BESQ[-1] starting from $f_J(\rho_J)$. By the strong Markov property of $\bN$, the process $\ShiftRestrict{\bN}{(T,\infty)}$ is a \PRM[\Leb\otimes\mClade], independent of $\Restrict{\bN}{[0,T]}$, as desired.
 
 For the case $T = T^{\geq y}$, note that $T^{\geq y} = \inf_{n \geq 1}T^y_{n,1}$. It follows from Proposition \ref{prop:nice_level} that this infimum is almost surely attained by some $n$. Thus, the result in this case follows from the previous case.
\end{proof}

\begin{lemma}\label{lem:clade:mass_ker}
 The It\^o measure $\mClade$ admits a unique $m^0$-disintegration $\mClade(\,\cdot\;|\;m^0)$ with the scaling property
 \begin{equation}
  \mClade(B\;|\;m^0 = a) = \mClade\left(\scaleH[\frac1a][B]\;\middle|\;m^0 = 1\right) \qquad \text{for }a>0,\ B\in\SHxc{\pm}.\label{eq:clade_mass_ker}
 \end{equation}
 Likewise, $\mClade^+$ and $\mClade^-$ admit unique $m^0$-disintegrations with this same scaling property.
\end{lemma}

\begin{proof}
 Lemmas \ref{lem:clade:invariance} and \ref{lem:heavy_clades_discrete} and the scaling property noted in Table \ref{tbl:clade_scaling} satisfy the hypotheses of Lemma \ref{lem:scl_ker}, which then implies the claimed result.
\end{proof}

\begin{proposition}\label{prop:clade_splitting}
 \begin{enumerate}[label=(\roman*), ref=(\roman*)]
  \item Fix $a>0$. Let $\ol N_a$ have distribution $\mClade(\,\cdot\;|\;m^0 = a)$, and let $(\ol N_a^+,\ol N_a^-)$ denote its decomposition into clade and anti-clade. Then $(\ol N_a^+,\ol N_a^-)$ has distribution
   \begin{equation}
    \mClade^+(\,\cdot\;|\;m^0 = a) \otimes \mClade^-(\,\cdot\;|\;m^0 = a);\label{eq:clade_split_joint_law}
   \end{equation}
   in particular, $\ol N_a^+$ and $\ol N_a^-$ are independent.\label{prop:c_s:split}
  \item Let $\hat f$ denote a \BESQ[-1] started at $\hat f(0) = a$ and absorbed upon hitting $0$, independent of $\bN$, and let $\widehat T^0 := \inf\big\{t>0\colon \xi_{\bN}(t) = -\zeta(\hat f)\big\}$. We define
   \begin{equation}
    \ol N_a^+ := \Dirac{0,\hat f} + \Restrict{\bN}{\left[0,\widehat T^0\right)}.
   \end{equation}
   Then $\ol N_a^+$ has the law $\mClade^+(\,\cdot\;|\;m^0 = a)$ and $\reverseH(\ol N_a^+)$ has the law $\mClade^-(\,\cdot\;|\;m^0 = a)$, where $\reverseH$ is time reversal as in \eqref{eq:clade:xform_def}.\label{prop:c_s:clade_law}
 \end{enumerate}
\end{proposition}

\begin{proof} 
 Let $T = T^y_{n,j}$ be as in \eqref{eq:Tnj}, with $y=0$, $n=j=1$, and let $S=\ell^0(T)$ and $N_S=\ShiftRestrict{\bN}{(\tau^0(S-)\tau^0(S))}$.
 Then $(S,N_S)$ is the earliest point of $\bF^0$ in $\{N\in\Hxc{\pm}\colon m^0(N) > 1\}$ 
  and $(T,f_T)$ is the spindle in $\bN$ that corresponds to the jump of $\xi(N_S)$ across level zero.  
 By the \PRM\ description of $\bF^0$, the bi-clade $N_S$ has law $\mClade(\,\cdot\;|\;m^0 > 1)$. 
 Moreover, by Lemma \ref{lem:scl_ker}, $\ol N:= \scaleH[\frac{1}{m^0(N_S)}][N_S]$ has distribution $\mClade(\,\cdot\;|\;m^0 = 1)$, 
  and $\ol N_a:=\scaleH[a][\ol N]$ has law $\mClade(\,\cdot\;|\;m^0=a)$.
 
 The marginal distributions of $\ol N_a^+$ and $\ol N_a^-$ stated in (i) follow straight from the definitions of $\mClade$, $\mClade^+$ and $\mClade^-$. 
 The independence of $\ol N_a^+$ and $\ol N_a^-$ asserted in (i) and the description of $\mClade^+$ stated in (ii) follow from Lemma \ref{lem:mid_spindle_Markov}. 
 To obtain the corresponding description of $\mClade^-$ in (ii), observe that
 \begin{equation}
  (\reverseH(N))^- = \reverseH(N^+) \quad \text{and} \quad m^0(\reverseH(N)) = m^0(N).
 \end{equation}
 (We refer the reader back to Figure \ref{fig:bi-clade_decomp} for an illustration; $\reverseH$ time-reversal corresponds to holding the page upside down.) 
 By Lemma \ref{lem:clade:invariance}, if $\ol N_a$ has law $\mClade(\,\cdot\;|\;m^0=a)$ then so does $\reverseH(\ol N_a)$. 
 Thus, $\reverseH(\ol N_a^+)$ has law $\mClade^-(\,\cdot\;|\;m^0=a)$, as desired.
\end{proof}

\begin{definition}\label{def:clade_stats}
 We define several statistics of bi-clades $N\in\Hxc{\pm}$.\vspace{-8pt}
 \begin{list}{}{\setlength{\itemsep}{2pt}\setlength{\rightmargin}{0pt}\setlength{\leftmargin}{5pt}\setlength{\topsep}{-6pt}}
 \begin{multicols}{2}
  \item[\emph{Overshoot:}] $J^+(N) := \xi_N(T_0^+(N))$.
  \item[\emph{Undershoot:}] $J^-(N) := -\xi_N(T_0^+(N)-)$.
  \item[\emph{Crossing size:}] $J(N) := J^+(N) + J^-(N)$.
  \item[\emph{Lifetime:}] $\displaystyle\life^+(N) := \sup_{t\in[0,\len(N)]} \xi_N(t)$.
  \item[\emph{Anti-clade lifetime:}] $\displaystyle\life^-(N) := -\!\!\!\!\inf_{t\in[0,\len(N)]}\!\xi_N(t)$.
 \end{multicols}\vspace{-13pt}
\end{list}
\end{definition}

\begin{figure}
 \centering
 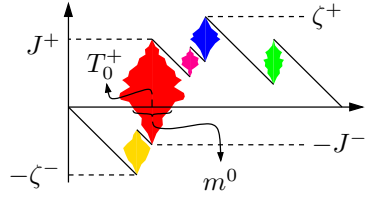
 \caption{A bi-clade, with the statistics of \eqref{eq:clade:mass_def} and Definition \ref{def:clade_stats} labeled.\label{fig:clade_stats}}
\end{figure}

We call $\life^+$ ``lifetime'' rather than ``maximum'' since values in the scaffolding function play the role of times in the evolving interval partitions $(\skewer(y,N),\ y\geq 0)$ that we ultimately wish to study. The above quantities appear labeled in Figure \ref{fig:clade_stats}. The rates at which they scale under $\scaleH$ are listed in Table \ref{tbl:clade_scaling}. By Lemma \ref{lem:scl_ker}, $\mClade$ admits unique kernels with scaling properties that allow us to condition on the exact value of any one of these quantities and get a resulting probability distribution.

\begin{table}
 \centering
 \begin{tabular}{|c|Sc|c|}\hline
  $J^+(\scaleH[c][N]) = cJ^+(N)$ & $J^-(\scaleH[c][N]) = cJ^-(N)$ & $J(\scaleH[c][N]) = cJ(N)$\\\hline
  $\life^+(\scaleH[c][N]) = c\life^+(N)$ & $\life^-(\scaleH[c][N]) = c\life^-(N)$ & $m^0(c\scaleH N) = cm^0(N)$\\\hline
  $T_0^+(c\scaleH N) = c^{3/2}T_0^+(N)$ & $\len(c \scaleH N) = c^{3/2}\len(N)$ & $\ell^y_{c \scaleH N}(t) = c^{1/2}\ell^{y/c}_{N}(tc^{-3/2})$\\\hline
 \end{tabular}\vspace{4pt}
 \caption{How statistics of \eqref{eq:clade:mass_def} and Definition \ref{def:clade_stats} scale as $N$ is scaled, as in \eqref{eq:clade:xform_def}.\label{tbl:clade_scaling}}
\end{table}

\begin{proposition}\label{prop:clade:stats}
 \begin{enumerate}[label=(\roman*), ref=(\roman*)]
  \item $\displaystyle \mClade\big\{ m^0 > a \big\} = \frac{1}{\sqrt{\pi}}a^{-1/2}.$\label{item:CS:mass}
  \item $\displaystyle \mClade\big\{ \life^+ > z \big\} = \frac{1}{\sqrt{2}}z^{-1/2}.$\label{item:CS:max}
  \item $\displaystyle \mClade\big\{ J^+ \in dy\;\big|\;m^0 = a\big\} = \frac{a^{3/2}}{\sqrt{2\pi}y^{5/2}}e^{-a/2y}dy.$\label{item:CS:over:mass}
  \item $\displaystyle \mClade\big\{ \life^+ \leq z\;\big|\;m^0 = a\big\} = e^{-a/2z}.$\label{item:CS:max:mass}
  \item $\displaystyle \mClade\big\{ m^0 \leq a\;\big|\;J^+ = y\big\} = 1 - e^{-a/2y}.$\label{item:CS:mass:over}
  \item $\displaystyle \mClade\big\{ \life^+ \leq z\;\big|\;J^+ = y\big\} = \cf\{z\geq y\}\sqrt{\frac{z-y}{z}}.$\label{item:CS:max:over}
  \item $\displaystyle \mClade\big\{ m^0 \in da\;\big|\;\life^+ \geq z\big\} = \frac{z^{1/2}}{\sqrt{2\pi}a^{3/2}}(1-e^{-a/2z})da.$\label{item:CS:mass:max}
 \end{enumerate}\vspace{-2pt}\ 
 
 \noindent All of these equations remain true if we replace all superscript `+'s with `-'s.
\end{proposition}

The proof of this is given in Appendix \ref{sec:clade_stats}; it is based on Proposition \ref{prop:clade_splitting} and well-known properties of our spectrally one-sided \Stable[\frac32] process.

\begin{corollary}\label{cor:clade_law_given_over}
 Take $y>0$. Let $A\sim \ExpDist[1/2y]$. Conditionally given $A$, let $\bff$ denote a \BESQ[-1] first-passage bridge from $\bff(0) = A$ to $\bff(y) = 0$, in the sense of \cite{BertChauPitm03}. Let $\bN$ be a \PRM[\Leb\otimes\mBxc], with $T^{-y}$ the hitting time of $-y$ in $\xi(\bN)$. Then $\Dirac{0,\bff}+\restrict{\bN}{[0,T^{-y}]}$ has law $\mClade^+(\,\cdot\;|\;J^+ = y)$.
\end{corollary}

\begin{proof}
 By Proposition \ref{prop:clade_splitting}, under $\mClade^+(\,\cdot\;|\;m^0 = a)$ a clade $N^+$ has the form $\DiracBig{0,\bff} + \bN'$. Here, $\bff$ is a \BESQ[-1] starting from $a$ and, conditionally given $\bff$, the point process $\bN'$ is distributed like $\bN$ stopped at time $T^{-\life(\bff)}$. Then $J^+(N^+) = \zeta(\bff)$. Thus, we may further condition $\mClade^+(\,\cdot\;|\;m^0 = a,\ J^+ = y)$. Under this new law, $N^+$ has the same form, and $\bff$ is now distributed like a \BESQ[-1] first-passage bridge from $a$ to $0$ in time $y$. So, since its lifetime is fixed, in this setting $\bff$ is independent of $\bN'$. Now,
 $$\mClade^+(\,\cdot\;|\;J^+ = y) = \int \mClade^+(\,\cdot\;|\;m^0 = a,\ J^+ = y)\mClade^+(m^0\in da\;|\;J^+ = y).$$
 The conditional law of $m^0$ above appears in Proposition \ref{prop:clade:stats} \ref{item:CS:mass:over}. In particular, under this law, $m^0\sim \ExpDist[1/2y]$.
\end{proof}

\subsection{The skewer map and type-1 evolutions}
\label{sec:type-1:def}

We now make a slight modification to Definition \ref{def:skewer} of the aggregate mass process, with the aim of having it apply nicely to anti-clades. Compare this to \eqref{eq:clade:mass_def}, which defines $m^0$.

\begin{definition}\label{def:skewer_2}
 For $N\in\H$ and $y\in\BR$, the \emph{aggregate mass process} of $N$ at level $y$ is
 $$M^y_N(t) := M^y_{N,\xi(N)}(t) := \int_{[0,t]\times\Exc} \max\Big\{\ f\big((y-\xi_N(u-))-\big),\ f\big(y-\xi_N(u-)\big)\ \Big\}dN(u,f)$$
 for $t\geq 0$. We leave the definition of the skewer map unchanged, but abbreviate it
 $$\skewer(y,N) := \skewer(y,N,\xi(N)) = \left\{\left(M^y_{N}(t-),M^y_{N}(t)\right)\colon t\geq 0,\ M^y_{N}(t-) < M^y_{N}(t)\right\}.$$
 Finally, we abbreviate $\skewerP(N) := \skewerP(N,\xi(N)) = \big(\skewer(y,N),\ y\geq 0\big)$.
\end{definition}

Recall the level $y$ inverse local time of $\bX$, $(\tau^y(s),\ s\geq 0)$ of Definition \ref{def:inverse_LT}. 

\begin{proposition}[Aggregate mass from $F^y(N)$]\label{prop:agg_mass_subord}
 Take $N\in\H$ and $y\in\BR$ and suppose that level $y$ is nice for $\xi(N)$, in the sense of Proposition \ref{prop:nice_level}. Suppose also that either $\len(N) = \infty$ or $\xi_N(\len(N)) < y$. We write $F^y_{N} := F^y(N)$. For every $s\geq 0$,
 \begin{equation}\label{eq:agg_mass_from_clades}
 \begin{split}
  M^y_{N}\circ\tau^y_N(s) &= M^y_N\circ\tau^y_N(0) + \int_{(0,s]\times\Hxc{\pm}} m^0(N')dF^y_N(r,N')\\
  \text{and} \quad \skewer(y,N) &= \big\{ \big(M^y_{N}\circ\tau^y_N(s-),M^y_{N}\circ\tau^y_N(s)\big) \colon s\geq 0,\ \tau^y(s) > \tau^y(s-) \big\},
 \end{split}
 \end{equation}
 where we take $\tau^y(0-) := 0$. In particular, for fixed $y\in\BR$ this holds for $N = \bN$ almost surely. For fixed $y\in\BR$, the process $\big(M^y_{\bN}\circ\tau^y_{\bN}(s) - M^y_{\bN}\circ\tau^y_{\bN}(0),\ s\geq 0\big)$ is a \Stable[\frac12] subordinator with Laplace exponent $\Phi(\lambda) = \sqrt{\lambda}$.
\end{proposition}

\begin{proof}
 As noted in Proposition \ref{prop:partn_spindles}, the bi-clades of $F^y(N)$, along with the potential initial and final incomplete bi-clades, partition the spindles of $N$. At most one of the spindles in the initial incomplete bi-clade crosses level $y$. Each subsequent excursion interval $I^y_N(a,b)$ with $[a,b]\in V^y(N)$ includes at most one jump of $\xi(N)$ that crosses level $y$. If $N' := \shiftrestrict{N}{I^y_N(a,b)}$ then this spindle crosses with mass $m^0(N')$. Finally, our requirement that either $\len(N) = \infty$ or $\xi_N(\len(N)) < y$ implies that either there is no final incomplete bi-clade about $y$ or, if there is, then this bi-clade dies during the incomplete anti-clade $N^{\leq y}_{\textnormal{last}}$, without contributing mass at level $y$. This gives us the claimed description of $M^y_N\circ\tau^y_N$. The subsequent description of $\skewer(y,N)$ follows from our assumption that level $y$ is nice for $\xi(N)$, whereby no two level $y$ bi-clades, complete or incomplete, arise at the same local time.
 
 If $N=\bN$ then by Proposition \ref{prop:nice_level}, level $y$ is nice for $\bX$ almost surely. By the Poisson property of $\bF^y$, $M^y_{\bN}$ is a shifted subordinator with L\'evy measure specified by Proposition \ref{prop:clade:stats} \ref{item:CS:mass}. Thus, its Laplace exponent follows from $\int_0^\infty (1-e^{-\lambda x})x^{-3/2}dx = 2\sqrt{\pi\lambda}$.
\end{proof}

\begin{theorem}[Scaffolding local time equals skewer diversity everywhere; Theorem 28 of \cite{Paper0}]\label{thm:LT_property_all_levels}
 There is an event of probability 1 on which, for every $y\in\BR$ and $s\geq 0$, the partition $\alpha^y_s := \skewer\big(y,\ \restrict{\bN}{[0,\tau^y(s)]}\big)$ 
 possesses the diversity property of \eqref{eq:IPLT}, and for all $t\in[0,\tau^y(s)]$,
 \begin{equation}
  \ell^y(t) = \IPLT_{\alpha^y_s}\left(M^y_{\bN}(t)\right) = \sqrt{\pi}\lim_{h\downto 0} \sqrt{h}\int \cf\{m^0(N) > h,\tau^y(r)\le t\}d\bF^y(r,N).\label{eq:inft_skewer_LT_cnvgc}
 \end{equation} 
\end{theorem}

The strength of the preceding result is that it holds a.s.\ simultaneously at every level $y$. Proposition \ref{prop:agg_mass_subord} implies the weaker result that \eqref{eq:inft_skewer_LT_cnvgc} holds a.s.\ for every $s\geq 0$, for any fixed $y$. Recall Definition \ref{def:assemblage_m} of the measurable spaces $(\H,\SH)$ and $(\Hfin,\SHfin)$. We are interested in diffusions on $(\IPspace,\dI)$. To that end we require measures $N\in \Hfin$ for which $\skewerP(N)$ is path-continuous in $(\IPspace,\dI)$. 

\begin{definition}[$\Hs,\ \Hfins$]\label{def:domain_for_skewer}
 Let $\Hs$ denote the set of all $N\in\H$ with the following additional properties.
 \begin{enumerate}[label=(\roman*), ref=(\roman*)]
  \item The aggregate mass $M^y_N(t)$ is finite for every $y\in\BR$ and $t\geq 0$.\label{item:d_f_s:fin}
  \item The occupation density local time $(\ell^y_N(t))$ is continuous on $t\geq 0$, $\inf_{u\in [0,\len(N)]}\xi_N(u) < y < \sup_{u\in [0,\len(N)]}\xi_N(u)$, and for every $(y,t)$ in this range,\label{item:d_f_s:div}
   \begin{equation}
    \IPLT_{\skewer(y,\restrict{N}{[0,t]})}\big( M^y_N(u) \big) = \ell^y_N(u), \qquad u\in [0,t].\label{eq:div_LT_condition}
   \end{equation}
  \item For every $t>0$, the skewer process $\skewerP\left(\restrict{N}{[0,t]}\right)$ is continuous in $(\IPspace,\dI)$.\label{item:d_f_s:cts}
 \end{enumerate}
 
 Let $\Hfins := \Hfin\cap\Hs$. Let $\Sigma(\Hs) := \{A\cap\Hs\colon A\in\SH\}$, and correspondingly define $\Sigma(\Hfins)$.
\end{definition}

In condition (ii) above, we restrict $y$ away from boundary values because \eqref{eq:div_LT_condition} can fail at $y = 0$ for the point processes $\bN_{\beta}$ constructed in Definition \ref{constr:type-1}.

\begin{proposition}\label{prop:skewer_measurable}
 The map $\skewerP$ is measurable from $(\Hfins,\SHfins)$ to the space of continuous functions $\mathcal{C}([0,\infty),\IPspace)$ under the Borel $\sigma$-algebra generated by uniform convergence.
\end{proposition}

\begin{proof}
 We have established in Proposition \ref{prop:JCCP_meas} that $N\mapsto \xi(N)$ is a measurable map from $(\H,\SH)$ to $(\DS,\SDS)$. This measurability passes to the restriction $\restrict{\xi}{\Hfins}$. Thence, and by Definition \ref{def:skewer_2} of the aggregate mass process, we deduce the measurability of the map $N\mapsto M^y_{N,\xi(N)}(t)$ from $\Hfins$ to $[0,\infty)$, for $(y,t)$ fixed. Let $\cD^*$ denote the set of non-decreasing \cadlag\ functions supported on intervals $[0,L]$, $L\in [0,\infty)$, whose closed ranges have zero Lebesgue measure. Since $M^y_{N,\xi(N)}(t)$ is non-decreasing in $t$, we conclude that for fixed $y$ the map $N\mapsto (M^y_{N,\xi(N)}(t),\ t\in [0,\len(N)])$ is measurable from $\Hfins$ to $\cD^*$.
 
 Recall the map $G$ of Definition \ref{def:Hausdorff} that takes $\alpha\in\HIPspace$ to a closed, bounded, Lebesgue-null set $G(\alpha)$. The map $g\mapsto G^{-1}(\text{range}(g))$ is continuous from the Skorokhod topology on $\cD^*$ to the Hausdorff topology on $\HIPspace$. Thus, $N\mapsto\skewer(y,N)$ is measurable from $\Hfins$ to $(\HIPspace,\dH)$ for fixed $y$. From Proposition \ref{prop:Hausdorff} \ref{item:Haus:sig}, the Borel $\sigma$-algebra generated by $\dI$ on $\IPspace$ equals that generated by the Hausdorff metric, $\dH$. Thus, the map $N\mapsto \skewer(y,N)$ is measurable from $\Hfins$ to $(\IPspace,\dI)$ for $y$ fixed. Finally, by Theorem \ref{thm:Lusin}, $(\IPspace,\dI)$ is separable. From this it follows that the $\sigma$-algebra on $\mathcal{C}([0,\infty),\IPspace)$ generated by uniform convergence equals that generated by the evaluation maps \cite[Theorem 14.5]{Billingsley}.
\end{proof}

We do not claim that $\Hs\in\SH$. However, we will show in Proposition \ref{prop:type-1:cts} that the processes that we wish to study admit $\Hs$-versions, in the following sense.

\begin{definition}\label{def:version}
 Let $(\Omega,\cF,\Pr)$ denote a probability space, $(\cS,\ScS)$ a measurable space, and $X,Y: \Omega \to \cS$ a pair of random variables. We call $Y$ a \emph{version} of $X$ if $X=Y$ a.s. If $Y$ only takes values in $B\subseteq \cS$ then we call $Y$ a \emph{$B$-version} of $X$. Recall that the \emph{outer measure} associated with $\Pr$ is a map $\Pr^{\circ}\colon 2^{\Omega}\to [0,1]$ given by $\Pr^{\circ}(A) := \inf\left\{ \Pr(B)\colon A\subseteq B\in\cF\right\}$.
\end{definition}

\begin{lemma}\label{lem:version}
 We follow the notation of the preceding definition. Suppose $\cS^*\subset \cS$ is not necessarily measurable, but it satisfies $\Pr^\circ\{ X \in \cS\setminus \cS^*\} = 0$. Suppose also that there is some $y\in\cS^*$ for which $\{y\}\in\ScS$. Then there exists an $\cS^*$-version of $X$.
\end{lemma}

A related notion may be found in \cite[Lemma II.(6.1)]{RogersWilliams}.

\begin{proof}
 Let $y$ be as above. There is some $A\in\cF$ with $X^{-1}(\cS\setminus \cS^*)\subseteq A$ and $\Pr(A) = 0$. For $\omega\in A$ define $X^*(\omega) := y$. For $\omega\in\Omega\setminus A$ define $X^*(\omega) := X(\omega)$. Then $X^*$ is such a version.
\end{proof}

Recall Definition \ref{constr:type-1} of pre-type-1 evolutions, with $\bN_{\beta} = \ConcatIL_{U\in\beta}\bN_U$. Comparing that construction to Proposition \ref{prop:clade_splitting}, we see that each $\bN_U$ has distribution $\mClade^+(\,\cdot\;|\;m^0=\Leb(U))$. 

\begin{definition}\label{constr:type-1_2}
 For $\bN_{\beta}$ as in Definition \ref{constr:type-1}, we abuse notation to write $F^{\geq 0}_0(\bN_{\beta}) := \sum_{U\in\beta}\Dirac{\IPLT_{\beta}(U),\bN_U}$, substituting diversities in the place of local times in Definition \ref{def:bi-clade_PP}. We write $\Pr^1_{\beta}\{F^{\geq 0}_0\in\cdot\,\}$ to denote its distribution, and correspondingly for $\Pr^1_{\mu}$.
\end{definition}

We will find from Propositions \ref{prop:type-1:LT_diversity} and \ref{prop:cts_lt_at_0} that 
 almost surely for all $t\geq 0$, if $t$ falls within the segment of $\bN_{\beta}$ corresponding to $\bN_U$, then
$$\IPLT_{\beta}(U) = \lim_{y\downto 0} \ell^y_{\bN_{\beta}}(t) = \lim_{h\downto 0}h^{-1}\Leb\{u\in [0,t]\colon \xi_{\bN_{\beta}}(u)\in [0,h]\}.$$

\begin{proposition}\label{prop:clade_lengths_summable}
 \begin{enumerate}[label=(\roman*), ref=(\roman*)]
  \item For every $\beta\in\IPspace$, the point process $\bN_{\beta}$ of Definition \ref{constr:type-1} a.s.\ has finite length: in the notation of that definition, $\sum_{U\in\beta}\len(\bN_U) < \infty$ a.s..\label{item:CLS:CLS}
  \item The map $\beta\mapsto \Pr^1_{\beta}$ is a stochastic kernel.\label{item:CLS:kernel}
  \item We have $\xi(\bN_{\beta}) = \ConcatIL_{U\in\beta}\xi(\bN_U)$, where concatenation is as in Definition \ref{def:scaff_concat}.\label{item:CLS:scaffold}
  \item For the purpose of the following, let $\cS := \{N\in\Hfin\colon \skewer(0,N)\in\IPspace\}$. The laws $\Pr^1_{\beta}$ are supported on $\cS$. There exists a measurable map $\phi\colon \cS\to \cNRHf$ such that $F^{\geq 0}_0(\bN_{\beta})=\phi(\bN_{\beta})$. Moreover, the map $\beta \mapsto \Pr^1_{\beta}\{F^{\geq 0}_0\in \cdot\,\}$ is a stochastic kernel.\label{item:CLS:PP}
 \end{enumerate}
\end{proposition}

\begin{proof}
 \ref{item:CLS:CLS} By Corollary \ref{cor:clade_law_given_over}, Proposition \ref{prop:stable_JCCP}, and standard fluctuation theory for L\'evy processes (\cite[Theorem VII.1]{BertoinLevy}; see also Proposition \ref{prop:hitting_time:subord} in our appendix),
 \begin{equation*}
  \mClade^+\left\{ \exp(-\theta \len(N))\;\middle|\;J^+ = y\right\} = \EV\left[e^{-\theta T^{-y}}\right] = \exp\left( -y\psi^{-1}(\theta) \right) = \exp\left( -y\left(\frac{\pi}{2}\right)^{1/3}\theta^{2/3} \right).
 \end{equation*}
 Let $\widetilde{\theta}=\theta\sqrt{\pi/2}$. Applying Proposition \ref{prop:clade:stats} \ref{item:CS:over:mass} and mixing,
 \begin{equation*} 
 \begin{split}
  &\mClade^+\left\{ \exp\left(-\theta \len(N)\right)\;\middle|\;m^0 = a\right\} = \int_0^{\infty} \exp\left(-y\widetilde{\theta}^{2/3}\right) \frac{a^{3/2}}{\sqrt{2\pi}y^{5/2}}e^{-a/2y}dy\\
  	&\qquad\qquad = \frac{1}{\sqrt{2\pi}} \int_0^{\infty} \exp\left(-av\widetilde{\theta}^{2/3}\right) v^{-5/2}e^{-1/2v}dv
  	= \left(1 + \sqrt{2a}\,\widetilde{\theta}^{1/3}\right)\exp\left(-\sqrt{2a}\,\widetilde{\theta}^{1/3}\right),
 \end{split}
 \end{equation*}
 where the second equality comes from the change of variables $y = av$ and the third from integrals that can be found e.g.\ in \cite[Example 33.15]{Sato}. Thus,
 \begin{equation*}
 \begin{split} \EV\left[\exp\left(-\theta \sum\nolimits_{U\in\beta}\len(\bN_U)\right)\right] 
   &=\exp\left(\sum_{U\in\beta}\left(\log\left(1+\sqrt{2\Leb(U)}\widetilde{\theta}^{1/3}\right)-\sqrt{2\Leb(U)}\widetilde{\theta}^{1/3}\right)\right)\\
   &\ge \exp\left( \sum\nolimits_{U\in\beta} -\Leb(U)\widetilde{\theta}^{2/3} -C(\Leb(U))^{3/2}\widetilde{\theta}\right)
 \end{split}
 \end{equation*}
 for some $C>0$ and $\theta$ sufficiently small, since a Taylor series approximation of $\log(1+x)$ gives that the exponent in the above is 
 $-a\widetilde{\theta}^{2/3} + O(a^{3/2}\widetilde{\theta})$.
 Since $\sum_{U\in\beta}\Leb(U) < \infty$, we conclude that the above goes to one as $\theta$ goes to zero. Thus, $\sum_{U\in\beta}\len(\bN_U)$ is a.s.\ finite.
 
 \ref{item:CLS:kernel} For $\beta\in\IPspace$, the counting measure $\sum_{(a,b)\in\beta}\Dirac{a,|b-a|}$ is a measurable function of $\beta$. By Lemma \ref{lem:meas_marking}, the map from $\beta$ to the law of $\bG_{\beta} := \sum_{(a,b)\in\beta}\Dirac{a,\bN_{(a,b)}}$ is a kernel, as we have marked the points via the kernel $|b-a|\mapsto \mClade^+(\,\cdot\,|\; m^0=|b-a|)$. And finally, the map that takes $\bG_{\beta}$ to $\bN_{\beta} = \ConcatIL \bN_{(a,b)}$ by concatenating over points of $\bG_{\beta}$ is measurable.
 
 \ref{item:CLS:scaffold} Let $(\bff_U,\,U\in\beta)$ denote an independent family of \BESQ[-1] processes absorbed at $0$, with each $\bff_U$ starting from $\Leb(U)$. By Lemma \ref{lem:BESQ:length},
 $$\EV[\life(\bff_U)] = \frac{1}{\Gamma(3/2)} \int_0^{\infty} \frac{\Leb(U)}{2x}\sqrt{x}e^{-x}dx = \Leb(U).$$
 For each $U\in \beta$, let $S(U-) := \sum_{V\in\beta\colon V<U}\life(\bff_V)$ and $S(U) := S(U-)+\Leb(U)$. Let $L:= \sup_{U\in\beta}S(U)$. Then
 \begin{equation}
  \EV[L] = \sum_{U\in\beta} \EV\big[\life(\bff_U)\big] = \sum_{U\in\beta} \Leb(U) = \IPmag{\beta} < \infty.
 \end{equation}
 
  Let $\bN$ denote a \PRM[\Leb\otimes\mBxc], independent of $(\bff_U,\,U\in\beta)$. Let $H_U$ denote the first hitting time of $-S(U-)$ in $\xi(\bN)$, let $T$ denote the hitting time of $-L$, and set
 $$\bN_{\beta}' := \Restrict{\bN}{[0,T]} + \sum_{U\in\beta}\DiracBig{S(U),\bff_U}.$$
 It follows from the strong Markov property of $\bN$ that $\bN_{\beta}'$ has law $\Pr^1_{\beta}$. Adding spindles $\DiracBig{S(U),\bff_U}$ to $\bN$ with summable lifetimes modifies the associated scaffolding $\xi(\bN_{\beta}')$ only by adding jumps of the corresponding heights. In particular, $\xi(\bN_{\beta}')$ is formed by concatenating the paths of the excursions $\xi\big(\shiftrestrict{\bN_{\beta}'}{[S(U-),S(U)]}\big)$. Thus, the claimed identity holds a.s.\ under $\Pr^1_{\beta}$.
 
 \ref{item:CLS:PP} First, $\skewer(0,\bN_{\beta}) = \beta$ for every $\beta$, so the laws $\Pr^1_{\beta}$ are supported on $\cS$, as claimed. Now, we need only construct the desired measurable map $\phi$, as the stochastic kernel claim follows from this and assertion (ii). We present $\phi(N)$ in the case $N=\bN_{\beta}$, but this construction applies to any $N\in\cS$. 
 Define
 $$\bG_1 := \sum_{[a,b]\in V^0_0(\bN_{\beta})}\Dirac{a,\shiftrestrict{\bN_{\beta}}{[a,b)},m^0\left(\shiftrestrict{\bN_{\beta}}{[a,b)}\right)}.$$
 Then for each $U \in\beta$ we get $\bN_U = \shiftrestrict{\bN_{\beta}}{[a,b)}$ for some $[a,b]\in V^0_0(\bN_{\beta})$, and
 $$\IPLT_{\beta}(U) = \sqrt{\pi}\lim_{h\downto 0} \sqrt{h}\bG_1\left([0,a)\times\Hxc{\pm}\times (h,\infty)\right) =:  D_{[a,b]}.$$
 Finally, $F^{\geq 0}_0(\bN_{\beta}) = \phi(\bN_{\beta}) := \sum_{[a,b]\in V^0_0(\bN_{\beta})}\DiracBig{D_{[a,b]},\shiftrestrict{\bN_{\beta}}{[a,b)}}$. The measurability of the preceding transformations follows from the measurability of restriction maps and the existence of measurable enumerations of points of a point process, per \cite[Proposition 9.1.XII]{DaleyVereJones2}.
\end{proof}

\begin{definition}[Type-1 evolution]\label{def:type-1}
 For $\beta\in\IPspace$, a \emph{type-1 evolution} starting from $\beta$ is a process $(\alpha^y,\,y\ge0)\in\cCRI$ distributed like $\skewerP(\bN_{\beta}^*)$, where $\bN_{\beta}^*$ is an $\Hfins$-version of $\bN_{\beta}$.
\end{definition}

We will confirm in Proposition \ref{prop:type-1:cts} that such $\Hfins$-versions exist for every $\beta$. We now relate point processes of clades to the skewer process. Recall the cutoff processes of \eqref{eq:cutoff_def}.

\begin{lemma}\label{lem:cutoff_skewer}
 Take $N\in\H$, $y,z\geq 0$, and suppose $M^y_N(t) < \infty$ for all $t < \len(N)$.
 \begin{enumerate}[label=(\roman*), ref=(\roman*)]
  \item We have \label{item:CPS:cutoff_skewer}
  $\displaystyle\skewer(y,N) = \left\{\begin{array}{ll}
  		\skewer\Big(y,\cutoffL{z}{N},\cutoffL{z}{\xi(N)}\Big)		& \text{if }y < z,\\[6pt]
  		\skewer\left(y-z,\cutoffG{z}{N},\cutoffG{z}{\xi(N)}\right)	& \text{if }y > z.
  	\end{array}\right.$\vspace{4pt}
  \item \label{item:CPS:cutoff_skewer_2}
  	For $\bN$ a \PRM[\Leb\otimes\mBxc], it is a.s.\ the case that for every $t\ge 0$,
  	\begin{equation*}
  	 \skewer\left(y,\Restrict{\bN}{[0,t]}\right) = \left\{\begin{array}{ll}
  		\skewer\Big(y,\cutoffL{z}{\restrict{\bN}{[0,t]}}\Big)		& \text{if }y \leq z,\\[6pt]
  		\skewer\left(y-z,\cutoffG{z}{\restrict{\bN}{[0,t]}}\right)	& \text{if }y \geq z.
  	 \end{array}\right.
  	\end{equation*}
  	The same holds for $\bN_{\beta}$, for any $\beta\in\HIPspace$, with $z>0$.
  \item \label{item:CPS:clades_skewer}
  	If level $z$ is nice for $\xi(N)$, in the sense of Proposition \ref{prop:nice_level}, then
  	\begin{equation}\label{eq:skewer_from_clade_PP}
  	 \skewer(y,N) = \Concat_{\text{points }(s,N^+_s)\text{ of }F^{\geq z}_0(N)} \skewer(y-z,N^+_s) \qquad \text{for }y\geq z.
  	\end{equation}
  \item \label{item:CPS:clades_skewer_0}
  	Suppose $\beta\in\IPspace$ is \emph{nice} in the sense that, for $U,V\in\beta$, if $U\neq V$ then $\IPLT_{\beta}(U)\neq \IPLT_{\beta}(V)$. Let $\bN_{\beta}$ and $F^{\geq 0}_0(\bN_{\beta})$ be as in Definition \ref{constr:type-1_2}. In the event that $M^y_{\bN_{\beta}}(t) < \infty$ for all $t < \len(\bN_{\beta})$, \eqref{eq:skewer_from_clade_PP} holds with $z=0$ and $N = \bN_{\beta}$.
 \end{enumerate}
\end{lemma}

\begin{proof}
 \ref{item:CPS:cutoff_skewer}. 
 This follows from the definitions in \eqref{eq:cutoff_def}. See Figure \ref{fig:cutoff}.
 
 Before proceeding to prove \ref{item:CPS:cutoff_skewer_2} we note that, in the general setting of \ref{item:CPS:cutoff_skewer}, we make no claim in the case $y=z$. These formulas fail in cases where $y=z$ and this is a level at which a spindle of $N$ is born or dies discontinuously. For example, if a spindle $f_t$ of $N$ dies discontinuously at level $z$, then there is no broken spindle $\hat f^z_t$ incorporated into $\cutoffG{z}{N}$. Thus, while $\skewer(z,N)$ has a block of size $f_t((z-\xi_N(t-))-)$ corresponding to this spindle, $\skewer(0,\cutoffG{z}{N},\cutoffG{z}{\xi(N)})$ has no such block.
 
 
 \ref{item:CPS:cutoff_skewer_2}. The argument in the proof of Proposition \ref{prop:clade_lengths_summable} \ref{item:CLS:scaffold} also shows that
 \begin{equation}\label{eq:cutoff_scaffold_commute}
  \xi\left(\cutoffL{z}{\restrict{\bN}{[0,t]}}\right) = \cutoffL{z}{\restrict{\xi(\bN)}{[0,t]}}, \qquad \xi\left(\cutoffL{z}{\restrict{\bN_{\beta}}{[0,t]}}\right) = \cutoffL{z}{\restrict{\xi(\bN_{\beta})}{[0,t]}},
 \end{equation}
 and likewise for $\cutoffG{z}{}$. In particular, that argument uses the a.s.\ summability of lifetimes of leftmost spindles $\life(\bff_U)$ to show that the concatenation of the scaffoldings of clades $\ConcatIL_U\xi(\bN_U)$ equals the scaffolding of their concatenation $\xi(\ConcatIL_U\bN_U)$. The same argument applies here.
 
 Next, since no spindles in $\bN$ or $\bN_{\beta}$ are born or die discontinuously at any level above 0, every spindle $f_t$ of $\bN$ or $\bN_{\beta}$ that is alive at level $z>0$ has non-trivial broken components $\check f^z_t$ and $\hat f^z_t$ about level $z$. Thus, the strict inequalities on the conditions for the two expressions in \ref{item:CPS:cutoff_skewer} can be replaced by weak inequalities.
 
 \ref{item:CPS:clades_skewer}. This follows from assertion \ref{item:CPS:cutoff_skewer} via Lemma \ref{lem:cutoff_vs_PP}, which relates $\cutoffG{y}{\cdot\,}$ to $F^{\geq y}_0$, and via the observation that the skewer map commutes with concatenation of bi-clades.
 
 \ref{item:CPS:clades_skewer_0}. Definition \ref{constr:type-1_2} of $F^{\geq 0}_0(\bN_{\beta})$ has the property that, if $\beta$ is nice in the sense described in the assertion, then $F^{\geq 0}_0(\bN_{\beta})$ does not have two points coinciding at the same time. Thus, the conclusion of Lemma \ref{lem:cutoff_vs_PP} applies to it, and $\cutoffG{0}{\bN_{\beta}}$, which equals $\bN_{\beta}$, may be recovered from $F^{\geq 0}_0(\bN_{\beta})$. Thus, the claim follows by the same argument as for assertion (ii).
%
\end{proof}


In light of Lemma \ref{lem:cutoff_skewer}, the skewer process $\skewerP$ as a map on $\Hfins$ is adapted to the filtration $(\cF^{y,*},\,y\geq 0)$, where we take $\cF^{y,*} := \{A\cap\Hfins\colon A\in\cF^y\}$.

\subsection{Path-continuous type-1 evolution from Poisson-Dirichlet initial distribution}
\label{sec:type-1:PRM}

Let $\bN$ be a \PRM[\Leb\otimes\mBxc] living on a probability space $(\Omega,\cA,\Pr)$. We continue to use the notation of the first paragraph of Section \ref{sec:biclade_PRM} for objects related to $\bN$. Let $(\ol\cF_t,\,t\ge0)$ and $(\ol\cF^y,\,y\ge0)$ denote $\Pr$-completions of the time- and level-filtrations on $(\Omega,\cA)$ generated by $\bN$, as in Definition \ref{def:filtrations}, augmented to allow an independent random variable $S$ measurable in $\ol\cF^0\cap\ol\cF_0$. That is, these are formed by augmenting the $\Pr$-completions of the pullbacks, via $\bN\colon \Omega\to\H$, of the time- and level-filtrations on $\H$. 

We define $\tdN := \restrict{\bN}{[0,T)}$, where $T$ is an a.s.\ finite $(\ol\cF_t)$-stopping time. We take ``twiddled versions'' of our earlier notation to denote the corresponding objects for $\tdN$; for instance, $\tdl$ will denote the jointly H\"older continuous version of the local time process associated with $\tdX := \xi(\tdN)$. 
It follows from Proposition \ref{prop:agg_mass_subord} and the a.s.\ finiteness of $T$ that for each $y\geq 0$ we have $\td\alpha^y := \skewer(y,\tdN)\in\IPspace$ almost surely. I.e.\ $\td\alpha^y$ is almost surely a finite interval partition with the diversity property. 

\begin{proposition}\label{prop:PRM:Fy-_Fy+}
 Suppose $T$ has the properties: (a) $S^0 := \ell^0(T)$ is measurable in $\ol\cF^0$, and (b) $\bX < 0$ on the time interval $(\tau^0(S^0-),T)$. Then for each $y\geq 0$, the measure $\tdF^{\geq y}_0 = F^{\geq y}_0(\tdN)$ is conditionally independent of $\ol\cF^{y}$ given $\td\alpha^y$, with the regular conditional distribution (r.c.d.) $\Pr^1_{\td\alpha^y}\{F^{\geq 0}_0\in\cdot\,\}$ of Definition \ref{constr:type-1_2}.
\end{proposition}

In light of Lemma \ref{lem:cutoff_skewer} \ref{item:CPS:cutoff_skewer}, this proposition is very close to a simple Markov property for $(\td\alpha^y,\ y\geq 0)$. In order to minimize our involvement with measure-theoretic technicalities, we will postpone pinning this connection down until Corollary \ref{cor:type-1:simple_Markov}.

\begin{proof}
 Step 1 of this proof establishes the claimed result at a fixed level $y\geq 0$ when $T = \tau^y(s-)$, where $s>0$ is fixed. Note that this time does not satisfy conditions (a) and (b). In Step 1, $T$ is specific to a fixed level $y$, whereas in the proposition, the result holds at each level for a single time $T$. 
 In Step 2, we extend this to describe the unstopped point process $\bF^{\geq y}_0$. Finally, in Step 3, we extend our results to the regime of the proposition.
 
 \emph{Step 1}: Assume $T = \tau^y(s-)$. Note that $\tdF^y_0 = \restrict{\bF^y_0}{[0,s)}$. The strong Markov property of $\bN$ tells us that $\restrict{\bN}{[0,T^y)}$ is independent of $\shiftrestrict{\bN}{[T^y,\infty)}$. Rephrasing this in the notation of Definition \ref{def:bi-clade_PP}, $(\bN^{\leq y}_{\textnormal{first}},\bN^{\geq y}_{\textnormal{first}})$ is independent of $(\bF^{\leq y},\bF^{\geq y})$. This will allow us to consider conditioning separately for the first pair and the second. Let $m^y\colon \H\to [0,\infty)$ denote the mass of the leftmost spindle at level $y$:
 \begin{equation}\label{eq:LMB_def}
  m^y(N) := M^y_N\big(\inf\{ t\geq 0\colon M^y_N(t) >0\}\big) \qquad \text{for }N\in\H.
 \end{equation}
 We apply the mid-spindle Markov property, Lemma \ref{lem:mid_spindle_Markov}, at time $T^{\geq y}$. Together with the description of $\mClade^+$ in Proposition \ref{prop:clade_splitting}, this implies that the clade $\bN^{\geq y}_{\textnormal{first}}$ has conditional law $\mClade^+(\,\cdot\;|\;m^0 = m^y(\bN^{\leq y}_{\textnormal{first}}))$ given $\bN^{\leq y}_{\textnormal{first}}$, as desired.
 
 Now, let $\td\gamma^y$ denote $\td\alpha^y$ minus its leftmost block, so that $\td\alpha^y = \{(0,m^y(\tdN))\}\concat\td\gamma^y$. Proposition \ref{prop:agg_mass_subord} indicates two properties of $\td\gamma^y$: (a) it is a \Stable[\frac12] interval partition with total diversity $s$, in the sense of Proposition \ref{prop:IP:Stable}, and (b) it a.s.\ equals a function of $\tdF^y$. For $\beta\in\IPspace$ let $(N^{\pm}_U,\ U\in\beta)$ denote a family of independent bi-clades with respective distributions $N^{\pm}_U \sim \mClade\{\,\cdot\;|\,m^0=\Leb(U)\}$,
 $$G_{\beta} := \sum_{U\in\beta} \Dirac{\IPLT_{\beta}(U),\Leb(U)}, \quad \text{and} \quad G^{\pm}_{\beta} := \sum_{U\in\beta} \Dirac{\IPLT_{\beta}(U),N^{\pm}_U}.$$
 Then $G_{\td\gamma^y}$ is a \PRM[\Leb\otimes \mClade\{m^0\in \cdot\,\}] on $[0,s)\times (0,\infty)$. Moreover, $G^{\pm}_{\td\gamma^y}$ is a \PRM[\Leb\otimes\mClade] on $[0,s)\times \Hxc{\pm}$, as it may be obtained by marking the points of $G_{\beta}$ via the stochastic kernel $a\mapsto \mClade\{\,\cdot\;|\,m^0 = a)$, and this is an $m^0$-disintegration of $\mClade$. By Proposition \ref{prop:bi-clade_PRM}, $\tdF^{y}$ has the same \PRM\ distribution as $G^{\pm}_{\td\gamma^y}$. Thus, the distribution of $G^{\pm}_{\td\gamma^y}$ is a regular conditional distribution for $\tdF^y$ given $\td\gamma^y$.
 
 Extending the preceding construction of $G^{\pm}_\beta$, for each $U\in\beta$ let $(N^+_U,N^-_U)$ denote the clade and anti-clade components of $N^{\pm}_U$, respectively. By Proposition \ref{prop:clade_splitting} these are independent. Thus,
 $$G^+_{\beta} := \sum_{U\in\beta} \Dirac{\IPLT_{\beta}(U),N^{+}_U} \quad \text{is independent of} \quad G^{-}_{\beta} := \sum_{U\in\beta} \Dirac{\IPLT_{\beta}(U),N^{-}_U}.$$
 Moreover, $G^+_{\beta}$ has law $\Pr^1_{\beta}\{F^{\geq 0}_0\in\cdot\,\}$, as in Definition \ref{constr:type-1_2}. Thus, given $\td\gamma^y$, the measure $\tdF^{\geq y}$ is conditionally independent of $\tdF^{\leq y}$ with regular conditional distribution $\Pr^1_{\td\gamma^y}\{F^{\geq 0}_0\in\cdot\,\}$. By another application of the strong Markov property of $\bN$ at time $T$, this conditional independence extends to conditional independence between $\tdF^{\geq y}$ and $\bF^{\leq y}$. 
 We finish with an appeal to the general principle that from $\cF_1\indep_{\cH_1}\cG_1$, $\cF_2\indep_{\cH_2}\cG_2$, and $(\cF_1,\cG_1,\cH_1)\indep (\cF_2,\cG_2,\cH_2)$, we may deduce $(\cF_1,\cF_2)\indep_{\cH_1,\cH_2}(\cG_1,\cG_2)$; see e.g.\ \cite[Propositions 6.6-6.8]{Kallenberg}. Thus, $\tdF^{\geq y}_0$ is conditionally independent of $\ol\cF^y$ given $\td\alpha^y$, with regular conditional distribution $\Pr^1_{\td\alpha^y}\{F^{\geq 0}_0\in\cdot\,\}$.
 
 
 \emph{Step 2}: For $s>0$ let $\alpha^y_s := \skewer(y,\restrict{\bN}{[0,\tau^y(s-))}$. We write $\alpha^y_\infty := (\alpha^y_n,\ n\in\BN)$; this takes values in the subset of $\IPspace^{\BN}$ comprising projectively consistent sequences. We equip $\IPspace^{\BN}$ with the product $\sigma$-algebra. In the regime of such projectively consistent sequences, Definition \ref{constr:type-1_2} extends naturally to define a kernel $\beta_\infty = (\beta_n,\ n\geq 1) \mapsto \Pr^1_{\beta_\infty}\{F^{\geq 0}_0\in\cdot\,\}$; i.e.\ a point process $G$ has this law if $\restrict{G}{[0,n)} \stackrel{d}{=} \sum_{U\in\beta_{n}} \Dirac{\IPLT_{\beta_{n}}(U),N^+_U}$ for every $n\geq 0$, 
 where the $(N^+_U)$ are as above. Extending the conditioning in the conclusion of Step 1, we find that $\restrict{\bF^{\geq y}_0}{[0,n)}$ is conditionally independent of $\ol\cF^y$ given $\alpha^y_\infty$. By consistency, $\bF^{\geq y}_0$ is conditionally independent of $\ol\cF^y$ given $\alpha^y_\infty$, with r.c.d.\ $\Pr^1_{\alpha^y_\infty}\{F^{\geq 0}_0\in\cdot\,\}$.
 
 \emph{Step 3}: Assume $T$ satisfies conditions (a) and (b) stated in the proposition. We now show that $S^y := \ell^y(T)$ is measurable in $\ol\cF^y$. For $y=0$, this is exactly condition (a), so assume $y>0$. 
 From condition (b), $S^y = \ell^y(\tau^0(S^0-))$. Thus, $\tau^0(S^0-)\in (\tau^y(S^y-),\tau^y(S^y))$. By monotonicity of $\ell^0$ we have $S^0\in [\ell^0(\tau^y(S^y-)),\ell^0(\tau^y(S^y))]$. In fact, we cannot have $S^0 = \ell^0(\tau^y(S^y-))$, since then we would have $\tau^y(S^y-)\in (\tau^0(S^0-),T)$ while $\bX(\tau^y(S^y-)) = y > 0$, which would violate condition (b). We conclude that $S^y = \inf\{s\ge0\colon \ell^0(\tau^y(s)) \ge S^0\}$. Finally,
 $$\ell^0(\tau^y(s)) = \ell^0_{N^{\leq y}_{\textnormal{first}}}(\infty) + \int_{[0,s]\times\Hxc{-}}\ell^{-y}_N(\infty)d\bF^{\leq y}(r,N),$$
 which is measurable in $\ol\cF^y$, as desired. 
 
 
 Condition (b) has the additional consequence that time $T$ occurs in the midst of a (possibly incomplete) bi-clade about level $y$ at local time $S^y$, no later than the jump across level $y$. Thus, the clade that follows at local time $S^y$ is entirely excluded from $\tdN$, so $\tdF^{\geq y}_0 = \restrict{\bF^{\geq y}_0}{[0,S^y)}$.
 
 Appealing to the result of Step 2, $S^y$ is conditionally independent of $\bF^{\geq y}_0$ given $\alpha^y_\infty$. Thus, $\Pr^1_{\alpha^y_\infty}\{F^{\geq 0}_0\in\cdot\,\}$ is a regular conditional distribution for $\bF^{\geq y}_0$ given $(\alpha^y_\infty,S^y)$. Consequently, for $f$ non-negative and measurable on the appropriate domain,
 $$\EV\big[ f\big(\bF^{\geq y}_0,S^y\big) \big] = \int f(G,s)\Pr^1_{\beta_\infty}\{F^{\geq 0}_0\in dG\}\Pr\{\alpha^y_\infty\in d\beta_\infty,S^y\in ds\}.$$
 For the purpose of the following, for $(G,s)$ as above we will write $G_{< s} := \restrict{G}{[0,s)}$ and $G_{\geq s} := \shiftrestrict{G}{[s,\infty)}$. Similarly, modifying our earlier notation, for $\beta_\infty = (\beta_n,\ n\geq 1)$ as in Step 2, we will write $\beta_{<s}$ to denote the set of blocks of $\beta_{\infty}$ prior to diversity $s$, and $\beta_{\geq s}$ will denote the remainder, shifted to start at left endpoint zero. More formally, $\beta_{<s} := \{U\in \beta_{\lceil s\rceil}\colon \IPLT_{\beta_{\lceil s\rceil}}(U) < s\}$ and $\beta_{\geq s} := (\beta_{\geq s,n},\ n\geq 1)$ where, for $n\geq 1$,
 $$\beta_{\geq s,n} := \big\{(a-\IPmag{\beta_{<s}},b-\IPmag{\beta_{<s}})\colon (a,b)\in \beta_{\lceil s\rceil + n},\ \IPLT_{\beta_{\lceil s\rceil + n}}(a) \in [s,s+n)\big\}.$$
 Now, suppose that in our earlier disintegration calculation, $f(G,s) := h(G_{<s})$. Then
 \begin{equation*}
 \begin{split}
  \EV\big[ h\big(\tdF^{\geq y}_0\big) \big] &= \int h(G_{<s})\Pr^1_{\beta_\infty}\{F^{\geq 0}_0\in dG\}\Pr\{\alpha^y_\infty\in d\beta_\infty,S^y\in ds\}\\
  		&= \int h(G_{<s})\Pr^1_{\beta_{<s}}\{F^{\geq 0}_0\in dG_{<s}\}\Pr^1_{\beta_{\geq s}}\{F^{\geq 0}_0\in dG_{\geq s}\}\Pr\{\alpha^y_\infty\in d\beta_\infty,S^y\in ds\}\\
  		&= \int h(G_{<s})\Pr^1_{\beta_{<s}}\{F^{\geq 0}_0\in dG_{<s}\}\Pr\{\alpha^y_\infty\in d\beta_\infty,S^y\in ds\}.
 \end{split}
 \end{equation*}
 The second line above comes from noting that $\restrict{F^{\geq 0}_0}{[0,s)}$ is independent of $\shiftrestrict{F^{\geq 0}_0}{[s,\infty)}$ under $\Pr^1_{\beta_\infty}$, and the third line comes from integrating out the $\Pr^1_{\beta_{\geq s}}$ term. Noting that $\td\alpha^y = \alpha^y_{< S^y}$, we conclude that $\Pr^1_{\td\alpha^y}\{F^{\geq 0}_0\in\cdot\,\}$ is a regular conditional distribution for $\tdF^{\geq y}_0$ given $(\alpha^y_\infty,S^y)$. We already have the desired conditional independence from $\ol\cF^y$. Finally, since this r.c.d.\ depends only on $\td\alpha^y$, it is also an r.c.d.\ given $\td\alpha^y$.
\end{proof}

We now revert to the general setting of $T$ being an a.s.\ finite $(\ol\cF_t)$-stopping time.

\begin{lemma}[Proposition 6 of \cite{Paper0}]\label{lem:spindle_piles}
 Fix $\theta\in \left(0,\frac12\right)$. For each spindle $(t,f)$ of $\tdN$, let $\bar f(y) := f(y-\tdX(t-))$, $y\in \BR$. These translated excursions can a.s.\ be partitioned into sequences $(g^n_j,\,j\geq 1)$, for $n\geq 1$, in such a way that in each sequence $(g^n_j,\,j\geq 1)$:
 (i) the excursions have disjoint support, and (ii) they are uniformly H\"older-$\theta$ with some constants $D_n$, with $\sum_{n\ge1}D_n < \infty$.
\end{lemma}

\begin{corollary}\label{cor:PRM:type-1_wd}
 It is a.s.\ the case that $M^y_{\tdN}(T)$ is finite for all $y\in\BR$. Moreover, $(\td\alpha^y,\ y\in\BR)$ a.s.\ takes values in $\IPspace$ for all $y\in\BR$.
\end{corollary}

\begin{proof}
 It suffices to check the first assertion: by Theorem \ref{thm:LT_property_all_levels}, this implies that $\td\alpha^y\in\IPspace$ for all $y$ simultaneously, almost surely. Fix $\theta\in \left(0,\frac12\right)$. For each $n\geq 1$ let the $(g^n_j,\ j\geq 1)$ and $D_n$ be as in Lemma \ref{lem:spindle_piles}. Definition \ref{def:skewer_2} of $M^y$ gives
 \begin{equation}
  M^y_{\tdN}(T) = \sum_{n\geq 1}\sum_{j\geq 1} g^n_j(y) \qquad \text{for }y\in\BR.\label{eq:agg_mass_sum_piles}
 \end{equation}
 Let $g^n := \sum_{j\geq 1} g^n_j$ for each $n\geq 1$. Since the $g^n_j$ in each sequence have disjoint support, $g^n$ is H\"older-$\theta$ with constant $D_n$. Proposition \ref{prop:agg_mass_subord} implies that $M^0_{\tdN}(T)$ is a.s.\ finite. Thus, by \eqref{eq:agg_mass_sum_piles}, $y\mapsto M^{y}_{\tdN}(T)$ is almost surely H\"older-$\theta$ with constant bounded by $\sum_{n\geq 1}D_n$.
\end{proof}

Recall Definition \ref{def:domain_for_skewer} of $\Hfins$, the subspace of $\H$ on which the skewer map measurably produces a continuously evolving interval partition.

\begin{proposition}\label{prop:PRM:cts_skewer}
 There exists an $\Hfins$-version of $\tdN$ and an $\Hs$-version of $\bN$. In particular, $(\td\alpha^y,\,y\ge0)$ is a.s.\ H\"older-$\theta$ in $(\IPspace,\dI)$ for every $\theta\in \big(0,\frac14\big)$.
\end{proposition}


\begin{proof}
 We have already shown in Theorem \ref{thm:LT_property_all_levels} and Corollary \ref{cor:PRM:type-1_wd} that $\tdN$ (respectively $\bN$) satisfies the first two conditions in Definition \ref{def:domain_for_skewer} for membership in $\Hfins$ (resp.\ $\Hs$). It remains only to prove the claimed H\"older continuity.
 
 Fix $\theta\in \left(0,\frac14\right)$. For $n,j\geq 1$, let $g^n_j$ and $D_n$ be as in Lemma \ref{lem:spindle_piles} and let $g^n := \sum_{j\geq 1} g^n_j$. Since $\tdN$ is stopped at an a.s.\ finite time, the path of $\tdX$ lies within a random bounded space-time rectangle. We restrict our attention to the intersection of the almost sure events posited by Lemma \ref{lem:spindle_piles}, Corollary \ref{cor:PRM:type-1_wd}, and Theorem \ref{thm:Boylan}: that the H\"older constants $D_n$ are summable, the process $(\td\alpha^y)$ lies in $\IPspace$, and the local times $(\tdl^y(t))$ are uniformly H\"older-$\theta$ in level and continuous in time. Let
 \begin{equation*}
  C := \sup_{-\infty<y<z<\infty}\sup_{t\in [0,T]}\frac{|\tdl^z(t) - \tdl^y(t)|}{|z-y|^{\theta}} \qquad \text{and} \qquad D := \sum_{n=1}^\infty D_n.
 \end{equation*}
 
 Fix $y,z\in \BR$ with $y<z$. Let $A := \{n\geq 1\colon \Restrict{g^n}{[y,z]} > 0\}$. That is, $A$ is the set of indices $n$ for which a single spindle in the sequence $(g^n_j)_{j\geq 1}$ survives the interval $[y,z]$. For each $n\in A$, let $t_n$ denote the time at which that particular spindle arises as a point in $\tdN$. Recall Definition \ref{def:IP:metric} of $\dI$ and correspondences between interval partitions. Consider the correspondence from $\td\alpha^y$ to $\td\alpha^z$ that, for each $n\in A$, pairs the block $U_n^y\in\td\alpha^y$ with $U_n^z\in\td\alpha^z$, where there are the blocks corresponding to $g^n$. This is indeed a correspondence, respecting order in the two interval partitions, since each paired block corresponds to the same spindle as its partner.
 
 Note that for $n\notin A$ there is some $x\in [y,z]$ for which $g^n(x) = 0$. Thus, by its H\"older continuity, both $g^n(y)$ and $g^n(z)$ are bounded by $D_n(z-y)^{\theta}$. Therefore,
 \begin{equation*}
 \begin{split}
  \sum_{n\in A} |g^n(z) - g^n(y)| + \max\left\{\sum_{n\notin A}g^n(y),\sum_{n\notin A}g^n(z)\right\} &\leq \sum_{n=1}^\infty D_n(z-y)^{\theta},\\
  \sup_{n\in A}\left|\IPLT_{\td\alpha^z}(U_n^z) - \IPLT_{\td\alpha^y}(U_n^y)\right| = \sup_{n\in A}\left|\tdl^z(t_n) - \tdl^y(t_n)\right| &\leq C(z-y)^{\theta},\\
  \text{and} \quad \left|\IPLT_{\td\alpha^z}(\infty) - \IPLT_{\td\alpha^y}(\infty)\right| = \left|\tdl^z(T) - \tdl^y(T)\right| &\leq C(z-y)^{\theta}.
 \end{split}
 \end{equation*}
 By Definition \ref{def:IP:metric} of $\dI$, we conclude that $(\td\alpha^y,\ y\in\BR)$ is H\"older-$\theta$ with constant bounded by $\max\{C,D\}$.
\end{proof}

We can now prove the existence of certain type-1 evolutions. For the purpose of the following, let $\tdN^*$ denote an $\Hfins$-version of $\tdN$ and let $(\alpha^{y,*},\ y\geq 0) := \skewerP\big(\tdN^*\big)$.

\begin{corollary}\label{cor:type-1:cts_from_Stable}
 Let $S>0$ be independent of $\bN$.
 \begin{enumerate}[label=(\roman*), ref=(\roman*)]
  \item If $T := \tau^0(S)$ then $\alpha^{0,*}$ is a \Stable[\frac12]\ interval partition with total diversity $S$.\label{item:c_f_S:determinstic}
  \item If $T := \inf\big\{t > 0\colon M^0_{\bN}(t) > S\big\}$ and $S\sim\ExpDist[\rho]$ for some $\rho>0$ then $\alpha^{0,*}$ is distributed like a \PDIP[\frac12,\frac12] multiplied by an independent $\GammaDist[\frac12,\rho]$ scaling factor.\label{item:c_f_S:mass}
 \end{enumerate}
 In either case, $(\alpha^{y,*},\ y\geq 0)$ is a type-1 evolution.
\end{corollary}

\begin{proof}
 First, the claimed distributions for $\td\alpha^0$ (and therefore for $\alpha^{0,*}$) follow from the \Stable[\frac12] description of $M^y_\bN$ in Proposition \ref{prop:agg_mass_subord} and the definitions of the \Stable\ and \PDIP\ interval partition laws in Propositions \ref{prop:IP:Stable} and \ref{prop:PDIP}.
 
 Next, note that in either of the cases, $\tdF^{\geq 0} = \restrict{\bF^{\geq 0}}{[0,\ell^0(T))}$ almost surely. In case (i) there is a.s.\ no bi-clade of $\bN$ about level 0 at local time $\ell^0(T) = S$. In case (ii), time $T$ occurs at the middle spindle of a bi-clade, so $\tdN$ cuts off before the clade component of that final incomplete bi-clade. From Proposition \ref{prop:PRM:Fy-_Fy+} applied at level $0$, we see that $\tdF^{\geq 0}_0$ has regular conditional distribution $\Pr^1_{\td\alpha^0}\{F^{\geq 0}_0\in\cdot\,\}$ given $\td\alpha^0$. Thus, it has law $\Pr^1_{\mu}\{F^{\geq 0}_0\in\cdot\,\}$, where $\mu$ is the law of $\td\alpha^0$. Therefore, $F^{\geq 0}_0(\tdN^*)$ has law $\Pr^{1,*}_{\mu}\{F^{\geq 0}_0\in\cdot\,\}$. From Lemma \ref{lem:cutoff_skewer} \ref{item:CPS:clades_skewer_0} and Proposition \ref{prop:nice_level}, since level 0 is a.s.\ nice for $\bN$ and thus for $\tdN^*$, we conclude that $(\alpha^{y,*},\ y\geq 0)$ has law $\Pr^{1,*}_{\mu}\{\skewerP\in\cdot\,\}$. Therefore, it satisfies Definition \ref{def:type-1} of a type-1 evolution.
\end{proof}

\subsection{Entrance law of type-1 evolution starting from a single block}
\label{sec:type-1:clade}


On a suitable probability space $(\Omega,\cA,\Pr)$ let $\bN$ be a \PRM[\Leb\otimes\mBxc]. We continue to use the notation of the first paragraph of Section \ref{sec:biclade_PRM} for objects related to $\bN$. 
%
%
Fix $a>0$ and let $\bff$ be a \BESQ[-1] starting from $a$ and absorbed upon hitting zero, independent of $\bN$. 
Let $\olN := \Dirac{0,\bff} + \bN$. We use barred versions of our earlier notation to refer to the corresponding objects associated with $\olN$. For example, $\olX = \bX + \zeta(\bff)$. Let $\ol T^0 = T^{-\life(\bff)}$ denote the first hitting time of 0 by $\olX$ and set $\widehat\bN := \restrict{\olN}{[0,\ol T^0)}$. By Proposition \ref{prop:clade_splitting}, $\widehat\bN$ has distribution $\mClade^+(\,\cdot\;|\;m^0 = a)$. We use hatted versions of our earlier notation to refer to the corresponding objects associated with $\widehat\bN$. Set $(\widehat\alpha^y,\,y\geq 0) := \skewerP(\widehat\bN)$.

Let $(\ol\cF_t,\,t\ge0)$ and $(\ol\cF^y,\,y\ge0)$ denote 
$\Pr$-completions of the pullbacks, via $\olN\colon \Omega\to\H$, of the time- and level-filtrations on $\H$, as in Definition \ref{def:filtrations}.

\begin{corollary}\label{cor:clade:cts_skewer}
 There exists an $\Hfins$-version of $\widehat\bN$. In particular, $(\widehat\alpha^y,\ y\geq 0)$ is a type-1 evolution starting from $\{(0,a)\}$, and it is a.s.\ H\"older-$\theta$ in $(\IPspace,\dI)$ for every $\theta\in \big(0,\frac14\big)$.
\end{corollary}

\begin{proof}
 For the purpose of the following let $\td\bN := \restrict{\bN}{[0,T^{-\life(\bff)}]} = \widehat\bN - \Dirac{0,\bff}$. Note that $\td\bN$ is in the regime of processes considered in Section \ref{sec:type-1:PRM}. 
 By Proposition \ref{prop:PRM:cts_skewer}, $\td\bN \in \Hfins$ almost surely. Let $\td\bX := \xi(\tdN)$ and $(\td\alpha^y,\,y\ge0) := \skewerP(\tdN)$. Then
 \begin{equation*}
  \widehat\bX = \td\bX + \life(\bff), \quad M^y_{\widehat\bN}\left(\ol T^0\right) = M^{y-\life(\bff)}_{\td\bN}\left(\ol T^0\right) + \bff(y), \quad \text{and} \quad \widehat\alpha^y = \big\{(0,\bff(y))\big\} \concat \td\alpha^{y-\life(\bff)}.
 \end{equation*}
 By Definition \ref{def:domain_for_skewer}, in order to have $\widehat\bN\in\Hfins$ we require that: $(\widehat\alpha^y,\,y\ge0)$ is continuous in $y$,
 $$M^y_{\widehat\bN}\left(\ol T^0\right) < \infty, \qquad \text{and} \qquad \wh\ell^y(t) = \IPLT_{\widehat\alpha^y}\left(M^y_{\widehat\bN}(t)\right)\qquad \text{for }t\geq 0,\ y\in\left(0,\life^+\left(\whN\right)\right),$$
 where $\life^+$ denotes clade lifetime, as in Definition \ref{def:clade_stats}. 
 In light of the connections between $\widehat\bN$ and $\td\bN$ mentioned above, these three properties follow from the corresponding properties for $\td\bN$, noted in Proposition \ref{prop:PRM:cts_skewer}. That proposition further implies that $(\td\alpha^y)$ is a.s.\ H\"older-$\theta$ for $\theta\in \big(0,\frac14\big)$. By Lemma \ref{lem:BESQ:Holder}, $\bff$ is a.s.\ H\"older-$\theta$ for $\theta\in \big(0,\frac12\big)$. Thus, $(\widehat\alpha^y)$ is a.s.\ formed by concatenating two H\"older-$\theta$ processes, so the claimed H\"older continuity follows from Lemma \ref{lem:IP:concat} on concatenation. 
\end{proof}

We consider a non-negative random variable $L^y$ with law
\begin{equation}
 \Pr(L^y\in db) = \frac{1}{\sqrt{2\pi}}\frac{\sqrt y}{b^{3/2}}\frac{e^{-b/2y}}{e^{a/2y}-1}
		\left(1 - \cosh\left(\frac{\sqrt{ab}}{y}\right) + \frac{\sqrt{ab}}{y}\sinh\left(\frac{\sqrt{ab}}{y}\right)\right)db.\label{eq:transn:LMB}
\end{equation}
Let $R^y$ be an inverse-Gaussian subordinator with L\'evy measure $\Pi^y$ and Laplace exponent $\Phi^y$ as follows:
\begin{equation}
 \Pi^y(dx) = \frac{1}{2\sqrt{\pi}}x^{-3/2}e^{-x/2y}dx	\qquad \text{and} \qquad \Phi^y(\lambda) = \left(\lambda+\frac{1}{2y}\right)^{1/2}-\left(\frac{1}{2y}\right)^{1/2}.\label{eq:transn:Levy_meas}
\end{equation}

\begin{proposition}[Entrance law for type-1 evolution from $\{(0,a)\}$]\label{prop:type-1:transn}
 The lifetime of $(\widehat\alpha^y)$ has \InvGammaDist[1,a/2] distribution, i.e.
 \begin{equation}
  \Pr\big\{\life^+\big(\whN\big)> y\big\} = \Pr(\widehat\alpha^y \not= \emptyset) = 1 - e^{-a/2y} \qquad \text{for }y>0.\label{eq:transn:lifetime}
 \end{equation}
 For $y>0$, consider an independent triple $(L^y,S^y,R^y)$, where $S^y \sim \ExpDist[(2y)^{-1/2}]$ and $L^y$ and $R^y$ are as in \eqref{eq:transn:LMB}-\eqref{eq:transn:Levy_meas}. 
 The conditional law of $\widehat\alpha^y$ given the event $\{\widehat\alpha^y\neq \emptyset\}$ equals the law of
 \begin{equation}
  \{(0,L^y)\} \concat \{(R^y(t-),R^y(t))\colon t\in [0,S^y],\ R^y(t-)<R^y(t)\}. \label{eq:transn:given_life}
 \end{equation}
\end{proposition}

\begin{remark}\label{rmk:transn:PDIP}
 Following \cite[Proposition 21]{PitmYorPDAT} and the characterization of the \PDIP\ laws in Proposition \ref{prop:PDIP} \ref{item:PDIP:Stable}, if we take $B^y\sim\GammaDist[\frac12,1/2y]$ independent of $\ol\gamma\sim\PDIP[\frac12,\frac12]$, then
 \begin{equation}\label{eq:transn:PDIP}
  \{(R^y(t-),R^y(t))\colon t\in [0,S^y],\ R^y(t-)<R^y(t)\} \stackrel{d}{=} B^y\scaleI\ol\gamma.
 \end{equation}
\end{remark}

We restate the claim regarding $L^y$ in terms of the leftmost spindle mass, $m^y(\whN)$, of \eqref{eq:LMB_def}.

\begin{lemma}\label{lem:LMB}
 For $a,b,y>0$,
 \begin{equation*}
  \mClade\left\{ m^y\in db\;\middle|\begin{array}{l} m^0 = a,\\ \life^+ > y \end{array}\right\} =
  		\frac{1}{\sqrt{2\pi}}\frac{\sqrt y}{b^{3/2}}\frac{e^{-b/2y}}{e^{a/2y}-1}
		\left(1 - \cosh\left(\frac{\sqrt{ab}}{y}\right) + \frac{\sqrt{ab}}{y}\sinh\left(\frac{\sqrt{ab}}{y}\right)\right)db.
 \end{equation*}
\end{lemma}

We prove this lemma at the end of Appendix \ref{sec:clade_stats}.

\begin{proof}[Proof of Proposition \ref{prop:type-1:transn}]
 By construction, $\life(\bff)$ is independent of $\bN$. Thus, by Proposition \ref{prop:nice_level}, level $y$ is a.s.\ nice for $\olX$; henceforth we restrict to that event. By Proposition \ref{prop:bi-clade_PRM} and the aforementioned independence, the point process $\olF^y = \bF^{y-\life(\bff)}$ is a \PRM[\Leb\otimes\mClade]. Let $\widehat S^y := \oll^y(\ol T^0)$. If $\widehat\bN$ survives past level $y$ then $\widehat S^y$ is the level $y$ local time at which some excursion of $\olX$ about level $y$ first reaches down to level zero:
 \begin{equation*}
  \widehat S^y = \cf\big\{\life^+\big(\whN\big) > y\big\}\inf\left\{s > 0\colon \olF^y\big([0,s]\times\{N\in\Hxc{\pm}\colon \life^-(N) \geq y\}\big) > 0 \right\}.
 \end{equation*}
 Conditionally given the event $\{\life^+(\widehat\bN) > y\}$ of survival beyond level $y$, it follows from the Poisson property of $\olF^y$ and the description of $\mClade\{\life^-\in\cdot\,\}$ in Proposition \ref{prop:clade:stats} \ref{item:CS:max} that $\widehat S^y \sim \ExpDist[(2y)^{-1/2}]$, which is the distribution of $S^y$. In light of this, up to null events,
 \begin{equation}\label{eq:type-1:no_revival}
  \left\{\life^+\left(\widehat\bN\right) \leq y\right\} = \left\{\wh S^y = 0\right\} = \left\{\widehat\bF^{\geq y} = 0\right\} = \left\{\widehat\alpha^y = \emptyset\right\}.
 \end{equation}
 Recall from Proposition \ref{prop:clade_splitting} that $\whN\sim\mClade\{\,\cdot\;|\;m^0=a\}$. Thus, \eqref{eq:transn:lifetime} follows from the formula for $\mClade^+\{\life^+ > z\;|\;m^0=a\}$ stated in Proposition \ref{prop:clade:stats} \ref{item:CS:max:mass}.
 
 Assuming $\life^+(\widehat\bN) > y$, time $\ol T^0$ occurs during an anti-clade of $\olN$ below level $y$ at local time $\widehat S^y$. In particular, the subsequent level $y$ clade, also at local time $\wh S^y$, is cut entirely from $\whN$. Thus, $\widehat\bF^{\geq y} = \restrict{\olF^{\geq y}}{[0,\widehat S^y)}$. That is, $\whF^{\geq y}$ is obtained from $\olF^{\geq y}$ by Poisson thinning. By Proposition \ref{prop:clade:stats} assertions \ref{item:CS:mass}, \ref{item:CS:max}, and \ref{item:CS:mass:max},
 \begin{equation*}
 \begin{split}
  \mClade\{ m^0\in db;\ \zeta^- < y \} &= \mClade\{ m^0\in db \} - \mClade\{m^0\in db\;|\;\zeta^- > y \} \mClade\{ \zeta^- > y \}\\
  	&= \frac{1}{2\sqrt{\pi}}b^{-3/2}db - \frac{1}{\sqrt{2}}y^{-1/2}\frac{y^{1/2}}{\sqrt{2\pi}}(1-e^{-b/2y})b^{-3/2}db\\
  	&= \frac{1}{2\sqrt{\pi}}b^{-3/2}e^{-b/2y}db = e^{-b/2y}\mClade\{ m^0\in db\}.
 \end{split}
 \end{equation*}
 Thus, it follows from \eqref{eq:agg_mass_from_clades} that the conditional law of $\big(M^y_{\whN}\circ\tau^y_{\whN}(s)-M^y_{\whN}\circ\tau^y_{\whN}(0),\ s\in [0,\wh S^y]\big)$ given $\{\life^+(\widehat\bN) > y\}$ equals the law of $\restrict{R^y}{[0,S^y]}$. Thus, appealing to \eqref{eq:type-1:no_revival}, 
 the conditional distribution of $\wh\alpha^y$ minus its leftmost block given $\{\wh\alpha^y\neq\emptyset\}$ is as described in \eqref{eq:transn:given_life}.
 
 The mass $m^y(\whN)$ of the leftmost block is a function of $\restrict{\whN}{[0,\wh T^y)}$, whereas $\wh\alpha^y$ minus its leftmost block is a function of $\shiftrestrict{\whN}{[\wh T^y,\infty)}$. 
 These are independent by the strong Markov property of $\whN$. We conclude by Lemma \ref{lem:LMB}.
\end{proof}


We now extend the Markov-like property of Proposition \ref{prop:PRM:Fy-_Fy+} to the present setting.

\begin{proposition}\label{prop:clade:Fy-_Fy+}
 $\widehat\bF^{\geq y}_0$ is conditionally independent of $\ol\cF^y$ given $\widehat\alpha^y$ with regular conditional distribution $\Pr^1_{\widehat\alpha^y}\big\{ F^{\geq 0}_0\in \cdot\,\big\}$, where this law is as in Definition \ref{constr:type-1_2}.
\end{proposition}

\begin{proof}
 By \eqref{eq:type-1:no_revival}, the claimed regular conditional distribution holds trivially on the event $\big\{\life^+\big(\widehat\bN\big) \leq y\big\}$. Likewise, the result is trivial for $y=0$.
 
 The mid-spindle Markov property for $T^{\geq y}$, Lemma \ref{lem:mid_spindle_Markov}, may be extended from $\bN$ to apply to $\olN$. Indeed, if $\ol T^{\geq y} > 0$ then the same proof goes through; otherwise, if $\ol T^{\geq y} = 0$, i.e.\ if $\life(\bff) > y$, then the lemma reduces to the Markov property of $\bff$ at $y$. We use this extension to split $\olN$ into three segments.
 
 Let $T := \ol T^{\geq y}$. Let $(T,f_T)$ denote the spindle of $\ol\bN$ at this time, which equals $(0,\bff)$ if $\life(\bff) > y$. Let $\hat f^y_T$ and $\check f^y_T$ denote the broken spindles of \eqref{eq:spindle_split}. Extending the notation of Section \ref{sec:clade_filtration}, set
 \begin{equation*}
  \olN^{\leq y}_{\textnormal{first}} := \Restrict{\olN}{[0,T)} + \Dirac{T,\check f_T^y},
  \qquad \olN^{\geq y}_{\textnormal{first}} := \ShiftRestrict{\olN}{\left(T,\ol T^y\right)} + \Dirac{0,\hat f_T^y},
  \qquad \olN^y_* := \ShiftRestrict{\olN}{[\ol T^y,\infty)}.
 \end{equation*}
 Let $H_1$ and $H_2$ be non-negative measurable functions on $\Hfin$, and likewise for $H_3$ on $\H$. Recall \eqref{eq:LMB_def} defining $m^y$. In the present setting $m^y(\olN) = f_T\big(y-\olX(T-)\big)$. By the preceding extension of the mid-spindle Markov property and the disintegration of $\mClade^+$ in Proposition \ref{prop:clade_splitting},
 \begin{equation*}
  \EV \left[ H_1\big(\olN^{\leq y}_{\textnormal{first}}\big)H_2\big(\olN^{\geq y}_{\textnormal{first}}\big) \right] = \EV \left[ H_1\big(\olN^{\leq y}_{\textnormal{first}}\big) \mClade^+\left[H_2\;\middle|\;m^0 = m^y\big(\olN^{\leq y}_{\textnormal{first}}\big)\right] \right],
 \end{equation*}
 Moreover, by the strong Markov property of $\olN$ applied at $\ol T^y$, $(\olN^{\leq y}_{\textnormal{first}},\olN^{\geq y}_{\textnormal{first}})$ is independent of $\olN^y_*$ and the latter is distributed like $\bN$. Thus,
 \begin{equation*}
  \EV \left[ H_1\big(\olN^{\leq y}_{\textnormal{first}}\big)H_2\big(\olN^{\geq y}_{\textnormal{first}}\big)H_3\big(\olN^y_*\big) \right] = \EV \left[ H_1\big(\olN^{\leq y}_{\textnormal{first}}\big) \mClade^+\left[H_2\;\middle|\;m^0 = m^y\big(\olN^{\leq y}_{\textnormal{first}}\big)\right] \right] \EV\big[ H_3(\bN) \big].
 \end{equation*}
 The event $\{\life^+(\widehat\bN) > y\}$ equals the event that the process $\xi\big(\olN^{\leq y}_{\textnormal{first}}\big)$ is non-negative. 
 In particular, this belongs to the $\sigma$-algebra $\sigma\big(\olN^{\leq y}_{\textnormal{first}}\big)$. Thus, the above formula also holds for the conditional expectation given this event. On this event, $\ol T^{\geq y} = \widehat T^{\geq y} < \infty$.
 
 Let $\tdN := \restrict{\olN^y_*}{[0,T^{-y}(\olN^y_*))}$. The stopping time $T^{-y}(\bN)$ satisfies the hypotheses of Proposition \ref{prop:PRM:Fy-_Fy+}. Thus, that proposition applies to the stopped \PRM\ $\tdN$. On the event $\{\life^+(\widehat\bN) > y\}$, which is independent of $\tdN$, we have $\tdN = \shiftrestrict{\whN}{[\wh T^y,\wh T^0)}$ and so $\whF^{y} = F^{0}_0(\tdN)$. Conditionally given this event, by Proposition \ref{prop:PRM:Fy-_Fy+}, the clade point process $\whF^{\geq y}$ is conditionally independent of $\olF^{\leq y}$ given $\skewer\big(0,\tdN\big) =: \beta$, with regular conditional distribution $\Pr^1_{\beta}\{F^{\geq 0}_0\in\cdot\,\}$. It follows from Proposition \ref{prop:nice_level} that level $y$ is a.s.\ nice for $\ol\bN$. Thus, by Lemma \ref{lem:cutoff_vs_PP}, $\big(\olN^{\leq y}_{\textnormal{first}},\olF^{\leq y}\big)$ generates $\ol\cF^y$ up to $\Pr$-null sets. Putting all of this together, $\whF^{\geq y}_0 = \whF^{\geq y} + \cf\big\{\life^+\big(\widehat\bN\big) > y\big\}\DiracBig{0,\olN^{\geq y}_{\textnormal{first}}}$ a.s., and this has the desired conditional independence and regular conditional distribution.
\end{proof}

\section{Type-1 and type-0 evolutions as Hunt processes in $(\IPspace,\dI)$}
\label{sec:type-1_gen}
\def\cFI{\cF_{\IPspace}}

\subsection{Type-1: total mass, path continuity, simple Markov property}
\label{sec:type-1_gen:cts}

Throughout this section we follow the notation of Definition \ref{constr:type-1} for $\beta\in\IPspace$, $(\bN_U,\,U\in\beta)$, $\bN_{\beta}$, and $(\alpha^y,\,y\geq 0)$. We treat these objects as maps on a probability space $(\Omega,\cA,\Pr)$. We additionally define
\begin{equation}\label{eq:concat_skewers:type-1}
 \alpha^y_U := \skewer(y,\bN_U) \quad \text{for }y\ge0,\ U\in\beta, \qquad \text{so} \qquad \alpha^y = \ConcatIL_{U\in\beta}\alpha^y_U.
\end{equation}
For each of the filtrations $(\cF_t)$, $(\cF_{t-})$, $(\cF^y)$, and $(\cF^{y-})$ on $\H$ introduced in Definition \ref{def:filtrations}, we accent with a bar, as in $(\ol\cF_t,\ t\geq 0)$, to denote the completion of the filtration under the family of measures $(\Pr^1_{\beta},\ \beta\in\IPspace)$.

We begin this section by showing that $(\alpha^y,\,y\geq 0)$ is a.s.\ an $\IPspace$-valued process, and we derive its transition kernel. Then we prove a simple Markov property of $(\alpha^y,\,y\geq 0)$ as a random element of the product space $\IPspace^{[0,\infty)}$. Next, we prove the type-1 assertion of Theorem \ref{thm:BESQ_total_mass}, which describes $(\IPmag{\alpha^y},\,y\ge0)$. Finally, we prove the existence of a continuous version of $(\alpha^y,\,y\geq 0)$, in the sense of Definition \ref{def:version}, as well as a simple Markov property for this continuous process.

\begin{lemma}\label{lem:finite_survivors}
 For $J\subseteq\beta$ and $y>0$,
 \begin{equation*}
 \begin{split}
  \EV [\#\{U\in J\colon \alpha_U^y \not= \emptyset\} ] &\leq \frac{1}{2y}\sum_{U\in J}\Leb(U) \quad \text{and} \quad \Pr \{\forall U\in J,\,\alpha_U^y = \emptyset \} \geq 1 - \frac{1}{2y}\sum_{U\in J}\Leb(U).
 \end{split}
 \end{equation*}
 In particular, a.s.\ only finitely many of the $\big(\alpha_U^z,\,z\geq 0\big)$ survive to level $y$.
\end{lemma}

\begin{proof}
 The variables $\cf\{\alpha^y_U = \emptyset\}$ are independent Bernoulli trials with respective parameters $e^{-\Leb(U)/2y}$, by \eqref{eq:transn:lifetime}. Thus, both inequalities follow from $e^{-x} \geq 1-x$.
\end{proof}

We can extend Theorem \ref{thm:LT_property_all_levels} to the present setting.

\begin{proposition}\label{prop:type-1:LT_diversity}
 It is a.s.\ the case that $\IPLT_{\alpha^y}\big(M^y_{\bN_{\beta}}(t)\big) = \ell^y_{\bN_{\beta}}(t)$ for all $t\geq 0$, $y>0$.
\end{proposition}


\begin{proof}
 Appealing to Corollary \ref{cor:clade:cts_skewer} and Lemma \ref{lem:finite_survivors}, we may restrict to an a.s.\ event on which:
 \begin{equation}\label{eq:good_clades_event}
  \forall U\in\beta,\ \bN_U\in\Hfins, \qquad \text{and} \qquad \forall n\in\BN,\ \#\left\{V\in\beta\colon \life^+\left(\bN_V\right)>1/n\right\} < \infty.
 \end{equation}
 Let $y>0$ and consider the left-to-right ordered sequence $U_1,\ldots,U_K$ of intervals $U\in\beta$ for which $\life^+(\bN_{U})>y$. For $U\in\beta$, define $S(U-):=\sum_{V\in\beta\colon V<U}\len(\bN_V)$ and $S(U):=S(U-) + \len(\bN_U)$.
 
 Since no clade prior to time $S(U_1-)$ survives to level $y$, $\IPLT_{\alpha^y}\big(M^y_{\bN_{\beta}}(t)\big) = \ell^y_{\bN_{\beta}}(t) = 0$ for $t\leq S(U_1-)$. We assume for induction that the same holds up to time $S(U_j-)$. Then
 \begin{equation*}
 \begin{split}
  \IPLT_{\alpha^y}\big(M^y_{\bN_{\beta}}(t)\big) &= \IPLT_{\alpha^y}\big(M^y_{\bN_{\beta}}(S(U_j-))\big) + \IPLT_{\alpha^y_{U_j}}\big(M^y_{\bN_{U_j}}(t-S(U_j-))\big)\\
  		&= \ell^y_{\bN_{\beta}}(S(U_j-)) + \ell^y_{\bN_{U_j}}(t-S(U_j-)) = \ell^y_{\bN_{\beta}}(t) \qquad \text{for }t\in [S(U_j-),S(U_j)],
 \end{split}
 \end{equation*}
 where the middle equality follows from our assumption $\bN_{U_j}\in\Hfins$ and the inductive hypothesis. For $t\in [S(U_j),S(U_{j+1}-)]$ or, if $j=K$, for all $t \geq S(U_j)$, no additional local time accrues and at most one skewer block arrives at level $y$ during this interval. Thus, on this interval,
 $$\IPLT_{\alpha^y}\big(M^y_{\bN_{\beta}}(t)\big) = \IPLT_{\alpha^y}\big(M^y_{\bN_{\beta}}(S(U_j))\big) = \ell^y_{\bN_{\beta}}(S(U_j)) = \ell^y_{\bN_{\beta}}(t).$$
 By induction, this proves that the identity holds at all $t\geq 0$ at level $y$, for all $y>0$.
\end{proof}

\begin{lemma}\label{lem:type-1:wd}
 It is a.s.\ the case that for every $y>0$, the collection of interval partitions $(\alpha_U^y,\,U\in \beta)$ is strongly summable in the sense of Definition \ref{def:IP:concat}, and $\alpha^y$ lies in $\IPspace$.
\end{lemma}

\begin{proof}
 This holds on the event in \eqref{eq:good_clades_event}, as 
 finite sequences in $\IPspace$ are strongly summable.
\end{proof}

\begin{proposition}[Transition kernel for type-1 evolutions]\label{prop:type-1:gen_transn}
 Fix $y>0$. Let $(\beta^y_U,U\!\in\!\gamma)$ denote an independent family of partitions, with each $\beta^y_U$ distributed like $\wh\alpha^y$ in Proposition \ref{prop:type-1:transn} with $a = \Leb(U)$. Then $\skewer(y,\bN_{\beta})\stackrel{d}{=}\ConcatIL_{U\in\beta}\beta^y_U$, and this law is supported on $\IPspace$.
%
\end{proposition}

\begin{proof}
 This follows from Definition \ref{constr:type-1} of $\Pr^1_{\beta}$ via the observation that the skewer map commutes with concatenation of clades. 
 By Lemma \ref{lem:type-1:wd}, the resulting law is supported on $\IPspace$.
\end{proof}

\begin{lemma}\label{lem:type-1:nice_lvl}
 For $y>0$, it is a.s.\ the case that level $y$ is nice for $\xi(\bN_{\beta})$ in the sense of Proposition \ref{prop:nice_level} and $\alpha^y$ is nice in the sense of Lemma \ref{lem:cutoff_skewer} \ref{item:CPS:clades_skewer_0}.
\end{lemma}

\begin{proof}
 Proposition \ref{prop:nice_level} implies that for each $U$, level $y$ is a.s.\ nice for $\xi(\bN_U)$. It follows from this and Lemma \ref{lem:finite_survivors} that $y$ is a.s.\ nice for $\xi(\bN_{\beta})$. In particular, no two level $y$ excursion intervals arise at the same local time. Proposition \ref{prop:agg_mass_subord} characterizes a correspondence between level $y$ excursion intervals of $\xi(\bN_{\beta})$, including the incomplete first excursion interval, and blocks in $\alpha^y$ whereby, via Proposition \ref{prop:type-1:LT_diversity}, the diversity up to each block $U\in\alpha^y$ equals the level $y$ local time up to the corresponding excursion interval. Thus, $\alpha^y$ is a.s.\ nice as well.
\end{proof}

We now extend the Markov-like property of Propositions \ref{prop:PRM:Fy-_Fy+} and \ref{prop:clade:Fy-_Fy+} to the present setting.

\begin{proposition}\label{prop:type-1:Fy-_Fy+}
 For $y>0$, the point process $F^{\geq y}_0(\bN_\beta)$ is conditionally independent of $\ol\cF^y$ given $\alpha^y$, with regular conditional distribution $\Pr^1_{\alpha^y}(F^{\geq 0}_0\in\cdot\,)$.
\end{proposition}

\begin{proof}
 By Lemma \ref{lem:type-1:nice_lvl}, we may restrict to the a.s.\ event that level $y$ is nice. For $U\in\beta$ let $\bF^{\geq y}_{0,U} := F^{\geq y}_0(\bN_U)$. By Proposition \ref{prop:clade:Fy-_Fy+} and the independence of the family $(\bN_U,\,U\in\beta)$, the process $\bF^{\geq y}_{0,U}$ is conditionally independent of $\ol\cF^y$ given $\alpha^y_U$, with regular conditional distribution $\Pr^1_{\alpha^y_U}\{F^{\geq 0}_0\in\cdot\,\}$, for each $U\in\beta$. By Lemma \ref{lem:finite_survivors}, only finitely many of the $\bF^{\geq y}_{0,U}$ are non-zero, so $F^{\geq 0}_0(\bN_{\beta}) = \ConcatIL_{U\in\beta}\bF^{\geq y}_{0,U}$. In light of this, the claimed conditional independence and regular conditional distribution follow from Definition \ref{constr:type-1_2} of the kernel $\gamma\mapsto\Pr^1_{\gamma}\{F^{\geq 0}_0\in\cdot\,\}$.
\end{proof}

\begin{corollary}[Simple Markov property for the skewer process under $\Pr^1_{\mu}$]\label{cor:type-1:simple_Markov_1}
 Let $\mu$ be a probability distribution on $\IPspace$. Take $z>0$ and $0\leq y_1<\cdots<y_n$. Let $\eta\colon \Hfin \to [0,\infty)$ be $\ol\cF^{z}$-measurable. Let $f\colon\IPspace^n\to [0,\infty)$ be measurable. Then
  \begin{equation*}
   \Pr^1_\mu\left[\eta\, f\left(\skewer(z+y_j,\cdot\,),j\in[n]\right) \right]
     = \int\!\eta(N)\Pr^1_{\skewer(z,N)}\left[f\left(\skewer(y_j,\cdot\,),j\in[n]\right)\right]d\Pr^1_{\mu}(N).
  \end{equation*}
\end{corollary}

\begin{proof}
 By Proposition \ref{prop:type-1:Fy-_Fy+}, for $\eta$ as above and $g\colon \cNRHf \to [0,\infty)$ measurable,
 \begin{equation}\label{eq:Fy-_Fy+:Markov}
  \Pr^1_\mu \left[ \eta\, g\big(F^{\geq z}_0\big)\right] = \Pr^1_\mu \left[ \eta\, \Pr^1_{\skewer(z,\cdot\,)}\big[g\big(F^{\geq 0}_0\big)\big]\right].
 \end{equation}
 By Lemma \ref{lem:cutoff_skewer} \ref{item:CPS:clades_skewer}, there is a measurable function $h$ for which $\left(\skewer(z+y_j,N),\,j\in[n]\right) = h(F^{\geq z}_0(N))$ identically on the event that level $z$ is nice for $N\in\Hfin$. Moreover, if $\beta\in\IPspace$ is nice in the sense 
 of Lemma \ref{lem:cutoff_skewer} \ref{item:CPS:clades_skewer_0}, 
 then that result gives $\left(\skewer(y_j,\bN_{\beta}),\,j\in[n]\right) = h(F^{\geq 0}_0(\bN_{\beta}))$. By Lemma \ref{lem:type-1:nice_lvl}, for $\bN_{\mu}\sim\Pr^1_{\mu}$, level $z$ is a.s.\ nice for $\bN_{\mu}$ and $\skewer(y_1,\bN_{\mu})$ is a.s.\ a nice interval partition. Thus, setting $g := f\circ h$ in \eqref{eq:Fy-_Fy+:Markov} gives the claimed result.
\end{proof}

\begin{proposition}\label{prop:BESQ_mass_temp}
 The process $(\IPmag{\alpha^y},\,y\ge0)$ admits a continuous version, which is a \BESQ[0].
\end{proposition}

Once we have shown that $(\alpha^y,\,y\ge0)$ admits a continuous version, 
this proposition will entail the type-1 assertion of Theorem \ref{thm:BESQ_total_mass}.

\begin{proof}
 Let $(\wh\alpha^y,\,y\geq 0)$ be as in Section \ref{sec:type-1:clade}. We proceed by establishing: (i) the desired 1-dimensional marginals; (ii) finite-dimensional marginals; and (iii) the existence of a continuous version. For each of these, we show the property first for $(\IPmag{\wh\alpha^y})$, then for $(\IPmag{\alpha^y})$.
 
 (i) By \cite[p.\ 441]{RevuzYor}, the Laplace transform of the marginal distribution at time $y > 0$ of a \BESQ[0] process $(B(z),\,z\geq 0)$ starting from $a$ is
 \begin{equation*}
  \EV\left[e^{-\lambda B(y)}\right] = \exp\left(-\frac{\lambda a}{2y\lambda+1}\right).
 \end{equation*}
 
 We wish to compare this to the Laplace transform of $\IPmag{\wh\alpha^y}$. In the notation of Proposition \ref{prop:type-1:transn}, given that it is not zero, $\IPmag{\alpha^y} \stackrel{d}{=} R^y(S^y)+L^y$. As noted in \eqref{eq:transn:PDIP}, $R^y(S^y) \sim \GammaDist[1/2,1/2y]$, which has Laplace transform $(2y\lambda+1)^{-1/2}$. 
 As for $L^y$, we note that
 \begin{align*}
  1-\cosh(u)+u\sinh(u) = \sum_{n=1}^\infty\left(\frac{1}{(2n-1)!}-\frac{1}{(2n)!}\right)u^{2n}=\sum_{n=1}^\infty\frac{2n-1}{(2n)!}u^{2n}.
 \end{align*}
 Plugging this into the probability density function for $L^y$ in \eqref{eq:transn:LMB},
 \begin{align}
  \EV\left[e^{-\lambda L^y}\right]
  	&= \frac{1}{\sqrt{2\pi}}\frac{\sqrt{y}}{e^{a/2y}-1}\sum_{n=1}^\infty\frac{2n-1}{(2n)!}\left(\frac{a}{y^2}\right)^n\int_0^\infty b^{n-3/2}e^{-(\lambda+1/2y)b}db\nonumber\\
  	&= \frac{\sqrt{2y\lambda+1}}{e^{a/2y}-1}\sum_{n=1}^\infty\frac{\Gamma(n+1/2)}{\sqrt{\pi}(2n)!}\left(\frac{2a}{y(2y\lambda+1)}\right)^n.\label{eq:clade:LMB_Laplace_1}
 \end{align}
 From \eqref{eq:transn:lifetime}, $\Pr\{\wh\alpha^y=\emptyset\} = e^{-a/2y}$. Now, to prove $\EV \left[\exp\left(-\lambda\IPmag{\wh\alpha}\right)\right] = \EV\left[\exp\left(-\lambda B(y)\right)\right]$ it suffices to show
  \begin{align}
   \EV\left[e^{-\lambda B(y)}\right] &= e^{-a/2y}+(1-e^{-a/2y})\EV\left[e^{-\lambda R^y(S^y)}\right]\EV\left[e^{-\lambda L^y}\right];\notag\\
   \text{i.e.} \quad \EV\left[e^{-\lambda L^y}\right] &= \frac{\EV[e^{-\lambda B(y)}]-e^{-a/2y}}{(1-e^{-a/2y})\EV[e^{-\lambda R^y(S^y)}]} = \frac{\sqrt{2y\lambda+1}}{1-e^{-a/2y}}\left(e^{-\lambda a/(2y\lambda+1)}-e^{-a/2y}\right)\label{eq:clade:LMB_Laplace_2}\\
   	&= \frac{\sqrt{2y\lambda+1}}{e^{a/2y}-1}\sum_{n=1}^\infty\frac{1}{n!}\frac{1}{2^{2n}}\left(\frac{2a}{y(2y\lambda+1)}\right)^n.\notag
  \end{align}
  By the Gamma duplication formula $\Gamma(2z)=\pi^{-1/2}2^{2z-1}\Gamma(z)\Gamma(z+\frac12)$ for $z = n+\frac12$, this equals the expression in \eqref{eq:clade:LMB_Laplace_1}. Hence, $\IPmag{\wh\alpha^y}$ is distributed like $B(y)$ for fixed $y$. This result extends to general initial states $\beta\in\IPspace$ by way of the independence of the clades $(\bN_U,\,U\in\beta)$ and \cite[Theorem 4.1 (iv)]{PitmYor82}, which states that an arbitrary sum of independent \BESQ[0] processes with summable initial values is a \BESQ[0].
 
 (ii) We now prove equality of finite-dimensional marginal distributions by an induction based on Corollary \ref{cor:type-1:simple_Markov_1} and the Markov property of \BESQ[0]. For $1$-dimensional marginals, we have proved the result. We now assume the result holds for all $n$-dimensional marginal distributions starting from any initial distribution. We write $\bQ_a$ to denote the law of a \BESQ[0] process $(B(y))$ starting from $a>0$. For all $0\leq y_1<\cdots<y_n<y_{n+1}$ and $\lambda_j\in[0,\infty)$, $j\in[n+1]$, we have
 \begin{align*}
  &\Pr^1_{\{(0,a)\}}\!\left[\exp\left(-\sum_{j=1}^{n+1}\lambda_j\IPmag{\skewer(y_j,\cdot\,)}\right)\right]\\
    &= \Pr^1_{\{(0,a)\}}\!\left[e^{-\lambda_1\IPmag{\skewer(y_1,\cdot\,)}}\Pr^1_{\skewer(y_1,\cdot\,)}\!\left[\exp\left(-\sum_{k=1}^{n}\lambda_{k+1}\IPmag{\skewer(y_{k+1}-y_1,\cdot\,)}\right)\right]\right]\\
    &= \Pr^1_{\{(0,a)\}}\!\left[e^{-\lambda_1\IPmag{\skewer(y_1,\cdot\,)}}\bQ_{\IPmag{\skewer(y_1,\cdot\,)}}\!\left[\exp\left(-\sum_{k=1}^{n}\lambda_{k+1}B(y_{k+1}-y_1)\right)\right]\right]\\
    &= \bQ_a\!\left[e^{-\lambda_1B(y_1)}\bQ_{B(y_1)}\!\left[\exp\left(-\sum_{k=1}^{n}\lambda_{k+1}B(y_{k+1}-y_1)\right)\right]\right]
    = \bQ_a\!\left[\exp\left(-\sum_{j=1}^{n+1}\lambda_jB(y_j)\right)\right]\!.
 \end{align*}
 Again, this extends to general initial distributions by \cite[Theorem 4.1 (iv)]{PitmYor82} and independence of clades. This completes the induction step and establishes equality of finite-dimensional distributions, hence equality of distributions of the processes.
 
 (iii) By Corollary \ref{cor:clade:cts_skewer}, there is an a.s.\ event on which all of the $(\alpha^y_U,\,y\ge0)$, $U\in\beta$, are continuous in $y$. If we take continuous versions of each, then the $\big(\IPmag{\alpha^y_U},\,y\ge0\big)$ are independent \BESQ[0] processes with summable initial states. Thus, by \cite[Theorem 4.1 (iv)]{PitmYor82}, their sum $\big(\IPmag{\alpha^y},\,y\ge0\big)$ is a.s.\ continuous and is a \BESQ[0].
\end{proof}

We proceed towards proving continuity of $(\alpha^y,\,y\ge0)$. We require the following.

\begin{lemma}\label{lem:IP:domination}
 Fix $\beta\in\IPspace$ and $\delta>0$, and let $\gamma$ denote a \Stable[\frac12] interval partition with total diversity $\IPLT_{\gamma}(\infty) = \IPLT_{\beta}(\infty)+\delta$, as in Proposition \ref{prop:IP:Stable}. Then with positive probability, there exists a matching between their blocks such that every block of $\beta$ is matched with a larger block in $\gamma$. (This is not a correspondence as in Definition \ref{def:IP:metric}, as it need not respect left-right order.) In this event, we say $\gamma$ \emph{dominates} $\beta$. If, on the other hand, $\gamma$ is a \Stable[\frac12] interval partition with total diversity $\IPLT_{\gamma}(\infty) = \IPLT_{\beta}(\infty)-\delta$ then with positive probability it is dominated by $\beta$.
\end{lemma}

\begin{proof}
 We begin with the case $\IPLT_{\gamma}(\infty) = \IPLT_{\beta}(\infty)+\delta$. We will abbreviate $\IPLT := \IPLT_{\beta}(\infty)$. By the diversity properties of these two partitions,
 \begin{equation*}
  \lim_{h\downto 0}\sqrt{h}\#\{U\in \beta\colon \Leb(U)>h\} = \frac{1}{\sqrt{\pi}}\IPLT \quad \text{and} \quad
  \lim_{h\downto 0}\sqrt{h}\#\{V\in \gamma\colon \Leb(V)>h\} = \frac{1}{\sqrt{\pi}}(\IPLT+\delta).
 \end{equation*}
 Thus, there is a.s.\ some $H > 0$ sufficiently small so that
 \begin{equation}\label{eq:IP:domination}
  \#\{U\in \beta\colon \Leb(U)>h\} < \#\{V\in \gamma\colon \Leb(V)>h\} \qquad \text{for all }h<H.
 \end{equation}
 Take $a>0$ sufficiently small that this holds for $H=a$ with positive probability. It follows from the definition of the \Stable[\frac12] interval partition that, conditionally given that \eqref{eq:IP:domination} holds for $H = a$, there is positive probability that all of the blocks in $\gamma$ with mass greater than $a$ also have mass greater than that of the largest block of $\beta$. In particular, there is positive probability that $\gamma$ dominates $\beta$ by matching, for each $n\geq 1$, the $n^{\text{th}}$ largest block of $\beta$ with that of $\gamma$.
 
 If we instead take $\IPLT_{\gamma}(\infty) = \IPLT_{\beta}(\infty)-\delta$ then there is a.s.\ some $H>0$ such that \eqref{eq:IP:domination} holds in reverse. Let $a$ be as before. Conditionally given that the reverse of \eqref{eq:IP:domination} holds for $H = a$, there is positive probability that no blocks in $\gamma$ have mass greater than $a$. In this event, $\beta$ dominates $\gamma$ by matching blocks in ranked order, as in the previous case.
\end{proof}

\begin{proposition}\label{prop:cts_lt_at_0}
 The diversity process $(\IPLT_{\alpha^y}(\infty),\,y\geq 0)$ is a.s.\ continuous at $y = 0$.
\end{proposition}

\begin{proof}
 Fix $\theta\in (0,\frac14)$, take $\delta > 0$, and abbreviate $\IPLT := \IPLT_{\beta}(\infty)$. Following the notation and situation of Corollary \ref{cor:type-1:cts_from_Stable} \ref{item:c_f_S:determinstic}, let $\tdN$ denote an $\Hfins$-version of a \PRM[\Leb\times\mBxc] stopped at an inverse local time $\tau^0(\IPLT+\delta)$ and let $(\td\alpha^y,\,y\ge0) := \skewerP(\tdN)$. 
 Then, as in Corollary \ref{cor:type-1:cts_from_Stable} \ref{item:c_f_S:determinstic}, $\td\alpha^0$ is a \Stable[\frac12] interval partition with total diversity $\IPLT+\delta$. By Lemma \ref{lem:IP:domination}, $\td\alpha^0$ dominates $\beta$ with positive probability. Since $\beta$ is deterministic, this domination event is independent of $(\alpha^y,\,y\ge0)$. We condition on this event.
 
 We now define an alternative construction of $(\alpha^y)$, coupled with $(\td\alpha^y)$. Let $(U_i)_{i\geq 1}$ and $(V_i)_{i\geq 1}$ denote the blocks of $\beta$ and $\td\alpha^0$ respectively, each ordered by non-increasing Lebesgue measure, with ties broken by left-to-right order. For each $i$ let $\tdN_{V_i}$ denote the clade of $\tdN$ corresponding to that block. By Proposition \ref{prop:PRM:Fy-_Fy+} 
 the $(\tdN_{V_i})_{i\ge1}$ are conditionally independent given $\td\alpha^0$, with conditional laws $\mClade^+(\cdot\ |\ m^0 = \Leb(V_i))$. Then
 \begin{equation*}
  \td\alpha^y = \Concat_{V\in \td\alpha^0} \td\alpha_V^y \qquad \text{where} \qquad \left(\td\alpha_V^y,\,y\ge0\right) = \skewerP\big(\tdN_V\big).
 \end{equation*}
 
 Let $(0,g_i)$ denote the left-most point in $\tdN_{V_i}$. This is the spindle associated with the block $V_i$. Conditionally given $V_i$, the process $g_i$ is a \BESQ[-1] starting from $\Leb(V_i)$. We define
 \begin{equation*}
  f_i := \frac{\Leb(U_i)}{\Leb(V_i)}\scaleB g_i,\quad
  T_i := \inf\left\{t\geq 0\colon \xi_{\tdN_{V_i}}(t) \leq \life(f_i)\right\},\quad
  \bN_{U_i} := \Dirac{0,f_i} + \ShiftRestrict{\tdN_{V_i}}{(T_i,\infty)}.
 \end{equation*}
 To clarify, $\bN_{U_i}$ is obtained from $\tdN_{V_i}$ by scaling down its leftmost spindle $g_i$ to get $f_i$ and cutting out the segment of $\tdN_{V_i}$ corresponding to the first passage of $\xi(\tdN_{V_i})$ down to level $\life(f_i)$. 
 From \BESQ\ scaling and the Poisson description of the laws $\mClade^+(\,\cdot\;|\;m^0)$ in Proposition \ref{prop:clade_splitting}, it follows that the $(\bN_{U_i})_{i\ge1}$ are jointly independent and have respective distributions $\bN_{U_i} \sim \mClade^+(\cdot\ |\ m^0 = \Leb(U_i))$. As in \eqref{eq:concat_skewers:type-1} we define
 \begin{equation*}
  \alpha^y := \Concat_{U\in\beta} \alpha_U^y \qquad \text{where} \qquad \left(\alpha_{U}^y,\,y\ge0\right) = \skewerP(\bN_U) \quad \text{for }U\in \beta.
 \end{equation*}
 The resulting $(\alpha^y,\,y\geq 0)\sim\Pr^1_{\beta}\{\skewerP\in\cdot\,\}$. 
 By virtue of this coupling, having conditioned on $\td\alpha^0$ dominating $\beta$, it is a.s.\ the case that $\IPLT_{\alpha_{U_i}^y}(\infty) \leq \IPLT_{\td\alpha_{V_i}^y}(\infty)$ for $i\geq 1$, $y\geq 0$. Thus, by the continuity in Proposition \ref{prop:PRM:cts_skewer},
 $$\limsup_{y\downto 0}\IPLT_{\alpha^y}(\infty) \leq \limsup_{y\downto 0}\IPLT_{\td\alpha^y}(\infty) = \IPLT + \delta\quad \text{a.s.}.$$
 Since this holds for all $\delta>0$, the left hand side expression is a.s.\ bounded above by $\IPLT$.
 
 If we repeat this argument but $\tdN$ stopped at $\tau^0(\IPLT-\delta)$ then we can condition on $\beta$ dominating $\td\alpha^0$ and reverse roles in the above coupling to show that
 \begin{equation*}
  \liminf_{y\downto 0}\IPLT_{\alpha^y}(\infty) \geq \liminf_{y\downto 0}\IPLT_{\td\alpha^y}(\infty) = \IPLT - \delta
 \end{equation*}
 almost surely for any positive $\delta$. The desired result follows.
\end{proof}

\begin{proposition}[Existence of type-1 evolutions]\label{prop:type-1:cts}
 For $\beta\in\IPspace$, $\bN_{\beta}$ belongs to $\Hfins$ almost surely. 
 In particular, $\skewerP(\bN_{\beta})$ is a.s.\ path-continuous in $(\IPspace,\dI)$. Moreover, this process is a.s.\ H\"older-$\theta$ for every $\theta\in\big(0,\frac14\big)$, except possibly at time zero.
\end{proposition}

\begin{proof}
 We have already checked properties (i) and (ii) of Definition \ref{def:domain_for_skewer} of $\Hfins$, in Lemma \ref{lem:type-1:wd} and Proposition \ref{prop:type-1:LT_diversity} respectively. It remains only to confirm the claimed path-continuity.
 
 By Lemma \ref{lem:finite_survivors}, for $z > 0$ the process $(\alpha^y,\,y\geq z)$ equals the concatenation of an a.s.\ finite subset of the processes $(\alpha_U^y,\,y\geq z)$ of \eqref{eq:concat_skewers:type-1}. By Corollary \ref{cor:clade:cts_skewer}, each of the $(\alpha_U^y,\,y\ge0)$ is a.s.\ H\"older-$\theta$ for $\theta\in(0,\frac14)$. This proves the a.s.\ H\"older continuity of $(\alpha^y,\,y\ge z)$, by way of \eqref{eq:IP:concat_dist}. Since this holds for every $z$, it remains only to establish a.s.\ continuity at $y=0$.
 
 Fix $\epsilon > 0$. Take a subset $\{U_1,\ldots,U_k\}\subseteq\beta$ 
 of sufficiently many large blocks so that $\IPmag{\beta} - \sum_{i=1}^k\Leb(U_i) < \epsilon/4$. We define a correspondence by pairing each $U_i$ with the leftmost block in $\alpha^y_{U_i}$. 
 Then there is a.s.\ some sufficiently small $\delta>0$ so that for $y<\delta$:
 \begin{enumerate}[label=(\roman*), ref=(\roman*)]
  \item for $i\in[k]$, $\left|\IPLT_{\alpha^0}(U_i) - \sum_{V\in\beta\colon V<U_i}\IPLT_{\alpha^y_V}(\infty)\right| < \epsilon$;
  \item $\big|\IPLT_{\alpha^0}(\infty) - \IPLT_{\alpha^y}(\infty)\big| < \epsilon$; and
  \item for $i\in[k]$, $\big|\Leb(U_i) - m^y(\bN_{U_i})\big| < \epsilon/4k$, where $m^y$ is as in \eqref{eq:LMB_def};
  \item $\big|\IPmag{\alpha^0} - \IPmag{\alpha^y}\big| < \epsilon/4$.
 \end{enumerate}
 The first and second of these can be controlled via Proposition \ref{prop:cts_lt_at_0}. The third can be controlled since each block $U_i$ is associated with the initial leftmost spindle of $\bN_{U_i}$, and said spindle evolves continuously as a \BESQ[-1]. Finally, the fourth comes from Proposition \ref{prop:BESQ_mass_temp}. Hence, $(\alpha^y,\,y\ge0)$ is a.s.\ continuous at $y=0$.
\end{proof}

\begin{definition}[$\BPr^1_{\beta}$, $\BPr^1_{\mu}$, $(\cFI^y)$]\label{def:IP_process_space_1}
 For $\beta\in\IPspace$, let $\BPr^1_{\beta}$ denote the distribution on $\cCRI$ of a continuous version of $\skewerP(\bN_{\beta})$. As in Definition \ref{constr:type-1}, for probability measures $\mu$ on $\IPspace$, let $\BPr^1_{\mu}$ denote the $\mu$-mixture of the laws $(\BPr^1_\beta)$. We write $(\cFI^y,\,y\geq 0)$ to denote the right-continuous filtration generated by the canonical process on $\cC([0,\infty),\IPspace)$. In integrals under the aforementioned laws, we will denote the canonical process by $(\alpha^y,\,y\geq 0)$.
\end{definition}

In this setting, Corollary \ref{cor:type-1:simple_Markov_1} extends via a monotone class theorem to the following.


\begin{corollary}[Simple Markov property for type-1 evolutions]\label{cor:type-1:simple_Markov}
 Let $\mu$ be a probability distribution on $\IPspace$. Fix $y>0$. Take $\eta,f\colon\cCRI\to [0,\infty)$ measurable, with $\eta$ measurable with respect to $\cFI^y$. Let $\theta_y$ denote the shift operator. Then 
  $\BPr^1_\mu\big[\eta\, f\circ\theta_y \big] = \BPr^1_{\mu}\left[\eta\, \BPr^1_{\alpha^y}[f]\right].$
%
\end{corollary}

\def\cvN{\cev\bN}
\def\cvX{\cev\bX}

\subsection{Type-0: construction, path continuity, total mass, simple Markov property}
\label{sec:type-0}

The type-0 evolution may be viewed as a type-1 evolution with immigration from the left. Let $\bN$ denote an $\Hs$-version of a \PRM[\Leb\otimes\mBxc] on $[0,\infty)\times\Exc$, the existence of which is guaranteed by Proposition \ref{prop:PRM:cts_skewer}. For $y\in\BR$, let $T^y$ denote the first hitting time of $y$ by $\bX := \xi(\bN)$. We define
\begin{equation*}
 \cev\alpha^y_j := \skewer\left(y,\restrict{\bN}{[0,T^{-j})},j+\restrict{\bX}{[0,T^{-j})}\right) \qquad \text{for }j\in\BN,\ y\in [0,j].
\end{equation*}
Note that for $k>j$ we have
\begin{equation}\label{eq:type-0:consistency}
\begin{split}
 \left(\restrict{\bN}{[T^{j-k},T^{-k})},k+\restrict{\bX}{[T^{j-k},T^{-k})}\right) &\stackrel{d}{=} \left(\restrict{\bN}{[0,T^{-j})},j+\restrict{\bX}{[0,T^{-j})}\right),\\
 \text{and thus} \quad (\cev\alpha^y_k,\ y\in [0,j]) &\stackrel{d}{=} (\cev\alpha^y_j,\ y\in [0,j]).
\end{split}
\end{equation}
Thus, by Kolmogorov's extension theorem and \cite[Lemma II.35.1]{RogersWilliams}, there exists a continuous process $(\cev\alpha^y,\ y\geq 0)$ such that for every $j\in\BN$ we have $(\cev\alpha^y,\ y\in [0,j]) \stackrel{d}{=} (\cev\alpha^y_j,\ y\in [0,j])$.

\begin{definition}[Type-0 evolution, $\BPr^0_{\beta}$, $\BPr^0_{\mu}$]
\label{constr:type_0}
 Let $\beta\in\IPspace$. Let $(\cev\alpha^y,\ y\geq 0)$ be as above and $(\vecc\alpha^y,\ y\geq 0)$ an independent type-1 evolution starting from $\beta$. Consider $(\alpha^y,\,y\geq 0) = (\cev\alpha^y\concat\vecc\alpha^y,\ y\geq 0)$. Let $\BPr^0_{\beta}$ its law on $\cCRI$. A \emph{type-0 evolution} starting from $\beta$ is a process with this law. 
 For probability measures $\mu$ on $\IPspace$ we write $\BPr^0_{\mu}$ to denote the $\mu$-mixture of the laws $\BPr^0_{\beta}$.
\end{definition}

Taking up the CMJ perspective of Section \ref{sec:intro:CRP} and the language of immigration, in the above definition, the type-1 component $(\vecc\alpha^y)$ may be viewed as all descendants of the population at time zero, whereas $(\cev\alpha^y)$ may be viewed as descendants of subsequent immigrants. 
The existence of continuous type-0 evolutions follows from our results for type-1 evolutions.

\begin{remark}
 It is possible to construct the type-0 evolution as the skewer of a point process of spindles, rather than via consistency and the extension theorem as we have done above. This would involve setting up a point process of spindles $\cev\bN$ on $(-\infty,0)\times\Exc$ such that, for a suitable extension of Definition \ref{def:JCCP} of $\xi$, the process $\xi(\cev\bN)$ could be understood as a \Stable[\frac32] first-passage descent from $\infty$ down to 0. Related processes have been studied in the literature. For example, Bertoin \cite[Section VII.2]{BertoinLevy} constructs spectrally negative L\'evy processes that are conditioned to stay positive. Transforming such a process via sign change and an increment reversal akin to $\reverseincr$, as in \eqref{eq:stable:reversal_def}, results in a spectrally positive process coming down from $\infty$ to $0$. We find this approach to be notationally friendlier.
\end{remark}

Note that $(\cev\alpha^y,\, y\geq 0)$ is itself a type-0 evolution with initial state $\emptyset$. We will see that $\emptyset$ is a reflecting boundary for type-0 evolutions, whereas it is absorbing for type-1.

\begin{proposition}[Transition kernel for type-0 evolutions]\label{prop:type-0:transn}
 Take $\gamma\in\IPspace$ and $y>0$. Let $(\gamma^y_U,\,U\in\gamma)$ denote an independent family of partitions, with each $\gamma^y_U$ distributed as $\wh\alpha^y$ in Proposition \ref{prop:type-1:transn} with $a = \Leb(U)$. Let $(R^y,S^y)$ be as in Proposition \ref{prop:type-1:transn}, independent of $(\gamma^y_U,\ U\in\gamma)$. Then under $\BPr^0_{\gamma}$, the partition $\alpha^y$ has the same distribution as
 $$\{(R^y(s-),R^y(s))\colon s\in [0,S^y], R^y(s-)<R^y(s)\} \concat \Concat_{U\in\gamma}\gamma^y_U.$$
\end{proposition}

\begin{proof}
 Let $(\cev\alpha^z)$ and $(\vecc\alpha^z)$ be as in Definition \ref{constr:type_0} with $\beta = \gamma$. By Proposition \ref{prop:type-1:gen_transn}, $\vecc\alpha^y \stackrel{d}{=} \ConcatIL_{U\in\gamma}\gamma^y_U$. By construction, this is independent of $\cev\alpha^y$. It remains only to show that $\cev\alpha^y$ is distributed like the interval partition arising from the range of $R^y$, up to time $S^y$.
 
 Let $\whN$ have law $\mClade^+\{\,\cdot\;|\;m^0 = 1,\ \life^+>y\}$. Let $\wh T^y$ denote the first hitting time of $y$ in $\whX = \xi(\whN)$. It follows from the description of $\mClade^+\{\,\cdot\;|\;m^0 = 1\}$ in Proposition \ref{prop:clade_splitting} and the strong Markov property applied at time $\wh T^{y}$ that
 \begin{gather}
  \left(\shiftrestrict{\whN}{[\wh T^y,\infty)},\shiftrestrict{\whX}{[\wh T^y,\infty)}\right) \stackrel{d}{=} \left(\restrict{\bN}{[T^{y-j},T^{-j})}, j+\restrict{\bX}{[T^{y-j},T^{-j})}\right) \qquad \text{for }j>y,\notag\\
  \text{so} \quad
  \skewer\left(y,\shiftrestrict{\whN}{[\wh T^y,\infty)},\shiftrestrict{\whX}{[\wh T^y,\infty)}\right) \stackrel{d}{=} \cev\alpha^y. \label{eq:type-0:restricted_skewer}
 \end{gather}
 Note that $\skewer(y,\whN)$ equals a single leftmost block, corresponding to the first jump of $\whX$ across level $y$, concatenated with the skewer on the left in \eqref{eq:type-0:restricted_skewer}. By Proposition \ref{prop:type-1:transn}, that term in \eqref{eq:type-0:restricted_skewer} 
 has the desired distribution.
\end{proof}

\begin{proposition}[Simple Markov property for type-0 evolutions]\label{prop:type-0:simple_Markov}
 Let $\mu$ be a probability distribution on $\IPspace$. Fix $y>0$. Take $\eta,f\colon\cCRI\to [0,\infty)$ measurable, with $\eta$ measurable in $\cFI^y$. Let $\theta_y$ denote the shift operator. Then 
  $\BPr^0_\mu\big[\eta\, f\circ\theta_y \big] = \BPr^0_{\mu}\left[\eta\, \BPr^0_{\alpha^y}[f]\right].$
\end{proposition}

\begin{proof}
 Take $\beta\in\IPspace$, $0\leq u_1 < \cdots < u_n \leq y$, and $0\leq v_1 < \cdots < v_m$. Suppose $\eta(\alpha^z,\,z\ge0) = \eta'(\alpha^{u_j},\,j\in[n])$ and $f(\alpha^z,\,z\ge0) = f'(\alpha^{v_j},\,j\in[m])$. We will show that in this case,
 \begin{equation}\label{eq:type-0:simple_Markov_1}
  \BPr^0_{\beta}\big[\eta'(\alpha^{u_j},\,j\in[n])f'(\alpha^{y+v_j},\,j\in[m])\big] = \BPr^0_{\beta}\big[\eta'(\alpha^{u_j},\,j\in[n])\BPr^0_{\alpha^y}\big[f'(\alpha^{v_j},\,j\in[m])\big]\big].
 \end{equation}
 Indeed, this will suffice to prove the proposition: we can extend to general $\eta$ and $f$ by a monotone class theorem, and we generalize the equation from $\BPr^0_{\beta}$ to $\BPr^0_{\mu}$ by mixing.
 
 For $x>0$, set $\cev\bN_x := \restrict{\bN}{[0,T^{-x})}$, similar to the point processes discussed in \eqref{eq:type-0:consistency}. Let $\bN_{\beta}$ be an $\Hfins$-version of a $\Pr^1_{\beta}$-distributed process, independent of $\cev\bN_x$. For the purpose of this argument, we define $\Pr^0_{x,\beta}$ to be the distribution of $\bN_{x,\beta} := \cev\bN_x\concat\bN_{\beta}$ on $\Hfins$. We work towards a type-0 version of Proposition \ref{prop:type-1:Fy-_Fy+}.
 
 Take $z>u_n$. Set $\vecc\alpha^y := \skewer(y,\bN_\beta)$, $\cev\alpha^y_{z+y} := \skewer\big(-z,\cev\bN_{z+y}\big)$, and $\alpha^y_{z+y} := \cev\alpha^y_{z+y}\concat\vecc\alpha^y$. Let $\cev\bN_{z} := \restrict{\cev\bN_{z+y}}{[0,T^{-z})}$ and $\cev\bN^z_y := \restrict{\cev\bN_{z+y}}{[T^{-z},T^{-z-y})}$, where $T^{-z}$ is the hitting time of $-z$ in $\xi(\cev\bN_{z+y})$. By the strong Markov property, these components are independent. Recall the cutoff processes of \eqref{eq:cutoff_def}. In our setting,
 \begin{equation}\label{eq:type-0:Fy-+_decomp}
 \begin{array}{r@{\ =\ }r@{\;\concat\;}l}
  \cutoffG{-z}{\bN_{z+y,\beta}} &\displaystyle\cev\bN_{z} \concat \cutoffGB{0}{\cev\bN^z_y} & \cutoffG{y}{\bN_{\beta}},\\[.2cm]
  \text{and} \quad
  \cutoffL{-z}{\bN_{z+y,\beta}} &\displaystyle					  \cutoffLB{0}{\cev\bN^z_y} & \cutoffL{y}{\bN_{\beta}}.
 \end{array}
 \end{equation}
 By Proposition \ref{prop:PRM:Fy-_Fy+}, $F^{\geq 0}_0\big(\cev\bN^z_y\big)$ is conditionally independent of $\cutoffLB{0}{\cev\bN^z_y}$ given $\cev\alpha^y_{z+y}$, with r.c.d.\ $\Pr^1_{\cev\alpha^y_{z+y}}\{F^{\geq 0}_0\in\cdot\,\}$. By Proposition \ref{prop:nice_level}, level 0 is a.s.\ nice for $\cev\bN^z_y$. Thus, by Lemma \ref{lem:cutoff_vs_PP}, $\cutoffGB{0}{\cev\bN^z_y}$ has this same conditional independence given $\cev\alpha^y_{z+y}$; and by Definitions \ref{constr:type-1} and \ref{constr:type-1_2}, it has r.c.d.\ $\Pr^1_{\cev\alpha^y_{z+y}}$. Analogously, substituting Proposition \ref{prop:type-1:Fy-_Fy+} and Lemma \ref{lem:type-1:nice_lvl} for Propositions \ref{prop:PRM:Fy-_Fy+} and \ref{prop:nice_level}, we see that $\cutoffG{y}{\bN_{\beta}}$ is conditionally independent of $\cutoffL{y}{\bN_{\beta}}$ given $\vecc\alpha^y$, with r.c.d.\ $\Pr^1_{\vecc\alpha^y}$. Thus, since $\big(\cev\bN_{z},\cev\bN^z_y,\bN_{\beta}\big)$ is an independent triple, $\cutoffG{-z}{\bN_{z+y,\beta}}$ is therefore conditionally independent of $\cutoffL{-z}{\bN_{{z+y},\beta}}$ given $\alpha^y_{z+y}$, with r.c.d.\ $\Pr^0_{z,\alpha^y_{z+y}}$. Now, \eqref{eq:type-0:simple_Markov_1} follows by Lemma \ref{lem:cutoff_skewer} \ref{item:CPS:cutoff_skewer_2}.
\end{proof}

We now show that for any $\beta\in\IPspace$, under $\BPr^0_{\beta}$ we have $(\IPmag{\alpha^y},y\geq0)\sim\BESQ[1]$.

\begin{proof}[Proof of type-0 assertion of Theorem \ref{thm:BESQ_total_mass}]
 By definition, type-0 evolutions are continuous, so it suffices to show that the total mass process is a Markov process with the same transition kernel as \BESQ[1]. First assume $\beta = \varnothing$. The marginal distribution of \BESQ[1] is given in \cite[(50)]{GoinYor03} as
 $$q_y^1(0,b)db=\frac{1}{\sqrt{2y}}\frac{1}{\sqrt{\pi}}b^{-1/2}e^{-b/2y}db,$$
 which is the \GammaDist[1/2,1/2y] distribution. Note that there is no point mass at $b=0$, as 0 is reflecting for \BESQ[1]. As noted in \eqref{eq:transn:PDIP}, $R^y(S^y)\sim\GammaDist[1/2,1/2y]$ as well. The extension to finite-dimensional marginals follows as in the proof of Proposition \ref{prop:BESQ_mass_temp}. This completes the proof when $\beta = \varnothing$. Now, by Definition \ref{constr:type_0}, the total mass process of a type-0 evolution from a general initial state is a \BESQ[1] added to the total mass process of an independent type-1 evolution, which by Proposition \ref{prop:BESQ_mass_temp} is a \BESQ[0]. Thus, the theorem follows from the well-known additivity property of \BESQ-processes; see e.g.\ \cite[Theorem XI.(1.2)]{RevuzYor}.
\end{proof}

We note one additional connection between type-0 and type-1 evolutions.

\begin{proposition}\label{prop:remaining_type-0}
 Fix $\beta\in\IPspace$, $a>0$, and let $\gamma := \{(0,a)\}\concat\beta$. Consider an independent pair $((\alpha^y,\,y\ge0),\,\bff)$, with $(\alpha^y)\sim\BPr^0_{\beta}$ and $\bff$ a \BESQ[-1] starting from $a$ and absorbed at 0. Let $(\td\alpha^y,\,y\ge0)\sim\BPr^1_{\gamma}$ and let $Y$ denote the lifetime of the original leftmost block in $(\td\alpha^y)$. Then $Y$ is an $(\cFI^y)$-stopping time and $(\td\alpha^y,\ y\in [0,Y))  \stackrel{d}{=} \big(\big\{(0,\bff(y))\big\} \concat \alpha^y,\ y\in[0,\life(\bff)) \big).$
\end{proposition}

\begin{proof}
 We begin with $\beta = \emptyset$. Let $\bN_{\gamma}\sim \Pr^1_{\{(0,a)\}} = \mClade^+(\,\cdot\;|\;m^0=a)$. By Proposition \ref{prop:clade_splitting}, $\bN_{\gamma}$ is distributed like $\Dirac{0,\bff}+\restrict{\bN}{[0,T^{-\life(\bff)})}$, where $\bff$ is a \BESQ[-1] starting from $a$, independent of $\bN$. Comparing this to the construction of the type-0 evolution $(\cev\alpha^y)$ around \eqref{eq:type-0:consistency} proves the claimed identity in this case. For other values of $\beta$, the type-1 and type-0 evolutions with respective laws $\BPr^1_{\gamma}$ and $\BPr^0_{\beta}$ may be constructed by concatenating each of the evolutions in the previous case with an independent type-1 evolution with law $\BPr^1_{\beta}$. 
\end{proof}

We define $L\colon \HIPspace\to [0,\infty)$ to map an interval partition to the mass of its leftmost block, or 0 if none exists. Let $R\colon\HIPspace\to [0,\infty)$ denote the remaining mass, $R(\alpha) = \IPmag{\alpha}-L(\alpha)$. It is not hard to see that these maps are measurable.

\begin{corollary}\label{cor:remaining_BESQ}
 Let $(\alpha^y,\,y\ge0)$ be a type-1 evolution. Let $Y := \inf\{y > 0\colon L(\alpha^y-) = 0\}$. Then $(L(\alpha^y),\,y\in [0,Y))$ and $(R(\alpha^y),\,y\in [0,Y))$ are jointly distributed as an independent \BESQ[-1] and \BESQ[1], stopped when the \BESQ[-1] hits zero.
\end{corollary}

\begin{proof} This follows from Proposition \ref{prop:remaining_type-0} and Theorem \ref{thm:BESQ_total_mass}.\end{proof}

\subsection{Continuity in initial state and strong Markov properties}
\label{sec:Markov}

\begin{proposition}[Continuity in the initial state]\label{prop:type-1:cts_in_init_state}
 For $f\colon\IPspace\to [0,\infty)$ bounded and continuous and $z>0$, the maps $\beta\mapsto\BPr^1_{\beta}[f(\alpha^z)]$ and $\beta\mapsto\BPr^0_{\beta}[f(\alpha^z)]$ are continuous on $(\IPspace,\dI)$.
\end{proposition}

\begin{proof}
 Fix $z>0$. We begin with the type-1 assertion. We will show that for every $\epsilon>0$ and $\alpha\in\IPspace$ there is some $\delta>0$ such that for $\beta\in\IPspace$, $\dI(\alpha,\beta)<\delta$ implies the existence of a pair of type-1 evolutions $(\alpha^y,\,y\geq 0)$ and $(\beta^y,\,y\geq 0)$ starting from these two initial states, with
 \begin{equation}
  \Pr\{\dI(\alpha^{z},\beta^{z}) \geq 3\epsilon\} < 6\epsilon.\label{eq:level_z_IP_cnvgc_claim}
 \end{equation}
 
 Fix $0 < \epsilon < z$ and $\alpha\in\IPspace$. Let $U_1,U_2,\ldots$ denote the blocks of $\alpha$, listed in non-increasing order by mass. Let $(\bN_{U_j})_{j\geq 1}$ be as in Definition \ref{constr:type-1}, let $(\alpha_{U_j}^y,\,y\geq 0) := \skewerP(\bN_{U_j})$, and set $a_j := \Leb(U_j)$. We take suitable a.s.\ versions so that the process $(\alpha^y,\,y\geq 0)$ formed by concatenating the $(\alpha_{U_j}^y)$ according to the interval partition order of the $U_j$ in $\alpha$, as in \eqref{eq:concat_skewers:type-1}, is a type-1 evolution starting from $\alpha$.
 
 We take $L$, $M$, and $K$ sufficiently large and $\delta > 0$ sufficiently small so that setting
 \begin{gather*}
  E_1 := \big\{\IPLT_{\alpha^{z}}(\infty) < L;\ \IPmag{\alpha^{z}} < M\big\}, \qquad E_2 := \big\{ \forall j > K,\,\zeta^+(\bN_{U_j}) < z\big\},\\
  \text{and}\quad E^j_3 := \left\{ \sup\nolimits_{y \in \left[\left(1-(\delta/a_K)\right)z,\left(1+(\delta/a_K)\right)z\right]}\dI\left(\alpha^y_{U_j},\alpha^z_{U_j}\right) < \frac{\epsilon}{K} \right\} \quad \text{for }j\in [K],
 \end{gather*}
 we have $\Pr(E_1)\geq 1-\epsilon$, $\Pr(E_2)\geq 1-\epsilon$, and $\Pr(E^j_3)\geq 1-(\epsilon/K)$ for each $j$. By Lemma \ref{lem:finite_survivors}, it suffices that we take the smallest $K$ large enough that $\sum_{j>K}a_j < 2z\epsilon$. The existence of such a $\delta$ is then guaranteed by the continuity of the type-1 evolution. We further require
 \begin{equation}\label{eq:partn_SM_1_delta_constraints}
  \delta < \min\left\{a_K,z\epsilon,\frac{\epsilon a_K}{K L}, \frac{\epsilon a_K}{K M}\right\}.
 \end{equation}
 
 Now take $\beta\in\IPspace$ with $\dI(\alpha,\beta) < \delta$. By definition of $\dI$, there exists a correspondence $(\td U_j,\td V_j)_{j\in [\td K]}$ from $\alpha$ to $\beta$ with distortion less than $\delta$. Since $\delta < a_K$, we get $\td K\geq K$ and
 $\{U_j\}_{j\in [K]} \subseteq \big\{\td U_j\big\}_{j\in [\td K]}$.
 Let $(V_j)_{j\in [K]}$ denote the terms paired with the respective $U_j$ in the correspondence; i.e.\ for each $j\in [K]$, the pair $(U_j,V_j)$ equals $(\td U_i,\td V_i)$ for some $i\in [\td K]$. For $j\in [K]$, let $b_j := \Leb(V_j)$.
 
 We assume w.l.o.g.\ that our probability space is sufficiently large for the following construction 
 of a type-1 evolution $(\beta^y,\,y\geq 0)$ starting from $\beta$, coupled with $(\alpha^y,\,y\ge0)$. For $j\in [K]$, set $\bN_{V_j} := (b_j/a_j)\scaleH\bN_{U_j}$.  We take $(\bN_V,\,V\in\beta\setminus\{V_j\colon j\in [K]\})$ to be an independent family, independent of $(\bN_U,\,U\in\alpha)$, with distributions as in Definition \ref{constr:type-1}. We write $(\beta_V^y,\,y\geq 0) := \skewerP(\bN_V)$ for each $V\in\beta$. From Lemma \ref{lem:clade:mass_ker} and the definition of $\scaleH$ in \eqref{eq:clade:xform_def}, we deduce that for $j\in [K]$ and $y\geq 0$,
 \begin{equation*}
  \bN_{V_j} \sim \mClade^+\left(\,\cdot\;\middle|\;m^0 = \Leb(V_j)\right) \qquad \text{and} \qquad
  \beta_{V_j}^y = \scaleI[\frac{b_j}{a_j}][\alpha_{U_j}^{y(a_j/b_j)}].
 \end{equation*}
 Then $(\beta^y,\,y\geq 0) := \left(\ConcatIL_{V\in \beta}\beta_V^y,\,y\ge0\right)$ is a type-1 evolution from $\beta$.

 By Definition \ref{def:IP:metric} of $\dI$ and our choices of $K$ and $\delta$,
 \begin{equation*}
  \IPmag{\beta} - \sum_{j=1}^K b_j \leq \dI(\alpha,\beta) + \IPmag{\alpha} - \sum_{j=1}^K a_j < \delta + 2z\epsilon < 3z\epsilon.
 \end{equation*}
 Thus, by Lemma \ref{lem:finite_survivors}, the event $E_4 := \big\{ \zeta^+(\bN_V) < z\text{ for every }V\in\beta\setminus \{V_j\colon j\in [K]\}\big\}$ 
 has probability at least $1-3\epsilon$. 
 On $E_2\cap E_4$, the partition $\alpha^z$ is formed by concatenating, in interval partition order, the $\alpha^z_{U_j}$, and correspondingly for $\beta^z$. 
 
 Inequality \eqref{eq:IP:scaling_dist_1} and the last two constraints on $\delta$ in \eqref{eq:partn_SM_1_delta_constraints} imply that on $E_1$,
 \begin{equation*}
  \dI\left(\alpha_{U_j}^z,\beta_{V_j}^{z(b_j/a_j)}\right) \leq \max\left\{\left|\sqrt{\frac{b_j}{a_j}} - 1\right|L, \left|\frac{b_j}{a_j}-1\right|M\right\} < \frac{\epsilon}{K}.
 \end{equation*}
 Moreover, \eqref{eq:IP:scaling_dist_2} implies that for each $j$, on $E^3_j\cap E_1$,
 \begin{equation*}
  \dI\left(\beta^{z(b_j/a_j)}_{V_j},\beta^{z}_{V_j}\right) < \max\left\{\frac{b_j}{a_j},\sqrt{\frac{b_j}{a_j}}\right\}\frac{\epsilon}{K} < \frac{2\epsilon}{K}, \qquad \text{so} \qquad
  \dI\left(\alpha^{z}_{U_j},\beta^{z}_{V_j}\right) < \frac{3\epsilon}{K}.
 \end{equation*}
 Finally, by Lemma \ref{lem:IP:concat}, 
 $\dI(\alpha^{z},\beta^{z}) < 3\epsilon$ on $E_1\cap E_2\cap E_4\cap\bigcap_{j=1}^K E_3^j$, 
 and this intersection has probability at least $1-6\epsilon$, as claimed in \eqref{eq:level_z_IP_cnvgc_claim}.
 
 For the type-0 result, recall Definition \ref{constr:type_0} of a type-0 evolution as $(\cev\alpha^y,\,y\geq0)$ concatenated with an independent type-1 evolution. Suppose $(\alpha^y)$ and $(\beta^y)$ are the coupled type-1 evolutions described above and $(\cev\alpha^y)$ is independent of both. Set $\td\alpha^y := \cev\alpha^y\concat\alpha^y$ and $\td\beta^y := \cev\alpha^y\concat\beta^y$ for $y\geq 0$. Then $\dI(\td\alpha^y,\td\beta^y) = \dI(\alpha^y,\beta^y)$, so our earlier argument applies to the type-0 evolutions $(\td\alpha^y,\,y\geq0)$ and $(\td\beta^y,\,y\geq 0)$.
\end{proof}

\begin{corollary}\label{cor:type-1:cts_init_2}
 Take $m\in\BN$, let $f_1,\ldots,f_m\colon\IPspace\rightarrow[0,\infty)$ be bounded and continuous, and take $0\le y_1<\cdots<y_m$. Then $\beta \mapsto \BPr^1_{\beta}\left[\prod_{i=1}^m f_i(\alpha^{y_i})\right]$ and $\beta \mapsto \BPr^0_{\beta}\left[\prod_{i=1}^m f_i(\alpha^{y_i})\right]$ are continuous.
 %
\end{corollary}

\begin{proof}
 The same proof applies to both type-0 and type-1, requiring only Proposition \ref{prop:type-1:cts_in_init_state} and the simple Markov property. We state our proof in terms of the type-1 case. The case $m=1$ is covered by Proposition \ref{prop:type-1:cts_in_init_state}. Assume for induction that for some $m\geq 1$, the assertion holds for all $m$-tuples $(f_1,\ldots,f_m)$ and $y_1 < \cdots < y_m$ as above. Now, fix $0\le y_1<\cdots<y_m<y_{m+1}$ and suppose $f_1,\ldots,f_m,f_{m+1}\colon\IPspace\rightarrow[0,\infty)$ are bounded and continuous. Then by the inductive hypothesis and the continuity of $f_1$, the function
 $$h(\beta) = f_1(\beta)\BPr^1_{\beta}\left[\prod_{i=1}^{m}f_{i+1}(\alpha^{y_{i+1}-y_1})\right]$$
 is bounded and continuous. The simple Markov property, noted in Corollary \ref{cor:type-1:simple_Markov}, and Proposition \ref{prop:type-1:cts_in_init_state} applied to $h$ yield that for all sequences $\beta_j\rightarrow\beta$,
 $$
  \BPr^1_{\beta_j}\left[\prod_{i=1}^{m+1}f_i(\alpha^{y_i})\right] = \BPr^1_{\beta_j}[h(\alpha^{y_1})] \rightarrow
  \BPr^1_{\beta}[h(\alpha^{y_1})] = \BPr^1_{\beta}\left[\prod_{i=1}^{m+1}f_i(\alpha^{y_i})\right].
 $$
 This proves the continuity of 
 $\beta \mapsto \BPr^1_{\beta}\left[\prod_{i=1}^{m+1}f_i(\alpha^{y_i})\right],$ 
 thereby completing the induction.
\end{proof}

\begin{proposition}[Strong Markov properties for type-1 and type-0]\label{prop:strong_Markov}
 Let $\mu$ be a probability distribution on $\IPspace$. Let $Y$ be an a.s.\ finite stopping time in $(\cFI^y,\,y\ge 0)$. Take $\eta,f\colon\cCRI\to [0,\infty)$ measurable, with $\eta$ measurable with respect to $\cFI^y$. Let $\theta_y$ denote the shift operator. Then 
  $\BPr^1_\mu\big[\eta\, f\circ\theta_Y \big] = \BPr^1_{\mu}\left[\eta\, \BPr^1_{\alpha^Y}[f]\right]$ and $\BPr^0_\mu\big[\eta\, f\circ\theta_Y \big] = \BPr^0_{\mu}\left[\eta\, \BPr^0_{\alpha^Y}[f]\right]$.
\end{proposition}

\begin{proof}
 If $Y$ only takes finitely many values, this is implied by the simple Markov property. In general, this follows via a standard discrete approximation of $Y$, as in the proof of \cite[Theorem 19.17]{Kallenberg}, in which we replace the Feller property by Corollary \ref{cor:type-1:cts_init_2}.
\end{proof}

We now prove our first main theorem, characterizing type-1 and type-0 evolutions as Hunt processes.

\begin{proof}[Proof of Theorems \ref{thm:diffusion_0} and \ref{thm:diffusion}]\label{page:pf_of_diffusions}
 Theorem \ref{thm:diffusion_0} follows from Corollary \ref{cor:type-1:cts_from_Stable} \ref{item:c_f_S:determinstic} and Proposition \ref{prop:strong_Markov}. For Theorem \ref{thm:diffusion}, 
 referring to Sharpe's definition of Borel-right Markov processes and Hunt processes, e.g.\ \cite[Definition A.18]{Li11}, we must check four properties.
 
 (i) The state space $(\IPspace,\dI)$ must be a Radon space. In fact it is Lusin, by Theorem \ref{thm:Lusin}.
 
 (ii) The semi-groups must be Borel measurable in the initial state. From Proposition \ref{prop:type-1:cts_in_init_state}, they are continuous.
 
 (iii) Sample paths must be right-continuous and quasi-left-continuous. In fact they are continuous, by Proposition \ref{prop:type-1:cts} and the construction around \eqref{eq:type-0:consistency}.
 
 (iv) The processes must be strong Markov under a right-continuous filtration. We have this from Proposition \ref{prop:strong_Markov}.
\end{proof}

\subsection{Interval partition evolutions started without diversity}
\label{sec:Hausdorff}

The construction in Definition \ref{constr:type-1} of $\bN_{\beta} = \ConcatIL_{U\in\beta}\bN_U$, for $\beta\in\IPspace$, can be carried out for $\beta\in\HIPspace$ as well. Extending the notation of that definition, let $\Pr^1_{\beta}$ denote the law of the resulting point process. The proof of Proposition \ref{prop:clade_lengths_summable} \ref{item:CLS:CLS} and \ref{item:CLS:kernel} that $\len(\bN_{\beta}) < \infty$ a.s.\ and $\beta\mapsto\Pr^1_{\beta}$ is a kernel still holds, without modification, in this generality. The same is true of the proofs of results in Section \ref{sec:type-1_gen:cts}, from Lemma \ref{lem:finite_survivors} up through Proposition \ref{prop:BESQ_mass_temp}. Several of these involve $\IPLT_{\alpha^y}(t)$ for $y>0$, but none take $y=0$. 
In particular, we note the extensions of Lemmas \ref{lem:finite_survivors} and \ref{lem:type-1:wd} to this setting.

\begin{lemma}\label{lem:Hausdorff_entrance_I}
 For $\beta\in\HIPspace$ and $y>0$, a.s.\ only finitely many of the $(\bN_U,\,U\in\beta)$ survive to level $y$. 
 Moreover, it is a.s.\ the case that for every $y>0$ we have $\skewer(y,\bN_{\beta})\in\IPspace$.
\end{lemma}

Proposition \ref{prop:cts_lt_at_0} deals specifically with the diversity of the interval partition at level $0$. It is used in the proof of Proposition \ref{prop:type-1:cts} to check continuous diversities as the type-1 evolution enters from its initial state. However, without this observation, what remains of the proof of Proposition \ref{prop:type-1:cts} includes a proof of the following.

\begin{proposition}\label{prop:cts_from_Hausdorff}
 For $\beta\in\HIPspace$, it is a.s.\ the case that $\skewerP(\bN_{\beta})$ is continuous in $(\HIPspace,\dH)$. Moreover, this process is H\"older-$\theta$ in $(\IPspace,\dI)$ on the time interval $(0,\infty)$, for every $\theta\in (0,\frac14)$.
\end{proposition}

\begin{proof}
 The first part of the proof of Proposition \ref{prop:type-1:cts} applies to show that for $z>0$ the process $(\skewer(y+z,\bN_{\beta}),\,y\ge0)$ has the claimed H\"older continuity on $(\IPspace,\dI)$. By Proposition \ref{prop:Hausdorff}, this implies continuity in $(\HIPspace,\dH)$. Then the latter part of the proof, and particularly the bounds on quantities (iii) and (iv), show that $(\skewer(y,\bN_{\beta}),\,y\ge0)$ enters $\dH'$-continuously from $y=0$, where $\dH'$ is as in Definition \ref{def:Hausdorff}. By Proposition \ref{prop:Hausdorff} \ref{item:Haus:equiv}, this yields $\dH$-continuity.
\end{proof}

We extend the notation of Definition \ref{def:IP_process_space_1} to define $\BPr^1_{\beta}$ for $\beta\in\HIPspace\setminus\IPspace$, to denote the law of a version of $\skewerP(\bN_{\beta})$ that enters $\dH$-continuously and is subsequently $\dI$-continuous. We call this continuous version a \emph{Hausdorff type-1 evolution}. Then, we can do the same for Definition \ref{constr:type_0} of $\Pr^0_{\beta}$, concatenating $(\cev\alpha^y,\,y\ge0)$ with a Hausdorff type-1 evolution to get a \emph{Hausdorff type-0 evolution}. Again all proofs of results in Section \ref{sec:type-0} apply without modification to this variant of the process. In Section \ref{sec:Markov}, 
the same coupling argument used to prove Proposition \ref{prop:type-1:cts_in_init_state} also proves the following variant.

\begin{proposition}\label{prop:type-1:cts_in_init_state_H}
 Let $\beta\in\HIPspace$. For $f\colon\HIPspace\to [0,\infty)$ bounded and continuous and $z>0$, the maps $\beta\mapsto\BPr^1_{\beta}[f(\alpha^z)]$ and $\beta\mapsto\BPr^0_{\beta}[f(\alpha^z)]$ are continuous on $(\HIPspace,\dH)$.
\end{proposition}

\begin{proof}
 By Proposition \ref{prop:Hausdorff} \ref{item:Haus:equiv}, it suffices to prove continuity under $d_H'$. 
 We follow the same argument, but omit the definition of $L$ and resulting bound on $\delta$ in \eqref{eq:partn_SM_1_delta_constraints}. So $E_1$ becomes $\{\IPmag{\alpha^z}\leq M\}$. Then we make the same coupling to define $(\beta^y)$ based on $(\alpha^y)$. In this setting, applying \eqref{eq:IP:Haus_scale}, the final two displays in the proof become:  for each $j$, on $E^3_j\cap E_1$,
 \begin{equation*}
  \dH'\left(\alpha_{U_j}^z,\beta_{V_j}^{z(b_j/a_j)}\right) \leq \left|\frac{b_j}{a_j}-1\right|M < \frac{\epsilon}{K}, \qquad \dH'\left(\beta^{z(b_j/a_j)}_{V_j},\beta^{z}_{V_j}\right) < \frac{b_j}{a_j}\frac{\epsilon}{K} < \frac{2\epsilon}{K},
 \end{equation*}
 and so $\dH'\big(\alpha^{z}_{U_j},\beta^{z}_{V_j}\big) < 3\epsilon/K$. Otherwise, the proof is as before.
\end{proof}

This result extends to a Hausdorff variant of Corollary \ref{cor:type-1:cts_init_2}, in the same manner as before, via the simple Markov property. Then the statement of the strong Markov property, Proposition \ref{prop:strong_Markov}, holds for initial distributions $\mu$ on $\HIPspace$, via the same standard argument.

\section{De-Poissonization and stationary interval-partition diffusions}\label{sec:dePoissonization}

\subsection{Pseudo-stationarity of type-1 and type-0 evolutions}\label{sec:pseudostat}

Neither the \BESQ[0] total mass law of the type-1 evolution nor the \BESQ[1] total mass law of the type-0 admit stationary distributions, so neither do the type-1 or type-0 evolutions themselves. However, we do have the following ``pseudo-stationarity'' results. These will help us prove Theorem \ref{thm:stationary}, which describes stationary variants of the interval partition evolutions. Recall the Poisson-Dirichlet interval partitions of Proposition \ref{prop:PDIP}.

\begin{theorem}[Pseudo-stationarity]\label{thm:pseudostat}
 Consider an independent pair $(B(\,\cdot\,),\ol\beta)$, where $\ol\beta\sim\PDIP[\frac12,0]$ (respectively \PDIP[\frac12,\frac12]) and $B$ is a \BESQ[0] (resp.\ \BESQ[1]) with an arbitrary initial distribution. Let $(\alpha^y,\,y\geq 0)$ be a type-1 (resp.\ type-0) evolution with $\alpha^0\stackrel{d}{=}B(0)\scaleI\ol\beta$. Then for each fixed $y\ge 0$ we have $\alpha^y \stackrel{d}{=} B(y)\scaleI\ol\beta$.
\end{theorem}

We prove this theorem in stages over the course of this section by considering different cases for the law of $B(0)$. Later, we demonstrate a stronger form of this statement in Theorem \ref{thm:pseudostat_strong}.

\begin{proposition}\label{prop:pseudostat:exp}
 Suppose that in the setting of the type-1 (respectively type-0) assertion of Theorem \ref{thm:pseudostat} we have $B(0)\sim \ExpDist[\rho]$ (resp.\ \GammaDist[\frac12,\rho]) for some $\rho\in(0,\infty)$. Then the conditional law of $\alpha^y$, given $\{\alpha^y\neq\emptyset\}$, equals the unconditional law of $(2y\rho+1)B(0)\scaleI\ol\beta$.
\end{proposition}

\begin{proof}
 We begin with the type-1 case. 
 We prove this by separately comparing the Laplace transforms of the leftmost blocks of the two interval partitions, comparing Laplace transforms of the subordinators of remaining block masses, and confirming that in each partition the leftmost block is independent of the remaining blocks. This is done in three steps.
 
 \emph{Step 1}. Following Proposition \ref{prop:PDIP} \ref{item:PDIP:Stable}, we may represent $\ol\beta$ as $\{(0,L)\}\concat \big((1-L)\scaleI \ol\gamma\big)$, where $L\sim\BetaDist[\frac12,\frac12]$ is independent of $\ol\gamma\sim\PDIP[\frac12,\frac12]$. Let $\beta := B(0)\scaleI\ol\beta$ and $\gamma := B(0)(1-L)\scaleI\ol\gamma$. We denote the leftmost block of $\beta$ by $U_0 := (0,B(0)L)$. Since $B(0) \sim \ExpDist[\rho]$, the masses $\Leb(U_0)$ and $\IPmag{\gamma}$ are i.i.d.\ \GammaDist[\frac12,\rho] variables. Thus, we can take $(\alpha^y,\,y\geq 0) := \skewerP(\bN_{U_0} \concat \bN_{\gamma})$, where $\bN_{\gamma} \sim \Pr^1_{\gamma}$ is independent of $\bN_{U_0} \sim \mClade^+(\,\cdot\;|\;m^0=B(0)L)$.
 
 By Proposition \ref{prop:PDIP} \ref{item:PDIP:Stable}, the partition $\gamma$ corresponds to the range of a \Stable[\frac12] subordinator stopped prior to crossing an independent random level $S\sim \ExpDist[\rho]$. This stopping corresponds to thinning the Poisson point process of jumps of the subordinator, tilting the L\'evy measure by a factor of $e^{-\rho x}$. In turn, appealing to Proposition \ref{prop:agg_mass_subord}, the point process $F^{\geq 0}_0(\bN_{\gamma})$ of Definition \ref{constr:type-1_2} may be viewed as a \PRM[\Leb\otimes\mu], where $\mu$ is given by $d\mu(N) = e^{-\rho m^0(N)}d\mClade^+(N)$, stopped at an independent random exponential time whose rate can be calculated from Proposition \ref{prop:clade:stats} \ref{item:CS:mass} to be
 \begin{equation*}
  \int_0^\infty (1-e^{-\rho a})\mClade^+\{m^0\in da\} = \int_0^\infty (1-e^{-\rho a})\frac{1}{2\sqrt{\pi}}a^{-3/2}da = \sqrt{\rho}.
 \end{equation*}
 
 Via Proposition \ref{prop:clade:stats} \ref{item:CS:max:mass} and the density function of $\Leb(U_0)\sim\GammaDist[\frac12,\rho]$,
 \begin{equation}\label{eq:p_s_E:LMB0_survival}
  \Pr\{\life^+(\bN_{U_0}) \leq y\} = \int_0^\infty e^{-a/2y} \frac{1}{\Gamma\left(\frac12\right)}\sqrt{\frac{\rho}{a}}e^{-\rho a}da = \sqrt{\frac{2y\rho}{2y\rho+1}}.
 \end{equation}
 To get the rate at which clades exceeding level $y$ arise in $F^{\geq 0}_0(\bN_{\gamma})$, we apply Proposition \ref{prop:clade:stats} \ref{item:CS:mass} and \ref{item:CS:max:mass}:
 \begin{equation*}
 \begin{split}
  \mu\{\life^+ > y\} &= \int_0^\infty(1-e^{-a/2y})\frac{1}{2\sqrt{\pi}}e^{-\rho a}a^{-3/2}da
  	= \sqrt{\rho+\frac{1}{2y}}-\sqrt{\rho} = \left(\sqrt{\frac{2y\rho+1}{2y\rho}}-1\right)\sqrt{\rho}.
 \end{split}
 \end{equation*}
 By competing exponential clocks, the probability of seeing no such clade before the independent $\ExpDist\big(\sqrt{\rho}\big)$ time is
 \begin{equation}\label{eq:p_s_E:competing_clocks}
  \Pr\{\life^+(\bN_{\gamma}) \leq y\} = \frac{\sqrt{\rho}}{\left(\sqrt{(2y\rho+1)/(2y\rho)}-1\right)\sqrt{\rho}+\sqrt{\rho}} = \sqrt{\frac{2y\rho}{2y\rho+1}}.
 \end{equation}
 Thus, by the independence of $\bN_{U_0}$ and $\bN_{\gamma}$,
 \begin{equation}\label{eq:p_s_E:survival}
  \Pr\{\alpha^y\neq\emptyset\} = \EV\left[\Pr_{\beta}\left\{\life^+>y\right\}\right] = 1-\left(\sqrt{\frac{2y\rho}{2y\rho+1}}\right)^2 = \frac{1}{2y\rho+1}.
 \end{equation}
 
 \emph{Step 2}. First, we compute the Laplace transform of the leftmost block mass in the event that it arises from $\bN_{U_0}$; then, we compute it in the event that the leftmost block arises from one of the clades that make up $\bN_{\gamma}$. 
 The Laplace transform $\mClade^+[e^{-\lambda m^y}\;|\;m^0=a]$, where $m^y$ is as in \eqref{eq:LMB_def}, may be read from \eqref{eq:clade:LMB_Laplace_2}. We multiply this by 
 the survival probability of Proposition \ref{prop:clade:stats} \ref{item:CS:max:mass} and integrate against the \GammaDist[\frac12,\rho] law of $\Leb(U_0)$:
 \begin{equation*}
 \begin{split}
  &\EV\left[ e^{-\lambda m^y(\bN_{U_0})}\cf\{\life^+(\bN_{U_0})>y\}\right] = \EV\left[\mClade^+\big[ e^{-\lambda m^y}\cf\{\life^+>y\}\;\big|\; m^0 = \Leb(U_0) \big] \right]\\
  &= \int_0^\infty\frac{\sqrt{2y\lambda+1}}{1-e^{-a/2y}}\left(e^{-\lambda a/(2y\lambda+1)}-e^{-a/2y}\right)\left(1-e^{-a/2y}\right)\frac{1}{\sqrt{\pi}}\sqrt{\frac{\rho}{a}}e^{-\rho a}da\\ 
  &= \sqrt{2y\lambda+1}\left(\sqrt{\frac{\rho(2y\lambda+1)}{\rho(2y\lambda+1)+\lambda}}-\sqrt{\frac{2y\rho}{2y\rho+1}}\right).
 \end{split}
 \end{equation*}
 The clades in $F^{\geq 0}_0(\bN_{\gamma})$ that exceed lifetime $y$ form an i.i.d.\ sequence with law $\mu(\,\cdot\;|\;\life^+ > y)$. Then the contribution to $\EV[\exp(-\lambda m^y(\bN_{U_0}\concat\bN_{\gamma}))]$ follows similarly:
 \begin{align}
  &\EV\left[ e^{-\lambda m^y(\bN_{\gamma})}\cf\{\life^+(\bN_{\gamma}) > y \geq \life^+(\bN_{U_0})\} \right] = \Pr\{\life^+(\bN_{\gamma}) > y \geq \life^+(\bN_{U_0})\} \mu\left[ e^{-\lambda m^y}\;\middle|\;\life^+>y\right]\notag\\
  &= \frac{\Pr\{\life^+(\bN_{\gamma}) > y \geq \life^+(\bN_{U_0})\}}{\mu\{ \life^+ > y\}} \int_0^\infty \mClade^+\big[ e^{-\lambda m^y}\cf\{\life^+ > y\}\;\big|\;m^0=a \big] \mu\{ m^0\in da \}\label{eq:p_s_E:LMB_Laplace_2}\\
  &= \frac{\sqrt{\frac{2y\rho}{2y\rho+1}}\left(1-\sqrt{\frac{2y\rho}{2y\rho+1}}\right)}{\left(\sqrt{\frac{2y\rho+1}{2y\rho}}-1\right)\sqrt{\rho}}\int_0^\infty\!\!\sqrt{2y\lambda+1}\left(e^{-\lambda a/(2y\lambda+1)}-e^{-a/2y}\right)\frac{1}{2\sqrt{\pi}}a^{-3/2}e^{-\rho a}da\notag\\
  &= \sqrt{2y\lambda+1}\left(\sqrt{\frac{2y\rho}{2y\rho+1}}-\left(\frac{2y\rho}{2y\rho+1}\right)\sqrt{\frac{\lambda + \rho(2y\lambda+1)}{\rho(2y\lambda+1)}}\right).\notag
 \end{align}
 Adding these terms and dividing by the formula for $\Pr\{\alpha^y\neq\emptyset\}$ in \eqref{eq:p_s_E:survival}, we get
 \begin{equation*}
 \begin{split}
  \EV\left[\Pr_{\beta}\big[e^{-\lambda m^y}\;\big|\;\life^+>y\big]\right] &= \sqrt{2y\lambda+1} \left(  (2y\rho+1)\sqrt{\frac{\rho(2y\lambda+1)}{\rho(2y\lambda+1)+\lambda}} - 2y\rho\sqrt{\frac{\rho(2y\lambda+1)+\lambda}{\rho(2y\lambda+1)}}  \right)\\
  &= \sqrt{\frac{\rho}{\rho(2y\lambda+1)+\lambda}} = \sqrt{\frac{\rho/(2y\rho+1)}{(\rho/(2y\rho+1))+\lambda}}\,.
 \end{split}
 \end{equation*}
 This equals the Laplace transform of $(2y\rho+1)\Leb(U_0)\sim\GammaDist[1/2,\rho/(2y\rho+1)]$.
 
 \emph{Step 3}. By Lemma \ref{lem:cutoff_skewer} \ref{item:CPS:clades_skewer_0}, $\alpha^y$ equals $\skewer(y,\bN_{U_0})$ concatenated with the skewers of the clades in $F^{\geq 0}_0(\bN_{\gamma})$ that survive to level $y$. By Lemma \ref{lem:finite_survivors}, only finitely many clades survive. Recall from Step 1 that we may view $F^{\geq 0}_0(\bN_{\gamma})$ as a \PRM[\Leb\otimes\mu] stopped at an independent $\ExpDist\big(\sqrt{\rho}\big)$ time. If we condition on $\{\alpha^y\neq\emptyset\}$ then, following the competing exponential clocks argument around \eqref{eq:p_s_E:competing_clocks}, we may view the surviving clades beyond the leftmost surviving clade as coming from an infinite sequence of independent clades with distribution $\mu(\,\cdot\;|\;\life^+>y)$ stopped 
 after an independent $\GeomDist\big(\sqrt{2y\rho/(2y\rho+1)}\big)$ number $G$ of clades, where $G$ can equal 0.
 
 By Propositions \ref{prop:type-1:transn} and \ref{prop:type-1:gen_transn}, for all $i\ge 1$, the $i^{\text{th}}$ clade with distribution $\mu(\,\cdot\;|\;\life^+>y)$ 
contributes its own leftmost block $(0,L^y_i)$ at level $y$, followed by masses from an independent subordinator $R^y_i$ stopped at an independent \ExpDist[(2y)^{-1/2}] time $S^y_i$. The leftmost surviving clade contributes a special leftmost block studied in Step 2, and independent masses from $(R_0^y,S_0^y)$, as for $i\ge 1$. We call the masses from $(R^y_i,S^y_i)$, $i\ge 0$, the ``remaining masses.'' 
So, we may view the masses in $\alpha^y$ beyond the far leftmost as arising from an alternating sequence of remaining masses of clades $i=0,\ldots,G$ and the leftmost blocks of clades $i=1,\ldots,G$.
 
The stopped $R_i^y$ from all clades $i\ge 0$ can be combined to capture all remaining masses in a single unstopped subordinator $R_{\rm rem}^y$ with Laplace exponent
$\Phi^y_{\rm rem} = \Phi^y$ of \eqref{eq:transn:Levy_meas}, independent of $(S_i^y,i\ge 0)$ and $G$, and hence of $\widetilde{S}^y:=S_0^y+\cdots+S_G^y$, which is exponential with parameter 
$$(2y)^{-1/2}\sqrt{2y\rho/(2y\rho+1)} = \sqrt{\rho/(2y\rho + 1)}.$$
This is the time that corresponds to stopping $R_{\rm rem}^y$ after the $G^{\text{th}}$ surviving clade. This independence also yields the independence of $R_{\rm rem}^y$ from 
the subordinator that has jumps of sizes $L_i^y$ at times $S_0+\cdots+S_{i-1}^y$, $i\ge 1$, with Laplace exponent 
$(2y)^{-1/2}\mu[ 1-e^{-\lambda m^y}\;|\;\life^+>y]$. Note
 \begin{equation*}
 \begin{split}
  \mu\left[ 1-e^{-\lambda m^y}\;\middle|\;\life^+>y\right] &= 1 - \int_0^\infty\mClade^+\big[ e^{-\lambda m^y}\cf\{\life^+ > y\}\;\big|\;m^0=a \big] \frac{\mu\{ m^0\in da \}}{\mu\{\life^+ > y\}}\\
  	&= 1-\sqrt{2y\lambda+1}\frac{\sqrt{\frac{2y\rho+1}{2y\rho}}-\sqrt{\frac{\rho+(2y\rho+1)\lambda}{\rho(1+2\lambda y)}}}{\sqrt{\frac{2y\rho+1}{2y\rho}}-1},
 \end{split}
 \end{equation*}
where the integral is a multiple of that in \eqref{eq:p_s_E:LMB_Laplace_2}. By an elementary thinning argument, this subordinator stopped after the $G^{\text{th}}$ jump can be viewed as a subordinator with Laplace exponent
 \begin{equation*}
 \begin{split}
  \Phi^y_{\textnormal{LMB}}(\lambda) &= \frac{1}{\sqrt{2y}}\left(1-\sqrt{\frac{2y\rho}{2y\rho+1}}\right)\mu\left[ 1-e^{-\lambda m^y}\;\middle|\;\life^+>y\right]\\
    &= \frac{1}{\sqrt{2y}}\left(1 - \sqrt{\frac{2y\rho}{2y\rho+1}} - \sqrt{2y\lambda + 1}\left(1 - \sqrt{\frac{2y\rho}{2y\rho + 1}}\sqrt{\frac{\rho + (2y\rho + 1)\lambda}{\rho(2y\lambda + 1)}}\right)\right).
 \end{split}
 \end{equation*}
stopped at the independent time $\widetilde{S}^y\sim\ExpDist\big(\sqrt{\rho/(2y\rho + 1)}\big)$.

 Putting these pieces together, $\Phi^y_{\textnormal{Rem}}(\lambda)+\Phi^y_{\textnormal{LMB}}(\lambda)$ is given by
 \begin{equation*}
 \begin{split}
  &\frac{1}{\sqrt{2y}}\left(\sqrt{2\lambda y+1}-1+1-\sqrt{\frac{2y\rho}{2y\rho+1}}-\sqrt{2y\lambda+1}+\sqrt{\frac{2y}{2y\rho+1}}\sqrt{\rho+(2y\rho+1)\lambda}\right)\\
  &\qquad= \sqrt{\lambda+\frac{\rho}{2y\rho+1}}-\sqrt{\frac{\rho}{2y\rho+1}} = \int_0^\infty(1-e^{-\lambda x})\frac{1}{2\sqrt{\pi}}e^{-x\rho/(2y\rho+1)}x^{-3/2}dx.
 \end{split}
 \end{equation*}
 The last expression above is the Laplace transform of a subordinator that, when stopped at an independent $\ExpDist\big(\sqrt{\rho/(2y\rho + 1)}\big)$ time, corresponds as in \eqref{eq:transn:PDIP} 
  to a \PDIP[\frac12,\frac12] scaled by an independent \GammaDist[1/2,\rho/(2y\rho+1)] variable. Putting this together with the result of Step 2 and the independence, in both $\alpha^y$ and $\beta$, of the leftmost block from the rest, we conclude that $\alpha^y$ is distributed like a \PDIP[\frac12,0] scaled by an independent \ExpDist[\rho/(2y\rho+1)] variable, as desired.
 
 Looking at the transition description in Proposition \ref{prop:type-0:transn}, Step 3 above also proves the claim for the type-0 evolution.
\end{proof}

\begin{lemma}[Scaling invariance of type-1 and type-0 evolutions]\label{lem:type-1:scaling}
 Fix $c>0$. If $(\alpha^y,\,y\geq 0)$ is a type-1 (respectively type-0) evolution then so is $(c\scaleI \alpha^{y/c},\,y\geq 0)$.
\end{lemma}

\begin{proof}
 Recall the construction $\bN_{\gamma} = \ConcatIL_{U\in\gamma}\bN_U$ of Definition \ref{constr:type-1}. By Lemma \ref{lem:clade:invariance} and the scaling of $m^0$ noted in Table \ref{tbl:clade_scaling}, $c\scaleH\bN_{\gamma} \sim \Pr_{c\scaleI\gamma}$. This proves the result for type-1. For type-0, recall the construction prior to Definition \ref{constr:type_0}, and note that
 $$c\scaleH\bN \stackrel{d}{=}\bN,\qquad \text{so} \qquad  (\cev\alpha^y,\,y\geq 0)\stackrel{d}{=}(c\scaleI\cev\alpha^{y/c},\,y\geq 0).$$
 This proves the result for type-0 evolutions.
\end{proof}

We can now invert Laplace transforms to deduce the following.

\begin{proposition}\label{prop:pseudostat:fixed}
 Both assertions of Theorem \ref{thm:pseudostat} hold if $B(0) = a \geq 0$ is fixed.
\end{proposition}

\begin{proof}
 \emph{Type-1 case}. The case $a=0$ is trivial. The transition density of \BESQ[0] can be read from \cite[equation (51)]{GoinYor03}. For $B(0) = a > 0$ we get $\Pr\{B(y)=0\} = e^{-a/2y}$ and
 \begin{equation}\label{eq:BESQ_0_transn}
  \Pr\{B(y)\in db\} = \frac{1}{2y}\sqrt{\frac{a}{b}}\exp\left(-\frac{a+b}{2y}\right)I_1\left(\frac{\sqrt{ab}}{y}\right)db\qquad \text{for }b\in (0,\infty),
 \end{equation}
 where $I_1$ is the Bessel function. Let $(\alpha^y_1,\,y\geq 0)$ denote a type-1 evolution with initial state $\ol\beta\sim\PDIP[\frac12,0]$. For $b>0$ and $y\geq 0$ let $\alpha^y_{b} := b \scaleI \alpha^{y/b}_1$; by Lemma \ref{lem:type-1:scaling} this is a type-1 evolution. For $\rho>0$ let $Z_{\rho}\sim \ExpDist[\rho]$ be independent of $(\alpha^y_1)$. By Proposition \ref{prop:pseudostat:exp}, for all $\rho>0$ and all bounded continuous $f\colon\IPspace\rightarrow[0,\infty)$ with $f(\emptyset)=0$ we have
 \begin{align}
  \int_0^\infty e^{-\rho a}\EV[f(\alpha^y_a)]da &= \frac{1}{\rho}\EV\left[f\left(\alpha^y_{Z_\rho}\right)\right] = \frac{1}{2y\rho+1}\int_0^\infty e^{-\rho a}\EV[f((2y\rho+1)a\scaleI\ol\beta)]da\notag\\
     &= \frac{1}{(2y\rho+1)^2}\int_0^\infty e^{-\rho b/(2y\rho+1)}\EV(f(b\scaleI\ol\beta))db.\label{lt1}
 \end{align}
 We want to identify this Laplace transform as the claimed
 \begin{align}
 &\int_0^\infty e^{-\rho a}\int_0^\infty\frac{1}{2y}\sqrt{\frac{a}{b}}\exp\left(-\frac{a+b}{2y}\right)I_1\left(\frac{\sqrt{ab}}{y}\right)\EV(f(b\scaleI\ol\beta))dbda\nonumber\\
    &= \int_0^\infty\frac{1}{2y}\frac{1}{\sqrt{b}}e^{-b/2y}\EV(f(b\scaleI\ol\beta))\int_0^\infty \sqrt{a}e^{-(\rho+(1/2y))a}I_1\left(\frac{\sqrt{ab}}{y}\right)dadb\nonumber\\
    &= \int_0^\infty\frac{1}{2y}\frac{1}{\sqrt{b}}e^{-b/2y}\EV(f(b\scaleI\ol\beta))\sqrt{\frac{b}{y^2}}\frac{1}{2(\rho+(1/2y))^2}\exp\left(\frac{b}{4y^2(\rho+(1/2y))}\right)db,\label{lt2}
 \end{align}
 where we use well-known formulas for integrals involving the Bessel function $I_1$: specifically, the normalization of \eqref{eq:BESQ_0_transn} and differentiation $d/dx$ under the integral sign give rise to
 $$\int_0^\infty \frac{1}{\sqrt{u}}e^{-xu}I_1(\sqrt{uv})du = \frac{2}{\sqrt{v}}\left(e^{v/4x}-1\right)\quad \text{and}\quad
 \int_0^\infty \sqrt{u}e^{-xu}I_1(\sqrt{uv})du = \frac{\sqrt{v}}{2x^2}e^{v/4x}$$
 for all $x,v\in(0,\infty)$. As desired, \eqref{lt1} and \eqref{lt2} can easily be seen to be equal. By Proposition \ref{prop:type-1:cts_in_init_state}, the map $a\mapsto\EV(f(\alpha^y_a))$ is continuous, so for all $a,y\in(0,\infty)$,
 $$\EV\left[f(\alpha^y_a)\right] = \int_0^\infty\frac{1}{2y}\sqrt{\frac{a}{b}}e^{-(a+b)/2y}I_1\left(\frac{\sqrt{ab}}{y}\right)\EV(f(b\scaleI\ol\beta))db = \EV\left[f(B(y)\scaleI\ol\beta)\right].$$
 Equality in distribution follows since, as noted in Theorem \ref{thm:Lusin}, $(\IPspace,\dI)$ is Lusin, so bounded continuous functions separate points in $\IPspace$.
 
 \emph{Type-0 case}. We begin with a similar argument, making the obvious adjustments of letting $(\alpha^y_a,\,y\geq 0)$ denote a type-0 evolution for $a>0$, taking $\ol\beta\sim\PDIP[\frac12,\frac12]$, and setting $Z_{\rho}\sim\GammaDist[\frac12,\rho]$. Then Proposition \ref{prop:pseudostat:exp} gives us
  \begin{equation*}
  \begin{split}
  \int_0^\infty\frac{1}{\sqrt{\pi}}\sqrt{\frac{\rho}{a}}e^{-\rho a}\EV[f(\alpha^y_a)]da &= \EV\left[f(\alpha^{y}_{Z_\rho})\right]
    = \int_0^\infty\frac{1}{\sqrt{\pi}}\sqrt{\frac{\rho}{a}}e^{-\rho a}\EV[f((2\rho y+1)a\scaleI\ol\beta)]da\\
    &= \int_0^\infty\frac{1}{\sqrt{\pi}}\frac{1}{\sqrt{b}}\sqrt{\frac{\rho}{2\rho y+1}}e^{-\rho b/(2\rho y+1)}\EV[f(b\scaleI\ol\beta)]db.
  \end{split}
  \end{equation*}
  Since the total mass evolution is \BESQ[1], 
  considering $f$ of the form $g\left(\IPmag{\,\cdot\,}\right)$ gives
  $$\frac{1}{\sqrt{\pi}}\frac{1}{\sqrt{b}}\sqrt{\frac{\rho}{2\rho y+1}}e^{-\rho b/(2\rho y+1)} = \int_0^\infty\frac{1}{\sqrt{\pi}}\sqrt{\frac{\rho}{a}}e^{-\rho a}q_y^1(a,b)da,$$
  where $q^1_y$ is the time-$y$ transition density of \BESQ[1]. Hence, after the cancellation of $\sqrt{\rho/\pi}$,
  $$\int_0^\infty a^{-1/2}e^{-\rho a}\EV[f(\alpha^y_a)]da = \int_0^\infty a^{-1/2}e^{-\rho a}\int_0^\infty q_y^1(a,b)\EV[f(b\scaleI\ol\beta)]dbda.$$
  Since this holds for all $\rho>0$, we conclude by uniqueness of Laplace transforms that
  $$a^{-1/2}\EV[f(\alpha^y_a)] = a^{-1/2}\int_0^\infty q_y^1(a,b)\EV[f(b\scaleI\ol\beta)]db,$$
  first for Lebesgue-a.e.\ $a>0$, then for every $a>0$ by continuity. Again, this gives equality in distribution, since $(\IPspace,\dI)$ is Lusin.
\end{proof}

\begin{proof}[Proof of Theorem \ref{thm:pseudostat}]
 The arguments for type-0 and type-1 are identical. We showed in the proof of Proposition \ref{prop:pseudostat:fixed} that $\alpha^y_a$ has the same distribution as $B(y)\scaleI\ol\beta$ for all $B(0) = a\ge 0$. Now consider any random $B(0)$ independent of $(\alpha^y_1,\,y\ge 0)$.
 \begin{equation*}
 \begin{split}
  \EV\left[f\left(\alpha^y_{B(0)}\right)\right] &= \int_0^\infty\EV[f(\alpha^y_a)]\Pr\{B(0)\in da\}\\
  	&= \int_0^{\infty}\EV[f(B(y)\scaleI\ol\beta)\;|\;B(0)=a]\Pr\{B(0)\in da\} = \EV[f(B(y)\scaleI\ol\beta)].
 \end{split}
 \end{equation*}\vspace{-20pt}
 
\end{proof}

\subsection{Two diffusions on $(\IPspace,\dI)$, stationary with \PDIP $\left(\frac12,0\right)$ and \PDIP $\left(\frac12,\frac12\right)$ laws}

\def\bsa{\boldsymbol{\alpha}}
\def\Absa{\rho_{\bsa}}
\def\dPF{\ol\cF_{\IPspace}}

Throughout this section we write $\IPspace_1 := \{\gamma\in\IPspace\colon\IPmag{\gamma}=1\}$. Recall the de-Poissonization transformation of Theorem \ref{thm:stationary}. In this section, we prove that theorem. We slightly update our earlier notation.

\begin{definition}[De-Poissonization]\label{def:dePois}
 For $\bsa = (\alpha^y,\,y\geq 0) \in \cC([0,\infty),\IPspace)$ with $\alpha^0\neq\emptyset$, we set
 $$\ol{\alpha}^u := \IPmag{\alpha^{\Absa(u)}}^{-1}\scaleI\alpha^{\Absa(u)}, \quad
  \text{where} \quad \Absa(u) = \inf\left\{  y\ge 0\colon \int_0^y \IPmag{\alpha^z}^{-1} dz > u  \right\}, \quad u\ge 0.$$
 We call the map $D$ sending $(\alpha^y,\,y\geq 0) \mapsto (\ol\alpha^u,\,u\geq 0)$ the \emph{de-Poissonization map} and we call $(\ol\alpha^u,\,u\geq 0)$ the de-Poissonized process.
\end{definition}

\begin{proposition}\label{prop:dePois:time_change}
 For $\bsa = (\alpha^y,\,y\geq 0)$ a type-0 or type-1 evolution with initial state $\beta\neq\emptyset$, the time-change $\Absa$ is continuous and strictly increasing, and $\lim_{u\upto\infty}\Absa(u) = \min\{y > 0\colon \alpha^y = \emptyset\}$.
\end{proposition}

This is really an assertion about integrals of inverses of the \BESQ\ total mass processes of Theorem \ref{thm:BESQ_total_mass}, and in that setting it is common knowledge. It can be read, for example, from \cite[p.\ 314-5]{GoinYor03}. 
The a.s.\ path-continuity claimed in Theorem \ref{thm:stationary} follows from Proposition \ref{prop:dePois:time_change} and the path-continuity of the type-1 and type-0 evolutions. It remains to prove the claimed Markov property and stationary distributions.



Take $\bsa = (\alpha^y,\,y\geq 0)\in\cCRI$ with $\alpha^0\neq\emptyset$. By changes of variables we see that
$$D(\bsa) = D(c\scaleI\alpha^{y/c},\,y\ge 0)\qquad\mbox{for all }c>0.$$
Consequently, a type-1 (respectively, type-0) evolution starting from $c\scaleI\beta$ has the same de-Poissonized process as a type-1 (resp.\ type-0) evolution starting from $\beta$. Thus, for laws $\mu$ on $\IPspace\setminus\{\emptyset\}$ we can denote by $\ol\BPr^1_{\ol{\mu}}$ (resp.\ $\ol\BPr^0_{\ol{\mu}}$) the distribution of a de-Poissonized type-1 (resp.\ type-0) evolution starting from the initial distribution $\ol{\mu}$ of $\IPmag{\beta}^{-1}\scaleI\beta$, where $\beta\sim\mu$.

Recall the filtration $(\cFI^y,\,y\ge0)$ of Definition \ref{def:IP_process_space_1}. Since $(\Absa(u),\,u\ge 0)$ is an increasing family of $(\cFI^y)$-stopping times, we can introduce the time-changed filtration $\dPF^u = \cFI^{\Absa(u)}$, $u\ge 0$. Note that, whereas in Section \ref{sec:type-1_gen}, $(\ol\cF^y)$ denotes the completion of the level filtration $(\cF^y)$, here we use the bar to indicate the relationship with the de-Poissonized $(\ol\alpha^u)$. 

\begin{proposition}[Strong Markov property of de-Poissonized evolutions]\label{prop:dePois:Markov}
 Let $\ol\mu$ be a probability distribution on $\IPspace_1$. Let $U$ be an a.s.\ finite $(\cFI^u)$-stopping time. Let $\ol\eta$ and $\ol f$ be non-negative, measurable functions on $\cCRIi$, with $\ol\eta$ being $\dPF^U$-measurable. Then
 \begin{equation*}
  \ol\BPr^1_{\ol\mu}\big[\ol\eta\,\ol f\circ\theta_U\big] = \ol\BPr^1_{\ol\mu}\left[\ol\eta\, \ol\BPr^1_{\ol\alpha^U}\left[\ol f\,\right]\right]\!,
  \qquad \text{and likewise} \qquad \ol\BPr^0_{\ol\mu}\big[\ol\eta\,\ol f\circ\theta_U\big] = \ol\BPr^0_{\ol\mu}\left[\ol\eta\, \ol\BPr^0_{\ol\alpha^U}\left[\ol f\,\right]\right]\!.
 \end{equation*}
\end{proposition}

\begin{proof}
 We begin by proving the type-1 assertion. In fact, we prove a stronger statement. Consider the canonical process $\bsa = (\alpha^y,\,y\ge0)$ under $\BPr^1_{\ol\mu}$, so $D(\bsa)$ is a de-Poissonized type-1 evolution with law $\ol\BPr^1_{\ol\mu}$. We show the strong Markov property of $D(\bsa)$ with respect to $(\dPF^u,u\!\ge\!0)$.
 
 Let $V$ be an a.s.\ finite $(\dPF^u)$-stopping time. Consider $\eta\colon\cCRI\to [0,\infty)$ measurable in $\dPF^{V}$ and set $f := \ol f\circ D$, where $\ol f$ is as in the statement above. 
 Let $Y := \Absa(V)$. Since $\Absa$ is $(\dPF^u)$-adapted, continuous and strictly increasing, \cite[Proposition 7.9]{Kallenberg} yields that $Y$ is an $(\cFI^y)$-stopping time and $\cFI^Y = \dPF^{V}$. Now, let $\theta$ denote the shift operator. For $u\geq 0$,
 $$\ol{\alpha}^{V+u} = \IPmag{\alpha^{\Absa(V+u)}}^{-1}\scaleI\alpha^{\Absa(V+u)}
    = \IPmag{\alpha^{Y+h(u)}}^{-1}\scaleI\alpha^{Y+h(u)},$$
    where $h(u) := \rho_{\theta_Y\bsa}(u)$. 
 Thus, $\theta_{V}\circ D = D\circ\theta_Y$. 
 Then
 \begin{equation*}
  \BPr^1_{\ol\mu}\left[\eta\,\ol{f}\circ\theta_V\circ D\right]
  = \BPr^1_{\ol\mu}\left[\eta\,f\circ\theta_Y\right]
  = \BPr^1_{\ol\mu}\left[\eta\,\BPr^1_{\alpha^Y}[f]\right]
  = \BPr^1_{\ol\mu}\left[\eta\,\BPr^1_{\alpha^Y}\left[\ol{f}\circ D\right]\right]
  = \BPr^1_{\ol\mu}\left[\eta\,\ol\BPr^1_{\ol{\bsa}^V}\left[\ol f\,\right]\right],
 \end{equation*}
 by the strong Markov property of the type-1 evolution, Proposition \ref{prop:strong_Markov}. The same argument works for the de-Poissonized type-0 evolution and the laws $\BPr^0_{\ol\mu}$.
\end{proof}

\begin{proof}[Proof of the Hunt assertion of Theorem \ref{thm:stationary}]
 We check the same four points as in the proof of Theorem \ref{thm:diffusion} at the end of Section \ref{sec:Markov}.
 
 (i) By Theorem \ref{thm:Lusin}, $(\IPspace,\dI)$ is Lusin. Since the mass map $\IPmag{\,\cdot\,}$ is continuous, the set $\IPspace_1$ is a Borel subset of this space, and is thus Lusin as well.
  
 (ii) From Proposition \ref{prop:type-1:cts_in_init_state}, the semi-group for the type-1 (respectively type-0) evolution is continuous in the initial state. Helland \cite[Theorem 2.6]{Helland78} shows that time-change operations of the sort considered here are a continuous maps from Skorokhod space to itself. Thus, the semi-group for the de-Poissonized type-1 (resp.\ type-0) is also continuous.
 
 (iii) Sample paths are continuous, as noted after the statement of Proposition \ref{prop:dePois:time_change}.
 
 (iv) Proposition \ref{prop:dePois:Markov} gives the required strong Markov property.
\end{proof}

To prove stationarity, we progressively strengthen the pseudo-stationarity results of Theorem \ref{thm:pseudostat}. Denote by $(\cF_{\rm mass}^y,y\ge 0)$ the right-continuous filtration on $\cCRI$ generated by $(\IPmag{\alpha^y},\,y\ge 0)$.


\begin{lemma}\label{strongstat}
 Let $\mu$ denote the law of $B\scaleI\ol\beta$, where $B$ is some non-negative random variable independent of $\ol\beta\sim\PDIP[\frac12,0]$. Then for all $y\ge 0$, all $\cF_{\rm mass}^y$-measurable $\eta\colon \cCRI\to [0,\infty)$, and all measurable $h\colon \IPspace_1\to [0,\infty)$, we have
 \begin{equation*}
  \BPr^1_{\mu}\left[\eta\cf\{\alpha^y\neq\emptyset\} h\left(\IPmag{\alpha^y}^{-1}\scaleI\alpha^y\right)\right] = \BPr^1_{\mu}\left[\eta \cf\{\alpha^y\neq\emptyset\}\right]\EV\left[h\left(\ol\beta\right)\right].
 \end{equation*}
 The same assertion holds if we replace superscript `0's with `1's and take $\ol\beta\sim\PDIP[\frac12,\frac12]$.
\end{lemma}

\begin{proof}
 We begin with the type-1 assertion. Let $(\gamma^y,\,y\ge0)$ denote a type-1 evolution with $\gamma^0 = \ol\beta\sim\PDIP[\frac12,0]$, and suppose this is independent of $B$, with both defined on $(\Omega,\cA,\Pr)$. Then $(B\scaleI\gamma^{y/B},\,y\ge0)$ has law $\BPr^1_{\mu}$. By Theorem \ref{thm:pseudostat}, for $f_0$, $f_1\colon [0,\infty)\to [0,\infty)$ measurable,
  \begin{equation*}
  \begin{split}
   &\BPr^1_{\mu}\left[f_0(\IPmag{\alpha^0})f_1\left(\IPmag{\alpha^y}\right)\cf\{\alpha^y\neq\emptyset\}h\left(\IPmag{\alpha^y}^{-1}\scaleI\alpha^y\right)\right]\\
   &\quad = \EV\left[f_0(B)f_1\left(B\IPmag{\gamma^{y/B}}\right)\cf\{\gamma^{y/B}\neq\emptyset\}h\left(\IPmag{\gamma^{y/B}}^{-1}\scaleI\gamma^{y/B}\right)\right]\\
   &\quad = \int_0^{\infty}f_0(m)\EV\left[f_1\left(m\IPmag{\gamma^{y/m}}\right)\cf\{\gamma^{y/m}\neq\emptyset\}h\left(\IPmag{\gamma^{y/m}}^{-1}\scaleI\gamma^{y/m})\right)\right]\Pr\{B\in dm\}\\
   &\quad = \int_0^{\infty}f_0(m)\EV\left[f_1\left(m\IPmag{\gamma^{y/m}}\right)\cf\{\gamma^{y/m}\neq\emptyset\}\right]\EV[h(\ol\beta)]\Pr\{B\in dm\}\\
   &\quad = \BPr^1_{\mu}\left[f_0\left(\IPmag{\alpha^0}\right)f_1\left(\IPmag{\alpha^y}\right)\cf\{\alpha^y\neq\emptyset\}\right]\EV[h(\ol\beta)].
  \end{split}
  \end{equation*} 
  An inductive argument based on the Markov property of the type-1 evolution then says that 
  for $0<y_1<\cdots<y_{n}$ and $f_0,\ldots,f_{n}\colon [0,\infty) \to [0,\infty)$ measurable, 
  \begin{equation*}
  \begin{split}
   &\BPr^1_\mu\Bigg[\prod_{j=0}^{n}f_j\left(\IPmag{\alpha^{y_j}}\right)\cf\{\alpha^{y_{n}}\neq\emptyset\}h\left(\IPmag{\alpha^{y_{n}}}^{-1}\scaleI\alpha^{y_n}\right)\Bigg]\\
    &\quad=\BPr^1_\mu\Bigg[f_0\left(\IPmag{\alpha^0}\right)\BPr^1_{\alpha^{y_1}}\Bigg[\prod_{j=1}^{n}f_{j}\left(\IPmag{\alpha^{y_{j}-y_1}}\right)\cf\{\alpha^{y_{n}-y_1}\neq\emptyset\}h\left(\IPmag{\alpha^{y_{n}-y_1}}^{-1}\scaleI\alpha^{y_{n}-y_1}\right)\Bigg]\Bigg]\\
    &\quad=\BPr^1_\mu\Bigg[f_0\left(\IPmag{\alpha^0}\right)\BPr^1_{\alpha^{y_1}}\Bigg[\prod_{j=1}^{n}f_{j}\left(\IPmag{\alpha^{y_{j}-y_1}}\right)\cf\{\alpha^{y_{n}-y_1}\neq\emptyset\}\Bigg]\EV[h(\ol\beta)]\Bigg]\\
    &\quad= \BPr^1_\mu\Bigg[\prod_{j=0}^{n}f_j\left(\IPmag{\alpha^{y_j}}\right)\cf\{\alpha^{y_{n}}\neq\emptyset\}\Bigg]\EV[h(\ol\beta)].
  \end{split}
  \end{equation*}
  A monotone class theorem completes the proof. The same argument works for type-0.
\end{proof}

To do de-Poissonization, we will replace $y$ by a stopping time in the filtration $(\cF_{\rm mass}^y)$, specifically the time-change stopping times $Y = \Absa(u)$.

\begin{theorem}[Strong pseudo-stationarity]\label{thm:pseudostat_strong}
 Let $\mu$ denote the law of $B\scaleI\ol\beta$, where $B$ is some non-negative random variable independent of $\ol\beta\sim\PDIP[\frac12,0]$. 
 Let $Y$ be an $(\cF_{\rm mass}^y,\,y\ge0)$-stopping time. Then for all $\cF_{\rm mass}^Y$-measurable $\eta\colon \cCRI\to [0,\infty)$ and all measurable $h\colon \IPspace_1\to [0,\infty)$, we have
 \begin{equation*}
  \BPr^1_{\mu}\left[\eta\cf\{\alpha^Y\neq\emptyset\} h\left(\IPmag{\alpha^Y}^{-1}\scaleI\alpha^Y\right)\right] = \BPr^1_{\mu}\left[\eta \cf\{\alpha^Y\neq\emptyset\}\right]\EV\left[h\left(\ol\beta\right)\right].
 \end{equation*}
 The same assertion holds if we replace superscript `0's with `1's and take $\ol\beta\sim\PDIP[\frac12,\frac12]$.
\end{theorem}

\begin{proof}
 We begin with the type-1 assertion. We use the standard dyadic approximation of $Y$ by $Y_n=  2^{-n}\lfloor 2^nY+1\rfloor \wedge 2^n$ which eventually tends to $Y$ from above. Since $Y$ and $Y_n$ are $(\cF_{\rm mass}^y)$-stopping times, the random variable $\eta_k = \eta\cf\{Y_n = k2^{-n}\}$ is $\cF_{\rm mass}^{k2^{-n}}$-measurable for $k\in[2^{2n}-1]$. By Lemma \ref{strongstat},
 \begin{equation*}
 \begin{split}
  &\BPr^1_{\mu}\left[\eta\cf\left\{\alpha^{Y_n}\neq\emptyset;\,Y_n=k2^{-n}\right\} h\left(\IPmag{\alpha^{Y_n}}^{-1}\scaleI\alpha^{Y_n}\right)\right]\\
    &\quad= \BPr^1_{\mu}\left[\eta_k\cf\left\{\alpha^{k2^{-n}}\neq\emptyset\right\} h\left(\IPmag{\alpha^{k2^{-n}}}^{-1}\scaleI\alpha^{k2^{-n}}\right)\right]
    = \BPr^1_{\mu}\left[\eta_k\cf\left\{\alpha^{k2^{-n}}\neq\emptyset\right\}\right]\EV[h(\ol\beta)]\\
    &\quad= \BPr^1_{\mu}\left[\eta \cf\left\{\alpha^{Y_n}\neq\emptyset;\,Y_n=k2^{-n}\right\}\right]\EV[h(\ol\beta)].
 \end{split}
 \end{equation*}
 Summing over $k\in[2^{2n}-1]$ and letting $n\rightarrow\infty$, the continuity of $(\alpha^y)$ and the observation that
 $\bigcup_{k\in [2^{2n}-1]}\left\{\alpha^{Y_n}\neq\emptyset;\,Y_n=k2^{-n}\right\}$ increases to $\left\{\alpha^{Y_n}\neq\emptyset\right\}$ complete the proof for type-1. The type-0 argument is identical.
\end{proof}

\begin{proof}[Proof of the stationarity assertions of Theorem \ref{thm:stationary}]
 We apply Theorem \ref{thm:pseudostat_strong} to $\eta=1$ and the stopping times $Y=\Absa(u)$, which satisfy $\alpha^Y\neq\emptyset$ a.s.. In the notation of Proposition \ref{prop:dePois:Markov},
 $$\ol\BPr^1_{\ol\mu}[h(\ol{\alpha}^u)] = \BPr^1_{\mu}\left[\cf\left\{\alpha^{\Absa(u)}\neq\emptyset\right\} h\left(\IPmag{\alpha^{\Absa(u)}}^{-1}\scaleI\alpha^{\Absa(u)}\right)\right] = \EV[h(\ol\beta)]$$
 for each $u>0$, as required. The same argument applies to type-0.
\end{proof}

\newcommand{\eq}{\begin{equation}}
\newcommand{\en}{\end{equation}}
\newcommand{\HIPspaceone}{\mathcal{I}_{H,1}}

\subsection{Connection with Petrov's Poisson-Dirichlet diffusions}\label{sec:Petrov}
Recall the Kingman simplex:
\begin{equation}
 \infsimp := \left\{ x=\left(x_1, x_2, \ldots   \right)\colon x_1 \ge x_2 \ge \cdots \ge 0,\; \sum_{i=1}^\infty x_i = 1     \right\}.
\end{equation}
This is a metric space under $\ell^{\infty}$. Its closure under $\ell^{\infty}$, denoted by $\infsimpcl$, is the set of non-increasing sequences in $[0,1]$ with sum at most 1. 
Petrov \cite{Petrov09} introduced a two-parameter family of diffusions on $\infsimpcl$ extending a previously existing one-parameter family introduced by Ethier and Kurtz \cite{EthiKurt81}. These processes can be described via their generators. 
Denote by $\ualgebra$ the commutative unital algebra of functions on $\infsimpcl$ generated by $q_k(x)=\sum_{i\ge 1}x_i^{k+1}$, $k\ge 1$, and the constant function $q_0(x):=1$. On this algebra, and for $0\le \alpha < 1$ and $\theta > -\alpha$, Petrov considers the linear operator
\begin{equation}\label{eq:genpetrov}
\petrovgen=\sum_{i\ge 1}x_i\frac{\partial^2}{\partial x_i^2}-\sum_{i,j\ge 1}x_ix_j\frac{\partial^2}{\partial x_i\partial x_j}-\sum_{i\ge 1}(\theta x_i+\alpha)\frac{\partial}{\partial x_i},
\end{equation}
and shows that there is a Markov process on $\infsimpcl$ that is symmetric with respect to the law $\PoiDir[\alpha, \theta]$ whose generator restricted to $\ualgebra$ is given by $\petrovgen$. Moreover, it has been shown in \cite{Eth14} that the process, starting at $x\in\infsimp$, never leaves $\infsimp$. We are going to denote the laws of this two-parameter family of diffusions on $\infsimp$ by $\EKP[\alpha, \theta]$. 

In this section we prove Theorem \ref{thm:petroviden}. The proofs for type-1 and type-0 are similar. We will mostly focus on type-1 and ask the reader to make the appropriate changes for type-0. We start by proving the Markov property of the process $W$. The same method allows us to prove that this is a Hunt process, which also follows from our identification with $\EKP$ diffusion. 

\begin{lemma}\label{lem:ordereddiff}
For type-1 evolutions the process $\left( W(y),\; y \ge 0   \right)$ is a path-continuous, $\infsimp$-valued Markov process that is stationary with respect to the $\PoiDir[\frac{1}{2},0]$ law. For type-0 evolutions a similar statement holds with $\PoiDir[\frac{1}{2},\frac{1}{2}]$ stationary law.
\end{lemma}

\begin{proof} Let $\HIPspaceone$ denote the set of elements in $\beta \in \HIPspace$ such that $\IPmag{\beta}=1$. 
For $\beta \in\HIPspaceone$, let $\textsc{ranked}(\beta)$ denote the decreasing order statistics of Lebesgue measures of intervals in $\beta$.  It is easy to see that $\textsc{ranked}\colon(\HIPspaceone,d_{\IPspace}) \to  \infsimp$ is continuous and consequently $(W(y)=\textsc{ranked}(\ol{\alpha}^y),y\geq 0)$ is continuous.  

For $x\in\infsimp$, consider any pair $\beta, \beta' \in \HIPspaceone$ such that $\textsc{ranked}(\beta)=\textsc{ranked}(\beta')=x$. Then there is a correspondence between intervals of $\beta$ and $\beta'$ given by the coordinates of $x$. For an interval $U\in \beta$ we will denote its corresponding interval in $\beta'$ by $U'$. Hence $\mathrm{Leb}(U)=\mathrm{Leb}(U')$.

Consider a coupled pair of type-1 evolutions $\left(\bsa_\beta, \bsa_{\beta'}  \right)$ starting with $\alpha^0_\beta=\beta$ and $\alpha^0_{\beta'}=\beta'$ constructed by 
\[
\bsa_\beta = \skewerP\left(\Concat_{U\in\beta }\bN_U\right), \qquad \bsa_{\beta'} = \skewerP\left(\Concat_{U'\in\beta' }\bN'_{U'}\right),
\]
where $\bN_U=\bN'_{U'}$, $U\in\beta$, are independent with $\bN_U \sim\mClade^+(\,\cdot\;|\;m^0 = \mathrm{Leb}(U))$.

Let $F\colon \infsimp \to [0,\infty)$ be bounded and measurable. By our coupling, we have $\BPr^1_\beta\left(F(W(y)) \right)= \BPr^1_{\beta'}\left(F(W(y)) \right)$. In particular, the conditional expectation of $F\left( W(y) \right)$ is a function of $W(0)=x\in \infsimp$. By \cite[Lemma I.14.1]{RogersWilliams}, $W$ is a Markov process.

The claims about stationary laws follow from Theorem \ref{thm:stationary}. 
\end{proof}

We will frequently employ the natural inclusion map $\iota\colon \infsimp\rightarrow \HIPspace$:
\[
\iota(x) = \left\{ \left( \sum_{i=1}^j x_i, \sum_{i=1}^{j+1} x_i   \right), \quad j=0,1,2,\ldots   \right\}, 
\]
where for $j=0$ the interval starts at $0$. Thus, when we say consider a type-1 (or type-0) evolution starting with $W(0)=x$, we mean that $\alpha^0$ or $\ol{\alpha}^0$ is given by $\iota(x)$. By an abuse of notation, we will denote the law $\BPr^1_{\iota(x)}$ on $\cC\left( [0,\infty), \HIPspace  \right)$ by $\BPr^1_{x}$.

We will also follow the convention of including finite-dimensional unit simplices in $\infsimp$ by appending zeros.

\begin{proposition}\label{lem:genidentify} Fix $x\in \infsimp$ and consider a type-1 evolution $\bsa$ with $\alpha^0=\iota(x)$. Let 
\[
W(y)=\textsc{ranked}\left(  {\IPmag{\alpha^{\Absa(y)}}}^{-1} \scaleI\alpha^{\Absa(y)}\right), \quad y\ge 0.
\]
For every $q \in \ualgebra$ we have 
\begin{equation}\label{eq:genvisa}
\lim_{y\rightarrow 0+} \frac{\BPr^1_x\left[ q\left( W(y) \right) \right] - q(x) }{y} =2 \petrovgen\left( q\right)(x), \quad \text{for every $x \in \infsimp$},
\end{equation}
where $\petrovgen$ is the generator for $\EKP[\frac{1}{2},0]$. The above convergence also holds in $\ltwo$ with respect to the law of $\PoiDir[\frac{1}{2},0]$. A similar statement holds for type-0 evolutions for the parameters  $\left(\frac{1}{2}, \frac{1}{2}\right)$.
\end{proposition}

Proposition \ref{lem:genidentify} is proved in two steps: first when $q=q_m$ for some $m$, and then the rest. The first case is shown in Lemma \ref{thm:identifygen}, and the latter is shown in Lemma \ref{thm:genalgebra}. 
%
%
%
%
%
%
First we need the following lemma, which is well-known under other conditions. See, for example, \cite{Volkonski58}.

\begin{lemma}\label{lem:gentimechange}
For any bounded measurable $f\colon\HIPspace \rightarrow \BR$, we have 
\begin{equation}\label{eq:changeclockgen}
 \lim_{y\rightarrow 0+} \frac{\BPr^1_{\beta}\left[ f\left( \alpha^{\Absa(y)} \right)\right] - f\left( \beta \right) }{y}
 = \IPmag{\beta}\lim_{y\rightarrow 0+} \frac{\BPr^1_{\beta}\left[ f\left( \alpha^{y} \right) \right] - f\left( \beta \right) }{y}, \qquad \beta\in\HIPspace\setminus\{\emptyset\},
\end{equation}
in the sense that if the limit on one side exists and coincides with a bounded, continuous function $g:\HIPspace\backslash\{\emptyset\}\rightarrow \BR$, then the limits on the other side also exist and coincide with $g$ on $\HIPspace\backslash\{\emptyset\}$.
\end{lemma}

\begin{proof} We will consider the process $\bsa$ to be killed the first time $\IPmag{\alpha^y}=0$. Therefore, the empty set $\emptyset$ will act as the cemetery state for our Markov process. Therefore, in the following proof, all real functions on this state will take value zero. 

Fix a bounded measurable $f$ and suppose that the limit on the right exists and equals $g\left( \beta \right)$ for a bounded, continuous function $g\colon \HIPspace\backslash\{\emptyset\} \rightarrow \BR$. Then, by \cite[Proposition 4.1.7]{EthKurtzBook}, we get that $M(y) = f\left( \alpha^y  \right) - f\left( \alpha^0 \right) - \int_0^y g\left(  \alpha^s \right) ds$
is a martingale. Since $g$ is bounded and $\Absa(y)$ has finite expectation for every $y>0$ (\cite[Section 4]{Yor92}), we can apply the Optional Sampling Theorem at the stopping time $\Absa(y)$ to find
\eq\label{eq:prelimiteqclock}
\frac{\BPr^1_{\beta} \left[f\left(  \alpha^{\Absa(y)} \right)\right] - f\left( \alpha^0 \right)}{y}= \BPr^1_{\alpha^0}\left[ \frac{1}{y}\int_0^{\Absa(y)} g\left( \alpha^s  \right) ds\right].  
\en
Since $g$ is bounded, continuous and $(\alpha^y,\,y\ge0)$ is continuous at $y=0$ it follows that
\[
\lim_{y\rightarrow 0+} \frac{1}{y} \int_0^{\Absa(y)} g\left( \alpha^s \right)ds = g\left(\alpha^0 \right) \Absa'\left( 0+ \right)= g\left( \alpha^0 \right)\IPmag{\alpha^0}\qquad \text{a.s..}
\]
An application of the Dominated Convergence Theorem to \eqref{eq:prelimiteqclock} settles the claim. 

The reverse implication is slightly more complex. Let $C(t):=\int_0^t \IPmag{\alpha^z}^{-1}dz$, $t\geq0$. This is the inverse process of $\Absa$, thus $C(t)$ is a $(\dPF^u)$-stopping time for each $t$. Now, consider the time-changed strong Markov process $\beta^y = \alpha^{\Absa(y)}$, $y \ge 0$, and repeat the previous argument reversing the roles of $\Absa$ and $C$. Note that $\beta^{C(y)}=\alpha^y$, for all $y\ge 0$ until $\bsa$ is killed. However, we cannot directly apply the Optional Sampling Theorem since $C(y)$ is not integrable for any $y$. To fix this, we apply the Optional Sampling Theorem to the bounded stopping time $C(y)\wedge 1$ to obtain
\eq\label{eq:stoppedequalitygen}
\frac{\BPr^1_{\alpha^0}\left[ f\left( \beta^{C(y)\wedge 1}\right) \right] - f(\alpha^0)}{y} = \BPr^1_{\alpha^0}\left[  \frac{1}{y}\int_0^{C(y)\wedge 1} g\left(  \alpha^{\Absa(s)} \right)ds   \right].
\en  
As $y\rightarrow 0+$, the right side converges to $g\left( \alpha^0 \right)\IPmag{\alpha^0}$ as before. For the left side, let $K$ denote a bound on $\abs{f}$. Then
\eq\label{eq:estimatetailhit}
\abs{\frac{\BPr^1_{\alpha^0} \left[ f\left( \beta^{C(y)\wedge 1}  \right) \right] - \BPr^1_{\alpha^0}\left[ f\left( \beta^{C(y)} \right)\right]}{y} }\le 2K \frac{\BPr^1_{\alpha^0}\left( C(y) > 1 \right)}{y}.
\en
We will now show that the limit of the right side above is zero as $y\rightarrow 0+$. We complete the argument in the case $\IPmag{\alpha^0} = 1$, from which the general case follows immediately.

By Theorem \ref{thm:BESQ_total_mass}, the total mass process $\IPmag{\alpha}$ is distributed as a $\BESQ[0]$ process starting from $1$. Let $X$ be a $\BESQ[0]$ process starting from $1$ on some probability space. Then
\[
\begin{split}
\BPr^1_{\alpha^0}(C(y) > 1) = \Pr\left( \int_0^y \frac{ds}{X(s)} > 1 \right) \le \Pr\left( \inf_{0\le s\le y} X(s) < y   \right) 
\le \Pr\left( \sup_{0\le s\le y} \abs{X(s)-1} >  1-y   \right).
\end{split}
\]
Let $\sigma_y$ denote the exit time of $X$ from the interval $[y, 2-y]$. The quadratic variation (\cite[Definition 3]{GoinYor03}) of the stopped martingale $X\left( \sigma_y \wedge t  \right)$ satisfies the inequality 
\[
\left\langle X \right\rangle\left( \sigma_y \wedge t  \right) \le 4 \int_0^{\sigma_y \wedge t} X(s)ds \le4(2-y) t.
\]
Therefore, by the martingale concentration inequality, we get 
\[
\Pr\left( \sup_{0\le s\le y} \abs{X(s)-1} \ge  1-y   \right)= \Pr\left( \sup_{0\le s\le y} \abs{X\left(s \wedge \sigma_y \right)-1} \ge  1-y   \right) \le 2 \exp\left[  - \frac{(1-y)^2}{8(2-y)y}   \right]. 
\] 
In particular, $\lim_{y\rightarrow 0+} \BPr^1_{\alpha^0}\left(  C(y)>1 \right)/y =0$. Plugging in this estimate in \eqref{eq:estimatetailhit} gives us that $\lim_{y\rightarrow 0+}$ of both sides is zero showing finally that 
\[
\lim_{y\rightarrow 0+} \frac{\BPr^1_{\alpha^0}\left[ f\left( \alpha^y \right) \right] - f(\alpha^0)}{y} =\lim_{y\rightarrow 0+} \frac{\BPr^1_{\alpha^0}\left[ f\left( \beta^{C(y)}\right) \right] - f(\alpha^0)}{y} = \lim_{y\rightarrow 0+} \frac{\BPr^1_{\alpha^0}\left[ f\left( \beta^{C(y)\wedge 1}\right) \right] - f(\alpha^0)}{y}.
\] 
The last expression is given by $g\left( \alpha^0\right)$ from \eqref{eq:stoppedequalitygen}. This completes the proof of the lemma. 
\end{proof}

We will also need a generalization of the previous lemma to a countable collection of independent type-1 evolution. Consider the space: 
\[
\mathcal{H}:=\left\{ \left( \alpha_i,\; i\ge 1  \right)\in \HIPspace^{\BN},\; \sum_{i\ge 1} \IPmag{\alpha_i} < \infty  \right\}
\]  
Suppose we have a countable independent collection of type-1 evolutions $\left(  \bsa_i,\; i \in \BN \right)$ such that $\left( \alpha^0_i,\; i\ge 1 \right)\in \mathcal{H}$. By Theorem \ref{thm:BESQ_total_mass} and the additivity of $\BESQ[0]$ processes, with probability one, $\left( \alpha^y_i,\; i\ge 1 \right)\in \mathcal{H}$ for all $y > 0$. Then, $\left( \bsa_i,\; i \in \BN  \right)$ is a strong Markov process in $\mathcal{H}$. Let  $\bsa^y=\ConcatIL_{i\ge1} \bsa^y_i$ denote the resulting type-1 evolution obtained by concatenating the individual ones. Let $\Absa$ refer to the de-Poissonization time-change for $\bsa$.

\begin{lemma}\label{lem:gentimechangeinf}
Suppose $F\colon\mathcal{H} \rightarrow \BR$ is a bounded function. Then \eqref{eq:changeclockgen} holds for the function $F$ applied to $\left(  \bsa_i,\; i \in \BN \right)$ by replacing the measure $\BPr^1_\beta$ by the product measure $\bigotimes_{i\ge 1} \BPr^1_{\bsa_i^0}$.
\end{lemma}

\begin{proof}
The proof remains exactly the same after replacing the strong Markov process $\bsa$ in \eqref{eq:changeclockgen} by the strong Markov process $\left(  \bsa_i,\; i \in \BN \right)$.
\end{proof}

To prove Proposition \ref{lem:genidentify} we will employ the transition description for the type-1 evolution, stated in Proposition \ref{prop:type-1:gen_transn}. On $\HIPspace$, consider the function
\[
f_m(\beta)= \sum_{U\in\beta}\left({\rm Leb}(U)/ \IPmag{\beta}\right)^{m+1},\quad m\ge 1, \quad \beta \in \HIPspace
\backslash\{\emptyset\},\quad f_m(\emptyset)=0.
\]
Then $f_m$ is a bounded function on $\HIPspace$ that is continuous on $\HIPspace\backslash\{\emptyset\}$.

Consider a type-1 evolution $\bsa=\left( \alpha^y,\; y \ge 0 \right)$ starting from $\iota(x)$. Then, clearly $\IPmag{\alpha^0}=1$. Then,  it follows by symmetry that $q_m\left( W(y)  \right)= f_m\left( \alpha^{\Absa(y)}  \right)$. There is a natural one-to-one correspondence between elements of the algebra formed by $\left( q_m,\; m\ge 0  \right)$ and $\left(  f_m, \; m \ge 0\right)$, with $f_0 \equiv 1$. If $(q,f)$ are corresponding functions from these algebras, we still have $q\left( W(y)  \right)= f\left( \alpha^{\Absa(y)}  \right)$. In particular, for that pair of functions $(q,f)$ and initial conditions $(x, \alpha^0)$ we have 
\eq\label{eq:changingvariables}
\lim_{y\rightarrow 0+} \frac{\BPr^1_x\left[ q\left( W(y) \right)\right] - q(x) }{y} = \lim_{y\rightarrow 0+} \frac{\BPr_x^1\left[ f\left( \alpha^{\Absa(y)} \right)\right] - f\left( \alpha^0 \right) }{y}. 
\en
Therefore, in order to show \eqref{eq:genvisa} we will perform our calculations on the algebra generated by $\left(  f_m, \; m \ge 0\right)$. Since the case of $m=0$ is trivial, we will restrict ourselves to $\left( f_m,\; m \ge 1  \right)$.

\begin{lemma}\label{thm:identifygen} Fix $x\in \infsimp$ and consider a type-1 evolution $\bsa$ starting with $W(0)=x$. 
\begin{enumerate}[label=(\roman*), ref=(\roman*)]
\item For $f=f_m$, the limit on the right of \eqref{eq:changeclockgen} is given by the bounded continuous function
\begin{equation}\label{eq:genfm}
\gen\left(f_m \right)(\alpha^0):=2(m+1)m q_{m-1}(x) - 2(m+1)m q_m(x) - (m+1) q_{m-1}(x).
\end{equation}
As a consequence, for any $x\in \infsimp$ and for any $m \ge 1$, 
\[
\lim_{y\rightarrow 0+} \frac{\BPr^1_{x}\left[ q_m\left( W(y)  \right)\right] - q_m(x)}{y}= 2 \petrovgen\left( q_m\right)(x),
\]
where $\petrovgen$ is the generator of the process $\EKP[\frac{1}{2},0]$.
\item The above convergence also holds in $\ltwo$ with respect to the law of $\PoiDir[\frac{1}{2},0]$.
\end{enumerate}
\end{lemma}

\begin{proof} Fix $x=\left( x_i,\; i\ge 1   \right)\in \infsimp$ and consider a type-1 evolution $\bsa$ with $W(0)=x$. The coordinate labeling of $x$ allows us to consider the countably many intervals in $\alpha^0$ as a sequence $\left( U_i,\; i \ge 1  \right)$ such that $x_i= \mathrm{Leb}\left(  U_i \right)$, $i\ge 1$. Hence, $\IPmag{\alpha^0}=\sum_{i \ge 1} x_i=1$.
This labeling allows us to identify the type-1 evolution of the $i$th interval in $\alpha^0$ for each $i\ge 1$.
Fix $y > 0$ and let us introduce the following notation following the description in Proposition \ref{prop:type-1:transn}. 
\begin{enumerate}[label=(\roman*), ref=(\roman*)]
\item Let $Z_i^y$ denote the total mass of the clade that starts with the $i$th interval in $\alpha^0$. In particular, $\IPmag{\alpha^y} = \sum_{i \ge 1} Z_i^y$. 
\item If $Z_i^y > 0$, the $i$th clade has two parts. The mass of the leftmost spindle will be denoted by $L_i^y$. The rest of the mass $Z^y_i- L^y_i$ will be denoted by $R^y_i$. Otherwise, we will define both $L^y_i$ and $R^y_i$ to be zero. 
\item When $R^y_i>0$, it is the sum of the jumps of an inverse Gaussian subordinator stopped at an independent exponential time. Let us arrange the resulting jump sizes of this subordinator in a decreasing order by $\big(  x^y_{ij}, \; j \ge 1  \big)$. Also, let 
\[
w_{ij}^y := \frac{x^y_{ij}}{R^{y}_i}, \; j \ge 1. \quad \text{Thus,}\quad w^y_{i\cdot}:=\big(w^y_{ij},\; j\ge 1\big) \in \infsimp, \quad \text{for each $i\ge 1$}.
\]
\item When $R_i^y=0$, by convention, we define $w^y_{ij}$ to be sampled from a $\PoiDir[\frac12,\frac12]$ distribution, independently for all $(i,j)$ and independently of all $\left( L_i^y, R_i^y,\; i \ge 1  \right)$.
\end{enumerate}

\begin{lemma}\label{lem:pdhalf}
 For every $y >0$, the family of sequences $\big(  w^y_{ij},\, j \ge 1 \big)$, $i \ge 1$, is i.i.d., with each $\big(  w^y_{ij},\, j \ge 1 \big)\sim\PoiDir[\frac12,\frac12]$. Moreover, this is jointly independent of $\big(L_i^y,\; R_i^y,\; i \ge 1 \big)$ and $\beta^0$. 
\end{lemma}

\begin{proof} This follows from \cite[Proposition 21]{PitmYorPDAT} for $\alpha=\theta=\frac{1}{2}$. The factor $1/2y$ in the jump measure of the inverse Gaussian subordinator is absorbed by scaling. That proposition shows that the sequence $w^y_{i \cdot}$ is independent of $R^y_i$ and distributed as $\PoiDir[\frac12,\frac12]$. By Proposition \ref{prop:type-1:transn}, we know that $L^y_i$ and $x^y_{i\cdot}$ are independent and they are independent across $i \ge 1$. 
\end{proof}

Let $Q^\delta_x$ denote the law of the squared Bessel process of dimension $\delta \in \mathbb{R}$ and starting from $x \ge 0$. We are also going to need the family of Wright-Fisher (\WF) diffusions on the unit simplex. Fix $k \ge 2$ and fix real parameters $\delta_1, \delta_2, \ldots, \delta_k$ and let $\delta_0=\sum_{i\in [k]} \delta_i$. Let $I^+$ and $I^-$ refer to the respective subsets of $[k]$ such that $\delta_i \ge 0$ for all $i \in I^+$ and $\delta_i < 0$ for all $i \in I^-$. Consider a diffusion on the following subset of the closed unit simplex in $\BR^k$:
\[
\left\{  \left( x_1, x_2, \ldots, x_k   \right):\; x_i \ge 0,\; i \in I^+,\; \text{and} \; x_i > 0,\; i \in I^-, \quad \sum_{i=1}^k x_i=1      \right\}.
\]
Consider a diffusion on the above state space with the generator:
\eq\label{eq:genwfgen}
\frac{1}{2} \sum_{i,j =1}^n x_i\left( \cf\{i=j\} - x_j\right) \frac{\partial^2 f}{\partial x_i \partial x_j} - \frac{1}{2} \sum_{i=1}^n \left( \delta_i - \delta_0 x_i  \right)\frac{\partial f}{\partial x_i}.
\en
Since the coefficients are smooth and bounded, starting from the open unit simplex, the process is well-defined and unique in law until the first time any coordinate $X_i$, $i \in I^-$, hits zero, at which time we kill the process. We will refer to this law as the \WF\ process with parameters $\left( \delta_1, \ldots, \delta_k  \right)$, or just \WF[\delta_1, \ldots, \delta_k]. This is a generalization of the usual \WF\ models with all nonnegative parameters and the extension in \cite{Pal13} to all negative parameters.

Fix an $m \ge 1$. By our definition of $f_m$ the following is true:
\eq\label{eq:definefn}
f_m\left( \alpha^y \right) =  \frac{1}{\IPmag{\alpha^y}^{m+1}}\left( \sum_{i \ge 1} \left( L^y_i  \right)^{m+1} + \sum_{i \ge 1} \left(R_i^y\right)^{m+1} \sum_{j \ge 1} \left( w_{ij}^y \right)^{m+1}\right)\cf\{\IPmag{\alpha^y} >0\}.
\en
Let $C^m_{1/2}$ denote a constant that represents the expectation of $\sum_{i \ge 1} w_i^{m+1}$, where $\left(  w_i,\; i \ge 1 \right)$ is sampled from $\PoiDir[\frac12,\frac12]$. Then, by Lemma \ref{lem:pdhalf}, we get that for all $y>0$ we have
\begin{equation}\label{eq:fmLR}
\BPr^1_x\left[ f_m\left( \alpha^y \right)  \right]= \sum_{i \ge 1} \BPr^1_x\left[ \left(  \frac{L_i^y}{\IPmag{\alpha^y}}  \right)^{m+1}\cf\{\IPmag{\alpha^y} >0\}\right] + C^m_{1/2} \sum_{i\ge 1} \BPr^1_x\left[ \left(  \frac{R_i^y}{\IPmag{\alpha^y}}  \right)^{m+1}\cf\{\IPmag{\alpha^y} >0\}\right]. 
\end{equation}  
Now, we need to compute the following limit:
\eq\label{eq:thirdtimechange}
\lim_{y\rightarrow 0+} \frac{1}{y}\left( \BPr^1_x\left[ f_m\left( \alpha^y \right)  \right]  - q_m(x) \right).
\en

Consider the countable collection of type-1 evolutions $\left( \bsa_i,\; i\ge 1  \right)$, where $\bsa_i$ is the evolution starting with the interval partition $\{ (0, x_i) \}$. Then, the concatenation $\left(\ConcatIL_{i\ge 1} \bsa_i^y,\; y\ge 0 \right)$ has law $\BPr^1_{x}$. For the rest of the proof we will work on a probability space $\left( \Omega, \mathcal{F}, \PV  \right)$ rich enough to support this countable collection. We will denote expectation with respect to $\PV$ by $\EV$. The dependence on $x$ will be suppressed from this notation.  

Consider the functions $L,R$ defined in Corollary \ref{cor:remaining_BESQ}. Consider the set-up of Lemma \ref{lem:gentimechangeinf}. For $\left(\beta_i,\; i\ge 1\right) \in \mathcal{H}$, consider the pairs $\left(L(\beta_i), R(\beta_i) \right)$. Define $F\colon \mathcal{H} \rightarrow [0,1]$ by 
\[
F\left(\left( \beta_i,\;i\ge 1   \right)\right)=\sum_{i\ge 1}\left(\frac{L(\beta_i)}{\sum_{i\ge 1} \IPmag{\beta_i}}\right)^{m+1} + C^m_{1/2}\sum_{i\ge 1}\left(\frac{R(\beta_i)}{\sum_{i\ge 1} \IPmag{\beta_i}}\right)^{m+1}, \; \left( \beta_i,\; i\ge 1  \right) \in \mathcal{H}\backslash\{ \emptyset\}.
\]
We take $F(\emptyset)=0$.
Then $F$ is a bounded function on $\mathcal{H}$. Then, by an application of Lemma \ref{lem:gentimechangeinf}, the limit in \eqref{eq:thirdtimechange} is equal to the following limit:
\[
\lim_{y\rightarrow 0+} \frac{1}{y}\left(  F\left( \bsa_i^{\Absa(y)}  \right) \cf\left\{ \IPmag{\alpha^{\Absa(y)}} > 0 \right\}- q_m(x)    \right),
\]
once the assumptions are satisfied.

However, note that with probability one $\IPmag{\alpha^{\Absa(y)}} > 0$ for all $y >0$. Hence, we can drop the indicator above and prove the following claim. 
\eq\label{eq:fourthtimechange}
\lim_{y\rightarrow 0+} \frac{1}{y}\left(  \sum_{i \ge 1} \EV\left[ \left(  \frac{L_i^{\Absa(y)}}{\IPmag{\alpha^{\Absa(y)}}}  \right)^{m+1} - x_i^{m+1} \right] + C^m_{1/2} \sum_{i\ge 1} \EV \left[ \left(  \frac{R_i^{\Absa(y)}}{\IPmag{\alpha^{\Absa(y)}}}  \right)^{m+1}\right]     \right)= 2\petrovgen\left( q_m \right)(x),
\en
where $\left(L_i^y=L\left( \alpha_i^y \right),\; y\ge 0 \right)$ and $\left(R^y_i=R\left( \alpha^y_i \right),\;y\ge 0\right)$, for $i\ge 1$. By our correspondence between symmetric functions on $\infsimp$ and those on $\HIPspace$, the limit $2\petrovgen\left( q_m \right)(x)$ can be extended to a bounded continuous function on $\HIPspace\backslash\{\emptyset\}$. This justifies the sufficiency of proving \eqref{eq:fourthtimechange}.

Now, either all $x_i>0$ or there exists a first $k$ such that $x_{k+i}=0$ for all $i \ge 1$. We start by assuming the latter. Then, there are initially only $k$ intervals in $\alpha^0$. The joint law of $\left( L^y_i, R^y_i,\; i\in [k] \right)$ is known from Corollary \ref{cor:remaining_BESQ}. Let $\varsigma_k$ be the first time $y$ such that $\min_{i\in [k]} L^{y-}_i=0$. Then, the process $\left(  L_i^y, R_i^y,\; i \in [k] \right)$, $0\le y < \varsigma_k$, has the same joint distribution as a vector of independent $\BESQ$ processes killed at $\varsigma_k$. Each $L_i$ is distributed as $Q^{-1}_{x_i}$ and each $R_i$ as $Q^1_0$.

We now claim that the time-changed vector of ratios
\[
 \left(\left( \frac{L^{\Absa(y)}_i}{\IPmag{\alpha^{\Absa(y)}}}, \frac{R^{\Absa(y)}_i}{\IPmag{\alpha^{\Absa(y)}}},\; i \in [k] \right), \; y \ge 0\right)
\]
is distributed as a \WF\ diffusion with parameters $-1/2$ for each $L_i$ and $1/2$ for each $R_i$, running at $4$ times the speed. This claim follows from an argument very similar to \cite[Theorem 4]{Pal13}. Although this cited result only has nonpositive parameters for \WF, the argument is valid in the case where we have some positive parameters. The factor of $4$ comes from the difference between $\Absa$ and the time-change used in \cite{Pal13}.

This \WF\ diffusion is a diffusion on the unit simplex in $\BR^{2k}$. Let us rearrange the coordinates of $\left( v_1, \ldots, v_{2k}  \right)\in \BR^{2k}$ such that the ones corresponding to $\left(L_1, \ldots, L_k\right)$ are given by $(v_1, \ldots, v_k)$ and those corresponding to $\left(R_1, \ldots, R_k \right)$ by $(v_{k+1}, \ldots, v_{2k})$. Then the generator of the diffusion that is the \WF\ diffusion, running $4$ times faster, can be identified from \cite[equation (13)]{Pal13} as an example of \eqref{eq:genwfgen} with $\delta_i=-1/2$ for $i\in [k]$, $\delta_i=1/2$ for $i=k+j$, $j \in [k]$, and $\delta_0=0$:
\begin{equation}\label{eq:gennwf}
\mathcal{L}= 2 \sum_{i,j=1}^{2k} v_i\left( \cf\{i=j\} - v_j \right) \frac{\partial^2}{\partial v_i \partial v_j} - \sum_{i=1}^k \frac{\partial}{\partial v_i} + \sum_{i=k+1}^{2k}  \frac{\partial}{\partial v_i}. 
\end{equation}

However, since $L_i^0=x_i$ and $R_i^0=0$, for each $i\in [k]$, the limit on the left of \eqref{eq:fourthtimechange} is given by the application of the generator $\mathcal{L}$ in \eqref{eq:gennwf} to the smooth function $g_m(v)= \sum_{i=1}^k v_i^{m+1} + C^m_{1/2} \sum_{i=k+1}^{2k} v_i^{m+1}$, at any point $v$ such that for all $i\in [k]$ we have $v_i=x_i$ and $v_{k+i}= 0$. 
However, at any such point we get 
\eq\label{eq:genwfonpoly}
\mathcal{L}g_m(v) = 2(m+1)m \sum_{i=1}^k x_i(1-x_i) x_i^{m-1} - (m+1) \sum_{i=1}^k x_i^{m}. 
\en
We get the right side of \eqref{eq:fourthtimechange} by a straightforward simplification.

\medskip

Next we will drop the assumption of finitely many initial intervals. For every $k \ge 1$, consider the infinite sequence $\left( L_i^y, R_i^y,\; i \ge 1   \right)$ and the finite sequence comprised of the first $k$ elements $\left( L_i^y, R_i^y,\; i \in [k]   \right)$. Recall that these correspond to the largest $k$ intervals in $\alpha^0$. 
%
%
%
Set
\begin{equation}
\Gamma^y_k:=\sum_{i=1}^k   \EV\left[ \left(\frac{L_i^{\Absa(y)}}{\IPmag{\alpha^{\Absa(y)}}} \right)^{m+1} - x_i^{m+1}\right] + C^m_{1/2} \sum_{i=1}^k \EV\left[ \left( \frac{R_i^{\Absa(y)}}{\IPmag{\alpha^{\Absa(y)}}}  \right)^{m+1}\right].
\end{equation}
Let $\mathbf{x}^k \in \infsimpcl$ be given by $\mathbf{x}^k_i=x_i$ for $i \in [k]$ and $\mathbf{x}^k_{k+i}=0$ for all $i \ge 1$. 
Finally, define
\[
 \Delta^{y}_k := \sum_{i=k+1}^\infty \EV\left[ \left(\frac{L_i^{\Absa(y)}}{\IPmag{\alpha^{\Absa(y)}}} \right)^{m+1} - x_i^{m+1}\right] + C^m_{1/2} \sum_{i=k+1}^\infty \EV\left[ \left( \frac{R_i^{\Absa(y)}}{\IPmag{\alpha^{\Absa(y)}}}  \right)^{m+1}\right].
\]

\begin{lemma}\label{lem:bnddeltak}
We have $\lim_{y\rightarrow 0+} \Gamma^y_k/y = 2\petrovgen\left( q_m \right)\left( \mathbf{x}^k \right)$.
Moreover,
\[
\sup_{y > 0} \left( y^{-1}\Delta^y_k\right) \le 2m(m+1)\sum_{i=k+1}^\infty x_i, \qquad \inf_{y > 0} \left(y^{-1}\Delta^{y}_k \right)  \ge - (m+1) \sum_{i=k+1}^\infty x_i.
\]
\end{lemma}

Before proving the lemma, let us show how Lemma \ref{thm:identifygen} (i) follows. Since the sequence $(x_1, x_2, \ldots)$ is summable, given any $\epsilon >0$, we can find a $k\ge 1$ such that $\sup_{y \ge 0} \abs{\Delta_k^y/y} \le \epsilon$. Letting $k \rightarrow \infty$ and $\epsilon \rightarrow 0$ and appealing to the continuity of $\petrovgen (q_m)(\cdot)$ shows \eqref{eq:fourthtimechange}. 

\begin{proof}[Proof of Lemma \ref{lem:bnddeltak}] The first claim follows in the same way as in the case of finitely many initial intervals by considering the vector of processes
\[
\left( L_i^y, R_i^y,\; i \in [k], \IPmag{\alpha^y} - \sum_{i=1}^k (L_i^y + R_i^y)   \right),
\]
which, until the first time the left limit of any $L_i$ hits zero, is an independent vector of squared Bessel processes killed when the corresponding coordinate hits zero. In particular, the last coordinate is distributed as $Q^0_{1-\sum_{i=1}^k x_i}$. We move on to the second claim.

For every $y$ and $i > k$ consider the function 
\[
h_y(i)=\EV \left[ \left( \frac{L_i^{\Absa(y)}}{\IPmag{\alpha^{\Absa(y)}}}  \right)^{m+1} - x_i^{m+1}\right] + C^m_{1/2} \EV\left[ \left( \frac{R^{\Absa(y)}_i} {\IPmag{\alpha^{\Absa(y)}}}  \right)^{m+1}\right].
\]
Then, $y^{-1}\Delta_k^{y}  = \sum_{i > k} y^{-1} h_y(i)$.

Recall $Z_i=L_i+R_i$. Since the constant $C^m_{1/2}$ must be less than one, it follows that 
\[
h_y(i) \le   \EV\left[ \left( \frac{L_i^{\Absa(y)} + R_i^{\Absa(y)}}{\IPmag{\alpha^{\Absa(y)}}}  \right)^{m+1} - x_i^{m+1}\right] = \EV\left[ \left( \frac{Z_i^{\Absa(y)}}{\IPmag{\alpha^{\Absa(y)}}}  \right)^{m+1} - x_i^{m+1}\right]. 
\]

We will now estimate the above expectation for every fixed $i \ge k+1$. By Theorem \ref{thm:BESQ_total_mass}, for every $i \ge k+1$, the joint distribution of the vector $\left( Z^y_i, T^y_i:=\IPmag{\alpha^y} - Z_i^y  \right)$ is that of an independent pair of diffusions where the first coordinate is distributed as $Q^{0}_{x_i}$ and the second is $Q^{0}_{1-x_i}$. By \cite[Proposition 8]{WarrYor98}, the de-Poissonized ratio $\xi_i^y:=Z^{\Absa(y)}_i\left/\IPmag{\alpha^{\Absa(y)}}\right.$, $y \ge 0$, is a Jacobi diffusion on $[0,1]$ with the generator $2u(1-u)\frac{d^2}{du^2}$ for $u\in (0,1)$ and starting from $x_i$. In particular, it is a martingale and $\EV\left( \xi_i^s \right)=x_i$ for all $s\ge 0$. Note that this is also the first coordinate of the \WF\ diffusion with parameters $(0,0)$, running at $4$ times the usual speed.

By It\^o's rule applied to the bounded function $u\mapsto u^{m+1}$ on $[0,1]$, we get 
\[
\begin{split}
y^{-1}h_y(i)=\frac{\EV\left[\left( \xi_i^y \right)^{m+1}\right] - x_i^{m+1}}{y} = \frac{2m(m+1)}{y}  \int_0^y \EV\left[ \left(\xi_i^s\right)^{m}\left(1-\xi_i^s \right) \right] ds. 
\end{split}
\]
In particular, $y^{-1} h_y(i) \le 2m(m+1)y^{-1}\int_0^y \EV\left[ \xi_i^s  \right] ds = 2m(m+1)x_i$.
Adding up these bounds over $i \ge k+1$ gives us the upper bound of the lemma.

A similar argument holds for the lower bound. Start by noting that 
\[
h_y(i) \ge \EV\left[\left( \frac{L_i^{\Absa(y)}}{\IPmag{\alpha^{\Absa(y)}}}  \right)^{m+1} - x_i^{m+1}\right]. 
\]
Define the stopping time $\chi_i=\inf\{y\ge 0\colon L_i^{y-}=0 \}$ to be the first time such that the mass of the initial spindle hits zero. Then 
\eq\label{eq:ilbnd}
h_y(i) \ge \EV\left[ \left(\frac{L_i^{\Absa(y)}}{\IPmag{\alpha^{\Absa(y)}}}\right)^{m+1} \cf\{ \chi_i > \Absa(y) \} - x^{m+1}_i \right].
\en

By de-Poissonization, the ratio $\left( \eta_i^y:=L_i^{\Absa(y)}/\IPmag{\alpha^{\Absa(y)}}   \right)$, $0 < y < \chi_i$, is a Jacobi diffusion with parameters $(-1/2, 1/2)$, i.e.\  a diffusion on $(0,1]$ with generator $2u(1-u) \frac{d^2}{du^2} - \frac{d}{du}$ and starting from $x_i>0$. Let $\tilde{\chi}_i=\Absa\left( \chi_i \right)$ be the hitting time of zero by $\eta_i$. Then, the stopped process $\eta^{\tilde{\chi}_i\wedge \cdot}_i$ it is a supermartingale. Moreover, 
\[
h_y(i) \ge \EV\left[ \left( \eta_i^y \right)^{m+1} \cf\{ \tilde{\chi}_i > y \} - x_i^{m+1} \right] = \EV\left[ \left(\eta_i^{y\wedge \tilde{\chi}_i}\right)^{m+1} - x_i^{m+1}   \right].
\]
By It\^o's rule and the Optional Sampling Theorem, we get
\[
\begin{split}
&\frac{\EV  \left[\left(\eta_i^{y\wedge \tilde{\chi}_i}\right)^{m+1}  - x_i^{m+1}\right]}{y} = \frac{1}{y}\EV\left[  \int_0^{y\wedge \tilde{\chi}_i} \left(  2m(m+1) \left(\eta_i^s\right)^{m}\left( 1- \eta_i^s  \right) - (m+1) \left(\eta_i^s\right)^{m} \right) ds\right]\\
&\ge - \frac{m+1}{y} \int_0^y \EV\left[\left( \eta_i^{s\wedge \tilde{\chi}_i}  \right)^{m}\right] ds \ge - \frac{m+1}{y} \int_0^y \EV\left[ \eta_i^{s\wedge \tilde{\chi}_i}  \right] ds \ge - (m+1)x_i.
\end{split}
\]
Summing up over $i > k$, gives us the lower bound in the statement of the lemma. 
\end{proof}

We now move on to prove Lemma \ref{thm:identifygen} (ii). We first argue that the limit in \eqref{eq:fourthtimechange} holds in $\ltwo:=\ltwo\left[\frac 12,0\right]$. Consider Lemma \ref{lem:bnddeltak} for $k=0$. Since $\sum_{i=1}^\infty x_i=1$, we get that for all $y >0$,
\[
\abs{\frac{1}{y}\left[  \sum_{i \ge 1} \EV\left[ \left(  \frac{L_i^{\Absa(y)}}{\IPmag{\alpha^{\Absa(y)}}}  \right)^{m+1} - x_i^{m+1} \right] + C^m_{1/2} \sum_{i\ge 1} \EV\left[ \left(  \frac{R_i^{\Absa(y)}}{\IPmag{\alpha^{\Absa(y)}}}  \right)^{m+1}  \right]   \right] } \le 2m(m+1), 
\]
Since $2\petrovgen\left( q_m \right)$ is also a bounded function on $\infsimp$, the difference 
\[
\frac{1}{y}\left[  \sum_{i \ge 1} \EV\left[ \left(  \frac{L_i^{\Absa(y)}}{\IPmag{\alpha^{\Absa(y)}}}  \right)^{m+1} - x_i^{m+1} \right] + C^m_{1/2} \sum_{i\ge 1} \EV\left[ \left(  \frac{R_i^{\Absa(y)}}{\IPmag{\alpha^{\Absa(y)}}}  \right)^{m+1}  \right]    \right]  - 2\petrovgen\left( q_m \right)(x),
\]
denoted by $y^{-1}\overline{\Delta}_m^y(x)$, is a family of bounded functions that converge to zero pointwise as $y \rightarrow 0+$. Therefore, by the Dominated Convergence Theorem, the convergence also holds in $\ltwo$. 

Now, consider the difference 
\[
\difftwo_m^y(x) :=\EV\left[ q_m\left( W(y) \right)  \right] - \left[  \sum_{i \ge 1} \EV\!\left[ \left(  \frac{L_i^{\Absa(y)}}{\IPmag{\alpha^{\Absa(y)}}}  \right)^{\!m+1}\!\! - x_i^{m+1} \right]\! + C^m_{1/2} \sum_{i\ge 1} \EV\!\left[ \left(  \frac{R_i^{\Absa(y)}}{\IPmag{\alpha^{\Absa(y)}}}  \right)^{\!m+1}  \right]     \right]\!. 
\]
Since $y^{-1}\left[\EV q_m\left( W(y) \right) - q_m\left( W(0) \right)\right] - 2\petrovgen\left(q_m \right)(x)= y^{-1} \overline{\Delta}_m^y(x) + y^{-1}\difftwo_m^y(x)$, once we show that $y^{-1} \difftwo_m^y$ converges to zero in $\ltwo$, we are done.

Towards this goal, note that the difference $\difftwo_m^y(x)$ can be written as
\[
\sum_{i \ge 1} \EV\left[  \left(  \frac{R_i^{\Absa(y)}}{\IPmag{\alpha^{\Absa(Y)}}} \right)^{m+1} \sum_{j\ge 1} \left(  w^{\Absa(y)}_{ij} \right)^{m+1}  \right] - C_m^{1/2}\sum_{i\ge 1} 
\EV\left[ \left(  \frac{R_i^{\Absa(y)}}{\IPmag{\alpha^{\Absa(y)}}}  \right)^{m+1}  \right]. 
\]
Since $\sum_{j\ge 1} \left(  w^{\Absa(y)}_{ij} \right)^{m+1}$ is always bounded above by one, we get 
\[
\abs{y^{-1}\difftwo_m^y}(x) \le \left(  1 + C_m^{1/2} \right) \frac{1}{y}\sum_{i\ge 1} \EV\left[  \left(  \frac{R_i^{\Absa(y)}}{\IPmag{\alpha^{\Absa(y)}}}  \right)^{m+1}  \right].
\] 
Now, using a similar logic as in the case of $\overline{\Delta}_m^y$, the right side above converges to zero in $\ltwo$ as $y \rightarrow 0+$. Hence, so does $y^{-1} \difftwo_m^y$ and we are done.  
\end{proof}

We will now extend Lemma \ref{thm:identifygen} to finite multiples of $\left( q_k,\; k \ge 1 \right)$. Consider the family of functions $\left( f_m,\; m \ge 1  \right)$ as before. Fix $l \ge 1$ and $\{m_1, m_2, \ldots, m_l\} \subset \BN$ and consider the product function from $\HIPspace$ to $\BR$: $f_{m_1, m_2, \ldots, m_l} = \prod_{j\in [l]} f_{m_j}$ and the corresponding function $q_{m_1,m_2, \ldots, m_l}=\prod_{j \in [l]} q_j$ from $\infsimp$ to $\BR$.

\begin{lemma}\label{thm:genalgebra} Fix $x\in \infsimp$ and consider a type-1 evolution $\bsa$ starting with $W(0)=x$. For every $m_1, \ldots, m_l \ge 1$, the limit on the right of  \eqref{eq:changeclockgen} exists for $f=f_{m_1, m_2, \ldots, m_l}$ and is given by $\gen\left(  f_{m_1,\ldots, m_l } \right)\left( \alpha^0 \right):= 2\petrovgen\left( q_{m_1, \ldots, m_l}  \right)(x)$, where $\petrovgen$ is the generator of $\EKP[\frac{1}{2},0]$. 
\begin{enumerate}[label=(\roman*), ref=(\roman*)]
\item In particular, for any $x\in \infsimp$ and for any $q\in \ualgebra$, 
\[
\lim_{y\rightarrow 0+} \frac{\BPr^1_x\left[ q\left( W(y)  \right)\right] - q(x)}{y}= 2 \petrovgen\left( q\right)(x).
\]
\item The above convergence also holds in $\ltwo$ with respect to the law of $\PoiDir[\frac{1}{2},0]$.
\end{enumerate}
\end{lemma}

\begin{proof} The proof is very similar to that of Lemma \ref{thm:identifygen}, only computationally more intensive. Hence, the strategy remains the same and we only highlight the differences. 

Start by noting that we can assume without loss of generality that each $m_j \ge 1$. The special case of $l=1$ and $m_1=0$ is trivial. Let us start with the case $l=2$. 

Fix $m,n > 1$ and $y > 0$. According to the notation in \eqref{eq:definefn} and taking expectations, and using the independence of $\left( w^y_{\cdot} \right)$ from $\left( L^y, R^y \right)$, we get
\begin{align}
&\EV\left[  f_m\left( \alpha^y  \right) f_n\left( \alpha^y  \right)  \right] = \EV\left[  \sum_{i \ge 1} \left( \frac{L_i^y}{\IPmag{\alpha}^y} \right)^{m+1}\sum_{j \ge 1} \left(\frac{L_j^y}{\IPmag{\alpha^y}}\right)^{n+1} \right] \nonumber\\
&+ C^n_{1/2} \sum_{i\ge 1} \sum_{j \ge 1} \EV\left[   \left( \frac{L_i^y}{\IPmag{\alpha^y}} \right)^{m+1}  \left( \frac{R_j^y}{\IPmag{\alpha^y}} \right)^{n+1}\right] + C^m_{1/2} \sum_{i\ge 1} \sum_{j \ge 1} \EV\left[   \left( \frac{L_i^y}{\IPmag{\alpha^y}} \right)^{n+1}  \left( \frac{R_j^y}{\IPmag{\alpha^y}} \right)^{m+1}\right]\nonumber\\
&+ \sum_{i\ge 1} \sum_{j \ge 1} C^{m,n}_{i,j} \EV\left[   \left( \frac{R_i^y}{\IPmag{\alpha^y}} \right)^{n+1}   \left( \frac{R_j^y}{\IPmag{\alpha^y}} \right)^{m+1} \right].\label{eq:genfmby}
\end{align}
Here, the constants $C^{m}_{1/2}, C^{n}_{1/2}$ are as before and $C^{m,n}_{i,j}$ are suitably defined.

Now, we change time in \eqref{eq:genfmby} and define $h_y(i,j)$ for $i,j\ge 1$ by
\begin{align}
&\EV\left[ \left( \frac{L_i^{\Absa(y)}}{\IPmag{\alpha^{\Absa(y)}}} \right)^{m+1} \left(\frac{L_j^{\Absa(y)}}{\IPmag{\alpha^{\Absa(y)}}}\right)^{n+1} - x_i^{m+1} x_j^{n+1} + C^n_{1/2}\left( \frac{L_i^{\Absa(y)}}{\IPmag{\alpha^{\Absa(y)}}} \right)^{m+1} \left(\frac{R_j^{\Absa(y)}}{\IPmag{\alpha^{\Absa(y)}}}\right)^{n+1}\right.\nonumber\\
&\quad\ \left. +\; C^m_{1/2}\left( \frac{L_i^{\Absa(y)}}{\IPmag{\alpha^{\Absa(y)}}} \right)^{\!n+1}\! \left(\frac{R_j^{\Absa(y)}}{\IPmag{\alpha^{\Absa(y)}}}\right)^{\!m+1}+ C^{m,n}_{i,j} \left( \frac{R_i^{\Absa(y)}}{\IPmag{\alpha^{\Absa(y)}}} \right)^{\!n+1}\! \left(\frac{R_j^{\Absa(y)}}{\IPmag{\alpha^{\Absa(y)}}}\right)^{\!m+1}\right]\!. \label{eq:whatishyij}
\end{align}

Consider $\mathfrak{G}_y(x)= \sum_{i\ge 1} \sum_{j\ge 1} \left(h_y(i,j) + x_i^{m+1} x_j^{n+1}\right)=\sum_{i\ge 1} \sum_{j\ge 1} h_y(i,j) + q_{m,n}(x)$. Then, it suffices to show that the limit $\lim_{y\rightarrow 0+} y^{-1}\left[\mathfrak{G}_y(x) - q_{m,n}(x)\right]$ is given by $2 \petrovgen\left(q_{m,n}\right)(x)$.

Now, as before we are going to start by assuming that there are only $k$ intervals in $\alpha^0$ and consider the following polynomial on the unit simplex in $\BR^{2k}$:
\begin{equation}\label{eq:derivedpoly}
\begin{split}
g_{m,n}(v) =& \sum_{i=1}^k v_i^{m+1} \sum_{i=1}^k v_i^{n+1} + C^{n}_{1/2} \sum_{i=1}^k v_i^{m+1} \sum_{i=k+1}^{2k} v_i^{n+1}  \\
&+ C^m_{1/2} \sum_{i=1}^k v_i^{n+1} \sum_{i=k+1}^{2k} v_i^{m+1} + \sum_{i=k+1}^{2k} \sum_{j=k+1}^{2k} C^{m,n}_{i,j} v_i^{n+1} v_j^{m+1}. 
\end{split}
\end{equation}
Then, as before, the limit $\gen\left( f_m f_n  \right)(\alpha^0)$, with $W(0)=x$, is equal to $\mathcal{L}g_{m,n}(v)$ where $\mathcal{L}$ is the generator of the \WF\ diffusion, running at $4$ times the usual speed, and $v$ is any point such that $v_i=x_i$ for $i \in [k]$ and $v_{k+i}\equiv 0$ for all $i \in [k]$.

\begin{lemma}\label{lem:genindivterm}
Let $g_{i,j}^{m,n}(v):= v_i^{m+1} v_j^{n+1}$ for $i,j \in [2k]$. If either $i$ or $j$ belongs to the set $\{ k+1, \ldots, 2k  \}$, then $\mathcal{L} g_{i,j}^{m,n}(v)=0$ for all $v$ such that $v_{k+i}\equiv 0$ for all $i \in [k]$.
\end{lemma}

\begin{proof} This follows from the polynomial form of the generator $\mathcal{L}$. Let us show this explicitly for the special case of $i=1, j=k+1$ and leave all other cases for the reader. 
\[
\begin{split}
\mathcal{L} g^{m,n}_{1,k+1}(v) =& 2 m(m+1) (1-v_1) v_1^{m} v_{k+1}^{n+1} + 2n(n+1) (1- v_{k+1}) v_1^{m+1} v_{k+1}^{n} \\
&- 2(m+1)(n+1) v_1^{m+1} v_{k+1}^{n+1} - m v_1^{m} v_{k+1}^{n+1} + n v_1^{m+1} v_{k+1}^{n}.
\end{split}
\]
Since $n \ge 1$ and $v_{k+1}=0$, every term above is zero proving that $\mathcal{L} g^{m,n}_{1,k+1}(v)=0$.
\end{proof}

By the above lemma and \eqref{eq:derivedpoly} it is clear that for all $v$ such that $v_{k+i}\equiv 0$ for all $i \in [k]$, we have 
\[
\mathcal{L} g_{m,n}(v) = \mathcal{L}\left[ \sum_{i=1}^k v_i^{m+1} \sum_{j=1}^k v_j^{n+1}  \right].
\]
However, restricted to functions of the first $k$ coordinates, the generator $\mathcal{L}$ is exactly $2\petrovgen$. 
\medskip

To drop the assumption of finitely many initial intervals, proceed as in the proof of Lemma \ref{thm:identifygen}. As in there, fix a $k \ge 1$ and consider $i \ge k+1$. Recall that $\mathbf{x}^k \in \infsimpcl$ denotes the vector such that $\mathbf{x}^k_i=x_i$ for $i \in [k]$ and $\mathbf{x}^k_{k+i}=0$ for all $i \ge 1$. Consider the quantity
\[
\overline{\Gamma}_{mn}^y(x) = \sum_{i=1}^k \sum_{j=1}^k \left[ h_y(i,j) + x_i^{m+1} x_j^{n+1} \right] = \sum_{i=1}^k \sum_{j=1}^k h_y(i,j) + q_{m,n}\left( \mathbf{x}^k \right),
\]
that depends only on the first $k$ coordinates, and the difference $\overline{\Delta}^y_{mn}(x)=\mathfrak{G}_y(x)-\overline{\Gamma}_{mn}^y(x)$.

\begin{lemma}\label{lem:bnddeltagen}
There is a constant $C_1 >0$ that does not depend on $x$ such that 
\[
\sup_{y > 0} y^{-1} \abs{\overline{\Delta}^y_{mn}(x)} \le C_1 \sum_{i=k+1}^\infty x_i. 
\]
\end{lemma}

\begin{proof} The difference $\overline{\Delta}^y_{mn}(x)=\sum_{i\ge k+1, j \ge 1} h_y(i,j) + \sum_{i \ge 1, j \ge k+1} h_y(i,j)$. We now give upper and lower bounds on each $h_y(i,j)$ as before. 

For the upper bound, if $Z_i=L_i+R_i$ denote the total mass of the clade, then 
\[
h_y(i,j) \le \EV\left[ \left( \frac{Z_i^{\Absa(y)}}{\IPmag{\alpha^{\Absa(y)}}} \right)^{m+1} \left(\frac{Z_j^{\Absa(y)}}{\IPmag{\alpha^{\Absa(y)}}}\right)^{n+1}  - x_i^{m+1} x_j^{n+1} \right].
\]
When $i\neq j$, the processes $\left( \frac{Z_i^{\Absa(y)}}{\IPmag{\alpha^{\Absa(y)}}}, \frac{Z_j^{\Absa(y)}}{\IPmag{\alpha^{\Absa(y)}}}   \right)$ are the first two coordinates of a \WF[0,0,0] diffusion, running at $4$ times the usual speed. Therefore, the product of the two is a martingale. Then, by It\^o's rule, $\sup_{y > 0} y^{-1}h_y(i,j) \le x_i x_j$.  When $i=j$, we are back to the proof of Lemma \ref{lem:bnddeltak}, and thus, $\sup_{y > 0} y^{-1} h_y(i,j) \le x_i$. Hence, 
\[
\overline{\Delta}^y_{mn}(x) \le \sum_{i\ge k+1, j\ge 1, i\neq j} x_i x_j + \sum_{i \ge 1, j\ge k+1, i\neq j} x_i x_j + \sum_{i\ge k+1} x_i \le 3 \sum_{i\ge k+1} x_i. 
\]

The lower bound also follows similarly to that in Lemma \ref{lem:bnddeltak}. Following the notation there,
\[
h_y(i,j) \ge \EV\left[ \left( \eta_i^{y\wedge \tilde{\chi}_i}\right)^{m+1} \left(\eta_j^{y\wedge \tilde{\chi}_j}\right)^{n+1}  - x_i^{m+1} x_j^{n+1} \right].
\]
The rest again follows by It\^o's rule. 
\end{proof}

For a general $l$, the product $\EV\big[ \prod_{j \in [l]}  f_{m_j}\left(\alpha^y\right) \big]$ can be expressed in terms of a polynomial of $\left(  L^y_i, R^y_i,\; i \ge 1  \right)$ whose individual terms will be of the type $\prod_{j \in [l]} v^{m_j}_{i_j}$ for different choices of $\{ i_1, \ldots, i_l  \} \subseteq [k]^l$. An extension of Lemma \ref{lem:genindivterm} applied to each such term will give us zero, unless every $i_j \in [k]$. For this remaining case, we already have exactly twice the generator of $\EKP[\frac{1}{2},0]$. The $\ltwo$ converges follows similarly. The extension to all $q\in \ualgebra$ follows by linearity. 
\end{proof}

We now argue that our process $W(\cdot)$ is the one that is described in \cite{FengSun10}. Let $\ltwo\left[\frac 12,0\right]$ refer to the Hilbert space of square integrable functions on $\infsimp$ with respect to the measure $\PoiDir[\frac{1}{2},0]$. Also, for this section, the corresponding norm will be denoted by $\norm{\cdot}_{1/2,0}$. 

Let $\Prm_x$ denote the probability measure on $\mathcal{C}\left([0,\infty),\infsimp \right)$ which is the law of $W(\cdot)$ under $\BPr^1_x$. We will continue to use the notation $W$ for this canonical random process.

\begin{lemma}\label{lem:strcontsemigp}
Let $\left(T_y,\; y \ge 0  \right)$ denote the transition semigroup of the process $W$. Then, for every $y>0$, $T_y$ is an operator on $\ltwo\left[\frac 12,0\right]$ and the semigroup is strongly continuous as a semigroup. 
\end{lemma}

\begin{proof} 
By definition, $T_y f(x) = \Prm_x[ f(W(y))] = \int_{\infsimp} f(v)p_y(x,dv)$, where $p_y(x,dv)$ is the transition operator of $W$.  We first show that $T_y$ is an operator on $\ltwo\left[\frac 12,0\right]$ in the sense that 
\begin{enumerate}[label=(\roman*), ref=(\roman*)]
\item If $f$ is square integrable with respect to $\PoiDir[\frac{1}{2},0]$ then so is $T_yf$,
\item If $f =0 $ $\PoiDir[\frac{1}{2},0]$-a.e. then so is $T_y f$.
\end{enumerate}
The second condition shows that the $\PoiDir[\frac{1}{2},0]$-equivalence class of $T_y f$ is determined by the $\PoiDir[\frac{1}{2},0]$-equivalence class of $f$, so that we may consider $T_y\colon\ltwo\left[\frac 12,0\right] \to\ltwo\left[\frac 12,0\right]$.  From Jensen's inequality we see that
\[ \int_{\infsimp} (T_yf(x))^2 \PoiDir[\frac{1}{2},0](dx) \leq \int_{\infsimp} T_yf^2(x)  \PoiDir[\frac{1}{2},0](dx) =  \int_{\infsimp} f^2(v) \PoiDir[\frac{1}{2},0](dv)\]
since $\PoiDir[\frac{1}{2},0]$ is the stationary distribution of $W$.  Both claims follow immediately.

It is easy to see that every element in the unital algebra $\ualgebra$ is in $\ltwo\left[\frac 12,0\right]$. As a corollary of Lemma \ref{thm:genalgebra} (ii), for any $q \in \ualgebra$, we have $\lim_{y\rightarrow 0+} \norm{T_yq- q}_{1/2,0} = 0$. Consider any $f \in \ltwo\left[\frac 12,0\right]$. Then, there exists a sequence $\{q_n \} \subseteq \ualgebra$ such that $\lim_{n\rightarrow\infty}q_n=f$ in $\ltwo\left[\frac 12,0\right]$. By the triangle inequality, $\norm{(T_y-I)f}_{1/2,0} \le \norm{(T_y - I)q_n}_{1/2,0} + \norm{(T_y-I)\left( q_n-f \right)}_{1/2,0}$.
Since $\left(T_y-I,\; y \ge 0\right)$ is a uniformly bounded family of operators, we get $\lim_{y\rightarrow0+}T_yf = f$ in $\ltwo\left[\frac 12,0\right]$. This proves strong continuity of the semigroup. 
\end{proof}

Hence, by \cite[Corollary 1.1.6]{EthKurtzBook}, the $\ltwo\left[\frac 12,0\right]$ generator $\procvgen$ of $\left( T_y,\; y\ge 0 \right)$ is closed and has a dense domain in $\ltwo\left[\frac 12,0\right]$. Moreover, by \cite[Proposition 2.1]{EthKurtzBook}, for any $\lambda >0$, the resolvent $\left( \lambda - \procvgen  \right)^{-1}$ exists as a bounded operator on $\ltwo\left[\frac 12,0\right]$ and is one-to-one and has dense range.  

\begin{lemma}\label{lem:resolventinverse} For any $\lambda >0$, we have $\left( \lambda - \procvgen  \right)\ualgebra = \ualgebra$.
\end{lemma}

\begin{proof} Since $\procvgen=2\petrovgen$ on $\ualgebra$, it is clear that $\left( \lambda - \procvgen  \right)\ualgebra \subseteq \ualgebra$. So, the nontrivial part is the converse.

For any $m \ge 1$, consider the subspace spanned by $\left\{ q_0, q_1, \ldots, q_m  \right\}$. Notice that $\left( \lambda - \procvgen  \right) q_i$ is a polynomial of degree at most $i$. In fact, 
\[
\left( \lambda - \procvgen  \right) q_i = \left( \lambda + 2i(i+1)  \right) q_i - (2i-1)(i+1) q_{i-1}.
\] 
Therefore, the operator $\left(\lambda - \procvgen \right)$ restricted to this finite-dimensional vector space is given by a lower bidiagonal matrix with positive elements on the diagonal. This matrix clearly has full rank and is invertible. By taking $m \rightarrow \infty$ we have shown that the vector space spanned by $\left\{ q_0, q_1, q_2, \ldots  \right\}$ is invariant under $\left( \lambda - \procvgen  \right)^{-1}$.

Now, fix any $m \ge 1$ and consider the finite-dimensional subspace of $\ualgebra$ spanned by the basis
\[
\left\{ q_iq_j,\; i\ge 1, j\ge 1, i+j\le m\right\} \cup \left\{  q_i,\; 1\le i \le m+1  \right\} \cup \{q_0\}. 
\]
By a calculation similar to the one in the proof of Lemma \ref{lem:genindivterm}, we get that, $\left(\lambda - \procvgen\right) q_iq_j(x)$ is given by a linear combination of $q_{i}q_j$, $q_{i-1}q_j$, $q_iq_{j-1}$, and $q_{i+j+1}$, with positive coefficient for $q_{i}q_j$, and therefore belongs to the above subspace when $i,j\ge 1$ and $i+j\le m$. The same is true for $\left(\lambda - \procvgen\right)q_i$, as we have verified in the previous paragraph. Hence, we can label the elements in the basis such that $\procvgen$ restricted to this subspace is a again a lower diagonal matrix with positive coefficients on the diagonal. Hence, the subspace spanned by $\{ q_i q_j,\; i\ge 0, j\ge 0 \}$ is invariant under the operator $\left( \lambda - \procvgen  \right)^{-1}$. 

The general pattern is now clear. For any $l,m\ge 1$, consider the subspace spanned by 
\[
\{q_0\}\cup \bigcup_{j\in [l]} \left\{  \prod_{k\in [j]} q_{i_k},\; i_k\ge 1,\; k\in [j],\; \sum_{k\in [j]} i_k \le m + (l-j) \right\}. 
 \]
 A similar argument verifies that this subspace, and that spanned by $\big\{\!\prod_{k\in [l]} q_{i_k},\; i_1, \ldots, i_l \ge 0  \big\}$, are invariant under $(\lambda- \procvgen)^{-1}$. Taking a union over $l \ge 1$ completes the proof. 
\end{proof}

The previous lemma allows us to avoid Dirichlet form techniques while identifying our process $W$ to be the one described in \cite{FengSun10}. For example, the symmetry of the resolvent $\left( \lambda - \procvgen  \right)^{-1}$ follows from the symmetry of $2\petrovgen$ on $\ualgebra$ (see the calculation in \cite[equation (2.5)]{FengSun10}) and Lemma \ref{lem:resolventinverse}. This shows that our process $W$ is reversible with respect to $\PoiDir[\frac{1}{2},0]$. We skip the proof.

\begin{proof}[Proof of Theorem \ref{thm:petroviden}] 
Lemmas \ref{lem:strcontsemigp} and \ref{lem:resolventinverse} together with \cite[Proposition 1.3.1]{EthKurtzBook} imply that $\ualgebra$ is a core for $\procvgen$.  Letting $(\widetilde{T}_y,\; y\geq 0)$ be the $\ltwo\left[\frac 12,0\right]$ semigroup considered by Feng and Sun \cite{FengSun10}, this shows that $(T_y,\; y\geq 0)$ and $(\widetilde{T}_{2y},\; y\geq 0)$ have the same generator (given by the closure of $(\procvgen,\ualgebra)$) and thus $T_y= \widetilde{T}_{2y}$.  Let $(\widetilde{W}(y), {y\geq 0})$ be the diffusion associated with $(\widetilde{T}_{y},\; y\geq 0)$ constructed in \cite{Petrov09} (which is the Feller version of the diffusion constructed in \cite{FengSun10}). Let $\widetilde{\mathrm{P}}_x$ denote the law of $\widetilde{W}$, when starting from $x$. Then we find that for every $f\in \ltwo\left[\frac 12,0\right]$ we have 
\[ 
\Prm_x\left[f(W(y))\right] = T_yf(x) = \widetilde{T}_{2y} f(x) = \widetilde{\mathrm{P}}_x[f(\widetilde{W}(2y))] \qquad \PoiDir[\tfrac{1}{2},0]\textrm{-a.e.}.
\] 
If $f$ is bounded and continuous then $x\mapsto \Prm_x [f(W(y))]$ is continuous by Proposition \ref{prop:type-1:cts_in_init_state_H} and $x\mapsto  \widetilde{\mathrm{P}}_x[f(\widetilde{W}(2y))]$ is continuous by \cite[Proposition 4.3]{Petrov09}.  Since any set of full $\PoiDir[\frac{1}{2},0]$-measure is dense in $\infsimp$, we get $\Prm_x [f(W(y))]=\widetilde{\mathrm{P}}_x[f(\widetilde{W}(2y))]$ for every bounded, continuous $f$ and every $x\in\infsimp$. Together with path-continuity and the Markov property, this identifies the laws of the processes. The argument for the type-0 evolution is similar.
\end{proof}

\begin{proof}[Proof of Corollary \ref{cor:Petrov_diversity}]
 By Theorem \ref{thm:stationary} and Definition \ref{def:IP:metric} of $\dI$, the de-Poissonized type-1 and type-0 evolutions have continuously evolving diversity. By Theorem \ref{thm:petroviden} and the compatibility between definitions of diversity in \eqref{eq:intro:diversity} and \eqref{eq:IPLT}, this extends to the \EKP[\frac12,0] and \EKP[\frac12,\frac12] processes.
\end{proof}

\appendix

\section{Topology of $(\IPspace,\dI)$}
\label{sec:IP_space}

It will be useful to separate the diversity of a partition from most of its mass in the following sense. 
For $\alpha\in\IPspace$ and $\epsilon>0$, let
\begin{equation}\label{eq:IP:mass_cutoff_for_sep}
 \delta(\alpha,\epsilon) := \sup\left\{m>0\colon\ \sum_{U\in\alpha}\cf\big\{\Leb(U) < m\big\}\Leb(U) < \epsilon \right\}
\end{equation}
For the purpose of the following, let $A := \{U\in\alpha\colon\ \Leb(U) \geq \delta(\alpha,\epsilon)\big\}$,
$$S_A(x) := \sum_{(a,b)\in A}(b-a)\cf\{b \leq x\}, \quad \text{and} \quad S_{\alpha\setminus A}(x) := \sum_{(a,b)\in\alpha\setminus A}(b-a)\cf\{b \leq x\}\quad \text{for }x\ge0.$$
We define
 \begin{equation}\label{eq:IP:separate_LT_mass}
 \begin{split}
   \alpha^{\IPLT}_{\epsilon} &:= \big\{ (a-S_A(a),b-S_A(a))\colon (a,b)\in\alpha\setminus A\big\}\\
   \text{and}\quad
   \alpha^{M}_{\epsilon} &:= \big\{ (a-S_{\alpha\setminus A}(a),b-S_{\alpha\setminus A}(a))\colon (a,b)\in A\big\}.
 \end{split}
 \end{equation}
Effectively, we form $\alpha^M_{\epsilon}$ by taking the large blocks of $\alpha$ and sliding them down to sit next to each other, and correspondingly for $\alpha^{\IPLT}_{\epsilon}$ with the small blocks. These partitions have the properties
$$
 \IPLT_{\alpha^{\IPLT}_{\epsilon}}(\infty) = \IPLT_{\alpha}(\infty),\qquad \IPmag{\alpha^{\IPLT}_{\epsilon}} \leq \epsilon,\qquad
 \IPLT_{\alpha^{M}_{\epsilon}}(\infty) = 0,\qquad \IPmag{\alpha^{M}_{\epsilon}} \geq \IPmag{\alpha} - \epsilon.
$$

\begin{proof}[Proof of Proposition \ref{prop:Hausdorff} (i)-(iii)]
 \ref{item:Haus:neq} Fix $\epsilon > 0$. Consider an arbitrary $\alpha\in\IPspace$ with $\IPLT_{\alpha}(\infty) > 1/\epsilon$. The pair $\left(\alpha,\alpha^M_{2\epsilon}\right)$ defined in \eqref{eq:IP:separate_LT_mass} has the desired property.
 
 \ref{item:Haus:leq} This is immediate from Definition \ref{def:Hausdorff} of $d_H'$.
 
 \ref{item:Haus:equiv} First, we show $\dH(\beta,\gamma)\leq 3d_H'(\beta,\gamma)$ for every $\beta,\gamma\in\HIPspace$. Suppose $d_H'(\beta,\gamma) < x$ for some $x>0$. Then there is some correspondence $(U_i,V_i)_{i\in [n]}$ from $\beta$ to $\gamma$ with Hausdorff distortion less that $x$. Recall from Definition \ref{def:IP:metric} that, in a correspondence, the $(U_i)$ and $(V_i)$ are each listed in left-to-right order. Let
 $$\beta' := \Concat_{i\in[n]} \{(0,\Leb(U_i))\}, \qquad \gamma' := \Concat_{i\in[n]} \{(0,\Leb(V_i))\}.$$
 By Definition \ref{def:Hausdorff} of Hausdorff distortion, $\IPmag{\beta} - \IPmag{\beta'} < x$, and likewise for $\gamma$ and $\gamma'$. Thus, for each $j\in [n-1]$, the right endpoint of $U_j$ and the left endpoint of $U_{j+1}$ are within distance $x$ of the corresponding point in $\beta'$, and similarly for the left endpoint of $U_1$ and the right endpoint of $U_n$. Thus, $d_H(\beta,\beta') < x$ and correspondingly for $\gamma$. Moreover, by definition of distortion, we also find $d_H(\beta',\gamma') < x$. By the triangle inequality, $\dH(\beta,\gamma)< 3x$, as desired.
 
 Now, consider $\beta\in\HIPspace$ and $\epsilon>0$. Take $\delta_0>0$ small enough that $\sum_{U\in\beta\colon \Leb(U)\leq 2\delta_0}\Leb(U) < \epsilon/3$. Let $K$ denote the number of blocks in $\beta$ with mass at least $2\delta_0$. Take $\delta := \min\{\delta_0,\epsilon/(6K+3)\}$. It suffices to show that for $\gamma\in\HIPspace$, if $d_H(\beta,\gamma)<\delta$ then $d_H'(\beta,\gamma)<\epsilon$.
 
 Suppose $d_H(\beta,\gamma)<\delta$ for some $\gamma\in\HIPspace$. Then for each $U\in\beta$ with $\Leb(U) > 2\delta_0 \geq 2\delta$, the midpoint of $U$ must lie within some block $V$ of $\gamma$. Consider the correspondence from $\beta$ to $\gamma$ that matches each such $(U,V)$. Then, by the bound on $d_H(\beta,\gamma)$, for each such pair, $|\Leb(U)-\Leb(V)|<2\delta\leq \epsilon/3K$. Moreover, by our choice of $\delta_0$, the total mass in $\beta$ excluded from the blocks in the correspondence is at most $\epsilon/3$. Similarly, the reader may confirm that the mass in $\gamma$ excluded from the correspondence is at most $(\epsilon/3)+2K\delta+\delta \leq 2\epsilon/3$. Thus, by Definition \ref{def:Hausdorff} of $d_H'$, we have $d_H'(\beta,\gamma)<\epsilon$, as desired.
\end{proof}
 


We prove Proposition \ref{prop:Hausdorff} \ref{item:Haus:sig} at the end of this appendix. 
We now proceed towards proving Theorem \ref{thm:Lusin}, that $(\IPspace,\dI)$ is Lusin.

\begin{proposition}\label{prop:Lusin2}
 The metric space $(\IPspace,\dI)$ is isometric to a path-connected Borel subset of a complete separable metric space $(\cJ,d_\cJ)$. This space $(\cJ,d_\cJ)$ cannot be chosen locally compact. 
\end{proposition}

We prove this proposition by identifying an explicit completion $(\cJ,d_\cJ)$ of $(\IPspace,\dI)$, which allows to identify $\IPspace$ with a Borel subset. 

\begin{lemma}\label{lem:IP:sep}
 $(\IPspace,\dI)$ is path-connected and separable.
\end{lemma}

\begin{proof}
 For path-connectedness, just note that $c\mapsto\scaleI[c][\alpha]$, $c\in[0,1]$, is a path from $\emptyset\in\IPspace$ to $\alpha\in\IPspace$. Specifically, continuity holds since Lemma \ref{lem:IP:scale} yields for $0<a<b\leq 1$
 $$
  \dI(\scaleI[a][\alpha],\scaleI[b][\alpha]) = \dI(\scaleI[a][\alpha], \scaleI[a][b/a\scaleI \alpha])
  	\leq \sqrt{a}\max\left\{\left|\frac{\sqrt{b}}{\sqrt{a}}-1\right|\IPLT_\alpha(\infty),\left|\frac{b}{a}-1\right|\IPmag{\alpha}\right\}
 $$
and $\dI(\emptyset,\scaleI[b][\alpha])=\max\{\sqrt{b}\IPLT_\alpha(\infty),b\IPmag{\alpha}\}$. 
 
 For separability, we fix a partition $\alpha\in\IPspace$ with $\IPLT_{\alpha}(\infty) > 0$ and such that $t\mapsto\IPLT_\alpha(t)$ is continuous on $[0,\IPmag{\alpha}]$.
 For the purpose of this proof we abbreviate our scaling notation from $\scaleI[c][\alpha]$ to $c\alpha$. We will construct a countable $S\subset\IPspace$ in which each element is formed by taking $(c\alpha)^{\IPLT}_{\epsilon}$, as in \eqref{eq:IP:separate_LT_mass}, for some $c\geq 0$ and $\epsilon > 0$, and inserting finitely many large blocks into the middle, via the following operation.
%
 For $s\in \left[0,\IPLT_{\alpha}(\infty)\right]$ and $m>0$, we define
 \begin{equation*}
  \alpha\oplus_s m := \big(\{U\in\alpha\colon \IPLT_{\alpha}(U) \leq s\} \concat \{(0,m)\}\big) \cup \{(a+m,b+m)\colon (a,b)\in\alpha,\ \IPLT_{\alpha}(a)>s\}.
 \end{equation*}
 This operation inserts a new interval $V$ of length $m$ into the middle of $\alpha$ in such a way that $\IPLT_{\alpha\oplus_s m}(V) = s$. Let
 $$
  S := \left\{ (c\alpha)^{\IPLT}_{\varepsilon} \oplus_{s_1} m_1 \cdots \oplus_{s_r} m_r\ \middle|
   		\begin{array}{l}
   			r\in\BN,\ s_1,\ldots,s_r \in [0,\IPLT_{c\alpha}(\infty))\cap\BQ,\\
   			c,\epsilon,m_1,\ldots,m_r\in (0,\infty)\cap\BQ
   		\end{array}\right\}.
 $$
 By Lemma \ref{lem:IP:scale}, $\IPLT_{c\alpha}(\infty) = \sqrt{c}\IPLT_\alpha(\infty)$ for $c \geq 0$. Thus, any $\beta\in\IPspace$ can be approximated in $S$ by the partitions constructed from the following rational sequences. First, take rational
 $$
  c_n \to \left(\frac{\IPLT_\beta(\infty)}{\IPLT_\alpha(\infty)}\right)^2,\qquad \epsilon_n = \frac1n \downto 0,\qquad\text{ and}\qquad r_n = \# \beta^M_{\epsilon_n}.
 $$
 Then let $\left\{U\in\beta\colon\ \Leb(U) > \delta\left(\beta,\epsilon_n\right)\right\} = \left\{\left(a_j^{(n)},a_j^{(n)}\!+\!k_j^{(n)}\right),\ j\in [r_n]\right\}$ with $a_1^{(n)}\le\cdots\le a_{r_n}^{(n)}$,
 where $\delta$ is as in \eqref{eq:IP:mass_cutoff_for_sep}. This is the sequence of blocks of $\beta$ that comprise $\beta^M_{\epsilon_n}$. Finally, we take rational sequences $\left(\left(s_j^{(n)},m_j^{(n)}\right),\ j\in [r_n]\right)$ so that 
$$
  \sup\nolimits_{j\in[r_n]}\left|s_j^{(n)}-\IPLT_\beta\left(a_j^{(n)}\right)\right| \le \epsilon_n \quad \text{and} \quad \sum\nolimits_{j\in[r_n]}\left|k_j^{(n)}-m_j^{(n)}\right| \le \epsilon_n.
 $$\vspace{-24pt}
 
\end{proof}

\begin{corollary}\label{cor3}
 There is a metric on $\IPspace$ that generates the same topology as $\dI$, for which $\IPspace$ is isometric to a subset of a compact metric space.
\end{corollary}

\begin{proof}
 Since $(\IPspace,\dI)$ is a separable metric space, Dudley's \cite[Theorem 2.8.2]{Dudley02} applies.
\end{proof}

Unfortunately, this argument is unsuitable to show that the subset can be chosen as a Borel subset. Indeed, the argument can be applied to non-Borel subsets of a compact metric space.

\begin{lemma}\label{lem:dJ_metric}
 Let $\cJ$ be the set of pairs $(\alpha,f)$, where $\alpha$ is an interval partition of $[0,\IPmag{\alpha}]$ with $\Leb([0,\IPmag{\alpha}]\setminus\bigcup_{U\in\alpha}U)=0$, and where $f\colon[0,\infty]\rightarrow[0,\infty)$ is a right-continuous increasing function that is constant on every interval 
$U\in\alpha$ and on $[\IPmag{\alpha},\infty]$. We replace $\IPLT_\alpha$ and $\IPLT_\beta$ in Definition \ref{def:IP:metric}, the definition of $\dI(\alpha,\beta)$, by $f$ and $g$, to define $d_\cJ((\alpha,f),(\beta,g))$. Then $d_\cJ$ is a metric on $\cJ$.  
\end{lemma}

\begin{proof}
 Given the proof of Proposition \ref{prop:IP_metric}, the only change needed for this lemma is in proving positive-definiteness, since now $f$ is not determined by $\alpha$. However, this follows easily since we assume that $f$ is right-continuous and constant on each $U\in\alpha$ and on $[\IPmag{\alpha},\infty]$, and $f$ is therefore determined by the values it takes on these sets.
\end{proof}

For $n\geq1$ and $\beta\in\HIPspace$, let $\beta_n$ denote the interval partition formed by deleting all but the $n$ largest blocks from $\beta$ (breaking ties via left-to-right order) and sliding these large blocks together, as in the construction of $\alpha^M_{\epsilon}$ in \eqref{eq:IP:separate_LT_mass}. For $t\ge0$, let 
\begin{equation}\label{eq:div_approx}
 D_{\beta,n}(t) := \sqrt{\pi x_n}\#\{(a,b)\in\beta_n\colon b\leq t\}, \qquad \text{where} \qquad x_n = \min\{\Leb(U)\colon U\in\beta_n\}.
\end{equation}
If the following two limits are equal, then we adapt Definition \ref{def:diversity_property} to additionally define
\begin{equation}\label{eq:div_rt_cts}
\begin{split}
 \IPLT_{\beta}^+(t) :=&\ \lim_{u\downto t}\limsup_{h\downto 0} \sqrt{\pi h}\#\{(a,b)\in\beta\colon (b-a)>h,\, b\leq u\}\\
 	=&\ \lim_{u\downto t}\liminf_{h\downto 0} \sqrt{\pi h}\#\{(a,b)\in\beta\colon (b-a)>h,\, b\leq u\}.
\end{split}
\end{equation}

\begin{lemma}\label{lem:dJ_meas}
 \begin{enumerate}[label=(\roman*), ref=(\roman*)]
  \item The maps $\beta\mapsto\beta_n$ and $\beta_n\mapsto D_{\beta,n}$ are Borel under $\dH'$.\label{item:dJm:order_stats}
  \item The set $\{(\beta,t)\in \HIPspace\times [0,\infty)\colon \IPLT_{\beta}(t)\text{ exists}\}$ is Borel under $\dH'$ in the first coordinate plus the Euclidean metric in the second. The map $(\beta,t)\mapsto\IPLT_{\beta}(t)$ is measurable on this set, under the same $\sigma$-algebra. The same assertions hold with $\IPLT_{\beta}(t)$ replaced by $\IPLT_{\beta}^+(t)$.
  %
  \label{item:dJm:diversity}
  \item For $\beta\in\IPspace$, the pairs $(\beta_n,D_{\beta,n})$ converge to $(\beta,\IPLT_{\beta}(\,\cdot\,+))$ under $d_{\cJ}$.\label{item:dJm:cnvgc}
 \end{enumerate}
\end{lemma}

\begin{proof}
 (i) The map $\ranked$ that sends $\beta\in\HIPspace$ to the vector of its order statistics is continuous under $\dH'$. The restriction map $(\beta,t)\mapsto \restrict{\beta}{[0,t]} := \{U\cap (0,t)\colon U\in\beta,\,U\cap (0,t)\neq\emptyset\}$ is continuous from $\dH'$ plus the Euclidean metric to $\dH'$. If $\ranked(\beta) = (x_1,x_2,\ldots)$, then we determine whether the block of mass $x_1$ is to the right of the block of mass $x_2$ by finding the least $t_1,t_2\in x_2\BN$ for which $\restrict{\beta}{[0,t_1]}$ has $x_1$ as its first order statistic and $\restrict{\beta}{[0,t_2]}$ has $(x_1,x_2)$ as its first two order statistics. If $t_1<t_2$ then $\beta_2 = \{(0,x_1),(x_1,x_1+x_2)\}$; otherwise, $\beta_2 = \{(0,x_2),(x_2,x_2+x_1)\}$. This method extends to give the desired measurability of $\beta\mapsto\beta_n$. The measurability of $\beta_n\mapsto D_{\beta,n}$ follows similarly from the measurability of $\ranked$ and restrictions.
 
 (ii) Consider $\beta\in\HIPspace$. For $n\ge 1$, let $U_1,\ldots,U_n$ denote the $n$ largest blocks of $\beta$, in left-to-right order. Let $\theta_{\beta,n}\colon [0,\IPmag{\beta}]\to [0,\IPmag{\beta_n}]$ denote the continuous time-change starting from $\theta_{\beta,n}(0) = 0$, increasing with slope 1 on $\bigcup_{i\in [n]}U_i$, and having slope 0 on $[0,\IPmag{\beta}]\setminus \bigcup_{i\in [n]}\ol U_i$, where $\ol U$ denotes closure. 
 Note that $\{\theta_{\beta,n}(U_1),\ldots,\theta_{\beta,n}(U_n)\} = \beta_n$. 
 It follows from similar arguments to those in the proof of (i) that $\beta\mapsto\theta_{\beta,n}$ is measurable from $(\HIPspace,\dH')$ to $\cC([0,\infty),[0,\infty))$.
 
 By comparing \eqref{eq:div_approx} to Definition \ref{def:diversity_property} of $\IPLT_{\beta}$, for every $t\ge0$ we see that $\lim_{n\upto\infty}D_{\beta,n}(\theta_{\beta,n}(t)) = \IPLT_{\beta}(t)$, with each limit existing if and only if the other exists. By (i), this proves the two claims for $\IPLT_{\beta}(t)$. 
 By monotonicity of the limiting terms in \eqref{eq:div_rt_cts}, $\IPLT_{\beta}^+(t)$ exists if and only if
 $$\lim_{m\upto\infty}\limsup_{n\upto\infty}D_{\beta,n}\left(\theta_{\beta,n}\left(t+m^{-1}\right)\right) = \lim_{m\upto\infty}\liminf_{n\upto\infty}D_{\beta,n}\left(\theta_{\beta,n}\left(t+m^{-1}\right)\right).$$
 If these limits are equal, then they equal $\IPLT_{\beta}^+(t)$. This proves the two claims for $\IPLT_{\beta}^+(t)$.

 
 (iii) This follows from the previous argument by taking the correspondences from $\beta$ to $\beta_n$ that pair $U_i$ with $\theta_{\beta,n}(U_i)$, for each $i\in [n]$.
 %
\end{proof}


\begin{lemma}\label{lmcompl}
 Consider the map $\iota\colon \IPspace\rightarrow\cJ$ given by $\iota(\alpha)=(\alpha,\IPLT_\alpha(\,\cdot\,+))$.
 \begin{enumerate}[label=(\roman*), ref=(\roman*)]
  \item Both $\iota(\IPspace)$ and $\cJ\setminus\iota(\IPspace)$ are dense in $(\cJ,d_{\cJ})$.
  \item Both $\iota(\IPspace)$ and $\cJ\setminus\iota(\IPspace)$ are Borel subsets of $\cJ$.
  \item The space $(\cJ,d_\cJ)$ is a completion of $(\IPspace,\dI)$, with respect to the isometric embedding $\iota$.
 \end{enumerate}
\end{lemma}

\begin{proof}
 (i) By the definitions of $\dI$ and $d_\cJ$, the map $\iota$ is an isometry. Take $(\beta,g)\in\cJ\setminus\iota(\IPspace)$ and $\alpha\in\IPspace$ with $\IPLT_\alpha(\infty) = g(\infty)\ge 0$ and such that $t\mapsto\IPLT_\alpha(t)$ is continuous on $[0,\IPmag{\alpha}]$. Using the notation of the proof of Lemma \ref{lem:IP:sep}, we consider
  $$
   \beta^{(n)} := \alpha_{1/n}^{\IPLT}\oplus_{g\left(a_1^{(n)}\right)}k_1^{(n)}\cdots\oplus_{g\left(a_{r_n}^{(n)}\right)}k_{r_n}^{(n)}.
  $$
  Then $d_\cJ\left(\left(\beta^{(n)},\IPLT_{\beta^{(n)}}\right),(\beta,g)\right)\rightarrow 0$, i.e. $\beta$ is in the closure of $\iota(\IPspace)$. The same argument, with roles of $(\alpha,\IPLT_\alpha(\,\cdot\,+))$ and $(\beta,g)$ swapped (now $\alpha\in\IPspace$ general and $(\beta,g)\in\cJ\setminus\iota(\IPspace)$ and such that $g$ is continuous), shows that $(\alpha,\IPLT_\alpha)$ is in the closure of $\cJ\setminus\iota(\IPspace)$. 
 
 
 (ii) Recall that for $\alpha\in\IPspace$ we have $\IPLT_{\alpha}(\,\cdot\,+)=\IPLT_{\alpha}^+$ identically. Thus,
 $$\iota(\IPspace)=\left\{(\alpha,f)\in\mathcal{J}\colon\mbox{for all }t\in[0,\IPmag{\alpha}],\; \IPLT_\alpha(t)\mbox{ exists and }\IPLT^+_\alpha(t)=f(t)\right\}.$$
 By Lemma \ref{lem:dJ_meas} \ref{item:dJm:diversity} and Proposition \ref{prop:Hausdorff} \ref{item:Haus:leq}, the following set is Borel under $\dJ$:
 $$A := \left\{(\alpha,f)\in\cJ\colon \mbox{for all }t\in[0,\IPmag{\alpha}]\cap\BQ,\; \IPLT^+_{\alpha}(t)\text{ exists and equals }f(t)\right\}.$$
 For $(\alpha,f)\in A$, 
 by the right-continuity and monotonicity of $f$ and $\IPLT_{\alpha}^+$ 
 we have $f = \IPLT_{\alpha}^+$ identically. 
 By comparing Definition \ref{def:diversity_property} of $\IPLT_{\alpha}$ with \eqref{eq:div_rt_cts}, we see that if $\IPLT_{\alpha}^+$ is continuous at some $t\in [0,\IPmag{\alpha}]$ then $\IPLT_{\alpha}(t)$ exists and equals $\IPLT_{\alpha}^+(t)$, by a sandwiching argument. Thus, $\iota(\IPspace)$ is the set of $(\alpha,f)\in A$ for which $\IPLT_{\alpha}(t)$ exists at each time $t$ at which $f$ jumps.
 
 By Lemma \ref{lem:dJ_meas} \ref{item:dJm:diversity} and Proposition \ref{prop:Hausdorff} \ref{item:Haus:leq}, $(\alpha,f)\mapsto f(s) = \IPLT_{\alpha}^+(s)$ is Borel measurable on $(A,\dJ)$. By \cite[Theorem 14.5]{Billingsley}, the Borel $\sigma$-algebra on Skorokhod space is generated by the evaluation maps, so $(\alpha,f)\mapsto f$ is measurable from $(A,\dJ)$ to Skorokhod space. 
 By \cite[Proposition II.(1.16)]{JacodShiryaev}, the map from $f$ to the point process of its jumps is measurable; and by \cite[Proposition 9.1.XII]{DaleyVereJones2}, we can measurably map the latter to a sequence $(t_1,\Delta_1),(t_2,\Delta_2),\ldots$ listing times and sizes of all jumps of $f$, though these may not be listed in chronological order. We write $\tau_i(\alpha,f) := t_i$, or $\tau_i(\alpha,f) := -1$ if $f$ has less than $i$ jumps. Then
 \begin{equation*}
  \iota(\IPspace) = \left\{(\alpha,f)\in A\colon \text{for all }i\in\BN,\;\tau_i(\alpha,f) = -1\text{ or }\IPLT_\alpha(\tau_i(\alpha,f))\text{ exists}\right\}.
 \end{equation*}
 By Lemma \ref{lem:dJ_meas} \ref{item:dJm:diversity}, this set is measurable.
 
 (iii) It is clear from the definition of $\dJ$, based on that of $\dI$, that $\iota$ is an isometry. Now consider any Cauchy sequence $((\alpha_n,f_n),\ n\geq 1)$ in $(\cJ,\d_{\cJ})$. Then $(f_n(\infty),n\ge 1)$ is a Cauchy sequence in $[0,\infty)$; let us denote the limit by $f(\infty)$. Consider $(\beta^{(0)},f^{(0)})=(\emptyset,f(\infty))\in\cJ$, i.e. the empty partition with the increasing function that is constant $f(\infty)$. Let $\fs_n=(s_n^{(i)},i\ge 1)=(\Leb (U),U\in\alpha_n)^\downarrow\in S^\downarrow$ be the decreasing rearrangement of interval sizes. Then for all correspondences $(U_j,V_j)_{j\in[k]}$,
 $$
  \ell^1(\fs_n,\fs_m)=\sum_{i\ge 1}^\infty\left|s_n^{(i)}-s_m^{(i)}\right|
                      \le\sum_{j\in[k]}\left|\Leb (U_j)-\Leb (V_j)\right|+\IPmag{\alpha_n}+\IPmag{\alpha_m} - \sum_{j\in[k]}\Leb(U_j)+\Leb(V_j),
 $$
 Let $\epsilon>0$. By the Cauchy property of $((\alpha_n,f_n),\ n\geq 1)$, there is some $N_1\ge 1$ so that $d_\cJ\big((\alpha_n,f_n),(\alpha_m,f_m)\big)<\epsilon/2$ for all $m,n\ge N_1$. Taking the infimum over all correspondences on the RHS of the display, this yields $\ell^1(\fs_n,\fs_m) < \epsilon$ for all $m,n\ge N_1$. By completeness of $(S^\downarrow,\ell^1)$, we have convergence $\fs_n \rightarrow \fs = (s^{(i)},i\ge 1)\in S^\downarrow$.
 
 Now consider any $r\ge 1$ such that $s^{(r)} > s^{(r+1)}$. Consider $\epsilon>0$ with $3\epsilon < s^{(r)}-s^{(r+1)}$. Then there is $N_2\ge 1$ such that for all $n\ge N_2$, there are precisely $r$ intervals $(a_1^{(n)},a_1^{(n)}+k_1^{(n)}),$ $\ldots,(a_r^{(n)},a_r^{(n)}+k_r^{(n)})\in\alpha_n$ of length greater than $s^{(r)}-\epsilon$. We define
 $
  \beta_n^{(r)} := \Concat_{j\in[r]} \left\{\left(0,k_j^{(n)}\right)\right\}
 $
 and associate to these intervals the $f_n$-values of the corresponding intervals in $\alpha_n$:
 $$
  f_n^{(r)}\left( k_1^{(n)} + \cdots + k_{j}^{(n)} + x \right)=\left\{\begin{array}{ll}f_n\left(a_j^{(n)}\right)& \text{for }x\in [0,k_j^{(n)}),\ j\in [0,r-1],\\
  f_n(\infty)&\text{for }x\geq k_1^{(n)} + \cdots + k_{r}^{(n)}.
  \end{array}\right.
 $$
 Then $d_\cJ((\beta_n^{(r)},f_n^{(r)}),(\beta_m^{(r)},f_m^{(r)}))\le d_\cJ((\alpha_n,f_n),(\alpha_m,f_m))$, so $(\beta_n^{(r)},f_n^{(r)})$, $n\ge 1$, is a Cauchy sequence in $(\cJ,d_\cJ)$; and since for $n\ge N_2$
 $$
  d_\cJ(\beta_n^{(r)},\beta_m^{(r)}) = \max\left\{ \sup\nolimits_{j\in[r]}\left|f_n(a_j^{(n)})-f_m(a_j^{(m)})\right| ,\ \sum\nolimits_{j\in[r]}\left|k_j^{(n)}-k_j^{(m)}\right| \right\},
 $$
 the vector $\left(\left(f_n\left(a_j^{(n)}\right),k_j^{(n)}\right),\ 1\le j\le r\right)$ is a Cauchy sequence in the metric space $(\BR^{2r},\|\cdot\|_\infty)$. By completeness of $(\BR^{2r},\|\cdot\|_\infty)$, we have convergence to a limit $((f_j,k_j),1\le j\le r)$, which gives rise to a $d_\cJ$-limit $(\beta^{(r)}, f^{(r)})\in\cJ$ of $((\beta_n^{(r)},f_n^{(r)}),n\ge 1)$. By construction, $(\beta^{(r)},f^{(r)})$ is consistent as $r$ varies, in the sense that they are related by insertions of intervals of sizes from $\fs$, and natural correspondences demonstrate that convergence $(\beta^{(r)},f^{(r)})\rightarrow(\beta,f)$ holds in $\cJ$ for a limiting $(\beta,f)\in\cJ$ that incorporates intervals of all sizes $s^{(i)}$, $i\ge 1$.
 
 Finally, let $\epsilon>0$ and $N_1\ge 1$ be as above. Then there is $r$ large enough so that, following the notation of \eqref{eq:IP:separate_LT_mass},
 $$\IPmag{\beta^{\IPLT}_{s^{(r)}}} = \sum_{j\ge r+1}s_j < \epsilon/4.$$
 Since $\fs_n\rightarrow\fs$, there is $N_3\ge N_1$ such that for all $n\ge N_3$, we have $\ell^1(\fs_n,\fs)<\epsilon/4$. Finally, there is $N_4\ge N_3$ so that for all $n\ge N_4$ we have $d_\cJ\left(\left(\beta^{(r)}_n,f_n^{(r)}\right),\ \left(\beta^{(r)},f^{(r)}\right)\right) < \epsilon/4.$ Then for all $n\ge N_4$, we have
 \begin{align*}
  d_\cJ((\alpha_n,\!f_n),(\beta,\!f)) &\leq d_\cJ((\alpha_n,\!f_n),(\beta_n^{(r)}\!,\!f_n^{(r)}))+d_\cJ((\beta^{(r)}_n\!,\!f_n^{(r)}),(\beta^{(r)}\!,\!f^{(r)}))+d_\cJ((\beta^{(r)}\!,\!f^{(r)}),(\beta,\!f))\\
  &< \sum_{j=r+1}^\infty s_j+\ell^1(\fs_n,\fs)+\frac{\epsilon}{4}+\frac{\epsilon}{4}<\epsilon.
 \end{align*}
 Hence, $((\alpha_n,f_n),\ n\ge 1)$ converges to $(\beta,f)$ in $(\cJ,d_\cJ)$. Therefore, $(\cJ,d_\cJ)$ is complete.
\end{proof}

\begin{corollary}\label{corcomplsep}
 $(\cJ,d_\cJ)$ is a complete and separable metric space.
\end{corollary}
\begin{proof} We have shown in the lemmas that $(\cJ,d_\cJ)$ is a complete metric space, and since the completion of a separable metric space is also separable, $(\cJ,d_\cJ)$ is also separable.
\end{proof}

\begin{lemma}
 $(\cJ,d_\cJ)$ is not locally compact.
\end{lemma}
\begin{proof}
 Consider the interval partitions $\alpha_n=\{((k-1)2^{-n},k2^{-n}),1\le k\le 2^n\}$ and $f_n\equiv 0$. For $m<n$, any correspondence for $\alpha_m$ and $\alpha_n$ that matches up any intervals of $\alpha_m$ and $\alpha_n$ attracts a term $2^{-m}-2^{-n}\ge 2^{-n}$, so it is best to use the trivial correspondence which gives $d_\cJ((\alpha_n,f_n),(\alpha_m,f_m))=1$. Now assume that $(\emptyset,0)\in\cJ$ has a compact neighbourhood $K$. Then $K$ contains an open ball of some radius $2\epsilon>0$, which contains $(\scaleI[\epsilon][\alpha_n],0)$ for all $n\ge 1$. Covering $K$ with open balls of radius $\epsilon/2$, the open balls around $(\scaleI[\epsilon][\alpha_n],0)$ are disjoint, so there cannot be a finite subcover. This contradicts the compactness of $K$. Hence $(\emptyset,0)$ does not have a compact neighbourhood, and $(\cJ,d_\cJ)$ is not locally compact. 
\end{proof}

Even though $(\cJ,d_\cJ)$ is not locally compact, we can now deduce that $(\IPspace,d_\IPspace)$ is Lusin:

\begin{proof}[Proof of Proposition \ref{prop:Lusin2} and Theorem \ref{thm:Lusin}.] Lemma \ref{lmcompl} and Corollary \ref{corcomplsep} yield that 
$(\IPspace,d_\IPspace)$ is isometric to a Borel subset of the Polish space $(\cJ,d_\cJ)$, completing the proof of Proposition \ref{prop:Lusin2}. By \cite[Theorem II.82.5]{RogersWilliams}, this implies $(\IPspace,d_\IPspace)$ is Lusin, which proves Theorem \ref{thm:Lusin}.
\end{proof}

We now proceed towards proving Proposition \ref{prop:Hausdorff} \ref{item:Haus:sig}. Recall the Skorokhod metric of \cite[equations (14.12), (14.13)]{Billingsley}; we denote this by $d_{\cD}$. For $n\ge1$, let $\cJ_n\subseteq\cJ$ denote the set of $(\beta,f)\in\cJ$ for which $\beta$ has exactly $n$ blocks.

\begin{lemma}\label{lem:Haus_Skor_fin}
 For $n\geq1$, the metric $\dJ$ on $\cJ_n$ is topologically equivalent to the maximum of $\dH'$ in the first coordinate and the $d_{\cD}$ in the second.
\end{lemma}

\begin{proof}
 Fix $(\beta,f)\in\cJ_n$. We denote the blocks of $\beta$ by $U_1,\ldots,U_n$, in left-to-right order. Take $r\in \big(0,\min_{j\in [n]}\Leb(U_j)\big)$. We will show that, for $(\gamma,g)\in\cJ_n$, we get $\dJ((\beta,f),(\gamma,g))<r$ if and only if both $\dH'(\beta,\gamma) < r$ and $d_{\cD}(f,g) < r$.
 
 Consider $(\gamma,g)\in\cJ_n$ with $\dH'(\beta,\gamma) < r$ and $d_{\cD}(f,g) < r$. Since we have required $r$ to be smaller than all block masses in $\beta$, the only correspondence from $\beta$ to $\gamma$ that can have Hausdorff distortion less than $r$ is $(U_i,V_i)_{i\in[n]}$, where $V_1,\ldots,V_n$ denote the blocks of $\gamma$ in left-to-right order. In particular, $\sum_{i\in[n]}|\Leb(V_i)-\Leb(U_i)| < r$. Thus, in order for a continuous time-change $\lambda\colon [0,\IPmag{\beta}]\to [0,\IPmag{\gamma}]$ to never deviate from the identity by $r$, it must map some time in each $U_i$ to a time in the corresponding $V_i$. Therefore, by our bound on $d_{\cD}$, we have $\max_{i\in[n]}|g(V_i)-f(U_i)| < r$. We conclude that $\dJ((\beta,f),(\gamma,g))<r$.
 
 Now, consider $(\gamma,g)\in\cJ_n$ with $\dJ((\beta,f),(\gamma,g)) < r$. Following our earlier notation, the only correspondence that can give distortion less than $r$ is $(U_i,V_i)_{i\in[n]}$. As in Proposition \ref{prop:Hausdorff} \ref{item:Haus:leq}, it follows immediately from the characterization of $\dJ$ in Lemma \ref{lem:dJ_metric} and Definition \ref{def:Hausdorff} of $\dH'$ that $\dH'(\beta,\gamma)\leq \dJ((\beta,f),(\gamma,g)) < r$. We define $\lambda\colon [0,\IPmag{\beta}]\to [0,\IPmag{\gamma}]$ by mapping the left and right endpoints of each $U_j$ to the corresponding left and right endpoints of $V_j$ and interpolating linearly. Since $\sum_{i\in[n]}|\Leb(V_i)-\Leb(U_i)| < r$, it follows that $|\lambda(t)-t| < r$ for $t\in [0,\IPmag{\beta}]$ as well. By definition of $\dJ$, we have $|g(V_i)-f(U_i)| < r$ for each $i\in [n]$. Thus, $|g(\lambda(t)) - f(t)| < r$ for $t\in [0,\IPmag{\beta}]$. This gives $d_{\cD}(f,g) < r$.
 %
 %
 %
\end{proof}

\begin{proof}[Proof of Proposition \ref{prop:Hausdorff} (iv)]
 In light of Proposition \ref{prop:Hausdorff} (ii)-(iii), we need only check that all $\dI$-balls are Borel sets with respect to $d_H'$. Recall the $\beta_n$ and $D_{\beta,n}$ of Lemma \ref{lem:dJ_meas}. By Lemma \ref{lem:dJ_meas} \ref{item:dJm:cnvgc}, the $\dI$-ball of radius $r>0$ about $\beta$ equals
 $$\bigcup_{m\ge1}\bigcup_{N\ge1}\bigcap_{n>N}\left\{\gamma\in\IPspace\colon d_{\cJ}\big((\beta_n,D_{\beta,n}),(\gamma_n,D_{\gamma,n})\big) < r-m^{-1}\right\}.$$
 The claimed measurability now follows by Lemmas \ref{lem:dJ_meas} \ref{item:dJm:order_stats} and \ref{lem:Haus_Skor_fin}.
\end{proof}

\section{Statistics of clades and \Stable$\left(\frac32\right)$ excursions}
\label{sec:clade_stats}

In this section we prove Proposition \ref{prop:clade:stats}, as well as some related results on the \Stable[\frac32] process $\bX$ of Proposition \ref{prop:stable_JCCP}. More results in the vein of Proposition \ref{prop:clade:stats2} may be derived from these in a similar manner. Several of the following may be construed as descriptions of the It\^o measure $\mSxc$ associated with $\bX$.

\begin{proposition}\label{prop:clade:stats2}
 \begin{enumerate}[label=(\roman*), ref=(\roman*)]
  \item $\displaystyle \mClade\big\{\len > x\big\} = \frac{3}{2^{2/3}\pi^{1/3}\Gamma\left(2/3\right)}x^{-1/3}.$\label{item:CS2:len}\vspace{-.1cm}
  \item $\displaystyle \mClade\big\{ J > y\big\} = \frac{3}{\pi\sqrt{2}}y^{-1/2}.$\label{item:CS2:crossing}
  \item $\displaystyle \mClade\big\{ J^+ \in dy\;\big|\;m^0 = a\big\} = \frac{a^{3/2}}{\sqrt{2\pi}y^{5/2}}e^{-a/2y}dy.$\label{item:CS2:over:mass}
  \item $\displaystyle \mClade\big\{ m^0 > a \big\} = \frac{1}{\sqrt{\pi}}a^{-1/2}.$\label{item:CS2:mass}
  \item $\displaystyle \mClade\big\{ J^+ > y \big\} = \frac{\sqrt{2}}{\pi}y^{-1/2}.$\label{item:CS2:over}
  \item $\displaystyle \mClade\{ m^0 \leq a\;|\;J^+ = y\} = 1 - e^{-a/2y}.$\label{item:CS2:mass:over}
  \item $\displaystyle \mClade\big\{ \life^+ \leq z\;\big|\;J^+ = y\big\} = \cf\{z\geq y\}\sqrt{\frac{z-y}{z}}.$\label{item:CS2:max:over}
  \item $\displaystyle \mClade\big\{ \life^+ \leq z\;\big|\;m^0 = a\big\} = e^{-a/2z}.$\label{item:CS2:max:mass}
  \item $\displaystyle \mClade\big\{ \life^+ > z \big\} = \frac{1}{\sqrt{2}}z^{-1/2}.$\label{item:CS2:max}\vspace{-.1cm}
  \item $\displaystyle \mClade\{ m^0 \in da\;|\;\life^+ \geq z\} = \frac{z^{1/2}}{\sqrt{2\pi}a^{3/2}}(1-e^{-a/2z})da.$\label{item:CS2:mass:max}
  %
\end{enumerate}
 Each of these identities also holds if we replace all superscript `+'s with `-'s.
\end{proposition}

The equivalence when replacing `+'s with `-'s follows from the time-reversal invariance of Lemma \ref{lem:clade:invariance}. Before proving these identities we note a pair of relevant properties of $\bX$. Recall that $(T^y,\,y\in\BR)$ denotes the first hitting times for $\bX$.

\begin{proposition}[Theorem VII.1 of \cite{BertoinLevy}]\label{prop:hitting_time:subord}
 The hitting time process $(T^{-y},\,y\geq 0)$ is \Stable[\frac23] subordinator, and its Laplace exponent is the inverse $\psi^{-1}$ of the Laplace exponent of $\bX$:
 \begin{equation}
  \EV\left[ e^{-\theta T^{-y}} \right] = e^{-y\psi^{-1}(\theta)}, \quad \text{where} \quad \psi^{-1}(\theta) = \left(\frac{\pi}{2}\right)^{1/3}\theta^{2/3}.\label{eq:hitting_time:Laplace}
 \end{equation}
\end{proposition}


\begin{proposition}\label{prop:inv_LT:subord}
 For each $y\in\BR$, the shifted inverse local time process $(\tau^y(s)-\tau^y(0),\,s\ge 0)$ is a \Stable[\frac13] subordinator with Laplace exponent $\Phi(\theta) = 3(\theta/4\pi)^{1/3}$.
\end{proposition}

\begin{proof}
 It is straightforward to check that this is a \Stable[\frac13] subordinator. For $f\colon \BR\to\BR$ bounded and measurable,
 \begin{equation*}
 \begin{split}
  \int_{-\infty}^\infty f(y)\ell^y(t)dy &= \int_0^t f(\bX(s))ds \stackrel{d}{=} \int_0^t f\left( c^{2/3}\bX\left(\frac{s}{c}\right) \right)ds = \int_0^{t/c} f\left(c^{2/3}\bX (r) \right)cdr\\
  	&= \int_{-\infty}^\infty cf\left(c^{2/3}y\right)\ell^y\left(\frac{t}{c}\right)dy = \int_{-\infty}^\infty f(z)c^{1/3}\ell^{c^{-2/3}z}\left(\frac{t}{c}\right)dz.
 \end{split}
 \end{equation*}
 Hence $(c^{1/3}\ell^{c^{-2/3}y}(t/c);\ t\ge 0,y\in\BR) \stackrel{d}{=}(\ell^y(t);\ t\ge 0,y\in\BR)$ and so
 \begin{equation*}
  \tau^0(s) \stackrel{d}{=} \inf\left\{t\ge 0\ :\ c^{1/3}\ell^0(t/c) > s\right\} = c\tau^0(s/c^{1/3})
 \end{equation*}
 satisfies \Stable[\frac13] self-similarity. Thus, $\EV(e^{-\theta\tau^0(s)}) = e^{-sb\theta^{1/3}}$ for some $b\in(0,\infty)$. To identify $b$, we use the property that $\Pr(\bX(t)\le 0) = 2/3$ for all $t>0$. This follows from an identity in \cite[p.\ 218]{BertoinLevy}. 
 Specifically, let $S_{\theta}$ be an $\ExpDist[\theta]$ random variable independent of $\bX$ and define $K_{\theta} := \EV[\ell^0(S_{\theta})]$. Then on the one hand,
 \begin{equation*}
  K_{\theta} = \int_0^\infty \Pr\{\ell^0(S_{\theta}) > s\}ds   =   \int_0^\infty \Pr\{S_{\theta} > \tau^{0}(s)\}ds   =   \int_0^\infty \EV \left[e^{-\theta \tau^0(s)}\right]ds = \frac{1}{b}\theta^{-1/3}.
 \end{equation*}
 On the other hand, by the Strong Markov Property of $\bX$ at the hitting time $T^y$, spatial homogeneity, and Proposition \ref{prop:hitting_time:subord},
 \begin{equation*}
  \EV[\ell^y(S_{\theta})] = \Pr\{T^y < S_{\theta}\}\EV[\ell^0(S_{\theta})] =e^{y\psi^{-1}(\theta)} K_{\theta} \quad \text{for }y \leq 0.
 \end{equation*}
 By Fubini's Theorem and the local time identity \eqref{eq:LT_int_ident},
 \begin{equation*}
 \begin{split}
  \frac{2}{3} &= \Pr\{\bX(S_{\theta}) \leq 0\} = \EV\left[ \int_0^\infty\theta e^{-\theta s}\cf\{\bX(s)\le 0\} ds \right]
  	= \EV\left[ \int_0^\infty\theta^2e^{-\theta t}\int_0^t\cf\{\bX(s)\le 0\}dsdt \right]\\
 	&= \EV\left[ \int_0^\infty\theta^2e^{-\theta t}\int_{-\infty}^0\ell^y(t)dydt \right]
 	= \theta\int_{-\infty}^0\EV \left[ \ell^y(S_{\theta})\right]dy\\
    &= \theta K_\theta\int_{-\infty}^0 e^{y\psi^{-1}(\theta)}dy = \theta K_\theta\left(\frac{\pi}{2}\right)^{-1/3}\theta^{-2/3}.
  \end{split}
  \end{equation*}
  Substituting in for $K_{\theta}$, we get $2/3 = (1/b)\left(\pi/2\right)^{-1/3}$; isolating $b$ gives the desired value.
\end{proof}

Note that $\Pr(\bX(t)\le 0) = 2/3$ while, by Lemma \ref{lem:stable:invariance}, the It\^o measure $\mSxc$ is invariant under increment reversal ($180^\circ$ rotation around the unique jump across 0). This means that typically, the process has spent half its time positive up to the last zero but is likely to be found in the first half of a much longer excursion. 
%
%
We now derive the results in Proposition \ref{prop:clade:stats2}.

\begin{proof}[Proof of Proposition \ref{prop:clade:stats2}.]
 \ref{item:CS2:len}. It follows from the scaling property of $\mClade$, noted in Lemma \ref{lem:clade:invariance}, that $\mClade(\len > x) = Cx^{-1/3}$, for some constant $C$. As noted in Proposition \ref{prop:inv_LT:subord}, the inverse local time process $(\tau^0(s))$ is a subordinator. Its L\'evy measure $\Pi$ equals $\mClade\{\len\in\cdot\,\}$. Then, recalling the identity $\Phi(\theta) = \int_0^\infty (1-e^{-\theta x})d\Pi(x)$, which may be read from \cite[Chapter 3]{BertoinLevy}, we obtain \ref{item:CS2:len} by solving for $C$ in $3(\theta/4\pi)^{1/3} = \int_0^\infty (1-e^{\theta x})(C/3)x^{-4/3}dx$.
 
 
 \ref{item:CS2:crossing}. The length of a bi-clade $N$ equals the time until the first crossing of zero, plus the subsequent time until the $\xi(N)$ hits zero. Suppose $N\sim\mClade(\,\cdot\;|\;J^+ = b)$. Then $\xi(N^+)$ is a \Stable[\frac32] first-passage path from $b$ down to zero independent of $\xi(N^-)$, by the strong Markov property under $\mSxc$ at the crossing time $T_0^+$. Thus, by Lemma \ref{lem:clade:invariance}, if $N\sim\mClade(\,\cdot\;|\;J^- = a)$ then $\xi(N^-)$ is the increment reversal of a \Stable[\frac32] first-passage path from $a$ down to zero, again independent of $\xi(N^+)$. 
Appealing to the subordinator property noted in Proposition \ref{prop:inv_LT:subord}, under $\mClade(\,\cdot\;|\;(J^-,J^+) = (a,b))$, the length $\len$ is distributed as the hitting time $T^{-a-b}$. Thus, $\mClade\{\len\in\cdot\;|\;J = y\}$ equals the law of $T^{-y}$. 
 
 
 It follows from the right-continuity of $\bX$ that $\mClade\{ J > y\}$ is finite for all $y>0$. By the scaling property \eqref{eq:stable:scaling_inv}, this equals $Cy^{-1/2}$ for some constant $C$. It remains to determine the value of $C$. By Proposition \ref{prop:inv_LT:subord}, our argument for \ref{item:CS2:len} above, and Proposition \ref{prop:hitting_time:subord},
 \begin{equation*}
  \frac{3}{2^{2/3}\pi^{1/3}}\theta^{1/3} = \mClade\left[1-e^{-\theta\len(N)}\right] = \int_0^\infty\left(1-\EV\left[e^{-\theta T^{-y}}\right]\right)\frac12Cy^{-3/2}dy = \frac{\pi^{2/3}}{2^{1/6}}C\theta^{1/3}.
 \end{equation*}
 Solving for $C$ gives the desired result.
 
 \ref{item:CS2:over:mass}. Let $N$ have law $\mClade(\,\cdot\;|\;m^0 = a)$. Let $\hat f$ denote the leftmost spindle in $N^+$, i.e.\ the top part of the middle spindle of $N$. By Proposition \ref{prop:clade_splitting}, $\hat f$ is a \BESQ[-1] started from $y$ and killed at zero. Then $J^+(N) = \life(\hat f)$; the law of the latter is specified in Lemma \ref{lem:BESQ:length}, which quotes \cite{GoinYor03}. In particular, this has distribution \InvGammaDist[3/2,a/2].
 
 \ref{item:CS2:mass}. We know this formula up to a constant from Lemma \ref{lem:clade:invariance} and the $m^0$ entry in Table \ref{tbl:clade_scaling} on page \pageref{tbl:clade_scaling}. To obtain the constant, we appeal to \ref{item:CS2:crossing} and \ref{item:CS2:over:mass}. In particular, it follows from Proposition \ref{prop:clade_splitting} that for $N$ with law $\mClade(\,\cdot\;|\;m^0=a)$, the over- and undershoot are i.i.d.\ with law \InvGammaDist[\frac32,\frac{a}{2}], as in \ref{item:CS2:over:mass} above. This gives
 \begin{equation*}
  \mClade\{ J\in dy\;|\;m^0 = a\} = \int_0^y \frac{a^3}{2\pi}\frac{1}{(zy-z^2)^{5/2}}\exp\left( -\frac{a}{2z}-\frac{a}{2(y-z)} \right)dz.
 \end{equation*}
 Integrating this against the law $\mClade\{m^0\in da\} = Ca^{-3/2}da$, we get
 \begin{equation*}
 \begin{split}
  \mClade\{ J\in dy\} &= dy\int_0^\infty Ca^{-3/2}\int_0^y \frac{a^3}{2\pi}(zy-z^2)^{-5/2}\exp\left( -\frac{a}{2z}-\frac{a}{2(y-z)} \right)dzda\\
  	&= \frac{Cdy}{2\pi} \int_0^y (zy-z^2)^{-5/2} \Gamma\left(\frac52\right)\left(\frac{y}{2(zy-z^2)}\right)^{-5/2}dz = \frac{3C}{\sqrt{2\pi}}y^{-3/2}dy.
 \end{split}
 \end{equation*}
 Setting this equal to \ref{item:CS2:crossing} gives $C = 1/(2\sqrt{\pi})$, as desired.
 
 \ref{item:CS2:over} and \ref{item:CS2:mass:over}. The former arises from integrating the product of formula \ref{item:CS2:over:mass} with the derivative of \ref{item:CS2:mass}. The latter is then computed by Bayes' rule.

 
 \ref{item:CS2:max:over}. 
 By the strong Markov property under $\mSxc$ at the crossing time $T_0^+$, this equals the probability that a \Stable[\frac32] process started from $y$ exits the interval $[0,z]$ out of the lower boundary first. 
 This is a standard calculation via scale functions \cite[Theorem VII.8]{BertoinLevy}, carried out for a spectrally negative \Stable[\alpha] process in 
\cite{Bertoin96}, from which the claimed result can be gotten by a sign change.
 
 \ref{item:CS2:max:mass}. This is computed by integrating the product of formulas \ref{item:CS2:over:mass} and \ref{item:CS2:max:over}, which can be reduced to a Gamma integral.
 
 \ref{item:CS2:max} and \ref{item:CS2:mass:max}. The former is computed by integrating the product of the derivative of formula \ref{item:CS2:mass} with \ref{item:CS2:max:mass}. The latter follows via Bayes' Rule.
\end{proof}

The remaining results in this section go towards proving Lemma \ref{lem:LMB} and thereby completing the proof of Proposition \ref{prop:type-1:transn}.
\begin{figure}
 \centering
 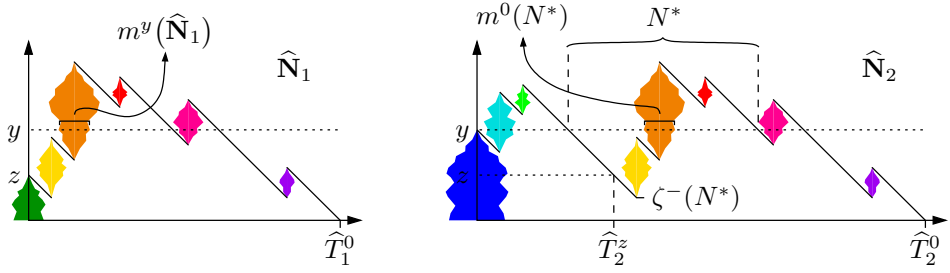
 \caption{Illustration of the coupling in the proof of Lemma \ref{lem:LMB_reversal}.\label{fig:LMB_reversal}}
\end{figure}
\begin{lemma}\label{lem:LMB_reversal}
 Recall \eqref{eq:LMB_def} defining $m^y(N)$ as the mass of the leftmost spindle at level $y$.
 \begin{equation}
  \mClade^+\{m^y\in db\;|\;J^+ = z,\,\life^+ > y\} = \frac{b^{-3/2}}{\sqrt{2\pi}} \left(\frac{e^{-b/2y} - e^{-b/2(y-z)}}{(y-z)^{-1/2} - y^{-1/2}}\right) db \qquad \text{for }0<z<y.\label{eq:LMB_reversal}
 \end{equation}
\end{lemma}

\begin{proof}
 We prove this by showing that $\mClade^+\{m^y\in db\;|\;J^+ = z,\,\life^+ > y\}$ equals
 \begin{equation*}
 \begin{split}
  &\mClade^+(m^0\in db\;|\;\life^+ \in (y-z,y))\\
  &=\frac{\mClade^+\{m^0\in db\;|\;\life^+ > y-z\}\mClade^+\{\life^+>y-z\} - \mClade^+\{m^0\in db\;|\;\life^+ > y\}\mClade^+\{\life^+>y\}}{\mClade^+\{\life^+>y-z\}-\mClade^+\{\life^+>y\}}.
 \end{split}
 \end{equation*}
 The latter equals the right hand side of \eqref{eq:LMB_reversal} by Proposition \ref{prop:clade:stats2} \ref{item:CS2:max} and \ref{item:CS2:mass:max}. We prove this by a coupling construction, illustrated in Figure \ref{fig:LMB_reversal}.
 
 Fix $y>z>0$. Let $\whN_1\sim\mClade^+(\,\cdot\;|\;J^+ = z)$. As in Corollary \ref{cor:clade_law_given_over}, this may be expressed as $\whN_1 = \Dirac{0,\bff_1} + \Restrict{\bN_1}{[0,T^{-z}_1]}$, where $\zeta(\bff_1) = z$, $\bN_1$ is a \PRM[\Leb\otimes\mBxc], and $T^{-z}_1 = \wh T^0_1$ is the hitting time of $-z$ by $\bX_1 := \xi(\bN_1)$, or that of 0 by $\whX_1 := \xi(\whN_1)$. Correspondingly, let $\whN_2 = \Dirac{0,\bff_2} + \Restrict{\bN_2}{[0,T^{-y}_2]}$ have distribution $\mClade^+(\,\cdot\;|\;J^+ = y)$. Let $\wh T^z_2$ denote the time at which $\whX_2 := \xi(\whN_2)$ first hits $z$ and $T^{-y}_2 = \wh T^0_2$ the time at which it hits zero. Then
 \begin{equation}\label{eq:LMB_reversal_dist_eq}
  \ShiftRestrict{\whN_2}{(\wh T^z_2,\wh T^0_2)} \stackrel{d}{=} \Restrict{\whN_1}{(0,\wh T^0_1)} \qquad \text{and} \qquad \ShiftRestrict{\whX_2}{[\wh T^z_2,\wh T^0_2]} \stackrel{d}{=} \whX_1,
 \end{equation}
 which is a \Stable[\frac32] first passage from $z$ down to zero.
 
 The time $\wh T^z_2 = T^{z-y}_2$ occurs during the first bi-clade $N^*$ of $\bN_2$ about level $0$ that has $\life^-(N^*) \geq y-z$. Now, consider the event $A_2$ that $\whX_2$ returns up to level $y$ during the time interval $[\wh T^z_2,\wh T^0_2]$. Then $A_2 = \{\life^-(N^*) < y\}$. Thus, conditionally given $A_2$, the mass $m^0(N^*)$ has law $\mClade^+\{m^0\in \cdot\;|\;\life^- \in (y-z,y)\}$. 
 This is equal, via the time-reversal invariance noted in Lemma \ref{lem:clade:invariance}, to $\mClade^+\{m^0\in \cdot\;|\;\life^+ \in (y-z,y)\}$.
 
 The quantity $m^0(N^*)$ and the event $A_2$ correspond, via \eqref{eq:LMB_reversal_dist_eq}, to the quantity $m^y(\whN_1)$ and the event $A_1$ that $\whX_1$ reaches level $y$ before reaching zero. Conditionally given $A_1$, the mass $m^y(\whN_1)$ has law $\mClade^+\{m^y\in \cdot\;|\;J^+ = z,\,\life^+ > y\}$. Thus, the two laws are equal, as desired.
\end{proof}

\begin{proof}[Proof of Lemma \ref{lem:LMB}]
 Fix $y>0$. We decompose the event $\{\life^+ > y\}$ into two components, based on whether $J^+ > y$:
 \begin{equation}
  \mClade^+\{m^y \in db\;|\;m^0 = a,\,\life^+ > y\}
  	= \dfrac{ \left[ \begin{array}{l}
  			\mClade\{m^y \in db,\,y \in [J^+,\life^+)\;|\;m^0 = a\}\\
  			\ +\ \mClade\{m^y \in db,\,y < J^+\;|\;m^0 = a\}
  		\end{array} \right] }
  		{ \mClade\{\life^+ > y\;|\;m^0 = a\} }.\label{eq:LMB_2_cases}
 \end{equation}
 The second summand in the above numerator describes the case in which the initial leftmost spindle of the clade survives to level $y$. Thus, this summand equals the density of the time $y$ distribution of a \BESQ[-1] started from $a$. We denote this by $q^{-1}_y(a,b)db$. From \cite[Proposition 3; Equation (49)]{GoinYor03},
 \begin{equation*}
  q^{-1}_y(a,b) = q^5_y(b,a) = \frac{1}{2y}\left(\frac ab\right)^{3/4}e^{-(a+b)/2y}I_{3/2}\left(\frac{\sqrt{ab}}{y}\right) \qquad \text{for }b>0.
 \end{equation*}
 Substituting in for the modified Bessel function $I_{3/2}(x)$ from \cite[equation 10.2.13]{AbraSteg},
 \begin{equation}\label{eq:BESQ_transn_2}
  \mClade(m^y \in db,\,y < J^+\;|\;m^0 = a) = \frac{\sqrt{y}}{\sqrt{2\pi}}\frac{1}{b^{3/2}}e^{-(a+b)/2y}(w\cosh(w) - \sinh(w))db,
 \end{equation}
 where $w:= \sqrt{ab}/y$.
 
 It remains to evaluate the first summand in the numerator in \eqref{eq:LMB_2_cases}. Via Corollary \ref{cor:clade_law_given_over}, under the law $\mClade^+(\,\cdot\;|\;J^+=z)$, the variables $\life^+$ and $m^y$ for $y>z$ are independent of $m^0$. Thus,
 \begin{equation*}
 \begin{split}
  &\mClade^+(m^y \in db,\,y \in [J^+,\life^+)\;|\;m^0 = a)\\
  	&\quad = \int_{z=0}^y \mClade(m^y\in db\;|\;J^+ = z,\,\life^+ > y) \mClade(\life^+ > y\;|\;J^+ = z) \mClade(J^+\in dz\;|\;m^0 = a).\notag
 \end{split}
 \end{equation*}
 We have formulas for these three conditional laws in Lemma \eqref{lem:LMB_reversal} and Proposition \ref{prop:clade:stats2} \ref{item:CS2:over:mass} and  \ref{item:CS2:max:over}. Plugging in, the above expression equals
 $$\left[\int_{z=0}^y \frac{b^{-3/2}}{\sqrt{2\pi}} \left(\frac{e^{-b/2y} - e^{-b/2(y-z)}}{(y-z)^{-1/2} - y^{-1/2}}\right) \left( 1 - \sqrt{\frac{y-z}{y}}\right) \frac{a^{3/2}}{\sqrt{2\pi}z^{5/2}}e^{-a/2z}dz\right] db.$$
 Set $u= z/y$ and then $v = (1-u)/u$. Note that $1/(1-u) = 1+(1/v)$. Our integral becomes
 \begin{equation*}
  \frac{1}{2\pi y} \left(\frac{a}{b}\right)^{3/2}\exp\left(-\frac{a+b}{2y}\right)db \int_{0}^{\infty} \left(1 - \exp\left(-\frac{b}{2y}\frac{1}{v}\right)\right)\exp\left(-\frac{a}{2y}v\right)\sqrt{v}dv.
 \end{equation*}
 We distribute the difference and compute the two resulting integrals separately:
 \begin{equation*}
  \int_{0}^{\infty} \exp\left(-\frac{a}{2y}v\right)\sqrt{v}dv = \Gamma(3/2)\left( \frac{2y}{a} \right) ^{3/2} = \sqrt{2\pi}\left( \frac{y}{a} \right)^{3/2}
 \end{equation*}
 and, via \cite[Example 33.15]{Sato}, 
 \begin{equation*}
  \int_{0}^{\infty} \exp\left(-\frac{b}{2y}\frac{1}{v} - \frac{a}{2y}v \right)\sqrt{v}dv = \sqrt{2\pi}\left( \frac{y}{a} \right)^{3/2}\left( \frac{\sqrt{ab}}{y} + 1\right) \exp\left(-\frac{\sqrt{ab}}{y}\right).
 \end{equation*}
 As in \eqref{eq:BESQ_transn_2}, let $w:=\sqrt{ab}/y$. Subtracting the second component from the first and multiplying in all constants,
 \begin{equation}\label{eq:mass_given_clade_but_not_spindle}
  \mClade(m^y \in db,\,y \in [J^+,\life^+)\;|\;m^0 = a) = \frac{\sqrt{y}}{\sqrt{2\pi}} \frac{1}{b^{3/2}} e^{-(a+b)/2y} ( 1 - (w + 1)e^{-w})db.
 \end{equation}
 

 Via Proposition \ref{prop:clade:stats2} \ref{item:CS2:max:mass}, the denominator in \eqref{eq:LMB_2_cases} is $1-e^{-a/2y}$. Adding \eqref{eq:BESQ_transn_2} to \eqref{eq:mass_given_clade_but_not_spindle} and dividing by $1-e^{-a/2y}$, the expression in \eqref{eq:LMB_2_cases} equals
 $$\frac{\sqrt{y}}{\sqrt{2\pi}} \frac{1}{b^{3/2}} \frac{e^{-(a+b)/2y}}{1-e^{-a/2y}}(1 - \cosh(w) + w\sinh(w)),$$
 since $(1 - (w+1)e^{-w}) + (w\cosh(w) - \sinh(w)) = 1 - \cosh(w) + w\sinh(w)$.
\end{proof}

This lemma completes the proof of Proposition \ref{prop:type-1:transn}.

\bibliographystyle{plain}
\bibliography{AldousDiffusion}
\end{document}